\documentclass[12pt]{article}
\usepackage{amssymb, amsmath, amsthm}
\usepackage{mathrsfs}

\newtheorem{remark}{Remark}[section]
\newtheorem{corollary}{Corollary}[section]
\newtheorem{lemma}{Lemma}[section]
\newtheorem{definition}{Definition}[section]
\newtheorem{proposition}{Proposition}[section]
\newtheorem{theorem}{Theorem}[section]
\newtheorem{example}{Example}[section]

\begin{document}

\title{Measures on General Codimensional Surfaces in Infinite Dimensions and Stokes-Type Theorems}
\author{Zhouzhe Wang\footnote{School of Mathematics,
Sichuan University, Chengdu 610064, Sichuan Province, China. E-mail
address: wangzhouzhe@stu.scu.edu.cn.}, Jiayang Yu\footnote{School of Mathematics,
Sichuan University, Chengdu 610064, Sichuan Province, China. E-mail
address: jiayangyu@scu.edu.cn.} and Xu Zhang\footnote{School of Mathematics,
Sichuan University, Chengdu 610064, Sichuan Province, China. E-mail
address: zhang\_xu@scu.edu.cn.}}
\date{}
\maketitle
\def\cc{\mathbb{C}}
\def\zz{\mathbb{Z}}
\def\nn{\mathbb{N}}
\def\rr{\mathbb{R}}
\def\qq{\mathbb{Q}}
\def\dd{\mathbb{D}}
\def\tt{\mathbb{T}}
\def\bb{\mathbb{B}}
\def\ff{\mathbb{F}}
\def\ll{\mathbb{L}}

\def\A{\mathcal{A}}
\def\B{\mathcal{B}}
\def\D{\mathcal{D}}
\def\L{\mathcal{L}}
\def\M{\mathcal{M}}

\def\N{\mathcal{N}}
\def\K{\mathcal{K}}
\def\E{\mathcal{E}}
\def\F{\mathcal{F}}
\def\G{\mathcal{G}}
\def\R{\mathcal{R}}
\def\s{\mathcal{S}}
\def\p{\mathcal{P}}
\def\P{\mathcal{P}}
\def\T{\mathcal{T}}
\def\O{\mathcal{O}}
\def\Z{\mathcal{Z}}
\def\C{\mathcal{C}}
\def\SS{(S, \mathcal{S})}

\def\al{\alpha}
\def\la{\lambda}
\def\ep{\epsilon}
\def\sig{\sigma}
\def\Sig{\Sigma}
\def\cd{\mathbb{C}^d}
\def\bm{\mathcal{M}}
\def\bn{\mathcal{N}}
\def\bc{\mathcal{C}}
\def\hN{H^2\otimes \mathbb{C}^N}
\def\ba{\mathcal{A}}
\def\supp{\hbox{$\,$\rm supp$\,$}}
\def\card{\hbox{$\,$\rm card$\,$}}
\def\det{\hbox{$\,$\rm det$\,$}}
\def\span{\hbox{$\,$\rm span$\,$}}

\def\bigpa#1{\biggl( #1 \biggr)}
\def\bigbracket#1{\biggl[ #1 \biggr]}
\def\bigbrace#1{\biggl\lbrace #1 \biggr\rbrace}

\def\papa#1#2{\frac{\partial #1}{\partial #2}}
\def\dbar{\bar{\partial}}

\def\oneover#1{\frac{1}{#1}}

\def\meihua{\bigskip \noindent $\clubsuit \ $}
\def\blue#1{\textcolor[rgb]{0.00,0.00,1.00}{#1}}
\def\red#1{\textcolor[rgb]{1.00,0.00,0.00}{#1}}
\def\xing{\heartsuit}
\def\tao{\spadesuit}
\def\lingxing{\blacklozenge}

\def\norm#1{||#1||}
\def\inner#1#2{\langle #1, \ #2 \rangle}

\def\divide{\bigskip \hrule \bigskip}

\def\bigno{\bigskip \noindent}
\def\medno{\medskip \noindent}
\def\smallno{\smallskip \noindent}
\def\bignobf#1{\bigskip \noindent \textbf{#1}}
\def\mednobf#1{\medskip \noindent \textbf{#1}}
\def\smallnobf#1{\smallskip \noindent \textbf{#1}}
\def\nobf#1{\noindent \textbf{#1}}
\def\nobfblue#1{\noindent \textbf{\textcolor[rgb]{0.00,0.00,1.00}{#1}}}
\def\purple#1{\textcolor[rgb]{1.00,0.00,0.50}{#1}}
\def\green#1{\textcolor[rgb]{0.00,1.00,0.00}{#1}}

\def\vector#1#2{\begin{pmatrix}  #1  \\  #2 \end{pmatrix}}

\begin{abstract}
In this paper, we give an explicit construction of surface
measures on a class of surfaces with arbitrary, possibly infinite,
codimension in \(\ell^2\). These measures are constructed from local representations associated with a fixed Gaussian product measure. We then
establish local and global Gauss--Green-type formulas, introduce
a notion of top-degree differential form, and
derive an associated Stokes-type identity. We also determine the orientation of the boundary induced by the orientation of the surface. Moreover, therelationship between \(\mathcal F\)-continuity and Borel measurability is examined,which reveals a phenomenon specific to the infinite-dimensionalsetting.
\end{abstract}

\tableofcontents

\section{Introduction}

In calculus, the integration-by-parts formula is of fundamental importance. Many important results in finite-dimensional analysis are based on this formula and its variants, including the Gauss--Green formula, the divergence theorem, and Stokes' formula. It is therefore natural to seek infinite-dimensional analogues of these fundamental identities. Such extensions, however, are far from straightforward.

One of the principal motivations for the development of infinite-dimensional mathematics comes from the mathematical foundations of quantum physics and, more broadly, from the analysis of problems involving integration over infinite-dimensional spaces. A prominent example is the longstanding problem of giving a rigorous mathematical formulation of the Feynman integral (see, e.g., \cite{GY}). The significance of infinite-dimensional mathematics was emphasized by M. Atiyah, who remarked in \cite[p.~14]{Atiyah}:
``I have said the 21st century might be the era of quantum mathematics or, if you like, of infinite-dimensional mathematics. What could this mean? Quantum mathematics could mean, if we get that far, understanding properly the analysis, geometry, topology, algebra of various non-linear function spaces and by ``understanding properly" I mean understanding it in such a way as to get quite rigorous proofs of all the beautiful things the physicists have been speculating about".

The classical integration-by-parts formula, the Gauss--Green formula, the divergence theorem, and Stokes' formula are all formulated with respect to Lebesgue measure on finite-dimensional Euclidean spaces. In infinite dimensions, however, there is no direct analogue of Lebesgue measure. Indeed, every nonempty open subset of an infinite-dimensional normed space contains countably many pairwise disjoint nonempty open balls of the same radius, a geometric feature with no finite-dimensional counterpart. Consequently, an infinite-dimensional separable normed space admits no nontrivial locally finite translation-invariant Borel measure.

The above mentioned fundamental obstruction makes the construction of meaningful measures in infinite-dimensional spaces substantially more delicate. In particular, any attempt to develop analogues of the Gauss--Green and Stokes formulas must begin with a suitable replacement for the ambient Lebesgue measure and, more importantly, with an appropriate notion of surface measure on infinite-dimensional surface. This naturally leads to the question of how such surface measures can be constructed and how the corresponding integration-by-parts identities should be formulated in the absence of a canonical translation-invariant reference measure.

A fundamental example of a measure on an infinite-dimensional function space is the Wiener measure. Wiener \cite{Wie} established its existence on \(C_0[0,1]\), the Banach space of real-valued continuous functions on \([0,1]\) vanishing at \(0\), endowed with the supremum norm. Building on Wiener's work, Cameron \cite{Cameron} derived an integration-by-parts formula for the Wiener measure on \(C_0[0,1]\), while Donsker \cite{Don}, by different methods, obtained related formulas in a more general setting.

A major subsequent development was Gross's theory of abstract Wiener spaces, which provided a general framework for Gaussian measures on infinite-dimensional Banach spaces. In this setting, Kuo \cite{Kuo1} extended the integration-by-parts formulas of Cameron and Donsker to abstract Wiener measures. Integration-by-parts identities on path spaces have also played an important role in the study of stochastic partial differential equations. In particular, Zambotti established such a formula on convex subsets of Wiener space and used it in the analysis of certain stochastic partial differential equations. Related results were subsequently developed by Funaki--Ishitani \cite{FI}, Hariya \cite{Har}, Otobe \cite{Oto}, and Bonaccorsi--Da Prato--Tubaro \cite{BDT}. In a different direction, motivated by an infinite-dimensional Hodge--Kodaira vanishing theorem, Shigekawa \cite{Shi03} also established an integration-by-parts formula in an infinite-dimensional setting.

These results demonstrate that integration by parts can be meaningfully formulated for a variety of infinite-dimensional Gaussian measures. They do not, however, by themselves provide a general theory of surface measure on infinite-dimensional surfaces. In particular, the corresponding formulas are not formulated as Gauss--Green or Stokes identities involving an independently constructed measure on the boundary or on surfaces of positive, possibly infinite, codimension. Thus, in order to develop genuine infinite-dimensional analogues of the Gauss--Green and Stokes formulas, one is naturally led to the problem of constructing suitable surface measures and identifying the geometric objects, including normal directions and orientations, that enter the associated boundary terms.

The notions of a surface and of the corresponding surface measure are indispensable in the classical Gauss--Green and Stokes formulas. Thus, any attempt to develop genuine infinite-dimensional counterparts of these formulas requires, as a first step, suitable notions of infinite-dimensional surfaces and surface measures. A number of approaches have been developed in this direction.

\begin{itemize}

\item
To the best of our knowledge, an early contribution appears to be due to Skorohod \cite{Sko70, Sko74}, who constructed surface measures in infinite-dimensional Hilbert spaces and established Gauss--Green-type formulas for a class of quasi-invariant measures, not necessarily Gaussian. His work was further developed by Uglanov \cite{Ugl1,Ugl2} and Yakhlakov \cite{Yak}. In particular, Uglanov \cite{Ugl1} constructed surface measures on codimension-one surfaces in Banach spaces, while Yakhlakov extended this construction to surfaces of arbitrary finite codimension. These approaches proceed by constructing surface measures locally on sufficiently small neighborhoods and then attempting to assemble the local objects into a global measure. A difficulty inherent in this strategy is that the existence of local surface measures does not, by itself, guarantee the existence of a globally defined surface measure; see, for instance, \cite[Introduction, p.~467]{CCKO}.

\item
Another early systematic treatment was developed by Goodman \cite{Goo} in the framework of Gross's abstract Wiener spaces (\cite{Gro65}). Goodman's construction is, in spirit, close to the classical finite-dimensional theory. He introduced \(H\)-\(C^1\) surfaces and their local surface measures, established the compatibility of these local measures, and thereby obtained a global surface measure on an \(H\)-\(C^1\) surface. He further developed an \(H\)-\(C^1\) partition-of-unity argument and proved a local divergence theorem, which was then globalized by means of a suitable exhaustion and approximation procedure.

Although conceptually close to the classical theory, Goodman's construction involves a substantial amount of abstract Wiener-space machinery and is not particularly explicit. As observed by Chaari--Cipriano--Kuo--Ouerdiane \cite[Introduction, p.~467]{CCKO}, the construction is rather complicated and difficult to compute with. In particular, some of the approximation arguments rely on properties of measurable seminorms and related tools from the theory of abstract Wiener spaces, and several steps in the construction are presented only briefly. We also note that the proof of the existence of an \(H\)-\(C^1\) partition of unity given in \cite[Lemma~2.1, p.~419]{Goo} contains a gap\footnote{This issue can be circumvented by a different argument, and we shall provide such a proof in this work.}.

\item
A different and highly influential approach was introduced by Airault and Malliavin \cite{AM}. They considered level sets of functions
\[
G\in \bigcap_{k\in\mathbb N, p\geqslant 1} W^{k,p}(C[0,1],\mu)
\]
satisfying an appropriate nondegeneracy condition, and constructed surface measures on finite-codimensional level sets by means of duality and image measures. In contrast with approaches based purely on local constructions, their method directly produces a measure on the entire surface. They also established an infinite-dimensional coarea formula and defined the integration of differential forms of degree \(n\) over surfaces of codimension \(n\), leading to a Stokes-type theorem. This approach was subsequently refined by Bogachev and Malofeev \cite{Bog90,BM}, extended to certain non-Gaussian settings by Pugach\"{e}v \cite{Pug99}, and further simplified by Da Prato--Lunardi--Tubaro \cite{DLT14,DLT18}. Related treatments can also be found in \cite{Bog98,Mal}, as well as in \cite{CCKO} for the space of tempered distributions over \(\mathbb R\).

\item
A substantially different construction was proposed by Feyel and de La Pradelle \cite{FL}, who introduced a Hausdorff--Gauss measure through approximation by finite-dimensional Hausdorff--Gauss measures. This construction exploits in an essential way the structure of Gaussian measures and therefore does not immediately extend to general non-Gaussian measures. The Feyel--de La Pradelle measure has played an important role in a number of subsequent works; see, for example, Celada--Lunardi \cite{CL} and Hino \cite{Hin}.

\item
Surface measures in infinite dimensions also arise naturally from geometric measure theory and the theory of functions of bounded variation. BV functions with respect to Gaussian measures on Banach spaces have been studied, for instance, in \cite{AMMP,Fuk,FH}. A Borel set \(E\) is said to have finite perimeter if its characteristic function is a BV function, in which case one obtains an associated perimeter measure concentrated, in an appropriate measure-theoretic sense, on the boundary of \(E\). Under suitable regularity assumptions, this perimeter measure agrees with the restriction to the boundary of the Feyel--de La Pradelle surface measure. Building on \cite{FH}, Hino \cite{Hin} introduced a measure-theoretic boundary for Borel sets in abstract Wiener spaces and established a Gauss--Green formula for sets of finite perimeter.

\item
Yet another approach is based on defining an integral associated with a given measure and a vector field. Bogdanskii \cite{Bogd13} first treated hypersurfaces in the setting of Banach manifolds. Bogdanskii and Moravetskaya \cite{BM18} subsequently extended the construction to surfaces of arbitrary finite codimension, and Bogdanskii \cite{Bogd21} obtained a divergence form of the Stokes formula for finite-codimensional surfaces embedded in Banach manifolds.

\end{itemize}

These approaches reveal several different mechanisms by which surface measures and boundary terms can arise in infinite-dimensional analysis. Their precise relationships, however, are subtle. As Hino observed in \cite[p.~1659]{Hin}, ``Although these apparently different expressions should be closely related to one another, it does not seem evident to derive one formula from another one directly. It would be an interesting problem to clarify such an involved situation.'' We refer to \cite{Bog17} for a survey of related developments.

The preceding results also leave open a number of structural questions that are particularly relevant to the present work. Most existing constructions are designed for hypersurfaces or, more generally, for surfaces of finite codimension, and many of them rely essentially on the Gaussian or abstract Wiener-space structure. Moreover, the resulting surface measures are often obtained through indirect procedures, such as duality, finite-dimensional approximation, or measure-theoretic boundary constructions. For the purposes of developing Gauss--Green and Stokes formulas on surfaces of arbitrary, possibly infinite, codimension, it is desirable to have a more direct construction that is both local and explicit, while at the same time admitting a consistent globalization.


In the present paper, we build on several ideas from Goodman's work \cite{Goo}, but work in a setting in which the relevant constructions and estimates can be carried out more explicitly than in the abstract Wiener-space framework. Whereas the existing theories discussed above are exclusively concerned with surfaces of finite codimension, we consider a class of surfaces in \(\ell^2\) of arbitrary, possibly infinite, codimension. Our first objective is to construct surface measures on such surfaces. We begin with a family of local surface measures, prove their compatibility on overlapping coordinate neighborhoods, and then patch them together to obtain a globally defined surface measure.

With these surface measures at hand, we establish both local and global Gauss--Green-type formulas. An important ingredient in the globalization argument is a simple construction of compact subsets adapted to the underlying Gaussian product measure. In contrast with Goodman's construction, which relies on measurable seminorms in abstract Wiener spaces, our argument uses only elementary tools from real analysis and is particularly convenient for explicit estimates. Even in the codimension-one setting treated by Goodman \cite{Goo}, our Gauss--Green theorem requires substantially weaker regularity assumptions than those imposed in his divergence theorem. In our formulation, we assume only that the function and its first-order partial derivatives are \(\mathcal F\)-continuous and Borel measurable.

The measurability assumption deserves separate attention. In finite-dimensional analysis, the continuity assumptions entering such formulas automatically imply Borel measurability. This implication becomes more delicate in the present infinite-dimensional setting. In Corollary~\ref{230720cor1}, we show that \(\mathcal G\)-continuity does not, in general, imply Borel measurability. It remains unclear whether Borel measurability follows from the remaining regularity assumptions of our main theorem. We therefore impose it as a separate hypothesis. This provides another indication that seemingly routine relations between topology and measurability may behave differently in infinite-dimensional spaces.

The question of orientation is one of the main geometric issues in passing from local constructions to global identities. In the classical theory, orientations determine the signs of boundary terms in the Gauss--Green and Stokes formulas. They also play a fundamental role in infinite-dimensional Fredholm theories, notably in the polyfold Fredholm theory of Hofer--Wysocki--Zehnder \cite{HWZ}. Thus, in an infinite-dimensional Gauss--Green or Stokes theory, it is necessary not only to construct surface measures, but also to identify the boundary orientation induced by the orientation of the underlying surface.

In finite codimension, this orientation is usually described through an ordered defining map. More precisely, if a surface is defined by \(n\) real-valued functions \(f_1,\ldots,f_n\), then the order of these functions determines the associated normal orientation; see, for instance, \cite{AM,FL}. Such a description is intrinsically finite-dimensional in the normal direction and breaks down for surfaces of infinite codimension. One of the advantages of the present approach is that it gives a direct definition of the boundary orientation induced by the orientation of the underlying surface of arbitrary, possibly infinite, codimension. This makes it possible to formulate a Stokes-type identity in a form that remains consistent with the finite-dimensional convention while extending beyond the finite-codimensional setting.

The present paper is also motivated by several problems arising in our earlier work on infinite-dimensional analysis. Variants of Gauss--Green-type formulas were used in our study of the infinite-dimensional \(\overline{\partial}\)-equation \cite{WYZ,YZ20}. In those applications, different versions of the formula were required under different regularity assumptions. The results developed here provide a unified framework for these identities and are intended to serve as basic tools for further problems in infinite-dimensional analysis, including those considered in \cite{YZ-b,YZ-c}.

A further objective is to develop an appropriate differential-form formalism and a corresponding Stokes theorem. For every positive integer \(n\), we define the integral of a differential form of codimension \(n\) over a surface of codimension \(n\), in accordance with the finite-dimensional convention; see, for example, \cite[Definition~3.4, p.~91]{CCL}. This differs from the formulation of Airault and Malliavin \cite{AM}, where differential forms of degree \(n\) are integrated over surfaces of codimension \(n\). Within our framework, the codimension-based formulation leads naturally to a notion of top-degree differential form and to a Stokes-type theorem, including the orientation on the boundary induced by the orientation of the surface.

We restrict the present paper to the regularity framework described above. Gauss--Green formulas for Sobolev functions and the corresponding trace theory have been studied in other infinite-dimensional settings; see, for example, \cite{CL}. The analogous questions in the present framework will be addressed separately.

The paper is organized as follows. Section~2 collects the notation and preliminary results. In Section~3, we introduce the class of surfaces in \(\ell^2\) of arbitrary, possibly infinite, codimension. Section~4 is devoted to the construction and globalization of the corresponding surface measures. In Section~5, we establish a local Gauss--Green-type formula, which is globalized in Section~6. Finally, Section~7 develops the differential-form formalism and proves the Stokes-type theorem for the surfaces considered in this paper.

\section{Preliminaries}

\subsection{Some notions and notations}
Denote by $\mathbb{N}$ and $\mathbb{N}_0$ respectively the sets of positive and nonnegative integers, by $\mathbb{Q}$ the set of all rational numbers, and by $\mathbb{R}$ the
$1$-dimensional (real) Euclidean space. Suppose that $\mathbb{X}$ is a nonempty set and $(\mathbb{X},\cal J)$ is a topological space. The smallest $\sigma$-algebra
generated by $\cal J$ is called the
Borel $\sigma$-algebra of
$\mathbb{X}$, denoted by $\mathscr{B}(\mathbb{X})$. For any nonempty open subset $\mathcal{O}$ of $\mathbb{X}$, we denote by $B_{\mathbb{X}}(\mathcal{O})$ and $B_{\mathbb{X}}^b(\mathcal{O})$ respectively the set of all real-valued Borel measurable functions on $\mathcal{O}$ and the set of all bounded, real-valued Borel measurable functions on $\mathcal{O}$. Recall that $\mathbb{X}$ is called a Lindel\"of space if every open cover of $\mathbb{X}$ has a countable subcover (See \cite[p. 50]{Kel}). From \cite[Theorem 15, p. 49]{Kel}, it follows that every separable metric space is a Lindel\"of space.

Suppose that $X$, $Y$, $Z$ and $W$ are four normed spaces such that $Z\subset X$ and $W\subset Y$, $U$ is a nonempty open subset of $X$, and $g$ is a mapping from $U$ into $Y$. Denote by ${\cal L}(X;Y)$ the class of all bounded linear operators from $X$ into $Y$ (which is a normed space with the usual operator norm), and by $C_X(U;Y)$ the set of all continuous mapping from $U$ to $Y$ (We simply denote $C_X(U;\mathbb{R})$ by $C_X(U)$). For $\textbf{x}\in U$, we recall that $g$ is called Fr\'echet differentiable (from $U$ into $Y$) at $\textbf{x}$, if there exists an element in ${\cal L}(X;Y)$, denoted by  $Dg(\textbf{x})$ (and called the Fr\'echet derivative of $g$ at $\textbf{x}$), such that
$$
\lim\limits_{\Delta \textbf{x}\to \textbf{0}}\frac{||g(\textbf{x}+\Delta \textbf{x})-g(\textbf{x})-Dg(\textbf{x})(\Delta \textbf{x})||_Y}{||\Delta \textbf{x}||_X}=0.
$$
Furthermore,  $g$ is called (continuously) Fr\'{e}chet differentiable from $U$ into $Y$ if $g$ is Fr\'{e}chet differentiable from $U$ into $Y$ at every point of $U$ (and the mapping $\textbf{x}\mapsto Dg(\textbf{x})$ from $U$ into ${\cal L}(X;Y)$ is continuous). Denote by $C_X^1(U;Y)$ the class of all continuously Fr\'{e}chet differentiable mappings from $U$ into $Y$. Inductively, people define $C_X^k(U;Y)$ for each $k\in\mathbb{N}$ and $C_X^{\infty}(U;Y)\triangleq \bigcap\limits_{k=1}^{\infty}C_X^k(U;Y)$. For each $r\in \mathbb{N}\cup\{\infty\}$, we simply denote $C_X^r(U;\mathbb{R})$ by $C_X^r(U)$. The following notions are variants of the classic Fr\'{e}chet derivatives, which will be used in the rest of this paper.

\begin{definition}\label{20241115def1}
For an $\mathbf{x}_0\in U$, if there is an open neighborhood $V_{\mathbf{x}_0}$ of $\mathbf{0}$ in $Z$ so that $\mathbf{x}_0+V_{\mathbf{x}_0}\subset U$ and the mapping $f(\cdot)\triangleq g(\mathbf{x}_0+\cdot)$ defined on $V_{\mathbf{x}_0}$
is Fr\'{e}chet differentiable from $V_{\mathbf{x}_0}$ into $Y$ at $\mathbf{0}$, we say that $g$ is Fr\'{e}chet differentiable along the $Z$-direction at $\mathbf{x}_0$, with derivative
\[
D_{Z}g(\mathbf{x}_0)
\triangleq Df(\mathbf{0}).
\]
Furthermore, for $r\in\mathbb{N}\cup\{\infty\}$, we denote by $C_Z^r(U;Y)$ the set of all mappings $g$ for which $g$ is Fr\'{e}chet differentiable along the $Z$-direction at every $\mathbf{x}_0\in U$, and $g(\mathbf{x}_0+\cdot)\in C_Z^r(V_{\mathbf{x}_0};Y)$ for the above $\mathbf{x}_0$ and $V_{\mathbf{x}_0}$. We shall simply denote $C_Z^r(U;\mathbb{R})$ by $C_Z^r(U)$.
\end{definition}

\begin{definition}\label{20250115def1}
Let \(x_0\in U\). We say that \(g\) is \(W\)-Fr\'echet
differentiable at \(x_0\) if there exists an open neighborhood
\(V\subset U\) of \(x_0\) such that
\[
g(x)-g(x_0)\in W,\qquad x\in V,
\]
and the mapping
\[
V\ni x\longmapsto g(x)-g(x_0)\in W
\]
is Fr\'{e}chet differentiable at \(x_0\).
\end{definition}

For any nonempty set $I\subset \mathbb{N}$ and $p\in[1,+\infty)$, write
 $$
 \mathbb{R}^{I}\triangleq \left\{(x_i)_{i\in I}:\;x_i\in \mathbb{R}\hbox{ for each }i\in I\right\},\quad\ell^p(I)\triangleq \left\{(x_i)_{i\in I}\in \mathbb{R}^{I}:\;\sum_{i\in I}|x_i|^p<\infty\right\}.$$
Then, $\ell^p(I)$ is a separable Banach space with the norm:
$$||\textbf{x}||_{\ell^p(I)}\triangleq  \left(\sum_{i\in I}|x_i|^p\right)^{1/p} ,\quad\forall\;\textbf{x}=(x_i)_{i\in I}\in\ell^p(I).
$$
Particularly, $\ell^2(I)$ is a Hilbert space. For $\textbf{x}\in\ell^2(I)$ and $r\in(0,+\infty)$,
put
 $$
 B_r^I(\textbf{x})\triangleq \left\{\textbf{y}\in\ell^2(I):\;||\textbf{y}-\textbf{x}||_{\ell^2(I)}<r \right\}.
 $$
For the special case that $I= \mathbb{N}$, we write respectively $\mathbb{R}^{\infty},\ell^p$ and $B_r(\textbf{x})$  instead
of $\mathbb{R}^{\mathbb{N}},\ell^p(\mathbb{N})$ and $B_r^{\mathbb{N}}(\textbf{x})$. For abbreviation, we write the inner product in $\ell^2$ as $\left\langle \cdot,\cdot\right\rangle$ (and its norm as $\left| \left|\;\cdot\;\right|  \right|$).

\subsection{Gaussian measures on $\ell^2$}

For any given $a>0$, we define a probability measure $\bn_a$ in $(\mathbb{R},\mathscr{B}(\mathbb{R}))$ by
$$\bn_a(B)\triangleq \frac{1}{\sqrt{2\pi a^2}}\int_{B}e^{-\frac{x^2}{2a^2}}\,\mathrm{d}x,\quad\,\forall\;B\in \mathscr{B}(\mathbb{R}),$$
where $\mathrm{d}x$ is the Lebesgue measure in $\mathbb{R}$.

In the rest of this paper, we fix a sequence $\{a_i\}_{i=1}^{\infty}$ of positive numbers such that
\begin{eqnarray}\label{20250130for1}
\sum_{i=1}^{\infty}a_i <+\infty.
\end{eqnarray}
Now we define a product measure on the space $\mathbb{R}^{\infty}$, endowed with the usual product topology $\mathscr{T}$.  By the discussion in \cite[p. 9]{Da}, $\mathscr{B}(\mathbb{R}^{\infty})$ is precisely the product $\sigma$-algebra generated by all sets in the following forms:
\begin{eqnarray*}
 B_1\times B_2\times \cdots \times B_k\times \mathbb{R}^{\mathbb{N}\setminus\{1,2,\cdots,k\}},
\end{eqnarray*}
where $k\in\mathbb{N}$, $B_i\in \mathscr{B}(\mathbb{R}),\,1\leqslant i\leqslant k$. Also, we fix any $s,t\in(0,+\infty)$, and let
$$\bn^t\triangleq\prod\limits_{i=1}^{\infty}\bn_{ta_i}$$
be the product measure on $(\mathbb{R}^{\infty},\mathscr{B}(\mathbb{R}^{\infty}))$. By the same arguments as in \cite[p. 524]{YZ20}, one can obtain the following two elementary results (and therefore we omit their proofs).
\begin{lemma}\label{230407lem1}
$\mathscr{B}(\ell^2)=\{B\cap\ell^2:\; B\in \mathscr{B}(\mathbb{R}^{\infty})\}$.
\end{lemma}

\begin{proposition}\label{230407prop1}
$\bn^t(\ell^2)=1$.
\end{proposition}
Thanks to Lemma \ref{230407lem1} and Proposition \ref{230407prop1}, we obtain a Borel probability measure $P_t$ on $(\ell^2,\mathscr{B}(\ell^2))$ by setting
\begin{eqnarray}\label{220817e1}
P_t(B)\triangleq \bn^t(B),\quad\forall\;B\in \mathscr{B}(\ell^2).
\end{eqnarray}
Since there does not exist a non-trivial translation-invariant measure on $\ell^2$, naturally, we shall consider the dilation and translation properties of $P_t$. The following two results (in Lemma \ref{dilation of Gaussian measure} and Proposition \ref{230421prop1} below) are essentially from \cite[Equation (3), p. 128]{Gro67} and \cite[Proposition 1, p. 127]{Gro67}. Nevertheless, for the reader's convenience, we shall give their proofs.
\begin{lemma}\label{dilation of Gaussian measure}
It holds that
\begin{eqnarray*}
P_{st}(B)=P_t(s^{-1}B),\quad \forall\, B\in \mathscr{B}(\ell^2),
\end{eqnarray*}
where $s^{-1}B\triangleq\{s^{-1}\mathbf{b}:\;\mathbf{b}\in B\}$.
\end{lemma}
\begin{proof}
Let $\mathscr{A}\triangleq \{E\in \mathscr{B}(\ell^2):\;P_{ts}(E)=P_t(s^{-1}E)\}$. By the $\pi$-$\lambda$ theorem, $\mathscr{A}$ is a $\sigma$-algebra. Denote by $\mathscr{C}$ the family of sets in the following form:
\begin{eqnarray*}
\{(x_i)_{i\in\mathbb{N}}\in\ell^2:b_i<x_i<d_i,\,1\leqslant i\leqslant n\},
\end{eqnarray*}
where $n\in \mathbb{N}$ and $b_i,d_i\in\mathbb{R}$ with $b_i<d_i$ for $1\leqslant i\leqslant n$.
For each set $E$ in the above form, it follows that
\begin{eqnarray*}
P_t(s^{-1}E)=\prod_{i=1}^{n}\left(\frac{1}{\sqrt{2\pi t^2a_i^2}}\int_{s^{-1}b_i}^{s^{-1}d_i}e^{-\frac{x_i^2}{2t^2a_i^2}}\,\mathrm{d}x_i\right)
=\prod_{i=1}^{n}\left(\frac{1}{\sqrt{2\pi s^2t^2a_i^2}}\int_{ b_i}^{ d_i}e^{-\frac{y_i^2}{2s^2t^2a_i^2}}\,\mathrm{d}y_i\right)=P_{st}(E).
\end{eqnarray*}
Thus $\mathscr{C}\subset \mathscr{A}$ and $\{\ell^2\cap \left(B_1\times B_2\times \cdots \times B_k\times  \mathbb{R}^{\mathbb{N}\setminus\{1,2,\cdots,k\}}\right):\;B_i\in \mathscr{B}(\mathbb{R}),\,1\leqslant i\leqslant k\}\subset \mathscr{A}$ for each $k\in\mathbb{N}$. Hence $\{B\cap\ell^2:\; B\in \mathscr{B}(\mathbb{R}^{\infty})\}\subset\mathscr{A}$. By Lemma \ref{230407lem1},
$\mathscr{B}(\ell^2)=\{B\cap\ell^2:\; B\in \mathscr{B}(\mathbb{R}^{\infty})\}$, we have $\mathscr{B}(\ell^2)\subset \mathscr{A}$. Therefore, $\mathscr{B}(\ell^2)= \mathscr{A}$ and this completes the proof of Lemma \ref{dilation of Gaussian measure}.
\end{proof}
For each $\textbf{x}=(x_i)_{i\in\mathbb{N}}\in\ell^2$ and $A\subset \mathbb{R}^{\infty}$, define $A+\textbf{x}\triangleq\{\textbf{a}+\textbf{x}:\;\textbf{a}\in A\}$ and
\begin{eqnarray*}
P_t(B,\textbf{x})&\triangleq &P_t(B+\textbf{x}),\quad\forall\, B\in \mathscr{B}(\ell^2), \\
\bn^t(E,\textbf{x})&\triangleq &\bn^t(E+\textbf{x}),\quad\forall\, E\in \mathscr{B}(\mathbb{R}^{\infty}).
\end{eqnarray*}
Then, by \eqref{220817e1}, $P_t(\cdot,\textbf{x})$ ({\it resp. } $\bn^t(\cdot,\textbf{x})$) is a measure on $(\ell^2,\mathscr{B}(\ell^2))$ ({\it resp. } $(\mathbb{R}^{\infty},\mathscr{B}(\mathbb{R}^{\infty}))$), and
\begin{eqnarray}\label{250210e1}
P_t(B,\textbf{x})\triangleq \bn^t(B,\textbf{x}),\quad\forall\;B\in \mathscr{B}(\ell^2).
\end{eqnarray}

\begin{proposition}\label{230421prop1}
Suppose that $\mathbf{x}_1=(x_{i,1})_{i\in\mathbb{N}},\mathbf{x}_2=(x_{i,2})_{i\in\mathbb{N}}\in\ell^2$.
If $t\not=s$ or $\sum\limits_{i=1}^{\infty}\frac{(x_{i,1}-x_{i,2})^2}{ a_i^2} =\infty$, then the measures $P_t(\cdot,\mathbf{x}_1)$ and $P_s(\cdot,\mathbf{x}_2)$ are mutually singular. Otherwise, the measures $P_t(\cdot,\mathbf{x}_1)$ and $P_s(\cdot,\mathbf{x}_2)$ are equivalent, and their
Radon-Nikod\'ym derivative satisfies that
\begin{eqnarray}\label{230702e1}
\frac{\mathrm{d} P_t(\cdot,s_1\mathbf{e}_i)}{\mathrm{d}P_t(\cdot)}(\mathbf{y})=e^{-\frac{2s_1 y_i +s_1^2}{2t^2a_i^2}},\quad\forall\,i\in\mathbb{N},\, s_1\in\mathbb{R},
\end{eqnarray}
and
\begin{eqnarray}\label{230702e2}
\frac{\mathrm{d} P_t(\cdot,s_1\mathbf{e}_1+\cdots+s_n\mathbf{e}_n)}{\mathrm{d}P_t(\cdot)}(\mathbf{y})=\prod_{k=1}^{n}e^{-\frac{2s_k y_k +s_k^2}{2t^2a_k^2}},\quad\forall\,n\in\mathbb{N},\, s_1,\cdots,s_n\in\mathbb{R},
\end{eqnarray}
for a.e. $\mathbf{y}=(y_i)_{i\in\mathbb{N}}\in \ell^2$ (with respect to $P_t$).
\end{proposition}
\begin{proof}
By \cite[Theorem, p. 218]{Kak},
\begin{eqnarray*}
H(\bn^t(\cdot,\textbf{x}_1),\bn^s(\cdot,\textbf{x}_2))=\prod_{i=1}^{\infty}H(\bn_{ta_i}(\cdot,x_{i,1}),\bn_{sa_i}(\cdot,x_{i,2})),
\end{eqnarray*}
where $\bn_a(E,x)\triangleq \bn_a(E+x)$ for all $E\in \mathscr{B}(\mathbb{R})$ and $x\in\mathbb{R}$, and $H(\cdot,\cdot)$ stands for the Hellinger integral of two measures. Note that for each $i\in\mathbb{N}$, we have
\begin{eqnarray*}
&&H(\bn_{ta_i}(\cdot,x_{i,1}),\bn_{sa_i}(\cdot,x_{i,2}))=\int_{\mathbb{R}}\frac{1}{\sqrt{2\pi sta_i^2}}e^{-\frac{(x+x_{i,1})^2}{4t^2a_i^2}-\frac{(x+x_{i,2})^2}{4s^2a_i^2}}\,\mathrm{d}x\\
&&=\int_{\mathbb{R}}\frac{1}{\sqrt{2\pi sta_i^2}}e^{-\frac{x^2}{4t^2a_i^2}-\frac{(x-x_{i,1}+x_{i,2})^2}{4s^2a_i^2}}\,\mathrm{d}x
=\int_{\mathbb{R}}\frac{1}{\sqrt{2\pi sta_i^2}}e^{-\frac{(t^2+s^2)x^2-2t^2(x_{i,1}-x_{i,2})x+t^2(x_{i,1}-x_{i,2})^2}{4t^2s^2a_i^2}}\,\mathrm{d}x\\
&&=\int_{\mathbb{R}}\frac{1}{\sqrt{2\pi sta_i^2}}e^{-\frac{ x^2-\frac{2t^2(x_{i,1}-x_{i,2})x}{t^2+s^2}+\frac{t^2(x_{i,1}-x_{i,2})^2}{t^2+s^2}}{\frac{4t^2s^2a_i^2}{t^2+s^2}}}\,\mathrm{d}x=\frac{\sqrt{2\pi \frac{2t^2s^2a_i^2}{t^2+s^2}}}{\sqrt{2\pi sta_i^2}}\cdot e^{-\frac{ -\left(\frac{t^2(x_{i,1}-x_{i,2})}{t^2+s^2}\right)^2+\frac{t^2(x_{i,1}-x_{i,2})^2}{t^2+s^2}}{\frac{4t^2s^2a_i^2}{t^2+s^2}}}\\
&&=\sqrt{\frac{2ts}{t^2+s^2}}\cdot e^{-\frac{(x_{i,1}-x_{i,2})^2}{4(t^2+s^2)a_i^2} }.
\end{eqnarray*}
If $t\neq s$, then $H(\bn_{ta_i}(\cdot,x_{i,1}),\bn_{sa_i}(\cdot,x_{i,2}))\leqslant\sqrt{\frac{2ts}{t^2+s^2}}<1$ for each $i\in\mathbb{N}$, and hence $$H(\bn^t(\cdot,\textbf{x}_1),\bn^s(\cdot,\textbf{x}_2))=0.$$ By \cite[Theorem, p. 218]{Kak}, $\bn^t(\cdot,\textbf{x}_1)$ and $\bn^s(\cdot,\textbf{x}_2)$ are mutually singular, and hence, by (\ref{250210e1}), so are $P_t(\cdot,\textbf{x}_1)$ and $P_s(\cdot,\textbf{x}_2)$.

If $t= s$, then $H(\bn_{ta_i}(\cdot,x_{i,1}),\bn_{ta_i}(\cdot,x_{i,2}))=e^{-\frac{(x_{i,1}-x_{i,2})^2}{8t^2a_i^2} }$ for each $i\in\mathbb{N}$ and hence $$
H(\bn^t(\cdot,\textbf{x}_1),\bn^t(\cdot,\textbf{x}_2))>0
$$
 if and only if
\begin{eqnarray*}
\sum_{i=1}^{\infty}\frac{(x_{i,1}-x_{i,2})^2}{ a_i^2} <\infty.
\end{eqnarray*}
In this case, by \cite[Theorem, p. 218]{Kak} again, $\bn^t(\cdot,\textbf{x}_1)$ and $\bn^s(\cdot,\textbf{x}_2)$ are equivalent, and hence, by (\ref{250210e1}) again, so are $P_t(\cdot,\textbf{x}_1)$ and $P_s(\cdot,\textbf{x}_2)$. It remains to prove \eqref{230702e1} and \eqref{230702e2}.
For simplicity, we only give the proof of \eqref{230702e1}. Note that for any $E\in \mathscr{B}(\ell^2)$, $i\in\mathbb{N}$ and $s_1\in\mathbb{R}$, we have
\begin{eqnarray}\label{230702e3}
P_t(E,s_1\textbf{e}_i)
&=&P_t(E+s_1\textbf{e}_i)
=\int_{\ell^2}\chi_{E}(\textbf{y})\,\mathrm{d}P_t(\textbf{y}+s_1\textbf{e}_i)
=\int_{\mathbb{R}}\int_{ \ell^{2}(\mathbb{N}\setminus\{i\})}\chi_{E}(\textbf{y})\,\mathrm{d}P_t^{\widehat{i}}(\textbf{y}^i)\mathrm{d}\bn_{ta_i}(y_i+s_1)\nonumber\\
&=&\int_{\mathbb{R}}\int_{\ell^{2}(\mathbb{N}\setminus\{i\})}\chi_{E}(\textbf{y})\,\mathrm{d}P_t^{\widehat{i}}(\textbf{y}^i)\frac{1}{\sqrt{2\pi t^2 a_i^2}}e^{-\frac{(y_i+s_1)^2}{2t^2a_i^2}}\mathrm{d}y_i\\
&=&\int_{\mathbb{R}}\int_{ \ell^{2}(\mathbb{N}\setminus\{i\})}\chi_{E}(\textbf{y})\cdot e^{-\frac{2s_1 y_i \nonumber +s_1^2}{2t^2a_i^2}}\,\mathrm{d}P_t^{\widehat{i}}(\textbf{y}^i)\mathrm{d}\bn_{ta_i}(y_i)
=\int_{E} e^{-\frac{2s_1 y_i +s_1^2}{2t^2a_i^2}}\,\mathrm{d}P_t(\textbf{y}),\nonumber
\end{eqnarray}
where $\textbf{y}^i=(y_j)_{j\in \mathbb{N}\setminus \{i\}}\in \ell^{2}(\mathbb{N}\setminus\{i\})$ and $P^{\widehat{i}}_t$ is the product measure without the $i$-th component, i.e., it is the restriction of the product measure $\Pi_{j\in\mathbb{N}\setminus\{i\}}\bn_{ta_j}$ on $\left(\ell^{2}(\mathbb{N}\setminus\{i\}),\mathscr{B}\big(\ell^{2}(\mathbb{N}\setminus\{i\})\big)\right)$ (as that in (\ref{220817e1})), $a_{i}$ is the $i$-th element in the sequence $\{a_j\}_{j=1}^\infty$ given in the beginning of this section. Since we can view $\ell^2$ as $\mathbb{R}\times \ell^{2}(\mathbb{N}\setminus\{i\})$, we have $P_t =\bn_{ta_i}\times P_t^{\widehat{i}}$ which implies the third equality in \eqref{230702e3}. The fourth equality follows from the one dimensional translation formula $\mathrm{d}\bn_{ta_i}(y_i+s_1)=\frac{1}{\sqrt{2\pi t^2 a_i^2}}e^{-\frac{(y_i+s_1)^2}{2t^2a_i^2}}\mathrm{d}y_i$.
The conclusion \eqref{230702e1} follows from \eqref{230702e3}. This completes the proof of Proposition \ref{230421prop1}.
\end{proof}


\subsection{Partial derivatives, some notions of continuous functions, and convolution on $\ell^2$}
For each $k\in\mathbb{N}$, set $\textbf{e}_k\triangleq(\delta_{k,i})_{i\in\mathbb{N}}$ where $\delta_{k,i}\triangleq0$ if $k\neq i$ and $\delta_{k,i}\triangleq 1$ if $k=i$. Then $\{\textbf{e}_k\}_{k=1}^{\infty}$ is an orthonormal basis of $\ell^2$. Suppose that $O$ is a nonempty open subset of $\ell^2$ and $h$ is a real-valued function defined on $O$. As in \cite[p. 528]{YZ20}, for each $\textbf{x}=\sum\limits_{j=1}^{\infty}x_j\textbf{e}_j\in O$ (here and henceforth $x_j\in \mathbb{R}$ for each $j\in \mathbb{N}$), we define the first order partial derivative of $h(\cdot)$ (at $\textbf{x}$) as follows:
$$
\begin{array}{ll}
\displaystyle\frac{\partial h(\textbf{x})}{\partial x_j}\triangleq\lim_{\mathbb{R}\ni \tau\to 0}\frac{h(\textbf{x}+\tau\textbf{e}_j )-h(\textbf{x})}{\tau},\quad \forall\,\,j\in\mathbb{N}.
\end{array}
$$
By induction, we can define the higher order partial derivatives of $h(\cdot)$ (at $\textbf{x}$). For each $\textbf{x}\in O$, $n\in\mathbb{N}$ and nonzero finite-dimensional subspace $L$ of $\ell^2$, write
\begin{eqnarray*}
&&O_{\textbf{x}}^n\triangleq\{ s\in\mathbb{R}:\textbf{x}+s\textbf{e}_n\in O\},\qquad O_{\textbf{x},n}\triangleq\{(s_1,\cdots,s_n)\in\mathbb{R}^n:\textbf{x}+s_1\textbf{e}_1+\cdots+s_n\textbf{e}_n\in O\},\\[2mm]
&&O_{\textbf{x}}(L)\triangleq\{\textbf{y}\in L:\textbf{x}+\textbf{y} \in O\}.
\end{eqnarray*}
Then $O_{\textbf{x}}^n$ is a nonempty open subset of $\mathbb{R}$, $O_{\textbf{x},n}$ is a nonempty open subset of $\mathbb{R}^n$ and $O_{\textbf{x}}(L)$ is a nonempty open subset of $L$.
\begin{definition}\label{230407def1}
We say that $h$ is $\mathcal{S}$-continuous, if for any $\mathbf{x}\in O$ and $n\in\mathbb{N}$ the following function (on $O_{\textbf{x}}^n$):
\begin{eqnarray*}
 s\mapsto h(\mathbf{x}+ s\mathbf{e}_n),\quad\forall\, s\in O_{\mathbf{x}}^n,
\end{eqnarray*}
is continuous. We say that $h$ is $\mathcal{F}$-continuous, if for any $\mathbf{x}\in O$ and $n\in\mathbb{N}$ the following function  (on $O_{\mathbf{x},n}$):
\begin{eqnarray*}
(s_1,\cdots,s_n)\mapsto h(\mathbf{x}+s_1\mathbf{e}_1+\cdots+s_n\mathbf{e}_n),\quad\forall\,(s_1,\cdots,s_n)\in O_{\mathbf{x},n},
\end{eqnarray*}
is continuous. We say that $h$ is $\mathcal{G}$-continuous, if for any $\mathbf{x}\in O$ and nonzero finite-dimensional subspace $L$ of $\ell^2$ the following function (on $O_{\mathbf{x}}(L)$):
\begin{eqnarray*}
\mathbf{y}\mapsto h(\mathbf{x}+\mathbf{y}),\quad\forall\,\mathbf{y}\in O_{\mathbf{x}}(L),
\end{eqnarray*}
is continuous.
\end{definition}

In what follows, we shall denote by $C_\mathcal{S}(O)$, $C_\mathcal{F}(O)$ and $C_\mathcal{G}(O)$ respectively  the collection of all $\mathcal{S}$-continuous, $\mathcal{F}$-continuous and $\mathcal{G}$-continuous functions on $O$.
Furthermore, for each $k\in\mathbb{N}$, we denote by $C_\mathcal{F}^k(O)$ all functions in $C_\mathcal{F}(O)$ whose partial derivatives up to order $k$ exist and belong to $C_\mathcal{F}(O)$, and write $C_\mathcal{F}^{\infty}(O)\triangleq \bigcap\limits_{k=1}^{\infty}C_\mathcal{F}^k(O)$. Similarly, we can define $C_\mathcal{S}^r(O)$ and $C_\mathcal{G}^r(O)$ for $r\in\mathbb{N}\cup\{\infty\}$.

\begin{remark}
In the above notations, the letter ``$\mathcal{S}$" stands for ``separately" and the $\mathcal{S}$-continuous functions are also called separately continuous functions (See \cite{CSV}); while the letter ``$\mathcal{F}$" ({\it resp. } ``$\mathcal{G}$") stands for ``special finite-dimensional" ({\it resp. }general finite-dimensional"), because we can view $\mathcal{F}$-continuous ({\it resp. }``$\mathcal{G}$"-continuous) functions as functions that are continuous on special ({\it resp. }general) finite-dimensional subspaces.
\end{remark}

For each $n\in\mathbb{N}$, let \(\mathscr S_n\) be the family of all sets of the form
\[
\{\mathbf{x}+s_1\mathbf{e}_1+\cdots+s_n\mathbf{e}_n:(s_1,\cdots,s_n)\in V\},
\]
where \(\mathbf{x}\in O\) and \(V\) is an open subset of \(O_{\mathbf{x},n}\).
By \cite[Theorem 11, p. 47]{Kel}, $\mathscr{S}_n$ is a base for some topology on $O$, denoted by $\mathscr{T}_n$. The usual topology (on $O$) generated by the $\ell^2$-norm will be denoted by $\mathscr{T}_{O}$. Let $ \mathscr{T}\triangleq \bigcap\limits_{n=1}^{\infty}\mathscr{T}_{n}$. Then $\mathscr{T}$ is also a topology space on $O$. The relationships between these topologies will be described in Proposition \ref{230703prop1} later.

We have the following simple result.

\begin{lemma}\label{230703lem1}
A real-valued function $f$ on $O$ is $\mathcal{F}$-continuous if and only if it is continuous with respect to $\mathscr{T}$.
\end{lemma}
\begin{proof}
Suppose that $f$ is an $\mathcal{F}$-continuous function on $O$. Then, for each open subset $U$ of $\mathbb{R}$, by Definition \ref{230407def1}, $f^{-1}(U)\in \mathscr{T}_{n}$ for each $n\in\mathbb{N}$, and hence $f^{-1}(U)\in \mathscr{T}$. Therefore, $f$ is continuous with respect to $\mathscr{T}$.

Conversely, suppose that $f$ is continuous with respect to $\mathscr{T}$. Then for each $n\in\mathbb{N}$, $\textbf{x}\in O$ and open subset $U$ of $\mathbb{R}$, we have $f^{-1}(U)\in \mathscr{T}\subset\mathscr{T}_{n}$. This implies that the following mapping
\begin{eqnarray*}
(s_1,\cdots,s_n)\mapsto f(\textbf{x}+s_1\textbf{e}_1+\cdots+s_n\textbf{e}_n),\quad\forall\,(s_1,\cdots,s_n)\in O_{\textbf{x},n},
\end{eqnarray*}
is continuous. Therefore, $f$ is $\mathcal{F}$-continuous. This completes the proof of Lemma \ref{230703lem1}.
\end{proof}

Further, we have the following result.

\begin{proposition}\label{230714lem1}
It holds that
$
C_{\ell^2}(O)\subsetneqq C_{\mathcal{G}}(O)\subsetneqq C_{\mathcal{F}}(O)\subsetneqq C_{\mathcal{S}}(O).
$
\end{proposition}
\begin{proof}
Clearly, $C_{\ell^2}(O)\subset C_{\mathcal{G}}(O)\subset C_{\mathcal{F}}(O)\subset C_{\mathcal{S}}(O)$. It suffices to prove the proper inclusions by giving three examples.

First of all, we choose a Hamel basis $\{\textbf{v}_{\alpha}:\alpha\in A\}$ for $\ell^2$ such that $\{\textbf{e}_i\}_{i=1}^{\infty}\subset \{\textbf{v}_{\alpha}:\alpha\in A\}$, where $A$ is an index set. For $\alpha\in A$, let
$$
c_{\alpha}  \triangleq
\begin{cases}
i, &\text{ if }\textbf{v}_{\alpha}=\textbf{e}_i\text{ for some }i\in\mathbb{N}, \\ 0,&\text{ otherwise}.
\end{cases}
$$
By the definition of Hamel basis, for each non-zero $\textbf{x}\in \ell^2$, there exists a unique finite subset $A_0\subset A$ and unique non-zero real numbers $\{a_{\alpha}\}_{\alpha\in A_0}$ such that $\textbf{x}=\sum\limits_{\alpha\in A_0}a_{\alpha}\textbf{v}_{\alpha}$, then let
$$
f(\textbf{x})\triangleq  \sum\limits_{\alpha\in A_0}a_{\alpha}c_{\alpha},
$$
and $f(\textbf{0})\triangleq 0.$ Clearly, $f$ is an $\mathbb{R}$-linear mapping from $\ell^2$ into $\mathbb{R}$ and hence $f\in C_\mathcal{G}(\ell^2)$. For each $\textbf{x}\in\ell^2$, we have
$$
\lim_{i\to\infty}\left(\textbf{x}+\frac{1}{\sqrt{i}}\textbf{e}_i\right)=\textbf{x}\,\,\text{ in }\,\ell^2,
$$
but the limit $\lim\limits_{i\to\infty}f\left(\textbf{x}+\frac{1}{\sqrt{i}}\textbf{e}_i\right)$ does not exist. This implies that $f$ is not continuous with respect to the usual $\ell^2$-norm topology at every point of $\ell^2$. Let $f|_O(\textbf{x})\triangleq f(\textbf{x}),\,\forall\, \textbf{x}\in O$. Since every finite-dimensional restriction of a linear map is automatically continuous, we obtain $f|_O\in C_\mathcal{G}(O)$ and $f|_O\notin C_{\ell^2}(O)$.

Next, motivated by the proof of \cite[Theorem 1.4, p. 13]{CSV}, for any $\textbf{x}=(x_i)_{i\in\mathbb{N}}\in O$, we put
$$
g\left( \textbf{x}\right)  \triangleq
\begin{cases}
\sum\limits^{\infty}_{i=1}x_i, &\text{ if }\sum\limits_{i=1}^{\infty}x_i\text{ converges}, \\ 0,&\text{ if }\sum\limits_{i=1}^{\infty}x_i\text{ does not converge}.
\end{cases}
$$
From the definition of $\mathcal{F}$-continuous functions in Definition \ref{230407def1}, it is easy to see that $g\in C_\mathcal{F}(O)$. Since $O$ is a nonempty open subset of $\ell^2$, it is easy to show that there exists $\textbf{x}_0=(x_i^0)_{i\in\mathbb{N}}\in O$ such that $\sum\limits_{i=1}^{\infty}x_i^0\text{ converges}$ and $\sum\limits_{i=1}^{\infty}x_i^0\neq 0$. Clearly, $\textbf{x}_1\triangleq\left(\frac{1}{i}\right)_{i\in\mathbb{N}}\in\ell^2$. Hence, we can find $\delta>0$ such that $\textbf{x}_0+\lambda\textbf{x}_1\in O$ for any $\lambda\in(-\delta,\delta)$. Thus, noting that $\sum\limits_{i=1}^{\infty}\frac{1}{i}\text{ does not converge}$, we have
$$
g(\textbf{x}_0+\lambda\textbf{x}_1)=0,\quad \forall \,\lambda\in (-\delta,0)\cup(0,\delta),\text{ and } g(\textbf{x}_0)=\sum\limits_{i=1}^{\infty}x_i^0\neq 0,
$$
which implies that $g\notin C_\mathcal{G}(O)$.

Finally, recall that A.~Genocchi and G.~Peano constructed  the following function on $\mathbb{R}^2$ which is discontinuous at $(0,0)$ even if it is continuous along every straight line through $(0,0)$ (See \cite[p. 31]{Ros}):
$$
h(x,y)  \triangleq
\begin{cases}
\frac{xy^2}{x^2+y^4}, &\text{ except for }x=y=0,\\
0,&\text{ for }x=y=0,
\end{cases}
$$
where $(x,y)\in \mathbb{R}^2$. Clearly, $h(x,y)$ is continuous everywhere except for $x=y=0.$ We can view $h$ as a cylinder function $\tilde{h}$ on $\ell^2$. Then, one can obtain a function
 $\tilde{h}|_O\in C_{\mathcal{S}}(O)\setminus C_{\mathcal{F}}(O)$.
This completes the proof of Proposition \ref{230714lem1}.
\end{proof}

\medskip

Write
\begin{equation}\label{20260725for2}
H\triangleq \left\{\sum\limits_{i=1}^{\infty}x_i\textbf{e}_i\in\ell^2:\;\sum\limits_{i=1}^{\infty}\frac{x_i^2}{a_i^2}<\infty\right\}.
\end{equation}
Then $H$ is a Hilbert space with the canonic inner product, and $H\subset\ell^2$. $H$ is the so-called Cameron-Martin space of the measure $P_t$ (See \eqref{220817e1}). Motivated by the definition in \cite[p. 133]{Gro67}, for any $f\in B_{\ell^2}^b(\ell^2)$, we introduce the following function (which can be viewed as the convolution of the function $f$ and the measure $P_t$):
\begin{eqnarray*}
(P_tf)(\textbf{x})\triangleq \int_{\ell^2}f(\textbf{x}-\textbf{y})\,\mathrm{d}P_t(\textbf{y}),\quad\textbf{x}\in\ell^2.
\end{eqnarray*}
Some basic properties of the function $P_tf$ will be given as follows:

\begin{proposition}\label{partial derivative of Ptf}
All order of partial derivatives of $P_tf$ exist at every point of $\ell^2$, and
\begin{itemize}
\item[$(1)$]$\frac{\partial (P_tf)(\mathbf{x})}{\partial x_i}=-\int_{\ell^2}f(\mathbf{x}-\mathbf{y})\cdot\frac{y_i}{t^2a_i^2}\,\mathrm{d}P_t(\mathbf{y})$ for all $\mathbf{x}\in\ell^2$ and $i\in\mathbb{N}$;
\item[$(2)$] $\left|\frac{\partial (P_tf)(\mathbf{x})}{\partial x_i}\right|\leqslant \frac{1}{ta_i}\sup\limits_{\mathbf{y}\in\ell^2}|f(\mathbf{y})|$  for all $\mathbf{x}\in\ell^2$ and $i\in\mathbb{N}$ and
\begin{eqnarray}\label{20250820for1}
 \sum_{i=1}^{\infty}a_{i}^2\left|\frac{\partial (P_tf)(\mathbf{x})}{\partial x_i}\right|^2&< &\frac{1}{t^4}\int |f(\mathbf{x}-\mathbf{y})|^2\,\mathrm{d}P_t(\mathbf{y}),\quad\forall\; \mathbf{x}\in\ell^2;
\end{eqnarray}
\item[$(3)$] $P_tf\in C_\mathcal{F}^{\infty}(\ell^2)$;
\item[$(4)$] $P_tf$ and all order of its partial derivatives belong to $B_{\ell^2}^b(\ell^2)$;
\item[$(5)$] $P_tf\in C_H^{\infty}(\ell^2)$.
\end{itemize}
\end{proposition}

\begin{proof}
The proofs of the results in Proposition \ref{partial derivative of Ptf} are almost the same as that in \cite[pp. 152--156]{Gro67}. For the reader's convenience, we shall provide below the key steps.

\smallskip

(1) For any $\textbf{x}\in\ell^2$, $i\in\mathbb{N}$ and $s\in\mathbb{R}$, by \eqref{230702e1} of Proposition \ref{230421prop1}, it holds that
\begin{eqnarray*}
(P_tf)(\textbf{x}+s\textbf{e}_i)
&=& \int_{\ell^2}f(\textbf{x}+s\textbf{e}_i-\textbf{y})\,\mathrm{d}P_t(\textbf{y})\\
&=& \int_{\ell^2}f(\textbf{x}-\textbf{y})\,\mathrm{d}P_t(\textbf{y}+s\textbf{e}_i)
= \int_{\ell^2}f(\textbf{x}-\textbf{y})J(s,y_i)\,\mathrm{d}P_t(\textbf{y}),
\end{eqnarray*}
where $J(s,y_i)\triangleq e^{-\frac{2s y_i +s^2}{2t^2a_i^2}}$. By
$\frac{\partial J(rs,y_i)}{\partial r}= -\frac{s y_i +rs^2}{t^2a_i^2}\cdot J(rs,y_i)$ for all $r\in\mathbb{R}$ and
\begin{eqnarray*}
(P_tf)(\textbf{x}+s\textbf{e}_i)-(P_tf)(\textbf{x})=-\int_{\ell^2}\int_0^1f(\textbf{x}-\textbf{y})\cdot\frac{s y_i +rs^2}{t^2a_i^2}\cdot J(rs,y_i)\,\mathrm{d}r \mathrm{d}P_t(\textbf{y}),
\end{eqnarray*}
we obtain that
\begin{eqnarray*}
&&\left|(P_tf)(\textbf{x}+s\textbf{e}_i)-(P_tf)(\textbf{x})+s\cdot\int_{\ell^2}f(\textbf{x}-\textbf{y})\cdot\frac{y_i}{t^2a_i^2}\,\mathrm{d}P_t(\textbf{y})\right|\\
&&\leqslant\left|\int_{\ell^2}\int_0^1f(\textbf{x}-\textbf{y})\cdot\frac{s y_i }{t^2a_i^2}\cdot (J(rs,y_i)-1)\,\mathrm{d}r \mathrm{d}P_t(\textbf{y})\right|\\
&&\quad+\left|\int_{\ell^2}\int_0^1f(\textbf{x}-\textbf{y})\cdot\frac{-rs^2}{t^2a_i^2}\cdot J(rs,y_i)\,\mathrm{d}r \mathrm{d}P_t(\textbf{y})\right|\\
&&\leqslant \frac{ \sup\limits_{\textbf{y}\in\ell^2}|f(\textbf{y})|}{t^2a_i^2}\cdot\left|\int_{\ell^2}\int_0^1|y_i\cdot s\cdot (J(rs,y_i)-1)|\,\mathrm{d}r \mathrm{d}P_t(\textbf{y})\right|\\
&&\quad+s^2\cdot\frac{\sup\limits_{\textbf{y}\in\ell^2}|f(\textbf{y})|}{t^2a_i^2}\left|\int_{\ell^2}\int_0^1  r\cdot J(rs,y_i)\,\mathrm{d}r \mathrm{d}P_t(\textbf{y})\right|\\
&&= \frac{ \sup\limits_{\textbf{y}\in\ell^2}|f(\textbf{y})|}{t^2a_i^2}\cdot\left|\int_0^1\int_{\ell^2}|y_i\cdot s\cdot (J(rs,y_i)-1)|\,\mathrm{d}P_t(\textbf{y})\mathrm{d}r \right|\\
&&\quad+s^2\cdot\frac{\sup\limits_{\textbf{y}\in\ell^2}|f(\textbf{y})|}{t^2a_i^2}\left|\int_0^1\int_{\ell^2}  r\cdot J(rs,y_i)\,\mathrm{d}P_t(\textbf{y})\mathrm{d}r \right|\\
&&\leqslant |s|\cdot \frac{ \sup\limits_{\textbf{y}\in\ell^2}|f(\textbf{y})|}{t^2a_i^2}\cdot \int_0^1\left(\int_{\ell^2}|y_i|^2\,\mathrm{d}P_t(\textbf{y})\right)^{\frac{1}{2}}\cdot \left(\int_{\ell^2}(J(rs,y_i)-1)^2\,\mathrm{d}P_t(\textbf{y})\right)^{\frac{1}{2}}\mathrm{d}r \\
&&\quad+s^2\cdot\frac{\sup\limits_{\textbf{y}\in\ell^2}|f(\textbf{y})|}{t^2a_i^2} \int_0^1 r\cdot\int_{\ell^2}  J(rs,y_i)\,\mathrm{d}P_t(\textbf{y})\mathrm{d}r.
\end{eqnarray*}
Note that for each $r\in[0,1]$, we have
\begin{eqnarray*}
\int_{\ell^2}(J(rs,y_i)-1)^2\,\mathrm{d}P_t(\textbf{y})=e^{\frac{r^2s^2}{t^2a_i^2}}-1,\qquad\int_{\ell^2}  J(rs,y_i)\,\mathrm{d}P_t(\textbf{y})=1.
\end{eqnarray*}
Hence,
\begin{eqnarray*}
&&\left|(P_tf)(\textbf{x}+s\textbf{e}_i)-(P_tf)(\textbf{x})+s\cdot\int_{\ell^2}f(\textbf{x}-\textbf{y})\cdot\frac{y_i}{t^2a_i^2}\,\mathrm{d}P_t(\textbf{y})\right|\\
&&\leqslant |s|\cdot \frac{ \sup\limits_{\textbf{y}\in\ell^2}|f(\textbf{y})|}{t^2a_i^2}\cdot \left(\int_{\ell^2}|y_i|^2\,\mathrm{d}P_t(\textbf{y})\right)^{\frac{1}{2}}\cdot \int_0^1 \left(e^{\frac{r^2s^2}{t^2a_i^2}}-1\right)^{\frac{1}{2}}\mathrm{d}r +s^2\cdot\frac{\sup\limits_{\textbf{y}\in\ell^2}|f(\textbf{y})|}{t^2a_i^2}\cdot \int_0^1 r\mathrm{d}r\\
&&\leqslant |s|\cdot \frac{ \sup\limits_{\textbf{y}\in\ell^2}|f(\textbf{y})|}{ta_i}\cdot\left(e^{\frac{s^2}{t^2a_i^2}}-1\right)^{\frac{1}{2}} +s^2\cdot\frac{\sup\limits_{\textbf{y}\in\ell^2}|f(\textbf{y})|}{t^2a_i^2}\cdot \frac{1}{2},
\end{eqnarray*}
which implies that
\begin{eqnarray*}
\left|(P_tf)(\textbf{x}+s\textbf{e}_i)-(P_tf)(\textbf{x})+s\cdot\int_{\ell^2}f(\textbf{x}-\textbf{y})\cdot\frac{y_i}{t^2a_i^2}\,\mathrm{d}P_t(\textbf{y})\right|=o(|s|),
\end{eqnarray*}
as $s\to 0$. Therefore, $\frac{\partial (P_tf)(\textbf{x})}{\partial x_i}=-\int_{\ell^2}f(\textbf{x}-\textbf{y})\cdot\frac{y_i}{t^2a_i^2}\,\mathrm{d}P_t(\textbf{y})$.

\smallskip

(2) By the conclusion (1) in this proposition, it follows that
\begin{eqnarray*}
 \left|\frac{\partial (P_tf)(\textbf{x})}{\partial x_i}\right|&\leqslant & \left|\int_{\ell^2}f(\textbf{x}-\textbf{y})\cdot\frac{y_i}{t^2a_i^2}\,\mathrm{d}P_t(\textbf{y})\right|\\
 &\leqslant&\frac{ \sup\limits_{\textbf{y}\in\ell^2}|f(\textbf{y})|}{t^2a_i^2}\cdot\left(\int_{\ell^2}y_i^2\,\mathrm{d}P_t(\textbf{y})\right)^{\frac{1}{2}}=\frac{ \sup\limits_{\textbf{y}\in\ell^2}|f(\textbf{y})|}{ta_i}.
 \end{eqnarray*}
Similar to the proof of the inequality (11) in \cite[Proposition 9, p. 152]{Gro67}, one can obtain \eqref{20250820for1}.

\smallskip

(3) For simplicity, we only prove that $P_tf\in C_\mathcal{F}(\ell^2)$. Note that for each $\textbf{x}\in \ell^2$, $n\in\mathbb{N}$ and $(s_1,\cdots,s_n)\in\mathbb{R}^n$, by \eqref{230702e2} of Proposition \ref{230421prop1}, we have
\begin{eqnarray*}
&&(P_tf)(\textbf{x}+s_1\textbf{e}_1+\cdots+s_n\textbf{e}_n)=\int_{\ell^2}f(\textbf{x}+s_1\textbf{e}_1+\cdots+s_n\textbf{e}_n-\textbf{y})\,\mathrm{d}P_t(\textbf{y})\\
&&=\int_{\ell^2}f(\textbf{x}-\textbf{y})\,\mathrm{d}P_t(\textbf{y}+s_1\textbf{e}_1+\cdots+s_n\textbf{e}_n)=\int_{\ell^2}f(\textbf{x}-\textbf{y})e^{-\frac{2s_1 y_1 +s_1^2}{2t^2a_1^2}}\cdots e^{-\frac{2s_n y_n +s_n^2}{2t^2a_n^2}}\,\mathrm{d}P_t(\textbf{y})\\
&&=e^{\frac{-s_1^2}{2t^2a_1^2}}\cdots e^{\frac{-s_n^2}{2t^2a_n^2}} \int_{\ell^2}f(\textbf{x}-\textbf{y})e^{-\frac{2s_1 y_1}{2t^2a_1^2}}\cdots e^{-\frac{2s_n y_n }{2t^2a_n^2}}\,\mathrm{d}P_t(\textbf{y}),
\end{eqnarray*}
from which we deduce that the mapping
$
(s_1,\cdots,s_n)\mapsto (P_tf)(\textbf{x}+s_1\textbf{e}_1+\cdots+s_n\textbf{e}_n)
$
is continuous from $\mathbb{R}^n$ into $\mathbb{R}$. This implies that $P_tf$ is $\mathcal{F}$-continuous.

\smallskip

(4) For simplicity, we only prove that $P_tf\in B_{\ell^2}(\ell^2)$. By the classic procedure in real analysis, there exists a sequence of uniformly bounded simple functions $\{f_n\}_{n=1}^{\infty}$ on $\ell^2$ such that $\lim\limits_{n\to\infty}f_n(\textbf{x})=f(\textbf{x})$ for any $\textbf{x}\in\ell^2.$ By the bounded convergence theorem, we have $\lim\limits_{n\to\infty}(P_tf_n)(\textbf{x})=(P_tf)(\textbf{x})$ for any $\textbf{x}\in\ell^2$. Therefore, it suffices to prove that if $E$ is a Borel subset of $\ell^2$, then $P_t\chi_E$ is Borel measurable. Let
$
\mathscr{M}\triangleq \left\{E\in\mathscr{B}(\ell^2):\; P_t\chi_E\text{ is Borel measurable}\right\}.
$
Then $\mathscr{M}$ is closed under monotone limits and complements. Denote by $\mathscr{C}$ the family of subsets of $\ell^2$:
$
\left\{(x_i)_{i\in\mathbb{N}}\in\ell^2:(x_1,\cdots,x_n)\in B\right\},
$
where $n\in \mathbb{N}$ and $B\in \mathscr{B}(\mathbb{R}^n)$.
It is easy to see that $P_t\chi_E$ is continuous for any $E\in \mathscr{C}$ and hence $\mathscr{C}\subset\mathscr{M}$. Therefore, $\mathscr{M}=\mathscr{B}(\ell^2)$.

\smallskip

(5) This conclusion is just a consequence of \cite[Proposition 9, p. 152]{Gro67}. This completes the proof of Proposition \ref{partial derivative of Ptf}.
\end{proof}

As an immediate application of Proposition \ref{partial derivative of Ptf}, we shall construct below a real-valued function on $\ell^2$ for which all of its partial derivatives exist, but it is not continuous with respect to the usual $\ell^2$-norm topology. Let $A\triangleq \{(x_i)_{i\in\mathbb{N}}\in\ell^2:x_i\in (0,+\infty)\text{ for each }i\in\mathbb{N}\}$, $\textbf{x}_0\triangleq(\sqrt{a_i})_{i\in\mathbb{N}}\in \ell^2$ and
\begin{eqnarray*}
\psi(\textbf{x})\triangleq (P_1\chi_A)(\textbf{x}),\,\quad \forall\,\textbf{x}\in\ell^2(\mathbb{N}).
\end{eqnarray*}
Since $A$ is a Borel set of $\ell^2$, by Proposition \ref{partial derivative of Ptf}, $\psi\in C_\mathcal{F}^{\infty}(\ell^2)$. However, as indicated by the following result, the above $\psi$ is not continuous with respect to the usual $\ell^2$-norm topology.
\begin{lemma}\label{not continuous}
 $\psi\in C_\mathcal{F}^{\infty}(\ell^2)\setminus C_{\ell^2}(\ell^2)$.
\end{lemma}
\begin{proof}
For each $n\in\mathbb{N}$, let
$x_{n,m}\triangleq\left\{\begin{array}{ll}
\sqrt{a_m},\quad&1\leqslant m\leqslant n,\\
(-1)^{m+1}\sqrt{a_m},&m>n.\end{array}\right.$
Set $\textbf{x}_n\triangleq (x_{n,m})_{m\in\mathbb{N}},\,\forall\,n\in\mathbb{N}$. Then, $\{\textbf{x}_n\}_{n=1}^{\infty}\subset \ell^2$ and $\lim\limits_{n\to\infty}||\textbf{x}_n-\textbf{x}_0||_{\ell^2}=0$. Note that for each $n\in\mathbb{N}$, we have
\begin{eqnarray*}
\psi(\textbf{x}_n)=\int_{\ell^2} \chi_A(\textbf{x}_n-\textbf{y})\,\mathrm{d}P_1(\textbf{y})=\prod_{i=1}^{\infty}\frac{1}{\sqrt{2\pi a_i^2}}\int_{-\infty}^{x_{n,i}}e^{-\frac{y_i^2}{2a_i^2}}\,\mathrm{d}y_i=0,
\end{eqnarray*}
where the last equality follows from the fact that $\frac{1}{\sqrt{2\pi a_i^2}}\int_{-\infty}^{x_{n,i}}e^{-\frac{y_i^2}{2a_i^2}}\,\mathrm{d}y_i\leqslant \frac{1}{\sqrt{2\pi a_i^2}}\int_{-\infty}^{0}e^{-\frac{y_i^2}{2a_i^2}}\,\mathrm{d}y_i=\frac{1}{2}$ for all $i=2(n+k)$ and $k\in\mathbb{N}$. We also note that
\begin{eqnarray*}
\psi(\textbf{x}_0)=\int_{\ell^2} \chi_A(\textbf{x}_0-\textbf{y})\,\mathrm{d}P_1(\textbf{y})=\prod_{i=1}^{\infty}\frac{1}{\sqrt{2\pi a_i^2}}\int_{-\infty}^{\sqrt{a_i}}e^{-\frac{y_i^2}{2a_i^2}}\,\mathrm{d}y_i= \prod_{i=1}^{\infty}\left(1-\frac{1}{\sqrt{2\pi a_i^2}}\int_{\sqrt{a_i}}^{+\infty} e^{-\frac{y_i^2}{2a_i^2}}\,\mathrm{d}y_i\right),
\end{eqnarray*}
and
\begin{eqnarray*}
&&\sum_{i=1}^{\infty} \frac{1}{\sqrt{2\pi a_i^2}}\int_{\sqrt{a_i}}^{+\infty} e^{-\frac{y_i^2}{2a_i^2}}\,\mathrm{d}y_i
=\frac{1}{\sqrt{\pi }}\sum_{i=1}^{\infty} \int_{\frac{1}{\sqrt{2a_i}}}^{+\infty} e^{- y_i^2 }\,\mathrm{d}y_i\\
&&\leqslant \frac{1}{\sqrt{\pi }}\sum_{i=1}^{\infty} \int_{\frac{1}{\sqrt{2a_i}}}^{+\infty} e^{- \frac{y_i}{\sqrt{2a_i}}  }\,\mathrm{d}y_i
=\frac{1}{\sqrt{\pi }}\sum_{i=1}^{\infty}\sqrt{2a_i}e^{-\frac{1}{2a_i}}
\leqslant\frac{1}{\sqrt{\pi }}\sum_{i=1}^{\infty}2\sqrt{2a_i} a_i <\infty.
\end{eqnarray*}
By \cite[Theorem 15.4, p. 299]{Rud87}, we have $\psi(\textbf{x}_0)>0$.  Therefore, $\psi\notin C_{\ell^2}(\ell^2)$. This completes the proof of Lemma \ref{not continuous}.
\end{proof}

Finally, we have the following result.

\begin{proposition}\label{230703prop1}
It holds that $\mathscr{T}_{O}\subsetneqq \mathscr{T}\subsetneqq \mathscr{T}_{n+1}\subsetneqq \mathscr{T}_n$ for all $n\in\mathbb{N}$.
\end{proposition}

\begin{proof}
For simplicity, we only consider the case that $O=\ell^2$. It is easy to see that $\mathscr{T}\subsetneqq \mathscr{T}_{n+1}\subsetneqq \mathscr{T}_n$ for all $n\in\mathbb{N}$. Note that $\mathscr{T}_{\ell^2}\subset \mathscr{T}$ is obvious and the proper inclusion follows from Lemmas \ref{230703lem1} and \ref{not continuous}. The proof of Proposition \ref{230703prop1} is completed.
\end{proof}

\subsection{Non-Borel measurable functions on $\ell^2$}

In Proposition \ref{partial derivative of Ptf}, we see that $P_tf\in C_\mathcal{F}^{\infty}(\ell^2)\cap  B_{\ell^2}^b(\ell^2)$ for each $f\in B_{\ell^2}^b(\ell^2)$, while Lemma \ref{not continuous} implies that $C_\mathcal{F}^{\infty}(\ell^2) \nsubseteqq C_{\ell^2}(\ell^2)$. Obviously, $C_{\ell^2}(O)\subsetneqq B_{\ell^2}(O)$ for each nonempty open subset $O$ of $\ell^2$, thus it is natural to ask, does $C_\mathcal{F}^{\infty}(O)\subset B_{\ell^2}(O)$ or $C_\mathcal{G}^{\infty}(O)\subset B_{\ell^2}(O)$? Since we will focus on integrable functions with respect to some Borel measures in this paper, these functions must be Borel measurable. Note that, for the functions under consideration in \cite[Theorem 2, p. 421]{Goo} both the $H$-continuity and the Borel measurability are assumed, while in \cite[Abstract, p. 279]{Zaj} both the G\^{a}teaux differentiability and the Borel measurability are assumed. In what follows, we shall use the tool of cardinality to prove that, many linear functions on $\ell^2$ are not Borel measurable!

For any set $E$, we denote by $\card E$ the cardinality of $E$. In particular, denote by $\aleph_0$ and $\mathfrak{c}$  respectively the cardinality of $\mathbb{N}$ and $\mathbb{R}$, and by $2^{\mathfrak{c}}$ the cardinality of all subsets of $\mathbb{R}$.
Motivated by the proof of Proposition \ref{230714lem1}, we shall prove that the cardinality of all linear maps from $\ell^2$ into $\mathbb{R}$ is $2^{\mathfrak{c}}$ but $\card B_{\ell^2}(O)=\mathfrak{c}$.

First, we have the following result.
\begin{lemma}\label{230717lem1}
For any nonempty open subset $O$ of $\ell^2$, it holds that $\card\mathscr{B}(O)=\mathfrak{c}$.
\end{lemma}
\begin{proof}
Let
$$
E_1\triangleq \{(x_i)_{i\in\mathbb{N}}\in O:x_i\in \mathbb{Q},\,\forall\,i\in\mathbb{N},x_j\neq 0\text{ only for finitely many
}j\in\mathbb{N}\}.
$$
Then $\card E_1=\aleph_0$ and $E_1$ is dense in $O$. Let
\begin{eqnarray*}
\mathscr{Q}&\triangleq &\{B_r(\textbf{x}): \textbf{x}\in E_1,\,r\in (0,+\infty)\cap \mathbb{Q},\,B_r(\textbf{x})\subset O\},\\
\mathscr{O}&\triangleq &\{O_1: O_1\text{ is an open subset of }O\}.
\end{eqnarray*}
Obviously, $\card \mathscr{Q}=\aleph_0$, and each element of $\mathscr{O}$ is a countable union of elements in $\mathscr{Q}$. Hence, $\mathscr{B}(O)$ is the $\sigma$-algebra generated by $\mathscr{Q}$, by \cite[Theorem 4, p. 129]{Kan}, it follows that $\card\mathscr{B}(O)=\mathfrak{c}$. This completes the proof of Lemma \ref{230717lem1}.
\end{proof}

\begin{lemma}\label{230717lem2}
For any nonempty open subset $O$ of $\ell^2$, it holds that $\card B_{\ell^2}(O)=\mathfrak{c}$.
\end{lemma}
\begin{proof}
Note that
$$
\{\chi_E:E\in \mathscr{B}(O)\}\subset B_{\ell^2}(O).
$$
Hence, we have $\card B_{\ell^2}(O)\geqslant\card \mathscr{B}(O)=\mathfrak{c}$.

On the other side, if $f,g\in B_{\ell^2}(O)$ and $f\neq g$, then it is easy to see that there exists $r\in \mathbb{Q}$ such that
$$
\{\textbf{x}\in O: f(\textbf{x})>r\}\neq \{\textbf{x}\in O: g(\textbf{x})>r\}.
$$
Let
$$
Tf\triangleq (\{\textbf{x}\in O: f(\textbf{x})>r\})_{r\in \mathbb{Q}},\quad\,\forall\, f\in B_{\ell^2}(O).
$$
Thus, $T$ is an injective mapping from $B_{\ell^2}(O)$ into $\prod\limits_{r\in \mathbb{Q}}\mathscr{B}(O)$, which implies that
$$
\card B_{\ell^2}(O)\leqslant \card\left(\prod\limits_{r\in \mathbb{Q}}\mathscr{B}(O)\right)=\mathfrak{c}.
$$
Therefore, $\card B_{\ell^2}(O)=\mathfrak{c}$, which completes the proof of Lemma \ref{230717lem2}.
\end{proof}

On the other hand, we shall prove that the cardinality of all linear maps from $\ell^2$ into $\mathbb{R}$ is $2^{\mathfrak{c}}$.
\begin{lemma}\label{230717lem3}
The cardinality of any Hamel basis of $\ell^2$ is $\mathfrak{c}$.
\end{lemma}
\begin{proof}
Suppose that the cardinality of some Hamel basis of $\ell^2$ is $\aleph_0$. Noting $\dim \ell^2=\infty$, we may choose this Hamel basis of $\ell^2$ as $\{\textbf{v}_i\}_{i=1}^{\infty}\subset \ell^2$. Then
$$
\ell^2=\bigcup_{n=1}^{\infty}\left\{\sum_{i=1}^{n}c_i\textbf{v}_i:c_i\in\mathbb{R},\,i=1,\cdots,n\right\},
$$
which implies that $\ell^2$ is a countable union of nowhere dense subsets. This contradicts the Baire category theorem (See \cite[2.2, p. 43]{Rud87}).

For every \(t\in(0,1)\), set
\[
v_t\triangleq\left(\frac{t^n}{n!}\right)_{n\in\mathbb N}.
\]
Then \(v_t\in \ell^2\). Indeed,
$
\sum_{n=1}^{\infty}\left(\frac{t^n}{n!}\right)^2
\leqslant
\sum_{n=1}^{\infty}\frac{1}{(n!)^2}
<\infty .
$

We claim that the family
$
\{v_t:t\in(0,1)\}
$
is linearly independent. Let \(m\in\mathbb N\), let
$
0<t_1<\cdots<t_m<1,
$
and suppose that, for some \(c_1,\ldots,c_m\in\mathbb R\),
\[
\sum_{j=1}^{m}c_jv_{t_j}=\mathbf{0}
\,\text{in }\ell^2.
\]
Comparing the \(n\)-th coordinates gives
$
\sum_{j=1}^{m}c_j\frac{t_j^n}{n!}=0,
\, n\in\mathbb N.
$
Equivalently,
$
\sum_{j=1}^{m}c_jt_j^n=0,
\, n\in\mathbb N.
$
In particular, for \(n=1,\ldots,m\), we obtain
\[
\begin{pmatrix}
t_1 & t_2 & \cdots & t_m\\
t_1^2 & t_2^2 & \cdots & t_m^2\\
\vdots & \vdots & \ddots & \vdots\\
t_1^m & t_2^m & \cdots & t_m^m
\end{pmatrix}
\begin{pmatrix}
c_1\\
c_2\\
\vdots\\
c_m
\end{pmatrix}
=
\begin{pmatrix}
0\\
0\\
\vdots\\
0
\end{pmatrix}.
\]
The determinant of the matrix on the left-hand side is
\[
\det(t_j^i)_{1\leqslant i,j\leqslant m}
=
\left(\prod_{j=1}^{m}t_j\right)
\prod_{1\leqslant j<k\leqslant m}(t_k-t_j).
\]
Since \(0<t_1<\cdots<t_m<1\), this determinant is nonzero. Hence
$
c_1=c_2=\cdots=c_m=0.
$
Therefore \(\{v_t:t\in(0,1)\}\) is a linearly independent subset of \(\ell^2\).

Moreover, the map \(t\mapsto v_t\) is injective, since the coordinate corresponding to \(n=1\) is \(t\). Hence
\[
\card \{v_t:t\in(0,1)\}
=
\card(0,1)
=
\mathfrak c.
\]
Consequently,
$
\dim_{\mathrm{Hamel}}\ell^2\geqslant \mathfrak c.
$
Equivalently, every Hamel basis of \(\ell^2\) has cardinality at least \(\mathfrak c\).

On the other hand,
$
\ell^2\subset \mathbb R^{\mathbb N},
$
and therefore
\[
\card\ell^2 \leqslant \card\mathbb R^{\mathbb N} =\mathfrak c^{\aleph_0}
=
(2^{\aleph_0})^{\aleph_0}
=
2^{\aleph_0}
=
\mathfrak c.
\]
Since every Hamel basis of \(\ell^2\) is a subset of \(\ell^2\), we have
$
\dim_{\mathrm{Hamel}}\ell^2\leqslant \mathfrak c.
$
Combining this with the previous inequality gives
$
\dim_{\mathrm{Hamel}}\ell^2=\mathfrak c.
$ This completes the proof of Lemma \ref{230717lem3}.
\end{proof}
\begin{lemma}\label{230717prop1}
The cardinality of all linear maps from $\ell^2$ into $\mathbb{R}$ is $2^{\mathfrak{c}}$.
\end{lemma}
\begin{proof}
Since $\card\ell^2=\mathfrak{c}$, the cardinality of all maps from $\ell^2$ into $\mathbb{R}$ is $2^{\mathfrak{c}}$. By Lemma \ref{230717lem3}, there exists a Hamel basis $\{\textbf{v}_{\alpha}:\alpha\in \mathbb{R}\}$ for $\ell^2$. Let
$$
\Phi(E)\triangleq \phi_E,\,\quad \forall\, E\in 2^\mathbb{R},
$$
where $2^\mathbb{R}$ denotes the collection of all subsets of $\mathbb{R}$ and $\phi_E$ is the linear mapping from $\ell^2$ into $\mathbb{R}$ defined by
$$
\phi_E\left(\sum_{\alpha\in\mathbb{R}}c_{\alpha}\textbf{v}_{\alpha}\right)\triangleq\sum_{\alpha\in\mathbb{R}}c_{\alpha}\chi_E(\alpha),
$$
where $c_{\alpha}\in \mathbb{R}$ for each $\alpha\in \mathbb{R}$ and $c_{\alpha}\neq 0$ only for finitely many $\alpha\in\mathbb{R}$. It is easy to see that $\Phi$ is an injective mapping. Since $\card 2^\mathbb{R}=2^{\mathfrak{c}}$, the cardinality of all linear maps from $\ell^2$ into $\mathbb{R}$ is $2^{\mathfrak{c}}$. The proof of Lemma \ref{230717prop1} is completed.
\end{proof}
Since all linear maps from $\ell^2$ into $\mathbb{R}$ belong to $C_\mathcal{G}^\infty(O)$, as a consequence of Lemmas \ref{230717lem2} and \ref{230717prop1}, we have the following result:
\begin{corollary}\label{230720cor1}
For any nonempty open subset $O$ of $\ell^2$, it holds that $C_\mathcal{G}^\infty(O)\nsubseteq B_{\ell^2}(O)$.
\end{corollary}

\begin{remark}
By the Banach--Pettis automatic continuity theorem, every Borel measurable
linear mapping between Polish normed spaces is continuous. In particular, every
Borel measurable linear mapping from \(\ell^2\) into \(\mathbb R\) is continuous;
see, for instance, Kechris \cite[Theorem 9.10, p. 61]{Kec}.

Recall that the \(\mathbb R\)-linear mapping
\[
f:\ell^2\longrightarrow \mathbb R
\]
constructed in the proof of Proposition \ref{230714lem1} is not continuous, while
\(f|_O\in C_{\mathcal G}^{\infty}(O)\). Hence \(f\) is not Borel measurable on
\(\ell^2\).

We further claim that \(f|_O\) is not Borel measurable on \(O\). Indeed, suppose
that \(f|_O\in B_{\ell^2}(O)\). Since \(O\) is nonempty and open, there exist
\(\mathbf{x}_0\in O\) and \(r>0\) such that
$
B_{r}(\mathbf{x}_0)\subset O.
$
Then \(f\) is Borel measurable on \(B_{r}(\mathbf{0} )\), because
$
f(\mathbf{x})=f(\mathbf{x}+\mathbf{x}_0)-f(\mathbf{x}_0),\, \mathbf{x}\in B_{r}(\mathbf{0} ).
$
By linearity, for every \(n\in\mathbb N\), \(f\) is Borel measurable on
\(B_{nr}(\mathbf{0} )\), since
$
f(\mathbf{x})=n f(\mathbf{x}/n),\,\mathbf{x}\in B_{nr}(\mathbf{0}).
$
As
\[
\ell^2=\bigcup_{n=1}^{\infty}B_{nr}(\mathbf{0}),
\]
it follows that \(f\) is Borel measurable on \(\ell^2\), a contradiction.
Therefore,
$
f|_O\notin B_{\ell^2}(O).
$
Consequently,
\[
f|_O\in C_{\mathcal G}^{\infty}(O)\setminus B_{\ell^2}(O).
\]
This gives another proof of Corollary \ref{230720cor1}.
\end{remark}

\subsection{A sequence of compactly supported functions on $\ell^2$}

For any real-valued function  $f$ on $\ell^2$, we denote by $\supp f$ the closure of the set $\{\textbf{x}\in\ell^2:f(\textbf{x})\neq 0\}$ in $\ell^2$.


By the assumption \eqref{20250130for1}, we have
\begin{eqnarray}\label{20250130for2}
\sum_{i=1}^{\infty}a_i^2\cdot \left|\ln \frac{1}{\sqrt{2\pi} a_i}\right|<\infty.
\end{eqnarray}
Choose a sequence of positive numbers $\{c_i\}_{i=1}^{\infty}$ such that $\lim\limits_{i\to\infty}c_i=0$,
 and
\begin{eqnarray}\label{20240920for1}
\sum_{i=1}^{\infty}\frac{a_i^2}{c_i}\cdot \left|\ln \frac{1}{\sqrt{2\pi} a_i}\right|<\infty,\qquad \sum_{i=1}^{\infty}\frac{a_i^2}{c_i}<\infty.
\end{eqnarray}
The following fact will be used in the proof of Proposition \ref{20241013thm1}.
\begin{lemma}\label{20250130lem1}
For each nonempty subset $I$ of $\mathbb{N}$, if $(x_i)_{i\in I}\in\ell^2(I)$ satisfies
\begin{eqnarray}\label{20250130for3}
\sum_{i\in I}\left|\ln \frac{1}{\sqrt{2\pi} a_i}-\frac{x_i^2}{2a_i^2}\right|<\infty,
\end{eqnarray}
then, $\sum\limits_{i\in I}\frac{x_i^2}{c_i}<\infty$.
\end{lemma}
\begin{proof}
Write $d_i=\ln \frac{1}{\sqrt{2\pi} a_i}-\frac{x_i^2}{2a_i^2}$ for $i\in I$, then combining \eqref{20250130for2},\eqref{20240920for1} and \eqref{20250130for3}, we have
\begin{eqnarray*}
\sum_{i\in I}\frac{x_i^2}{c_i}=\sum_{i\in I}\frac{2a_i^2}{c_i}\left(\ln \frac{1}{\sqrt{2\pi} a_i}-d_i\right)
\leq \sum_{i\in I}\frac{2a_i^2}{c_i}\left|\ln \frac{1}{\sqrt{2\pi} a_i} \right|+\sum_{i\in I}\frac{2a_i^2}{c_i}\cdot |d_i|<\infty,
\end{eqnarray*}
which completes the proof of Lemma \ref{20250130lem1}.
\end{proof}
Write
\begin{eqnarray}\label{230708e1}
K_n\triangleq \left\{(x_i)_{i\in\mathbb{N}}\in \ell^2:\sum\limits_{i=1}^{\infty}\frac{x_i^2}{c_i}\leqslant n^2\right\},\quad \forall\,n\in\mathbb{N},\quad K\triangleq \bigcup\limits_{n=1}^{\infty}K_n.
\end{eqnarray}
\begin{proposition}\label{basic propeties of Knm}
For any $n\in\mathbb{N}$, $K_n$ is a compact subset of $\ell^2$ and $\lim\limits_{n\to\infty}P_t(K_n)=1$ for any $t\in(0,+\infty)$.
\end{proposition}
\begin{proof}
First, it is easy to prove that $K_n$ is a closed subset of $\ell^2$ and hence we omit the details.

Second, let us prove that $K_n$ is a compact subset of $\ell^2$. Fix arbitrarily a sequence
$\{\textbf{y}_m\}_{m=1}^{\infty}=\{(y_{i,m})_{i\in\mathbb{N}}\}_{m=1}^{\infty}\subset K_n$.
For any $i,m\in\mathbb{N}$, it holds that $y_{i,m}^2\leqslant c_i \sum\limits_{j=1}^{\infty}\frac{y_{j,m}^2}{c_j}\leqslant c_i n^2\leqslant cn^2$, where $c\triangleq \sup\limits_{i\in\mathbb{N}}c_i$. By the diagonal process, there exists a subsequence $\{\textbf{y}_{m_k}\}_{k=1}^{\infty}$ of $\{\textbf{y}_m\}_{m=1}^{\infty}$ such that $\lim\limits_{k\to\infty}y_{i,m_k}$ exists for each $i\in\mathbb{N}$. Denote $y_i\triangleq\lim\limits_{k\to\infty}y_{i,m_k}$ for each $i\in\mathbb{N}$ and $\textbf{y}\triangleq (y_i)_{i\in\mathbb{N}}$. For each fixed $N_1\in\mathbb{N}$, we have
\begin{eqnarray*}
\sum_{i=1}^{N_1}\frac{y_i^2}{c_i}=\lim_{k\to\infty}\sum_{i=1}^{N_1}\frac{y_{i,m_k}^2}{c_i}\leqslant n^2.
\end{eqnarray*}
Letting $N_1\to\infty$ in the above, we obtain that $\sum\limits_{i=1}^{\infty}\frac{y_i^2}{c_i}\leqslant n^2$. Since $\sum\limits_{i=1}^{\infty} y_i^2\leqslant c\cdot\sum\limits_{i=1}^{\infty}\frac{y_i^2}{c_i}\leqslant cn^2<\infty,$ we have $\textbf{y}\in\ell^2$. Hence, $\textbf{y}\in K_n$. Since $\lim\limits_{i\to\infty}c_i=0$, for any $\varepsilon>0$, there exists $N_2\in\mathbb{N}$ such that $4c_i n^2<\frac{\varepsilon}{2}$ for all $i\geqslant N_2.$ Noting that $\lim\limits_{k\to\infty}\sum\limits_{i=1}^{N_2}|y_{i,m_k}-y_i|^2=0$, we can find $N_3\in\mathbb{N}$ such that $\sum\limits_{i=1}^{N_2}|y_{i,m_k}-y_i|^2<\frac{\varepsilon}{2}$ for all $k\geqslant N_3.$ Therefore, for any $k\geqslant N_3$, it holds that
\begin{eqnarray*}
\sum_{i=1}^{\infty}|y_{i,m_k}-y_i|^2&=&\sum_{i=1}^{N_2}|y_{i,m_k}-y_i|^2+\sum_{i=N_2 +1}^{\infty}|y_{i,m_k}-y_i|^2\\
&<&\frac{\varepsilon}{2}+\sum_{i=N_2 +1}^{\infty}2(|y_{i,m_k}|^2+|y_i|^2)
=\frac{\varepsilon}{2}+\sum_{i=N_2 +1}^{\infty}c_i\cdot\frac{2(|y_{i,m_k}|^2+|y_i|^2)}{c_i}\\
&\leqslant&\frac{\varepsilon}{2}+4\left(\sup_{i>N_2}c_{i}\right)\cdot n^2
\leqslant\varepsilon,
\end{eqnarray*}
which implies that $\lim\limits_{k\to\infty}\textbf{y}_{m_k}=\textbf{y}$ in $\ell^2$. Thus $K_n$ is a compact subset of $\ell^2$.

Finally, we prove that $\lim\limits_{n\to\infty}P_t(K_n)=1$. Note that
 \begin{eqnarray*}
\int_{\ell^2}\sum_{i=1}^{\infty}\frac{x_i^2}{c_i}\,\mathrm{d}P_t(\textbf{x})
=\sum_{i=1}^{\infty}\int_{\ell^2}\frac{x_i^2}{c_i}\,\mathrm{d}P_t(\textbf{x})
=t^2\sum_{i=1}^{\infty}\frac{a_i^2}{c_i}<\infty,
\end{eqnarray*}
where the last inequality follows from \eqref{20240920for1}. Hence,
\begin{eqnarray*}
P_t\left(\left\{(x_i)_{i\in\mathbb{N}}\in\ell^2:\sum_{i=1}^{\infty}\frac{x_i^2}{c_i}<\infty\right\}\right)=1.
\end{eqnarray*}
Since
\begin{eqnarray*}
\bigcup_{n=1}^{\infty}K_n=\left\{(x_i)_{i\in\mathbb{N}}\in\ell^2:\sum_{i=1}^{\infty}\frac{x_i^2}{c_i}<\infty\right\},
\end{eqnarray*}
we have $\lim\limits_{n\to\infty}P_t(K_n)=1$. This completes the proof of Proposition \ref{basic propeties of Knm}.
\end{proof}

The proof of the following result is motivated by \cite[pp. 419-421]{Goo} (but our approach is simpler and more readable).
\begin{theorem}\label{230215Th1}
There exists $\{X_n\}_{n=1}^{\infty}\subset C_\mathcal{F}^{\infty}(\ell^2)\cap  B_{\ell^2}^b(\ell^2)$ with the following properties:
\begin{itemize}
\item[$(1)$] For each $n\in\mathbb{N}$, $X_n\notin C_{\ell^2}(\ell^2)$ and $\supp X_n$ is a compact subset of $\ell^2$;
\item[$(2)$] There is a positive constant $C$ so that $\sum\limits_{i=1}^{\infty}a_i^2\left|\frac{\partial X_n(\mathbf{x})}{\partial x_i}\right|^2\leqslant C^2$ holds for every $\mathbf{x}\in\ell^2$ and $n\in\mathbb{N}$.
\end{itemize}
\end{theorem}
\begin{proof}
Choosing a sequence of positive numbers $\{c_i\}_{i=1}^{\infty}$ (with $\lim\limits_{i\to \infty}c_i=0$) as in \eqref{20240920for1}. Let
\begin{eqnarray*}
g_n(\textbf{x})\triangleq \int_{\ell^2}\chi_{K_n}(\textbf{x}-\textbf{y})\,\mathrm{d}P_1(\textbf{y})=P_1(\textbf{x}-K_n),\quad \forall\,n\in\mathbb{N},\quad\textbf{x}\in\ell^2,
\end{eqnarray*}
where $K_n$ is defined at \eqref{230708e1}, and $\textbf{x}-K_n\triangleq\{\textbf{x}-\textbf{a}:\;\textbf{a}\in K_n\}$.

\medskip

(1) By Proposition \ref{basic propeties of Knm}, we have $\lim\limits_{n\to\infty}P_1(K_n)=1$. Choose $N_1\in\mathbb{N}$ such that $P_1(K_{N_1})>\frac{4}{5}$, and fix any $k>N_1$. If $\textbf{x}\in K_{k-N_1}$, then $K_{N_1}\subset K_k-\textbf{x}\triangleq\{\textbf{a}-\textbf{x}:\;\textbf{a}\in K_k\}$ and hence
\begin{eqnarray*}
g_k(\textbf{x})=P_1(\textbf{x}-K_k)=P_1(K_k-\textbf{x})\geqslant P_1(K_{N_1})>\frac{4}{5}.
\end{eqnarray*}
Similarly, if $\textbf{x}\notin K_{k+N_1}$, then $K_k-\textbf{x}\subset \ell^2\setminus K_{N_1}$ and hence
\begin{eqnarray*}
g_k(\textbf{x})=P_1(\textbf{x}-K_k)=P_1(K_k-\textbf{x})\leqslant 1-P_1(K_{N_1})<\frac{1}{5}.
\end{eqnarray*}

Let
$$h(t)\triangleq\left\{\begin{array}{ll}
 e^{\frac{1}{(t-\frac{1}{16})(t-\frac{9}{16})}},\quad&\frac{1}{16}<t<\frac{9}{16},\\
0,&t\geqslant\frac{9}{16}\text{ or }t\leqslant\frac{1}{16}.\end{array}\right.$$
Set $H(x)\triangleq \frac{\int_{-\infty}^{x^2}h(t)\,\mathrm{d}t}{\int_{-\infty}^{+\infty}h(t)\,\mathrm{d}t}$ for $x\in\mathbb{R}$. Then $H\in C^{\infty}(\mathbb{R};[0,1])$, $H(x)=0$ for $|x|<\frac{1}{4}$ and $H(x)=1$ for $|x|>\frac{3}{4}$.

Let $X_n\triangleq H\circ g_{n+N_1}$ for each $n\in\mathbb{N}$. By Proposition \ref{partial derivative of Ptf}, $\{X_n\}_{n=1}^{\infty}\subset C_F^{\infty}(\ell^2)\cap  B_{\ell^2}^b(\ell^2)$. Moreover, for any $n\in\mathbb{N}$,
\begin{eqnarray}\label{230409eq1}
X_n(\textbf{x})=1,\quad\frac{\partial X_n(\textbf{x})}{\partial x_i}=0,\quad \forall\,i\in\mathbb{N},\quad\textbf{x}\in K_{n},\quad\text{and}\quad\supp X_n\subset K_{n+2N_1},
\end{eqnarray}
and hence $\supp X_n$ is a compact subset of $\ell^2$.

Let us use the contradiction argument to prove that $X_n\notin C_{\ell^2}(\ell^2)$. Suppose otherwise that $X_n\in C_{\ell^2}(\ell^2)$, then $X_n^{-1}(\mathbb{R}\setminus \{0\})$ would be a nonempty open subset of $\ell^2$ and $X_n^{-1}(\mathbb{R}\setminus \{0\})\subset\supp X_n\subset K_{n+2 N_1}$. Since every nonempty open subset of $\ell^2$ contains countably pairwise disjoint nonempty open balls with the same radius, the interior of any compact subset of $\ell^2$ is empty, we have $X_n^{-1}(\mathbb{R}\setminus \{0\})=\emptyset$, which is a contradiction. Therefore, $X_n\notin C_{\ell^2}(\ell^2)$.

\medskip

(2) Let $C\triangleq \sup\limits_{x\in\mathbb{R}}|H'(x)|<\infty$. For any $i,n\in\mathbb{N}$, we have
\begin{eqnarray*}
\left|\frac{\partial X_n(\textbf{x})}{\partial x_i}\right|=|H'(g_{n+N_1}(\textbf{x}))|\cdot\left|\frac{\partial g_{n+N_1}(\textbf{x})}{\partial x_i}\right|,\quad \forall\,\textbf{x}\in\ell^2.
\end{eqnarray*}
Hence, by the inequality \eqref{20250820for1} in the conclusion (2) of Proposition \ref{partial derivative of Ptf}, we have
\begin{eqnarray*}
\sum\limits_{i=1}^{\infty}a_i^2\cdot\left|\frac{\partial X_n(\textbf{x})}{\partial x_i}\right|^2\leqslant C^2\sum\limits_{i=1}^{\infty}a_i^2\cdot \left|\frac{\partial g_{n+N_1}(\textbf{x})}{\partial x_i}\right|^2= C^2\sum\limits_{i=1}^{\infty}a_i^2\cdot \left|\frac{\partial }{\partial x_i}P_1(\chi_{K_{n+N_1}})(\textbf{x})\right|^2\leq C^2,
\end{eqnarray*}
which completes the proof of Theorem \ref{230215Th1}.
\end{proof}

\section{General Codimensional Surfaces}
Suppose that $\ell^2=M_1+ M_2$, where $M_1,M_2$ are two closed linear subspaces of $\ell^2$ which are orthogonal to each other (and hence both $M_1$ and $M_2$ are Hilbert spaces). Suppose that $S$ is a nonempty subset of $\ell^2$ and $\textbf{x}=\textbf{x}_1+\textbf{x}_2\in S$, where $\textbf{x}_1\in M_1$ and $\textbf{x}_2\in M_2$. If there exists an open neighborhood $U_0$ of $\textbf{x}$ in $\ell^2$, an open neighborhood $U$ of $\textbf{x}_2$ in $M_2$ and $f\in C^1_{M_2}(U;M_1)$ such that $f(\textbf{x}_2)=\textbf{x}_1$ and
$$
S\cap U_0=\{\textbf{x}_2'+f(\textbf{x}_2'):\textbf{x}_2'\in U\},
$$
then for $\textbf{x}'=\textbf{x}_2'+f(\textbf{x}_2')\in S\cap U_0$ with $\textbf{x}_2'\in U$,
\begin{equation}\label{240323for1}
\begin{array}{ll}
\textbf{x}'-\textbf{x}= \textbf{x}_2'-\textbf{x}_2+f(\textbf{x}_2')-f(\textbf{x}_2)\\[2mm]
=\textbf{x}_2'-\textbf{x}_2+(Df(\textbf{x}_2))(\textbf{x}_2'-\textbf{x}_2)+o(||\textbf{x}_2'-\textbf{x}_2||),\hbox{ as }\textbf{x}_2'\to\textbf{x}_2\hbox{ in }U.
\end{array}
\end{equation}
We will use the notations $P$ and $Q$ to denote respectively the orthogonal normal projections in $\ell^2$ such that their ranges are given by
\begin{eqnarray}\label{20240705for5}
P\ell^2=\{\textbf{x}\in\ell^2:\;\left\langle \textbf{x},\textbf{x}_2'+(Df(\textbf{x}_2))\textbf{x}_2'\right\rangle=0,\,\forall\;\textbf{x}_2'\in M_2\}
\end{eqnarray}
and
\begin{eqnarray}\label{20240705for4}
Q\ell^2=\overline{\{\textbf{x}_2'+(Df(\textbf{x}_2))\textbf{x}_2':\;\textbf{x}_2'\in M_2\}}.
\end{eqnarray}
Clearly, $P\ell^2$ and $Q\ell^2$ are two closed linear subspaces of $\ell^2$ which are orthogonal to each other.

The following simple result will play a key role in the sequel.
\begin{proposition}\label{240618prop1}
For any orthogonal normal projection $P'$ in $\ell^2$, if
\begin{eqnarray}\label{240326for1}
P'(\mathbf{x}'-\mathbf{x})=o(||\mathbf{x}'-\mathbf{x}||)\quad\,\text{ as }\mathbf{x}'\to\mathbf{x}\hbox{ in }S,
\end{eqnarray}
then $P'\ell^2\subset P\ell^2$.
\begin{proof}
We use the contradiction argument. Suppose that $P'\ell^2\nsubseteq P\ell^2$, then there would exist $\textbf{x}''\in P'\ell^2\setminus P\ell^2$ and $\textbf{x}''=\textbf{x}''_1+\textbf{x}''_2$, where $\textbf{x}''_1\in P\ell^2$, $\textbf{x}''_2\in Q\ell^2$ and $\textbf{x}''_2\not=0$. Note that there exists $\textbf{x}'=\textbf{x}_2'+(Df(\textbf{x}_2))\textbf{x}_2'\in \{\textbf{x}_2'+(Df(\textbf{x}_2))\textbf{x}_2':\textbf{x}_2'\in M_2\}(\subset Q\ell^2)$ such that $\langle \textbf{x}''_2,\textbf{x}'\rangle\neq 0$. Denote by $P''$ the orthogonal normal projection in $\ell^2$ such that its range is given by $\{r\textbf{x}'':\;r\in \mathbb{R}\}$. Then, $||P''\textbf{x}||\leqslant||P'\textbf{x}||$ for any $\textbf{x}\in \ell^2$. Since $||P''\textbf{x}||=\frac{1}{||\textbf{x}''||}\left|\langle \textbf{x}'',\textbf{x}\rangle\right|$, it follows that $||P'\textbf{x}'||
\geqslant\frac{1}{||\textbf{x}''||}\left|\langle\textbf{x}'',\textbf{x}'\rangle\right|$.
Thus for sufficiently small $t\in\mathbb{R}$, we have $\textbf{x}_2+t\cdot \textbf{x}_2'\in U$ and as $t\to 0$, it holds that
\begin{eqnarray*}
&&||P'(\textbf{x}_2+t\cdot \textbf{x}_2'+f(\textbf{x}_2+t\cdot \textbf{x}_2')-\textbf{x}_2-f(\textbf{x}_2))||\\
&&=||P'(t\cdot \textbf{x}_2'+f(\textbf{x}_2+t\cdot \textbf{x}_2')-f(\textbf{x}_2))||
=||P'(t\cdot \textbf{x}_2'+Df(\textbf{x}_2)(t\cdot \textbf{x}_2')+o(t)) ||\\
&&=||P'(t\cdot \textbf{x}_2'+Df(\textbf{x}_2)(t\cdot \textbf{x}_2'))+o(t) ||
=||P'(t\cdot \textbf{x}')+o(t) ||\\
&&\geqslant||P'(t\cdot \textbf{x}')||-o(t)
\geqslant\frac{1}{||\textbf{x}''||}\left|\langle\textbf{x}'',t\cdot \textbf{x}'\rangle\right|-o(t)
=\frac{|t|}{||\textbf{x}''||}\cdot|\langle\textbf{x}''_2,\textbf{x}'\rangle|-o(t),
\end{eqnarray*}
which contradicts \eqref{240326for1}. This completes the proof of Proposition \ref{240618prop1}.
\end{proof}
\end{proposition}

As a consequence of Proposition \ref{240618prop1}, one can deduce the following result.
\begin{corollary}\label{20240705cor1}
The spaces in \eqref{20240705for5}--\eqref{20240705for4} are independent of the choice of $M_1,M_2$ and $f$.
\begin{proof}
It suffices to prove the uniqueness of $P$. Suppose that $\widetilde P$ is the corresponding orthogonal normal projection in $\ell^2$ for another choice of $\widetilde M_1',\widetilde M_2'$ and $\tilde f'$ (instead of the above $M_1,M_2$ and $f$). Then, similar to (\ref{240323for1}), we deduce that $\widetilde P(\textbf{x}'-\textbf{x})=o(||\textbf{x}'-\textbf{x}||)$ as $\textbf{x}'\to\textbf{x}$ in $S$. Hence, in view of Proposition \ref{240618prop1}, it follows that $\widetilde P\ell^2\subset P\ell^2$. Similarly, we have $P\ell^2\subset \widetilde P\ell^2$. Hence, $P= \widetilde P$.
\end{proof}
\end{corollary}

\begin{remark}
Roughly speaking, we can view \eqref{20240705for4} and \eqref{20240705for5} as the tangent space and the normal space of $S$ at $\mathbf{x}\in S$, respectively. Corollary \ref{20240705cor1} says that these two spaces are independent of the local coordinates and the corresponding coordinate functions, and hence they are geometric invariants.
\end{remark}

Recall that $\{\textbf{e}_k \}_{k=1}^\infty$ is an orthonormal basis of $\ell^2$. For a subset $I$ of $\mathbb{N}$, if $I=\emptyset$ ({\it resp. }$I\neq\emptyset$), then we denote by $P_I$ the orthogonal normal projection (in $\ell^2$) whose range is the closed linear linear space spanned by the zero vector ({\it resp. } $\{\textbf{e}_i:i\in I\}$).

Suppose that $I_1$ and $I_2$ are two nonempty subsets of $\mathbb{N}$. The following notion will play an important role in the sequel.

\begin{definition}\label{def3.1}
We say that $I_2$ can be obtained from $I_1$ by changing only a finite number of elements (in $\mathbb{N}$), denoted by $I_1\sim I_2$, if there exists $k\in\mathbb{N}_0$, $K_1 \subset I_1$ and $K_2\subset \mathbb{N}\setminus I_1$ such that $\card K_1=\card K_2=k$ and $I_2=(I_1\setminus K_1)\sqcup K_2$.
\end{definition}

\begin{proposition}\label{20250808prop1}
It holds that $I_1\sim I_2$ is equivalent to $\card(I_1\setminus I_2)=\card(I_2\setminus I_1)\in \mathbb{N}_0$.
\end{proposition}
\begin{proof}
If $I_1\sim I_2$, then there exists $k\in\mathbb{N}_0$, $K_1 \subset I_1$ and $K_2\subset \mathbb{N}\setminus I_1$ such that $\card K_1=\card K_2=k$ and $I_2=(I_1\setminus K_1)\sqcup K_2$. Clearly, $I_1\setminus I_2=K_1$ and $I_2\setminus I_1=K_2$, and hence, $\card(I_1\setminus I_2)=\card K_1=\card K_2=\card(I_2\setminus I_1)=k\in\mathbb{N}_0$.

Conversely, if $\card(I_1\setminus I_2)=\card(I_2\setminus I_1)=k\in \mathbb{N}_0$, then $K_1\triangleq I_1\setminus I_2\subset I_1$, $K_2\triangleq I_2\setminus I_1\subset \mathbb{N}\setminus I_1$, $\card K_1=\card K_2=k$ and $I_2=(I_1\setminus K_1)\sqcup K_2$, and hence $I_1\sim I_2$.
\end{proof}

\begin{remark}\label{20250505rem1}
Assume  that $I_1\sim I_2$.  Put
\begin{eqnarray}\label{20250505for3}
\hat I\triangleq\left\{
\begin{array}{ll}
\{i\in I_1\cap I_2: i\leq \max((I_1\setminus I_2)\sqcup(I_2\setminus I_1))\},\quad&\hbox{if }I_1\not=I_2,\\[2mm]
\emptyset,\quad&\hbox{if }I_1=I_2.
\end{array}
\right.
\end{eqnarray}
By Proposition \ref{20250808prop1}, $\card(I_1\setminus I_2)=\card(I_2\setminus I_1)\in \mathbb{N}_0$, and hence
$$
\card((I_1\setminus I_2)\sqcup \hat I)=\card(I_1\setminus I_2)+\card \hat I=\card(I_2\setminus I_1)+\card \hat I=\card ((I_2\setminus I_1)\sqcup \hat I)\triangleq t\in\mathbb{N}_0.
$$
Choose $i_1,\cdots,i_{t},j_1,\cdots,j_{t}\in\mathbb{N}$
such that $i_1<\cdots<i_{t}$, $j_1<\cdots<j_{t}$ and
$$
(I_1\setminus I_2)\sqcup \hat I=\{i_1,\cdots,i_{t}\},\qquad (I_2\setminus I_1)\sqcup \hat I=\{j_1,\cdots,j_{t}\}.
$$
By \eqref{20250505for3}, it follows that
$$
i>\max\{i_t,j_t\},\qquad \,\forall\,i\in I_1\setminus\{i_1,\cdots,i_{t}\}=(I_1\cap I_2)\setminus \hat I=I_2\setminus\{j_1,\cdots,j_{t}\}.
$$
Further, it is easy to see that $t=0$ if and only if $I_1=I_2$.
\end{remark}

\begin{lemma}\label{20241011lem2}
The above $\cdot\sim\cdot$ (in Definition \ref{def3.1}) is an equivalence relation.
\end{lemma}
\begin{proof}
We only need to prove that if $I_3$ is a nonempty subset of $\mathbb{N}$, $I_1\sim I_2$ and $I_2\sim I_3$, then $I_1\sim I_3$. Note that, by Proposition \ref{20250808prop1},  $I_1\sim I_3$ is equivalent to $\card(I_1\setminus I_3)=\card(I_3\setminus I_1)\in \mathbb{N}_0$. Put
$$
L_1\triangleq I_1\setminus(I_2\cup I_3),\quad L_2\triangleq I_2\setminus(I_1\cup I_3),\quad L_3\triangleq I_3\setminus(I_1\cup I_2).
$$
and
$$
L_4\triangleq(I_1\cap I_2)\setminus I_3,\quad L_5\triangleq(I_1\cap I_3)\setminus I_2,\quad L_6\triangleq (I_2\cap I_3)\setminus I_1.
$$
It is easy to see that $L_1,L_2,L_3,L_4,L_5$ and $L_6$ are disjoint finite subsets of $\mathbb{N}$, $I_1\setminus I_2=L_1\sqcup L_5$, $I_2\setminus I_1=L_2\sqcup L_6$,
$I_3\setminus I_1=L_3\sqcup L_6$, $I_1\setminus I_3=L_1\sqcup L_4$, $I_2\setminus I_3=L_2\sqcup L_4$ and $I_3\setminus I_2=L_3\sqcup L_5$.
By $\card(I_1\setminus I_2)=\card(I_2\setminus I_1)\in \mathbb{N}_0$,  it holds that
$\card(L_1\sqcup L_4\sqcup L_5)=\card(L_2\sqcup L_4\sqcup L_6)$. Similarly, $\card(L_2\sqcup L_4\sqcup L_6)=\card(L_3\sqcup L_5\sqcup L_6)$. Hence, $\card(L_1\sqcup L_4\sqcup L_5)=\card(L_3\sqcup L_5\sqcup L_6)$, which gives $\card(I_1\setminus I_3)=\card(I_3\setminus I_1)\in \mathbb{N}_0$. This completes the proof of Lemma \ref{20241011lem2}.
\end{proof}

In the rest of this paper, unless otherwise stated, we shall fix a nonempty subset $I_0$ of $\mathbb{N}$. Write
\begin{eqnarray}\label{20241231def1}
\Gamma_{I_0}\triangleq \{I:I\subset \mathbb{N}, \, I\sim I_0\}.
\end{eqnarray}
We introduce the following notion.

\begin{definition}\label{20241121def1000}
We say that $S$ is a surface of $\ell^2$ with codimension $\Gamma_{\mathbb{N}\setminus I_0}$, if for each $\mathbf{x}\in S$ there is an open neighborhood $U$ of $\mathbf{x}$ in $\ell^2$, $I\in\Gamma_{I_0}$,  an open neighborhood $U_{I}$ of $P_{I}\mathbf{x}$ in $P_{I}\ell^2$ and a mapping $f\in C_{P_{I}\ell^2}^1(U_{I};P_{\mathbb{N}\setminus I}\ell^2)$ such that $P_{I}|_{S\cap U}$ is a homeomorphism from $S\cap U$ onto $U_{I}$, and
$$
S\cap U=\left\{\mathbf{y}_{I}+f(\mathbf{y}_{I}):\;\mathbf{y}_{I}\in U_{I}\right\}.
$$
We call \((S\cap U,I,f)\) a local coordinate triple of \(S\),
and call \(f\) the corresponding local graph map.
\end{definition}

\begin{definition}\label{20241121def1}
A family of local coordinate triples $\mathcal {A}=\{(S\cap U_{\alpha}, I_{\alpha},f_{\alpha}):\alpha\in A\}$ (with an index set $A$) is called a $C^1$-differentiable structure on a surface $S$  of $\ell^2$ with codimension $\Gamma_{\mathbb{N}\setminus I_0}$, if $S=\bigcup\limits_{\alpha\in A}\left(S\cap U_{\alpha}\right)$.
If a $C^1$-differentiable structure is given on $S$, then $S$ is called a $C^1$-differentiable surface with codimension $\Gamma_{\mathbb{N}\setminus I_0}$.
\end{definition}
By Lemma \ref{20250105lem1} below, the transition maps between any two local coordinate triples are automatically
$C^1$. Thus the above family indeed defines an atlas in the usual sense.

In the rest of this paper, unless otherwise stated, we shall fix a $C^1$-differentiable surface $S$ with codimension $\Gamma_{\mathbb{N}\setminus I_0}$ and a given family of local coordinate triples $\mathcal {A}=\{(S\cap U_{\alpha}, I_{\alpha},f_{\alpha}):\alpha\in A\}$.

\section{Surface Measures}\label{20241113sec1}
In this section, we shall construct a natural surface measure on the $C^1$-differentiable surface $S$ by two steps: 1) Constructing a family of compatible local measures according to the given local coordinate triples; 2) Sticking together these local measures into a surface measure.

In what follows, unless otherwise stated, we suppose that $I_1$ and $I_2$ are two nonempty subsets of $\mathbb{N}$ such that $I_1\sim I_2$ (Recall Remark \ref{20250505rem1} for the corresponding $\hat I$,  $t$ and $i_1,\cdots,i_{t},j_1,$ $\cdots,j_{t}$, which will be used below). Denote by $F(P_{I_1}\ell^2,P_{I_2}\ell^2)$ the set of all bounded linear operators $T$ from $P_{I_1}\ell^2$ into $P_{I_2}\ell^2$ such that
\begin{eqnarray}\label{20241011for3}
P_{ I_1\cap I_2 }TP_{ I_1\cap I_2 }=P_{ I_1\cap I_2 },
\quad P_{I_1\cap I_2}TP_{I_1\setminus( I_1\cap I_2) }=0.
\end{eqnarray}
We need some determinant tools for this kind of operators.
\begin{definition}\label{20241113def1}
For any $T\in F(P_{I_1}\ell^2,P_{I_2}\ell^2)$, we define $\det T=1$ if $t=0$;  otherwise we define
\begin{eqnarray}
\det T\triangleq  \left|\begin{array}{cccc}
\langle T\mathbf{e}_{i_1},\mathbf{e}_{j_1}\rangle&\langle T\mathbf{e}_{i_2},\mathbf{e}_{j_1}\rangle&\cdots&\langle T\mathbf{e}_{i_{t}},\mathbf{e}_{j_1}\rangle\\
\langle T\mathbf{e}_{i_1},\mathbf{e}_{j_2}\rangle&\langle T\mathbf{e}_{i_2},\mathbf{e}_{j_2}\rangle&\cdots&\langle T\mathbf{e}_{i_{t}},\mathbf{e}_{j_2}\rangle\\
\vdots&\vdots&\ddots&\vdots\\
\langle T\textbf{e}_{i_1},\textbf{e}_{j_{t}}\rangle&\langle T\mathbf{e}_{i_2},\mathbf{e}_{j_{t}}\rangle &\cdots& \langle T\mathbf{e}_{i_{t}},\mathbf{e}_{j_{t}}\rangle
\end{array}\right|,
\end{eqnarray}
where the right side in the above equality is the determinant of the $t\times t$ matrix.
\end{definition}

\begin{remark}\label{20250505rem2}
If there exists $t'\in\mathbb{N}_0$
and $\widetilde{i_1},\cdots,\widetilde{i_{t'}},\widetilde{j_1},\cdots,\widetilde{j_{t'}}\in\mathbb{N},$
such that $\widetilde{i_1}<\cdots<\widetilde{i_{t'}},\widetilde{j_1}<\cdots<\widetilde{j_{t'}}$, $I_1\setminus I_2\subset\{\widetilde{i_1},\cdots,\widetilde{i_{t'}}\}\subset I_1$ and $I_2\setminus I_1\subset\{\widetilde{j_1},\cdots,\widetilde{j_{t'}}\}\subset I_2$, $I_1\setminus\{\widetilde{i_1},\cdots,\widetilde{i_{t'}}\}=I_2\setminus\{\widetilde{j_1},\cdots,\widetilde{j_{t'}}\}$ and $i>\max\{\widetilde{i_{t'}},\widetilde{j_{t'}}\}$ for any $i\in I_1\setminus\{\widetilde{i_1},\cdots,\widetilde{i_{t'}}\}$, then
$$
\hat I\subset\{\widetilde{i_1},\cdots,\widetilde{i_{t'}}\},\quad \hat I\subset\{\widetilde{j_1},\cdots,\widetilde{j_{t'}}\},
$$
where  $\hat I$ was given in \eqref{20250505for3}. Indeed, for any $j\in I_1\setminus \{\widetilde{i_1},\cdots,\widetilde{i_{t'}}\}(\subset I_1\cap I_2)$, one has $j>\max\{\widetilde{i_{t'}},\widetilde{j_{t'}}\}\geqslant \max((I_1\setminus I_2)\sqcup(I_2\setminus I_1))$, and therefore $j\not\in \hat I$. By the construction in Remark \ref{20250505rem1},  we see that
$\{i_1,\cdots,i_{t}\}=(I_1\setminus I_2)\sqcup\hat I\subset\{\widetilde{i_1}, \cdots,\widetilde{i_{t'}}\}$ and $\{j_1,\cdots,j_{t}\}= (I_2\setminus I_1)\sqcup \hat I\subset\{\widetilde{j_1},\cdots,\widetilde{j_{t'}}\}$.
Hence, there exists $s\in\mathbb{N}_0$ and $i_{t+1},\cdots,i_{t+s},j_{t+1},\cdots,j_{t+s}\in\mathbb{N}$
such that $i_t<i_{t+1}<\cdots<i_{t+s},j_t<j_{t+1}<\cdots<j_{t+s}$, $t'=t+s$, and
\begin{eqnarray}\label{202508101}
\{\widetilde{i_1},\cdots,\widetilde{i_{t'}}\}=\{i_1,\cdots,i_{t},i_{t+1},\cdots,i_{t+s}\},\quad
\{\widetilde{j_1},\cdots,\widetilde{j_{t'}}\}=\{j_1,\cdots,j_{t},j_{t+1},\cdots,j_{t+s}\}.
\end{eqnarray}

By \eqref{20241011for3}, it follows that
$$
\langle T\mathbf{e}_{i_{s_1}},\mathbf{e}_{j_{t_1}}\rangle=0,\quad \forall\,(s_1,t_1)\in\{1,\cdots,t\}\times\{t+1,\cdots,t+s\}.
$$
Indeed, for any $s_1\in\{1,\cdots,t\}$, one has $s_1\in I_1\setminus I_2$ or $s_1\in\hat I$. If $s_1\in I_1\setminus I_2$, then $P_{I_1\setminus( I_1\cap I_2)}\mathbf{e}_{i_{s_1}}=\mathbf{e}_{i_{s_1}}$. Hence, for any $t_1\in\{t+1,\cdots,t+s\}$, $j_{t_1}\in I_2\setminus\{j_1,\cdots,j_{t}\}=(I_1\cap I_2)\setminus \hat I\subset I_1\cap I_2$, and therefore,
$$
\langle T\mathbf{e}_{i_{s_1}},\mathbf{e}_{j_{t_1}}\rangle =\langle TP_{I_1\setminus( I_1\cap I_2)}\mathbf{e}_{i_{s_1}},P_{I_1\cap I_2}\mathbf{e}_{j_{t_1}}\rangle =\langle P_{I_1\cap I_2}TP_{I_1\setminus( I_1\cap I_2)}\mathbf{e}_{i_{s_1}},\mathbf{e}_{j_{t_1}}\rangle =0.
$$
If $s_1\in\hat I\subset I_1\cap I_2$, then for the above $j_{t_1}$, one has $s_1\not=j_{t_1}$, and hence
$$
\langle T\mathbf{e}_{i_{s_1}},\mathbf{e}_{j_{t_1}}\rangle =\langle TP_{I_1\cap I_2}\mathbf{e}_{i_{s_1}},P_{I_1\cap I_2}\mathbf{e}_{j_{t_1}}\rangle =\langle P_{I_1\cap I_2}TP_{I_1\cap I_2}\mathbf{e}_{i_{s_1}},\mathbf{e}_{j_{t_1}}\rangle =\langle \mathbf{e}_{i_{s_1}},\mathbf{e}_{j_{t_1}}\rangle =0.
$$

By \eqref{20241011for3} again, one has
\begin{eqnarray*}
\left(\begin{array}{cccc}
\langle T\mathbf{e}_{i_{t+1}},\mathbf{e}_{j_{t+1}}\rangle &\langle T\mathbf{e}_{i_{t+2}},\mathbf{e}_{j_{t+1}}\rangle &\cdots&\langle T\mathbf{e}_{i_{t+s}},\mathbf{e}_{j_{t+1}}\rangle \\
\langle T\mathbf{e}_{i_{t+1}},\mathbf{e}_{j_{t+2}}\rangle &\langle T\mathbf{e}_{i_{t+2}},\mathbf{e}_{j_{t+2}}\rangle &\cdots&\langle T\mathbf{e}_{i_{t+s}},\mathbf{e}_{j_{t+2}}\rangle \\
\vdots&\vdots&\ddots&\vdots\\
\langle T\mathbf{e}_{i_{t+1}},\mathbf{e}_{j_{t+s}}\rangle &\langle T\mathbf{e}_{i_{t+2}},\mathbf{e}_{j_{t+s}}\rangle  &\cdots&\langle T\mathbf{e}_{i_{t+s}},\mathbf{e}_{j_{t+s}}\rangle
\end{array}\right)=
\left(\begin{array}{cccc}
1&0&\cdots&0\\
0&1&\cdots&0\\
\vdots&\vdots&\ddots&\vdots\\
0&0&\cdots& 1
\end{array}\right).
\end{eqnarray*}
Indeed, for any $s_2,t_2\in\{t+1,\cdots,t+s\}$, one has $i_{s_2}\in I_1\setminus\{i_1,\cdots,i_t\}\subset (I_1\cap I_2)\setminus \hat I$ and $j_{t_2}\in I_2\setminus\{j_1,\cdots,j_t\}\subset (I_1\cap I_2)\setminus \hat I$, and therefore
$$
\langle T\mathbf{e}_{i_{s_2}},\mathbf{e}_{j_{t_2}}\rangle =\langle TP_{I_1\cap I_2}\mathbf{e}_{i_{s_2}},P_{I_1\cap I_2}\mathbf{e}_{j_{t_2}}\rangle =\langle P_{I_1\cap I_2}TP_{I_1\cap I_2}\mathbf{e}_{i_{s_2}},\mathbf{e}_{j_{t_2}}\rangle =\langle \mathbf{e}_{i_{s_2}},\mathbf{e}_{j_{t_2}}\rangle =\delta_{i_{s_2}}^{j_{t_2}}.
$$
On the other hand, by \eqref{202508101},
$$
I_1\setminus\{\widetilde{i_1},\cdots,\widetilde{i_{t'}}\}=\left(I_1\setminus\{i_1,\cdots,i_{t}\}\right)\setminus\{i_{t+1},\cdots,i_{t+s}\}=\left((I_1\cap I_2)\setminus \hat I\right)\setminus\{i_{t+1},\cdots,i_{t+s}\}.
$$
Similarly, $
I_2\setminus\{\widetilde{j_1},\cdots,\widetilde{j_{t'}}\}=\left((I_1\cap I_2)\setminus \hat I\right)\setminus\{j_{t+1},\cdots,j_{t+s}\}$. Hence $\{i_{t+1},\cdots,i_{t+s}\}=\{j_{t+1},\cdots,$ $j_{t+s}\}$, and the desired result follows.

Combining the above, we conclude that
\begin{eqnarray*}
&&\left|\begin{array}{cccc}
\langle T\mathbf{e}_{\widetilde{i_1}},\mathbf{e}_{\widetilde{j_1}}\rangle &\langle T\mathbf{e}_{\widetilde{i_2}},\mathbf{e}_{\widetilde{j_1}}\rangle &\cdots&\langle T\mathbf{e}_{\widetilde{i_{t'}}},\mathbf{e}_{\widetilde{j_1}}\rangle \\
\langle T\mathbf{e}_{\widetilde{i_1}},\mathbf{e}_{\widetilde{j_2}}\rangle &\langle T\mathbf{e}_{\widetilde{i_2}},\mathbf{e}_{\widetilde{j_2}}\rangle &\cdots&\langle T\mathbf{e}_{\widetilde{i_{t'}}},\mathbf{e}_{\widetilde{j_2}}\rangle \\
\vdots&\vdots&\ddots&\vdots\\
\langle T\mathbf{e}_{\widetilde{i_1}},\mathbf{e}_{\widetilde{j_{t'}}}\rangle &\langle T\mathbf{e}_{\widetilde{i_2}},\mathbf{e}_{\widetilde{j_{t'}}}\rangle  &\cdots& \langle T\mathbf{e}_{\widetilde{i_{t'}}},\mathbf{e}_{\widetilde{j_{t'}}}\rangle
\end{array}\right|\\
&=&\left|\begin{array}{cccc}
\langle T\mathbf{e}_{i_1},\mathbf{e}_{j_1}\rangle &\langle T\mathbf{e}_{i_2},\mathbf{e}_{j_1}\rangle &\cdots&\langle T\mathbf{e}_{i_{t}},\mathbf{e}_{j_1}\rangle \\
\langle T\mathbf{e}_{i_1},\mathbf{e}_{j_2}\rangle &\langle T\mathbf{e}_{i_2},\mathbf{e}_{j_2}\rangle &\cdots&\langle T\mathbf{e}_{i_{t}},\mathbf{e}_{j_2}\rangle \\
\vdots&\vdots&\ddots&\vdots\\
\langle T\mathbf{e}_{i_1},\mathbf{e}_{j_{t}}\rangle &\langle T\mathbf{e}_{i_2},\mathbf{e}_{j_{t}}\rangle &\cdots&\langle T\mathbf{e}_{i_{t}},\mathbf{e}_{j_{t}}\rangle
\end{array}\right|\cdot \left|\begin{array}{cccc}
1&0&\cdots&0\\
0&1&\cdots&0\\
\vdots&\vdots&\ddots&\vdots\\
0&0&\cdots& 1
\end{array}\right|=\det T.
\end{eqnarray*}
\end{remark}

\begin{proposition}\label{20241011prop1}
Suppose that $I_3\subset \mathbb{N}$ satisfy $I_2\sim I_3$, $T_1\in F(P_{I_1}\ell^2,P_{I_2}\ell^2)$, $T_2\in F(P_{I_2}\ell^2,P_{I_3}\ell^2)$ and $T_2T_1\in F(P_{I_1}\ell^2,P_{I_3}\ell^2)$. Then, $\det(T_2T_1)=\det(T_2)\cdot\det(T_1)$.
\end{proposition}
\begin{proof}
If $I_1=I_2=I_3$, then by \eqref{20241011for3}, we obtain $T_1=T_2=T_3$ which is the identity operator on $P_{I_1}\ell^2$ and hence $\det(T_2T_1)=\det(T_2)\cdot\det(T_1)=1$.

Otherwise, since $I_1\sim I_2$ and $I_2\sim I_3$, by the proof of Lemma \ref{20241011lem2}, we have
\begin{eqnarray*}
&&
I_1\setminus (I_1\cap I_2\cap I_3)=(I_1\setminus I_2)\cup (I_1\setminus I_3)=(L_1\sqcup L_5)\cup (L_1\sqcup L_4)=(I_1\setminus I_2)\sqcup L_4,\\
&&I_2\setminus (I_1\cap I_2\cap I_3)=(I_2\setminus I_1)\cup (I_2\setminus I_3)=(L_2\sqcup L_6)\cup (L_2\sqcup L_4)=(I_2\setminus I_1)\sqcup L_4=L_6\sqcup (I_2\setminus I_3),\\
&&I_3\setminus (I_1\cap I_2\cap I_3)=(I_3\setminus I_1)\cup (I_3\setminus I_2)=(L_3\sqcup L_6)\cup (L_3\sqcup L_5)=L_6\sqcup (I_3\setminus I_2),
\end{eqnarray*}
which gives
$\card(I_1\setminus (I_1\cap I_2\cap I_3)) =\card(I_2\setminus (I_1\cap I_2\cap I_3)) =\card(I_3\setminus (I_1\cap I_2\cap I_3)) \in\mathbb{N}_0$.  Put
\begin{eqnarray*}
I'\triangleq \{i\in I_1\cap I_2\cap I_3: i\leqslant \max(L_1\sqcup L_2\sqcup L_3\sqcup L_4\sqcup L_5\sqcup L_6)\}.
\end{eqnarray*}
Write $s=\card(I_1\setminus (I_1\cap I_2\cap I_3))+\card I'$. Then, $s\in\mathbb{N}_0$ and there is $i_1,\cdots,i_s,i_1',\cdots,i_s',i_1'',\cdots,i_s''\in\mathbb{N}$ such that
\begin{eqnarray*}
&&i_1<\cdots<i_s,\quad i_1'<\cdots <i_s',\quad i_1''<\cdots <i_s'',\\
&&(I_1\setminus (I_1\cap I_2\cap I_3))\sqcup I'=\{i_1,\cdots,i_s\}\subset I_1,\\
&&(I_2\setminus (I_1\cap I_2\cap I_3))\sqcup I'=\{i_1',\cdots,i_s'\}\subset I_2,\\
&& (I_3\setminus (I_1\cap I_2\cap I_3))\sqcup I'=\{i_1'',\cdots,i_s''\}\subset I_3,\\
&&I_1\setminus\{i_1,\cdots,i_s\}=(I_1\cap I_2\cap I_3)\setminus I'=
I_2\setminus\{i_1',\cdots,i_s'\}=
I_3\setminus\{i_1'',\cdots,i_s''\},\\
&&i>\max\{i_s,i_s',i_s''\},\,\,\forall\,\,i\in I_1\setminus\{i_1,\cdots,i_s\}.
\end{eqnarray*}
Note that $$
P_{I_1\cap I_2}T_1=P_{I_1\cap I_2}T_1P_{I_1}
=P_{I_1\cap I_2}T_1(P_{I_1\cap I_2}+P_{I_1\setminus(I_1\cap I_2)})=P_{I_1\cap I_2}T_1P_{I_1\cap I_2}.
$$
Hence, for any $j\in I_2\setminus\{i_1',\cdots,i_s'\}$ and $k\in\{1,\cdots,s\}$, it holds that $P_{I_1\cap I_2}\textbf{e}_{j}=\textbf{e}_{j}$ and
$
\langle T_1\textbf{e}_{i_k},\textbf{e}_{j}\rangle =\langle T_1\textbf{e}_{i_k},P_{I_1\cap I_2}\textbf{e}_{j}\rangle =\langle P_{I_1\cap I_2}T_1\textbf{e}_{i_k},P_{I_1\cap I_2}\textbf{e}_{j}\rangle =\langle P_{I_1\cap I_2}T_1P_{I_1\cap I_2}\textbf{e}_{i_k},\textbf{e}_{j}\rangle =\langle P_{I_1\cap I_2}\textbf{e}_{i_k},\textbf{e}_{j}\rangle =0,
$
where the last equality follows from the fact that if $i_k\in I_1\cap I_2$, then
$\langle P_{I_1\cap I_2}\textbf{e}_{i_k},\textbf{e}_{j}\rangle =\langle \textbf{e}_{i_k},\textbf{e}_{j}\rangle =0$, and if $i_k\notin I_1\cap I_2$, then
$\langle P_{I_1\cap I_2}\textbf{e}_{i_k},\textbf{e}_{j}\rangle =\langle \textbf{0},\textbf{e}_{j}\rangle =0$.
By Remark \ref{20250505rem2}, we have
\begin{eqnarray*}
&&\det (T_2T_1)\\
&=& \left |\begin{array}{cccc}
\langle T_2T_1\textbf{e}_{i_1},\textbf{e}_{i_1''}\rangle &\langle T_2T_1\textbf{e}_{i_2},\textbf{e}_{i_1''}\rangle &\cdots&\langle T_2T_1\textbf{e}_{i_{s}},\textbf{e}_{i_1''}\rangle \\
\langle T_2T_1\textbf{e}_{i_1},\textbf{e}_{i_2''}\rangle &\langle T_2T_1\textbf{e}_{i_2},\textbf{e}_{i_2''}\rangle &\cdots&\langle T_2T_1\textbf{e}_{i_{s}},\textbf{e}_{i_2''}\rangle \\
\vdots&\vdots&\ddots&\vdots\\
\langle T_2T_1\textbf{e}_{i_1},\textbf{e}_{i_{s}''}\rangle &\langle T_2T_1\textbf{e}_{i_2},\textbf{e}_{i_{s}''}\rangle &\cdots&\langle T_2T_1\textbf{e}_{i_{s}},\textbf{e}_{i_{s}''}\rangle
\end{array}\right |\\
&=& \left |\begin{array}{cccc}
\langle T_1\textbf{e}_{i_1},T_2^*\textbf{e}_{i_1''}\rangle &\langle T_1\textbf{e}_{i_2},T_2^*\textbf{e}_{i_1''}\rangle &\cdots&\langle T_1\textbf{e}_{i_{s}},T_2^*\textbf{e}_{i_1''}\rangle \\
\langle T_1\textbf{e}_{i_1},T_2^*\textbf{e}_{i_2''}\rangle &\langle T_1\textbf{e}_{i_2},T_2^*\textbf{e}_{i_2''}\rangle &\cdots&\langle T_1\textbf{e}_{i_{s}},T_2^*\textbf{e}_{i_2''}\rangle \\
\vdots&\vdots&\ddots&\vdots\\
\langle T_1\textbf{e}_{i_1},T_2^*\textbf{e}_{i_{s}''}\rangle &\langle T_1\textbf{e}_{i_2},T_2^*\textbf{e}_{i_{s}''}\rangle  &\cdots& \langle T_1\textbf{e}_{i_{s}},T_2^*\textbf{e}_{i_{s}''}\rangle
\end{array}\right |\\
&=&  \left|\begin{array}{cccc}
\sum\limits_{j\in I_2} \langle \textbf{e}_{j} ,T_2^* \textbf{e}_{i_1''}\rangle \cdot\langle T_1\textbf{e}_{i_1},\textbf{e}_{j}\rangle &\sum\limits_{j\in I_2} \langle \textbf{e}_{j} ,T_2^* \textbf{e}_{i_1''} \rangle \cdot\langle T_1\textbf{e}_{i_2},\textbf{e}_{j}\rangle &\cdots&\sum\limits_{j\in I_2} \langle  \textbf{e}_{j} ,T_2^* \textbf{e}_{i_1''} \rangle \cdot\langle T_1\textbf{e}_{i_s},\textbf{e}_{j}\rangle \\
\sum\limits_{j\in I_2} \langle \textbf{e}_{j} ,T_2^* \textbf{e}_{i_2''} )\cdot\langle T_1\textbf{e}_{i_1},\textbf{e}_{j})&\sum\limits_{j\in I_2}\langle \textbf{e}_{j} ,T_2^* \textbf{e}_{i_2''} \rangle \cdot\langle T_1\textbf{e}_{i_2},\textbf{e}_{j}\rangle &\cdots&\sum\limits_{j\in I_2} \langle \textbf{e}_{j} ,T_2^* \textbf{e}_{i_2''} \rangle \cdot\langle T_1\textbf{e}_{i_s},\textbf{e}_{j}\rangle \\
\vdots&\vdots&\ddots&\vdots\\
\sum\limits_{j\in I_2} \langle \textbf{e}_{j} ,T_2^* \textbf{e}_{i_s''}\rangle \cdot\langle T_1\textbf{e}_{i_1},\textbf{e}_{j}\rangle &\sum\limits_{j\in I_2} \langle \textbf{e}_{j} ,T_2^* \textbf{e}_{i_s''}\rangle \cdot\langle T_1\textbf{e}_{i_2},\textbf{e}_{j}\rangle &\cdots& \sum\limits_{j\in I_2} \langle \textbf{e}_{j} ,T_2^* \textbf{e}_{i_s''}\rangle \cdot\langle T_1\textbf{e}_{i_s},\textbf{e}_{j}\rangle
\end{array}\right |\\
&=&  \left|\begin{array}{cccc}
\sum\limits_{j\in I_2} \langle T_2 \textbf{e}_{j} ,\textbf{e}_{i_1''}\rangle \cdot\langle T_1\textbf{e}_{i_1},\textbf{e}_{j}\rangle &\sum\limits_{j\in I_2} \langle T_2 \textbf{e}_{j} ,\textbf{e}_{i_1''}\rangle \cdot\langle T_1\textbf{e}_{i_2},\textbf{e}_{j}\rangle &\cdots&\sum\limits_{j\in I_2}\langle T_2 \textbf{e}_{j} ,\textbf{e}_{i_1''}\rangle \cdot\langle T_1\textbf{e}_{i_s},\textbf{e}_{j}\rangle \\
\sum\limits_{j\in I_2}\langle T_2 \textbf{e}_{j} ,\textbf{e}_{i_2''}\rangle \cdot\langle T_1\textbf{e}_{i_1},\textbf{e}_{j}\rangle &\sum\limits_{j\in I_2}\langle T_2 \textbf{e}_{j} ,\textbf{e}_{i_2''}\rangle \cdot\langle T_1\textbf{e}_{i_2},\textbf{e}_{j}\rangle &\cdots&\sum\limits_{j\in I_2}\langle T_2 \textbf{e}_{j} ,\textbf{e}_{i_2''} \rangle \cdot\langle T_1\textbf{e}_{i_s},\textbf{e}_{j}\rangle \\
\vdots&\vdots&\ddots&\vdots\\
\sum\limits_{j\in I_2}\langle T_2 \textbf{e}_{j} ,\textbf{e}_{i_s''}\rangle \cdot\langle T_1\textbf{e}_{i_1},\textbf{e}_{j}\rangle &\sum\limits_{j\in I_2} \langle T_2 \textbf{e}_{j} ,\textbf{e}_{i_s''} \rangle \cdot\langle T_1\textbf{e}_{i_2},\textbf{e}_{j}\rangle  &\cdots& \sum\limits_{j\in I_2} \langle T_2 \textbf{e}_{j} ,\textbf{e}_{i_s''}\rangle \cdot\langle T_1\textbf{e}_{i_s},\textbf{e}_{j}\rangle
\end{array}\right |\\
&=&  \left|\begin{array}{cccc}
\sum\limits_{j=1}^{s} \langle T_2 \textbf{e}_{i_j'} ,\textbf{e}_{i_1''}\rangle \cdot\langle T_1\textbf{e}_{i_1},\textbf{e}_{i_j'}\rangle &\sum\limits_{j=1}^{s}  \langle T_2 \textbf{e}_{i_j'} ,\textbf{e}_{i_1''}\rangle \cdot\langle T_1\textbf{e}_{i_2},\textbf{e}_{i_j'}\rangle &\cdots&\sum\limits_{j=1}^{s} \langle T_2 \textbf{e}_{i_j'} ,\textbf{e}_{i_1''}\rangle \cdot\langle T_1\textbf{e}_{i_s},\textbf{e}_{i_j'}\rangle \\
\sum\limits_{j=1}^{s}  \langle T_2 \textbf{e}_{i_j'} ,\textbf{e}_{i_2''}\rangle \cdot\langle T_1\textbf{e}_{i_1},\textbf{e}_{i_j'}\rangle &\sum\limits_{j=1}^{s} \langle T_2 \textbf{e}_{i_j'} ,\textbf{e}_{i_2''} \rangle \cdot\langle T_1\textbf{e}_{i_2},\textbf{e}_{i_j'})&\cdots&\sum\limits_{j=1}^{s} \langle T_2 \textbf{e}_{i_j'} ,\textbf{e}_{i_2''} \rangle \cdot\langle T_1\textbf{e}_{i_s},\textbf{e}_{i_j'}\rangle \\
\vdots&\vdots&\ddots&\vdots\\
\sum\limits_{j=1}^{s}  \langle T_2 \textbf{e}_{i_j'} ,\textbf{e}_{i_s''} \rangle \cdot\langle T_1\textbf{e}_{i_1},\textbf{e}_{i_j'}\rangle &\sum\limits_{j=1}^{s} \langle T_2 \textbf{e}_{i_j'} ,\textbf{e}_{i_s''} \rangle \cdot\langle T_1\textbf{e}_{i_2},\textbf{e}_{i_j'}\rangle  &\cdots& \sum\limits_{j=1}^{s} \langle T_2 \textbf{e}_{i_j'} ,\textbf{e}_{i_s''}\rangle \cdot\langle T_1\textbf{e}_{i_s},\textbf{e}_{i_j'}\rangle
\end{array}\right |\\
&=& \left|\begin{array}{cccc}
\langle T_2 \textbf{e}_{i_1'} ,\textbf{e}_{i_1''} \rangle & \langle T_2 \textbf{e}_{i_2'} ,\textbf{e}_{i_1''} \rangle &\cdots& \langle T_2 \textbf{e}_{i_s'} ,\textbf{e}_{i_1''} \rangle \\
\langle T_2 \textbf{e}_{i_1'} ,\textbf{e}_{i_2''} \rangle &\langle T_2 \textbf{e}_{i_2'} ,\textbf{e}_{i_2''} \rangle &\cdots& \langle T_2 \textbf{e}_{i_s'} ,\textbf{e}_{i_2''} \rangle \\
\vdots&\vdots&\ddots&\vdots\\
 \langle T_2 \textbf{e}_{i_1'} ,\textbf{e}_{i_s''} \rangle &\langle T_2 \textbf{e}_{i_2'} ,\textbf{e}_{i_s''} \rangle &\cdots& \langle T_2 \textbf{e}_{i_s'} ,\textbf{e}_{i_s''} \rangle
\end{array}\right |\cdot
\left|\begin{array}{cccc}
  \langle T_1 \textbf{e}_{i_1} ,\textbf{e}_{i_1'} \rangle &\langle T_1 \textbf{e}_{i_2} ,\textbf{e}_{i_1'} \rangle &\cdots& \langle T_1 \textbf{e}_{i_s} ,\textbf{e}_{i_1'} \rangle \\
  \langle T_1 \textbf{e}_{i_1} ,\textbf{e}_{i_2'} \rangle &\langle T_1 \textbf{e}_{i_2} ,\textbf{e}_{i_2'} \rangle &\cdots& \langle T_1 \textbf{e}_{i_s} ,\textbf{e}_{i_2'} \rangle \\
\vdots&\vdots&\ddots&\vdots\\
  \langle T_1 \textbf{e}_{i_1} ,\textbf{e}_{i_s'} \rangle &\langle T_1 \textbf{e}_{i_2} ,\textbf{e}_{i_s'}\rangle &\cdots& \langle T_1 \textbf{e}_{i_s} ,\textbf{e}_{i_s'}\rangle
\end{array}\right |\\
&=&\det (T_2)\cdot \det (T_1),
\end{eqnarray*}
where $T_2^*$ stands for the Hilbert dual of $T_2$. This completes the proof of Proposition \ref{20241011prop1}.
\end{proof}

\begin{lemma}\label{20250105lem1}
Suppose that $\mathbf{x}\in S$, $(S\cap U_1,I_1,f_1)$ is a local coordinate triple of $S$ such that $\mathbf{x}\in S\cap U_1$, and $I_2\sim I_1$. Then,
$$
D (P_{I_2}P_{I_1}^{-1})(\mathbf{x}_{I_1})\in F(P_{I_1}\ell^2,P_{I_2}\ell^2),
$$
where $\mathbf{x}_{I_1}=P_{I_1}\mathbf{x}$.
\end{lemma}
\begin{proof}
By Definition \ref{20241121def1000}, there is an open neighborhood $U_{I_1}$ of $\textbf{x}_{I_1}$ in $P_{I_1}\ell^2$, $f_1\in C^1_{P_{I_1}}(U_{I_1};P_{\mathbb{N}\setminus I_1}\ell^2)$, and
$$
S\cap U_1=\left\{\textbf{y}_{I_1}+f_1(\textbf{y}_{I_1}):\textbf{y}_{I_1}\in U_{I_1}\right\},
\qquad \textbf{x}=\textbf{x}_{I_1}+f_1(\textbf{x}_{I_1}).
$$
Since $I_1\sim I_2$, one can find $K_1\subset I_1$ and $K_2 \subset \mathbb{N}\setminus I_1$ such that $\card K_1=\card K_2\in\mathbb{N}_0$ and $I_2=(I_1\setminus K_1)\cup K_2$.
Hence,
$$
 P_{I_2}P_{I_1}^{-1}\textbf{y}_{I_1}=P_{I_2}(\textbf{y}_{I_1}+f_1(\textbf{y}_{I_1}))=\textbf{y}_{{I_1}\setminus K_1}+\sum_{j\in K_2}f_{1,j}(\textbf{y}_{I_1})\textbf{e}_j,\quad\forall\;\textbf{y}_{I_1}=\sum\limits_{i\in {I_1}}y_i\textbf{e}_i\in U_{I_1},
$$
where $\textbf{y}_{{I_1}\setminus K_1}=\sum\limits_{i\in {I_1}\setminus K_1}y_i\textbf{e}_i$ and $f_{1,j}(\textbf{y}_{I_1})\triangleq(f_1(\textbf{y}_{I_1}),\textbf{e}_j)$ for any $j\in K_2$. Thus, $D (P_{I_2}P_{{I_1}}^{-1})(\textbf{x}_{I_1})$ is a bounded linear operator from $P_{I_1}\ell^2$ into $P_{I_2}\ell^2$, and
$$
P_{I_1\cap I_2}D (P_{I_2}P_{{I_1}}^{-1})(\textbf{x}_{I_1}) P_{I_1\cap I_2} =id_{P_{I_1\cap I_2}},\quad
P_{I_1\cap I_2}D (P_{I_2}P_{{I_1}}^{-1})(\textbf{x}_{I_1}) P_{I_1\setminus(I_1\cap I_2) } =0,
$$
and therefore $D (P_{I_2}P_{I_1}^{-1})(\textbf{x}_{I_1})\in F(P_{I_1}\ell^2,P_{I_2}\ell^2)$, which completes the proof of Lemma \ref{20250105lem1}.
\end{proof}

\begin{corollary}\label{20241015cor1}
Under the assumptions of Lemma \ref{20250105lem1}, suppose that $(S\cap U_2,I_2,f_2)$ is another local coordinate triple of $S$ such that $\mathbf{x}\in S\cap U_1 \cap U_2$, $I_3\subset \mathbb{N}$ and $I_2\sim I_3$. Then
\begin{eqnarray}\label{20250105for1}
\det\left(D(P_{I_3}P_{I_1}^{-1})(\mathbf{x}_{I_1})\right)
=\det \left(D(P_{I_3}P_{I_2}^{-1})(\mathbf{x}_{I_2})\right)\cdot \det \left(D(P_{I_2}P_{I_1}^{-1})(\mathbf{x}_{I_1})\right),
\end{eqnarray}
where $\mathbf{x}_{I_1}=P_{I_1}\mathbf{x}$ and $\mathbf{x}_{I_2}=P_{I_2}\mathbf{x}$.
\end{corollary}
\begin{proof}
Clearly, $P_{I_3}P_{I_1}^{-1}=(P_{I_3}P_{I_2}^{-1})(P_{I_2}P_{I_1}^{-1})$. Hence by the chain rule, we have
$$
 D(P_{I_3}P_{I_1}^{-1})(\textbf{x}_{I_1})
= D(P_{I_3}P_{I_2}^{-1})(\textbf{x}_{I_2}) \circ  D(P_{I_2}P_{I_1}^{-1})(\textbf{x}_{I_1}) .
$$
Combining Proposition \ref{20241011prop1} and Lemma \ref{20250105lem1}, we obtain \eqref{20250105for1}.
\end{proof}

Since we shall construct at the first step a family of compatible local measures according to the given local coordinate triples, the choose of local coordinate triples for $S$ is of crucial importance. The following example indicates that if two local coordinate triples for a surface are chosen to be ``too different" from each other, then the corresponding local measures may be mutually singular and hence they are not compatible.

\begin{example}\label{20250204exa1}
Let $I=\{2n-1:n\in\mathbb{N}\},\,J=\{2n:n\in\mathbb{N}\}$ and
$$
S=\left\{\sum_{j=1}^{\infty}\left(x_j\mathbf{e}_{2j-1}+x_j\mathbf{e}_{2j}\right):\;x_j\in\mathbb{R}\hbox{ for each }j\in\mathbb{N}\hbox{ and }\sum_{j=1}^{\infty}x_j^2<\infty\right\}.
$$
Let
\begin{eqnarray*}
U_I&\triangleq &\left\{\sum\limits_{j=1}^{\infty} x_j\mathbf{e}_{2j-1}:\;x_j\in\mathbb{R}\hbox{ for each }j\in\mathbb{N}\hbox{ and }\sum\limits_{j=1}^{\infty}x_j^2<\infty\right\},\\ U_J&\triangleq &\left\{\sum\limits_{j=1}^{\infty} x_j\mathbf{e}_{2j}:\;x_j\in\mathbb{R}\hbox{ for each }j\in\mathbb{N}\hbox{ and }\sum\limits_{j=1}^{\infty}x_j^2<\infty\right\},
\end{eqnarray*}
and
\begin{eqnarray*}
f_I(\mathbf{x}_{I})&\triangleq & \sum\limits_{j=1}^{\infty} x_j\mathbf{e}_{2j},\quad \forall\, \mathbf{x}_{I}=\sum\limits_{j=1}^{\infty} x_j\mathbf{e}_{2j-1}\in U_I,\\
f_J(\mathbf{x}_{J})&\triangleq & \sum\limits_{j=1}^{\infty} x_j\mathbf{e}_{2j-1},\quad \forall\, \mathbf{x}_{J}=\sum\limits_{j=1}^{\infty} x_j\mathbf{e}_{2j}\in U_J.
\end{eqnarray*}
Then $(S\cap \ell^2,I,f_I)$ and $(S\cap \ell^2,J,f_J)$ are two local coordinate triples of $S$.  Denote by $\mu_{I}$ and $\mu_{J}$ respectively the restriction of the product measure $\prod\limits_{i=1}^{\infty}\frac{e^{-\frac{x_i^2}{4a_i^2}}}{\sqrt{4\pi a_i^2}}$ on $P_{I}S$ and $P_{J}S$, where $\{a_k\}_{k=1}^{\infty}$ was give in \eqref{20250130for1}. If we identify $S$ with $P_{I}S$ and $P_{J}S$, and $a_{2j-1}=2a_{2j}$ for all $j\in\mathbb{N}$, then by the same arguments as in the proof of Proposition \ref{230421prop1}, one can see that $f\cdot\mu_{I}$ and $g\cdot\mu_{J}$ are mutually singular for any two non-negative, Borel measurable functions $f$ and $g$ on $S$.
\end{example}

Note that ${I}\nsim {J}$ in Example \ref{20250204exa1}. In order to avoid the above phenomenon, we add the condition that the indexes in any two local coordinate triples are equivalent in the sense of Definition
\ref{def3.1}. We shall explain more why we choose equivalent indexes in a more general context in our forthcoming work \cite{WYZ-c}.

The following quantity will play a key role in the construction of a family of compatible local measures on the $C^1$-differentiable surface $S$ with codimension $\Gamma_{\mathbb{N}\setminus I_0}$ (Recall Definition \ref{20241121def1}).
\begin{definition}\label{20241012def1}
For each $\mathbf{x}\in S$ and local coordinate triple $(S\cap U, I,f)$ of $S$ satisfying $\mathbf{x}\in S\cap U$,
write
\begin{equation}\label{20260225e1}
n_{I}(\mathbf{x})\triangleq \sqrt{ 1+\sum_{K_1\subset I, \,K_2 \subset \mathbb{N}\setminus I, \,\card K_1=\card K_2\,\in\mathbb{N}} \left|\det\left( D_{x_i} f_j(\mathbf{x}_{I})\right)_{i\in K_1,\,j\in K_2}\right|^2},
\end{equation}
where $\mathbf{x}_{I}=P_{I}\mathbf{x}$ and $f_j(\mathbf{x}_{I})\triangleq \langle f(\mathbf{x}_{I}),\mathbf{e}_j\rangle$ (for $j\in \mathbb{N}\setminus I$). Here and henceforth, unless otherwise stated, the elements of the sets $K_1$ and $K_2$ are arranged in ascending order.
\end{definition}

\begin{remark}\label{20250208rem1}
Let us, for the time being, return to the case of finite dimensions. Suppose that, for example, $S$ is a surface in $\mathbb{R}^3$ with codimension 1, i.e., for any $\mathbf{x}_0=(x_1^0,x_2^0,x_3^0)\in S$, there exists an open neighborhood $U$ of $\mathbf{x}_0$ in $\mathbb{R}^3$ and $g\in C^1_{\mathbb{R}^3}(U)$ such that $D_{x_1}g\neq 0$ on $U$ and $S\cap U=\{\mathbf{x}\in U:g(\mathbf{x})=0\}$. Then, by the Implicit Function Theorem, there exists an open neighborhood $U_{\{2,3\}}\subset \mathbb{R}^2$ and $f\in C^1_{\mathbb{R}^2}(U_{\{2,3\}})$ such that $f(x_2^0,x_3^0)=x_1^0$ and $S\cap U=\{(f(x_2,x_3),x_2,x_3):(x_2,x_3)\in U_{\{2,3\}}\}$. By $g(f(x_2,x_3),x_2,x_3)=0$, for any $(x_2,x_3)\in U_{\{2,3\}}$, it follows that
\begin{eqnarray}\label{20250208for1}
D_{x_1}g(f(x_2,x_3),x_2,x_3)\cdot D_{x_2}f(x_2,x_3)+D_{x_2}g(f(x_2,x_3),x_2,x_3)=0,
\end{eqnarray}
and
\begin{eqnarray}\label{20250208for2}
D_{x_1}g(f(x_2,x_3),x_2,x_3)\cdot D_{x_3}f(x_2,x_3)+D_{x_3}g(f(x_2,x_3),x_2,x_3)=0.
\end{eqnarray}
Then let
\begin{eqnarray*}
M_1&\triangleq&\span\{Dg(f(x_2,x_3),x_2,x_3)\}\\
&=&\{t\cdot(D_{x_1}g(f(x_2,x_3),x_2,x_3),D_{x_2}g(f(x_2,x_3),x_2,x_3),D_{x_3}g(f(x_2,x_3),x_2,x_3):t\in\mathbb{R}\},
\end{eqnarray*}
and
\begin{eqnarray*}
M_2&\triangleq& \{ (Df(x_2,x_3)(y_2,y_3),y_2,y_3):(y_2,y_3)\in\mathbb{R}^2\}\\
&=& \{ (D_{x_2}f(x_2,x_3)y_2+D_{x_3}f(x_2,x_3)y_3,y_2,y_3):(y_2,y_3)\in\mathbb{R}^2\}.
\end{eqnarray*}
Then \eqref{20250208for1} together with \eqref{20250208for2} imply that $M_1$ and $M_2$ are orthogonal to each other. Similar to the arguments before Corollary \ref{20240705cor1}, one can see that $M_1$ and $M_2$ are independent of $f$ and $g$. Actually, $M_1$ and $M_2$ can be viewed as the space spanned by the ``normal vectors" and  the ``tangent place" of $S$ at $\mathbf{x}=(f(x_2,x_3),x_2,x_3))$, respectively.

Furthermore, write
$\frac{Dg(\mathbf{x})}{||Dg(\mathbf{x})||_{\mathbb{R}^3}}=(\theta_1(\mathbf{x}),\theta_2(\mathbf{x}),\theta_3(\mathbf{x}))$.
By \eqref{20250208for1}--\eqref{20250208for2},
$$
|\theta_1(f(x_2,x_3),x_2,x_3))|=\frac{1}{\sqrt{1+|D_{x_2}f(x_2,x_3)|^2+|D_{x_3}f(x_2,x_3)|^2}},
$$
the surface area of $S\cap U$ equals
\begin{eqnarray*}
\int_{ U_{\{2,3\}}}\frac{1}{|\theta_1(f(x_2,x_3),x_2,x_3))|}\,\mathrm{d}x_2\mathrm{d}x_3,
\end{eqnarray*}
and hence the quantity $n_{I}(\mathbf{x})$ given in Definition \ref{20241012def1} is an infinite-dimensional version of the following quantity
\begin{eqnarray}\label{20250208for3}
n_{\{2,3\}}(\mathbf{x})=\frac{1}{|\theta_1(\mathbf{x})|}.
\end{eqnarray}
\end{remark}

Several examples are in order.

\begin{example}\label{20250205exa1}
Suppose that $S$ is a $C^1$-differentiable surface of $\ell^2$ with codimension $\Gamma_{\{i_0\}}$ for some $i_0\in\mathbb{N}$ (In this case, we choose $I_0=\mathbb{N}\setminus\{i_0\}$). Then, $S$ is the usual $C^1$-differentiable surface of $\ell^2$ with codimension 1. If $(S\cap U, \mathbb{N}\setminus\{i\},f)$ (for some $i\in\mathbb{N}$) is a local coordinate triple of $S$ (Recall Definition \ref{20241121def1}), then
\begin{eqnarray}\label{20250205for2}
n_{\mathbb{N}\setminus\{i\}}(\mathbf{x})
=\sqrt{ 1+\sum_{k\in \mathbb{N}\setminus\{i\}} \left| D_{x_k} f_{i}(\mathbf{x}_{\mathbb{N}\setminus\{i\}})\right|^2}
=\sqrt{ 1+ \left|\left| Df_{i}(\mathbf{x}_{\mathbb{N}\setminus\{i\}})\right|\right|^2},
\end{eqnarray}
where $\mathbf{x}\in S\cap U$, $\mathbf{x}_{\mathbb{N}\setminus\{i\}}=P_{\mathbb{N}\setminus\{i\}}\mathbf{x}$, $f_{i}(\mathbf{x}_{\mathbb{N}\setminus\{i\}})= \langle f(\mathbf{x}_{\mathbb{N}\setminus\{i\}}),\mathbf{e}_{i}\rangle$, $||\cdot||$ means the corresponding operator norm and hence the above function is a locally continuous function. Furthermore, if for any $\mathbf{x}_0\in S$, there exists an open neighborhood $U_{\mathbb{N}\setminus\{i\}}$ of $P_{\mathbb{N}\setminus\{i\}}\mathbf{x}_0$ in $P_{\mathbb{N}\setminus\{i\}}\ell^2$ and $g\in C^1_{ \ell^2}(U )$ such that $D_{x_{i}}g\neq 0$ on $U$ and $S\cap U=\{\mathbf{x}\in U:g(\mathbf{x})=0\}$, then similar to the arguments in Remark \ref{20250208rem1}, for $\frac{D g(\mathbf{x})}{||D g(\mathbf{x})||_{\ell^2}}\triangleq (\theta_j(\mathbf{x}))_{j=1}^{\infty}$  for $\mathbf{x}\in S\cap U$, it follows that
$$
n_{\mathbb{N}\setminus\{i\}}(\mathbf{x})=\frac{1}{|\theta_{i}(\mathbf{x})|},
$$
which enjoys the same form as its three-dimensional counterpart at \eqref{20250208for3}.
\end{example}
\begin{example}\label{20250110exa1}
Suppose that $S$ is a $C^1$-differentiable surface of $\ell^2$ with codimension $\Gamma_{\{1,2,\cdots,n\}}$ for some $n\in\mathbb{N}$  (In this case, we choose $I_0=\mathbb{N}\setminus\{1,2,\cdots,n\}$). Then $S$ is the usual $C^1$-differentiable surface of $\ell^2$ with codimension $n$.  If $(S\cap U, \mathbb{N}\setminus\{1,2,\cdots,n\},f)$ is a local coordinate triple of $S$ and $f\in C_{P_{I_0}\ell^2}^1(U_{I_0};P_{\mathbb{N}\setminus I_0}\ell^2)$ be the corresponding local graph map, then
\begin{equation}\label{20250111for1}
\begin{array}{ll}\displaystyle
n_{\mathbb{N}\setminus\{1,2,\cdots,n\}}(\mathbf{x})\\[3mm]
\displaystyle=\sqrt{ 1+\sum_{k=1}^{n}\sum_{I\subset\{1,2,\cdots,n\}, \card I=k}\sum_{J \subset\mathbb{N}\setminus\{1,2,\cdots,n\}, \card J=k} \left|\det \left( D_{x_j} f_i(\mathbf{x}_{\mathbb{N}\setminus\{1,2,\cdots,n\}})\right)_{i\in I,j\in J}\right|^2}\\[7mm]
\displaystyle\leqslant\sqrt{ 1+\sum_{k=1}^{n}\sum_{I\subset \{1,2,\cdots,n\}, \card I=k}\left(\prod_{i\in I}\left|\left| Df_i(\mathbf{x}_{\mathbb{N}\setminus\{1,2,\cdots,n\}})\right|\right|^2\right)},
\end{array}
\end{equation}
where $\mathbf{x}\in S\cap U$, $\mathbf{x}_{\mathbb{N}\setminus\{1,2,\cdots,n\}}=P_{\mathbb{N}\setminus\{1,2,\cdots,n\}}\mathbf{x}$, $f_{i}(\mathbf{x}_{\mathbb{N}\setminus\{1,2,\cdots,n\}})= (f(\mathbf{x}_{\mathbb{N}\setminus\{1,2,\cdots,n\}}),\mathbf{e}_{i})$, $||\cdot||$ means the corresponding operator norm and the last function in above inequality is a locally continuous function. For the reader's convenience, we give the proof of \eqref{20250111for1} below. Let us recall that
\begin{itemize}
\item[(1)]The Cauchy-Binet formula (\cite{Bin, Cau}): Suppose that $m\in\mathbb{N},\,m>n$, $A$ is an $n\times m$ matrix and $A^\top$ is the transpose of $A$. Then,
$$
\det (AA^{\top})=\sum_{1\leqslant j_1<\cdots<j_n\leqslant m}\left(\det A_{j_1,\cdots,j_n}\right)^2,
$$
where $A_{j_1,\cdots,j_n}$ denotes the $n\times n$ matrix consisting of the columns $j_1,\cdots,j_n$ of $A$;

\item[(2)]The Hadamard inequality (\cite{Had}): Suppose that $A=(a_{i,j})_{n\times n}$ is an $n\times n$ positive matrix (e.g., \cite[p. 4]{Bha} for the definition). Then,
$$
\det A\leqslant  \prod_{i=1}^{n}a_{i,i}.
$$
\end{itemize}
Now for $k\in\{1,2,\cdots,n\}$ and $I=\{i_1,i_2,\cdots,i_k\}\subset\{1,2,\cdots,n\}$ with $i_1<i_2<\cdots<i_k$, we see that
\begin{eqnarray*}
&&\sum_{J \subset\mathbb{N}\setminus\{1,2,\cdots,n\}, \card J=k} \left|\det \left( D_{x_j} f_i(\mathbf{x}_{\mathbb{N}\setminus\{1,2,\cdots,n\}})\right)_{i\in I,j\in J}\right|^2\\
&&=\lim_{m\to\infty}\sum_{J=\{j_1,\cdots,j_k\},\,n<j_1<\cdots<j_k\leqslant n+m} \left|\det \left( D_{x_j} f_i(\mathbf{x}_{\mathbb{N}\setminus\{1,2,\cdots,n\}})\right)_{i\in I,j\in J}\right|^2.
\end{eqnarray*}
Let $A_m=(a_{r,s})_{k\times m}$, where $a_{r,s}=D_{x_{n+s}} f_{i_r}(\mathbf{x}_{\mathbb{N}\setminus\{1,2,\cdots,n\}})$ for any $r=1,\cdots, k$ and $s=1,\cdots,m$. Then, for $m>k$, by the Cauchy-Binet formula, it follows that
\begin{eqnarray*}
\det (A_mA_m^\top)
&=&\sum_{J=\{j_1,\cdots,j_k\},\, n<j_1<\cdots<j_k\leqslant n+m} \left|\det\left( D_{x_j} f_i(\mathbf{x}_{\mathbb{N}\setminus\{1,2,\cdots,n\}})\right)_{i\in I,j\in J}\right|^2.
\end{eqnarray*}
Note that the $A_mA_m^\top$ is a positive matrix with the $(r,r)$-th ($r=1,2,\cdots,k$) element given by
\begin{eqnarray*}
\sum_{s=1}^{m} \left| D_{x_{n+s}} f_{i_r}(\mathbf{x}_{\mathbb{N}\setminus\{1,2,\cdots,n\}})\right|^2=\sum_{n<j\leqslant n+m} \left| D_{x_j} f_{i_r}(\mathbf{x}_{\mathbb{N}\setminus\{1,2,\cdots,n\}}) \right|^2.
\end{eqnarray*}
Hence, the Hadamard inequality implies that
\begin{eqnarray*}
&&\sum_{J=\{j_1,\cdots,j_k\},\,n<j_1<\cdots<j_k\leqslant n+m} \left|\det \left( D_{x_j} f_i(\mathbf{x}_{\mathbb{N}\setminus\{1,2,\cdots,n\}})\right)_{i\in I,j\in J}\right|^2\\
&&\leqslant \prod_{i\in I}\left(\sum_{n<j\leqslant n+m} \left| D_{x_j} f_i(\mathbf{x}_{\mathbb{N}\setminus\{1,2,\cdots,n\}}) \right|^2\right)\\
&&\leqslant \prod_{i\in I}\left(\sum_{n<j<\infty} \left| D_{x_j} f_i(\mathbf{x}_{\mathbb{N}\setminus\{1,2,\cdots,n\}}) \right|^2\right)=\prod_{i\in I}\left(  \left|\left| D f_i(\mathbf{x}_{\mathbb{N}\setminus\{1,2,\cdots,n\}})\right| \right|^2\right).
\end{eqnarray*}
Letting $m\to\infty$ in the above, we arrive at
\begin{eqnarray*}
\sum_{J \subset\mathbb{N}\setminus\{1,2,\cdots,n\}, \card J=k} \left|\det \left( D_{x_j} f_i(\mathbf{x}_{\mathbb{N}\setminus\{1,2,\cdots,n\}})\right)_{i\in I,j\in J}\right|^2
&\leqslant&\prod_{i\in I}\left(  \left|\left| D f_i(\mathbf{x}_{\mathbb{N}\setminus\{1,2,\cdots,n\}})\right| \right|^2\right).
\end{eqnarray*}
\end{example}

The following example provides a surface with infinite codimension£¬ for which the quantity \eqref{20260225e1} in Definition \ref{20241012def1} is a positive constant.
\begin{example}\label{20250205exa2}
Let $I_0=\{2n-1:n\in\mathbb{N}\}$ and $J_0=\mathbb{N}\setminus I_0$. Choose
$\mathbf{x}^0_{J_0}=\sum\limits_{j\in J_0}x_j^0\mathbf{e}_j$ and $\Delta\mathbf{x}^0_{J_0}=\sum\limits_{j\in J_0}(\Delta x_j^0)\mathbf{e}_j\in P_{J_0}\ell^2$ with $\Delta x_j^0\neq 0$ for all $j\in J_0$ such that
\begin{eqnarray*}
\sum_{j\in J_0}\left|\ln \frac{1}{\sqrt{2\pi} a_j}-\frac{(x_j^0)^2}{2a_j^2}\right|<\infty,\qquad\sum_{j\in J_0}\frac{|\Delta x_j^0|^2}{a_j^4}<\infty.
\end{eqnarray*}
Let
\begin{eqnarray*}
S=\left\{\mathbf{x}_{I_0}+\mathbf{x}^0_{J_0}+x_1\cdot\Delta\mathbf{x}^0_{J_0}:\mathbf{x}_{I_0}=\sum\limits_{i\in I_0}x_i\mathbf{e}_i\in P_{I_0}\ell^2\right\},\qquad U_{I_0}=P_{I_0}\ell^2,
\end{eqnarray*}
and $f(\mathbf{x}_{I_0})=\mathbf{x}^0_{J_0}+x_1\cdot\Delta\mathbf{x}^0_{J_0}$ for every $\mathbf{x}_{I_0}=\sum\limits_{i\in I_0}x_i\mathbf{e}_i\in U_{I_0}$.
Then $S$ is a surface of $\ell^2$ with codimension $\Gamma_{J_0}$ and $(S,I_0,f)$ is a local coordinate triple of $S$. It is easy to see that
\begin{eqnarray}\label{20250205for3}
n_{I_0}(\mathbf{x})=\sqrt{ 1+\sum_{j\in J_0} \left|\Delta x_j^0\right|^2}=\sqrt{ 1+ ||\Delta\mathbf{x}^0_{J_0}||^2},
\quad\forall \; \mathbf{x}=\mathbf{x}_{I_0}+\mathbf{x}^0_{J_0}+x_1\cdot\Delta\mathbf{x}^0_{J_0}\in S.
\end{eqnarray}
\end{example}

\begin{lemma}\label{20241016rem1}
The $n_{I }(\mathbf{x})$ in Definition \ref{20241012def1} can be rewritten as follows:
\begin{eqnarray}\label{20250205for6}
n_{I }(\mathbf{x})\triangleq \sqrt{ \sum_{I'\in \Gamma_{I_0}} \left|\det (D (P_{I'}P_{I }^{-1})(\mathbf{x}_{I }))\right|^2}.
\end{eqnarray}
\end{lemma}
\begin{proof}
For each $ K_1\subset I $ and $K_2 \subset \mathbb{N}\setminus I$ with $s\triangleq\card K_1=\card K_2\in\mathbb{N}$, let $I'=(I \setminus K_1)\cup K_2$. Then $I'\in \Gamma_{I_0}$
and hence by  Lemma \ref{20250105lem1}, $ D (P_{I'}P_{I }^{-1})(\textbf{x}_{I })\in F(P_{I}\ell^2,P_{I'}\ell^2)$. Similar to Remark \ref{20250505rem1}, put
\begin{eqnarray*}
\tilde I\triangleq \{i\in I\cap I': i\leqslant \max((I\setminus I')\sqcup(I'\setminus I))\}.
\end{eqnarray*}
Write $t\triangleq \card((I \setminus I')\sqcup \tilde I)\in\mathbb{N}$. Choose $i_1,\ldots,i_{t},j_1,\ldots,j_{t}\in\mathbb{N}$
such that $i_1<\cdots<i_{t}$, $j_1<\cdots<j_{t}$ and
\begin{eqnarray}\label{20260702for6}
(I \setminus I')\sqcup \tilde I=\{i_1,\ldots,i_{t}\},\qquad (I'\setminus I )\sqcup \tilde I=\{j_1,\ldots,j_{t}\}.
\end{eqnarray}
Further, choose $i_1',\ldots,i_{s}',j_1',\ldots,j_{s}',k_1,\ldots,k_{t-s}\in\mathbb{N}$
such that $i_1'<\cdots<i_{s}'$, $j_1'<\cdots<j_{s}'$, $k_1<\cdots<k_{t-s}$, and
\begin{eqnarray}\label{20260702for7}
\tilde I=\{k_1,\ldots,k_{t-s}\},\qquad I \setminus I'=\{i_1',\ldots,i_{s}'\},\qquad  I'\setminus I =\{j_1',\ldots,j_{s}'\}.
\end{eqnarray}
By $ D (P_{I'}P_{I }^{-1})(\textbf{x}_{I })\in F(P_{I}\ell^2,P_{I'}\ell^2)$, for any $s_2,t_2\in \tilde I$, one has
\begin{eqnarray}\label{20260702for1}
\langle D (P_{I'}P_{I }^{-1})(\textbf{x}_{I })\textbf{e}_{s_2},\textbf{e}_{t_2}\rangle&=&\langle D (P_{I'}P_{I }^{-1})(\textbf{x}_{I })P_{I\cap I'}\textbf{e}_{s_2},P_{I\cap I'}\textbf{e}_{t_2}\rangle\\
&=&\langle P_{I\cap I'}D (P_{I'}P_{I }^{-1})(\textbf{x}_{I })P_{I\cap I'}\textbf{e}_{s_2},\textbf{e}_{t_2}\rangle
=\langle \textbf{e}_{s_2},\textbf{e}_{t_2}\rangle=\delta_{s_2 t_2},\nonumber
\end{eqnarray}
and for any $s_2 \in \tilde I$ and $t_2\in I \setminus I'$, one has
\begin{eqnarray}\label{20260702for2}
\langle D (P_{I'}P_{I }^{-1})(\textbf{x}_{I })\textbf{e}_{t_2},\textbf{e}_{s_2}\rangle&=&\langle D (P_{I'}P_{I }^{-1})(\textbf{x}_{I })P_{{I }\setminus (I \cap I' )}\textbf{e}_{t_2},P_{I\cap I'}\textbf{e}_{s_2}\rangle\\
&=&\langle P_{I\cap I'}D (P_{I'}P_{I }^{-1})(\textbf{x}_{I })P_{{I }\setminus (I \cap I' )}\textbf{e}_{t_2},\textbf{e}_{s_2}\rangle=\langle \textbf{0},\textbf{e}_{s_2}\rangle=0.\nonumber
\end{eqnarray}
By Definition \ref{20241113def1},  $\det (D (P_{I'}P_{I }^{-1})(\textbf{x}_{I }))$ equals the following determinant
\begin{eqnarray}\label{20260702for4}
 \left|\begin{array}{cccc}
\langle D (P_{I'}P_{I }^{-1})(\textbf{x}_{I })\textbf{e}_{i_1},\textbf{e}_{j_1}\rangle&\langle D (P_{I'}P_{I }^{-1})(\textbf{x}_{I })\textbf{e}_{i_2},\textbf{e}_{j_1}\rangle&\cdots&\langle D \langle P_{I'}P_{I }^{-1})(\textbf{x}_{I })\textbf{e}_{i_{t}},\textbf{e}_{j_1}\rangle\\
\langle D (P_{I'}P_{I }^{-1})(\textbf{x}_{I })\textbf{e}_{i_1},\textbf{e}_{j_2}\rangle&\langle D (P_{I'}P_{I }^{-1})(\textbf{x}_{I })\textbf{e}_{i_2},\textbf{e}_{j_2}\rangle&\cdots&\langle D (P_{I'}P_{I }^{-1})(\textbf{x}_{I })\textbf{e}_{i_{t}},\textbf{e}_{j_2}\rangle\\
\vdots&\vdots&\ddots&\vdots\\
\langle D (P_{I'}P_{I }^{-1})(\textbf{x}_{I })\textbf{e}_{i_1},\textbf{e}_{j_{t}}\rangle&\langle D (P_{I'}P_{I }^{-1})(\textbf{x}_{I })\textbf{e}_{i_2},\textbf{e}_{j_{t}}\rangle &\cdots& \langle D (P_{I'}P_{I }^{-1})(\textbf{x}_{I })\textbf{e}_{i_{t}},\textbf{e}_{j_{t}}\rangle
\end{array}\right|.
\end{eqnarray}
Write $A\triangleq \left(\langle D (P_{I'}P_{I }^{-1})(\textbf{x}_{I })\textbf{e}_{i_l},\textbf{e}_{j_k}\rangle\right)_{k,l=1}^{t}$ which is a $t\times t$ matrix. Combining \eqref{20260702for6}, \eqref{20260702for7}, \eqref{20260702for1}, and \eqref{20260702for2}, we see that, after finitely many row and column interchanges on the determinant \eqref{20260702for4}, we obtain the following block upper triangular matrix
\begin{eqnarray*}
 B=\left(\begin{array}{cc}
B_{11}&B_{12}\\
\textbf{0}&B_{22}
\end{array}\right),
\end{eqnarray*}
where
\begin{eqnarray*}
 B_{11}&=&\left(\begin{array}{cccc}
\langle D (P_{I'}P_{I }^{-1})(\textbf{x}_{I })\textbf{e}_{i_1'},\textbf{e}_{j_1'}\rangle&\langle D (P_{I'}P_{I }^{-1})(\textbf{x}_{I })\textbf{e}_{i_2'},\textbf{e}_{j_1'}\rangle&\cdots&\langle D (P_{I'}P_{I }^{-1})(\textbf{x}_{I })\textbf{e}_{i_{s}'},\textbf{e}_{j_1'}\rangle\\
\langle D (P_{I'}P_{I }^{-1})(\textbf{x}_{I })\textbf{e}_{i_1'},\textbf{e}_{j_2'}\rangle&\langle D (P_{I'}P_{I }^{-1})(\textbf{x}_{I })\textbf{e}_{i_2'},\textbf{e}_{j_2'}\rangle&\cdots&\langle D (P_{I'}P_{I }^{-1})(\textbf{x}_{I })\textbf{e}_{i_{s}'},\textbf{e}_{j_2}'\rangle\\
\vdots&\vdots&\ddots&\vdots\\
\langle D (P_{I'}P_{I }^{-1})(\textbf{x}_{I })\textbf{e}_{i_1'},\textbf{e}_{j_{s}'}\rangle&\langle D (P_{I'}P_{I }^{-1})(\textbf{x}_{I })\textbf{e}_{i_2'},\textbf{e}_{j_{s}'}\rangle &\cdots& \langle D (P_{I'}P_{I }^{-1})(\textbf{x}_{I })\textbf{e}_{i_{s}'},\textbf{e}_{j_{s}'}\rangle
\end{array}\right),\\
B_{22}&=&\left(\begin{array}{cccc}
1&0&\cdots&0\\
0&1&\cdots&0\\
\vdots&\vdots&\ddots&\vdots\\
0&0&\cdots& 1
\end{array}\right)_{(t-s)\times(t-s)}.
\end{eqnarray*}
Since row or column interchanges of the determinant affect only its sign (not its absolute value), we have
$$
|\det (D (P_{I'}P_{I }^{-1})(\textbf{x}_{I }))|=|\det B|=|\det (B_{11})|=\left|\det \left( D_{x_i} f_j(\textbf{x}_{I })\right)_{i\in K_1,j\in K_2}\right| .
$$
Conversely, for $I'\in \Gamma_{I_0}$ with ${I }\neq I'$, by the definition of $\Gamma_{I_0}$ in \eqref{20241231def1}, there exists
$K_1\subset {I },K_2 \subset \mathbb{N}\setminus {I }$ such that $\card K_1=\card K_2\in\mathbb{N}$ and $I'=({I }\setminus K_1)\cup K_2$. Hence
 \begin{equation}\label{20260720for1}
\left|\det \left( D_{x_i} f_j(\textbf{x}_{I })\right)_{i\in K_1,j\in K_2}\right|=\left|\det (D (P_{I'}P_{I }^{-1})(\textbf{x}_{I }))\right|.
\end{equation}
For $ I  =I'$, it is easy to see that $\left|\det (D (P_{I'}P_{I }^{-1})(\textbf{x}_{I }))\right|^2=1.$
Thus we arrive at \eqref{20250205for6}. This completes the proof of Lemma \ref{20241016rem1}.
\end{proof}

\begin{lemma}\label{20241012lem3}
If in Definition \ref{20241012def1} one chooses another local coordinate triple $(S\cap U', I',f')$ of $S$ such that $\mathbf{x}\in S\cap U'$, then
\begin{eqnarray*}
n_{I}(\mathbf{x})=\left|\det (D(P_{I'}P_{I}^{-1})(\mathbf{x}_{I}))\right|\cdot n_{I'}(\mathbf{x}),\qquad
n_{I'}(\mathbf{x})=\left|\det (D(P_{I}P_{I'}^{-1})(\mathbf{x}_{I'}))\right|\cdot n_{I}(\mathbf{x}),
\end{eqnarray*}
where $\mathbf{x}_{I}=P_{I}\mathbf{x}$ and $\mathbf{x}_{I'}=P_{I'}\mathbf{x}$.
\end{lemma}
\begin{proof}
For each $J\in \Gamma_{I_0}$, by Corollary \ref{20241015cor1},
\begin{eqnarray}\label{20250205for4}
\det (D (P_{J}P_{I}^{-1})(\textbf{x}_{I}))
&=&\det (D (P_{J}P_{I'}^{-1})(\textbf{x}_{I'}))\cdot\det (D(P_{I'}P_{I}^{-1})(\textbf{x}_{I})).
\end{eqnarray}
By Lemma \ref{20241016rem1}, it follows that
\begin{eqnarray}
n_{I}(\textbf{x})&=&\left|\det (D(P_{I'}P_{I}^{-1})(\textbf{x}_{I}))\right|\cdot\sqrt{ \sum_{J\in \Gamma_{I_0}} \left|\det (D (P_{J}P_{I'}^{-1})(\textbf{x}_{I'}))\right|^2}\nonumber\\
&=&\left|\det (D(P_{I'}P_{I}^{-1})(\textbf{x}_{I}))\right|\cdot n_{I'}(\textbf{x}).\label{20241012for5}
\end{eqnarray}
Note that $(P_{I'}P_{I}^{-1})(P_{I}P_{I'}^{-1})=\text{id}_{P_{I'}\ell^2}$ and hence
$D (P_{I'}P_{I}^{-1})(\textbf{x}_{I})\circ D(P_{I}P_{I'}^{-1})(\textbf{x}_{I'})=\text{id}_{P_{I'}\ell^2}$. Combining Proposition \ref{20241011prop1}, we have
\begin{eqnarray}\label{20241013for1}
 \det (D(P_{I'}P_{I}^{-1})(\textbf{x}_{I})) \cdot  \det (D(P_{I}P_{I'}^{-1})(\textbf{x}_{I'}))=1.
\end{eqnarray}
By \eqref{20241012for5}, we have
\begin{eqnarray*}
n_{I'}(\textbf{x})=\left|\det (D(P_{I}P_{I'}^{-1})(\textbf{x}_{I'}))\right|\cdot n_{I}(\textbf{x}).
\end{eqnarray*}
This completes the proof of Lemma \ref{20241012lem3}.
\end{proof}

Now, we suppose that $(S\cap U_1, I_1, f_1)$ and $(S\cap U_2, I_2, f_2)$ are two local coordinate triples of $S$ such that $W\triangleq S\cap U_1\cap U_2\neq \emptyset$. Then there exists two nonempty sets $K_1\subset I_1$ and $K_2\subset \mathbb{N}\setminus I_1$ such that $\card K_1=\card K_2\in\mathbb{N}$ (the case of $\card K_1=\card K_2=0$ is trivial) and $I_2=(I_1\setminus K_1)\cup K_2$.
Furthermore, one can find open sets $U_{I_k}$ of $P_{I_k}\ell^2$ ($k=1,2$) such that $f_k\in C_{P_{I_k}\ell^2}^1(U_{I_k};P_{\mathbb{N}\setminus I_k}\ell^2)$ and
$$
S\cap U_k=\left\{\textbf{x}_{I_k}+f_k(\textbf{x}_{I_k}): \textbf{x}_{I_k}\in U_{I_k}\right\}.
$$
Clearly, $T\triangleq P_{I_1}P_{I_2}^{-1}$ is a homeomorphism from $P_{I_2}W$ onto $P_{I_1}W$.
For any $\textbf{x}_{I_2}\in P_{I_2}W$, it holds that
\begin{eqnarray}
T\textbf{x}_{I_2}
&=&P_{I_1}P_{I_2}^{-1}\textbf{x}_{I_2}=P_{I_1}(\textbf{x}_{I_2}+f_2(\textbf{x}_{I_2}))\label{220916e1}\\
&=&\textbf{x}_{I_1\setminus K_1}+\sum_{i\in K_1}f_{2,i}(\textbf{x}_{I_2}) \textbf{e}_i=\textbf{x}_{I_1\setminus K_1}+\sum_{i\in K_1}f_{2,i}(\textbf{x}_{I_1\setminus K_1}+\textbf{x}_{K_2}) \textbf{e}_i,\nonumber
\end{eqnarray}
where $f_{2,i}(\textbf{x}_{I_2})=(f_2(\textbf{x}_{I_2}),\textbf{e}_i)$ for each $i\in K_1$. Also, for any $\textbf{x}_{I_1}\in P_{I_1}W$, it holds that
\begin{eqnarray}
T^{-1}\textbf{x}_{I_1}
&=&P_{I_2}P_{I_1}^{-1}\textbf{x}_{I_1}=P_{I_2}(\textbf{x}_{I_1}+f_1(\textbf{x}_{I_1}))\label{220916e2}\\
&=&\textbf{x}_{I_1\setminus K_1}+\sum_{j\in K_2}f_{1,j}(\textbf{x}_{I_1}) \textbf{e}_j=\textbf{x}_{I_1\setminus K_1}+\sum_{j\in K_2}f_{1,j}(\textbf{x}_{I_1\setminus K_1}+\textbf{x}_{K_1}) \textbf{e}_j,\nonumber
\end{eqnarray}
where $f_{1,j}(\textbf{x}_{I_1})=(f_1(\textbf{x}_{I_1}),\textbf{e}_j)$ for each $j\in K_2$.

Since $f_1$ and $f_2$ are Fr\'{e}chet differentiable, it holds that
\begin{eqnarray*}
&&(DT(\textbf{x}_{I_2}))\textbf{y}_{I_2}=\textbf{y}_{I_1\setminus K_1}+\sum_{i\in K_1}\left(\sum_{j\in I_2}(D_{x_j}f_{2,i}(\textbf{x}_{I_2} ))\cdot y_j \right)\textbf{e}_i\in P_{I_1}\ell^2,\quad \textbf{y}_{I_2}=\sum_{i\in I_2}y_i \textbf{e}_i\in P_{I_2}\ell^2,\\
&&(DT^{-1}(\textbf{x}_{I_1}))\textbf{y}_{I_1}=\textbf{y}_{I_1\setminus K_1}+\sum_{i\in K_2}\left(\sum_{j\in I_1}(D_{x_j}f_{1,i}(\textbf{x}_{I_1} ))\cdot y_j \right)\textbf{e}_i\in P_{I_2}\ell^2,\quad \textbf{y}_{I_1}=\sum_{i\in I_1}y_i \textbf{e}_i\in P_{I_1}\ell^2,
\end{eqnarray*}
where $\textbf{y}_{I_1\setminus K_1}=\sum\limits_{i\in I_1\setminus K_1}y_i \textbf{e}_i$.

Since $TT^{-1}\textbf{y}_{I_1}=\textbf{y}_{I_1},\,\forall\, \textbf{y}_{I_1}\in P_{I_1}\ell^2$ and $T^{-1}T\textbf{y}_{I_2}=\textbf{y}_{I_2},\,\forall\, \textbf{y}_{I_2}\in P_{I_2}\ell^2$. By the chain rule, we have
\begin{eqnarray*}
&&(DT(\textbf{x}_{I_2}))(DT^{-1}(\textbf{x}_{I_1}))\textbf{y}_{I_1}=\textbf{y}_{I_1},\quad\forall\, \textbf{y}_{I_1}\in P_{I_1}\ell^2,\\
&&(DT^{-1}(\textbf{x}_{I_1}))(DT(\textbf{x}_{I_2}))\textbf{y}_{I_2}=\textbf{y}_{I_2},\quad\forall\, \textbf{y}_{I_2}\in P_{I_2}\ell^2,
\end{eqnarray*}
where $T\textbf{x}_{I_2}=\textbf{x}_{I_1}\in P_{I_1}W$.
In particular, for $\textbf{y}_{K_1}=\sum\limits_{i\in K_1}y_i \textbf{e}_i\in P_{K_1}\ell^2$, it holds that
\begin{eqnarray}
(DT(\textbf{x}_{I_2}))(DT^{-1}(\textbf{x}_{I_1}))\textbf{y}_{K_1}
&=&(DT(\textbf{x}_{I_2}))\left(\sum_{i\in K_2}\left(\sum_{j\in K_1}(D_{x_j}f_{1,i}(\textbf{x}_{I_1} ))\cdot y_j \right)\textbf{e}_i\right)\nonumber\\
&=&\sum_{k\in K_1}\left( \sum_{i\in K_2} \sum_{j\in K_1}(D_{x_i}f_{2,k}(\textbf{x}_{I_2} ))\cdot(D_{x_j}f_{1,i}(\textbf{x}_{I_1} ))\cdot y_j    \right)\textbf{e}_k\label{20241011for1}\\
&=&\sum_{k\in K_1}  y_k\textbf{e}_k.\nonumber
\end{eqnarray}

Write $|K_1|=|K_2|\triangleq r\in\mathbb{N}$ and choose $i_1,\cdots,i_r,j_1,\cdots,j_r\in \mathbb{N}$ such that $i_1<i_2<\cdots<i_r$, $j_1<j_2<\cdots<j_r$, $K_1=\{i_1,i_2,\cdots,i_r\}$ and $K_2=\{j_1,j_2,\cdots,j_r\}$. Set
\begin{equation}\label{260119e1}
J_{I_1,I_2}(\textbf{x}_{I_2})\triangleq\left(\begin{array}{cccc}
D_{x_{j_1}}f_{2,i_1}(\textbf{x}_{I_2} )&D_{x_{j_2}}f_{2,i_1}(\textbf{x}_{I_2} )&\cdots&D_{x_{j_r}}f_{2,i_1}(\textbf{x}_{I_2} )\\
D_{x_{j_1}}f_{2,i_2}(\textbf{x}_{I_2} )&D_{x_{j_2}} f_{2,i_2}(\textbf{x}_{I_2} )&\cdots&D_{x_{j_{r}}} f_{2,i_2}(\textbf{x}_{I_2} )\\
\vdots&\vdots&\ddots&\vdots\\
D_{x_{j_1}}f_{2,i_r}(\textbf{x}_{I_2} )&D_{x_{j_2}} f_{2,i_r}(\textbf{x}_{I_2} ) &\cdots& D_{x_{j_r}} f_{2,i_r}(\textbf{x}_{I_2} )
\end{array}\right)_{r\times r}.
\end{equation}
Hence,
\begin{eqnarray*}
J_{I_2,I_1} (\textbf{x}_{I_1})\triangleq
 \left(\begin{array}{cccc}
D_{x_{i_1}}f_{1,j_1}(\textbf{x}_{I_1} )&D_{x_{i_2}}f_{1,j_1}(\textbf{x}_{I_1} )&\cdots&D_{x_{i_r}}f_{1,j_1}(\textbf{x}_{I_1} )\\
D_{x_{i_1}}f_{1,j_2}(\textbf{x}_{I_1} )&D_{x_{i_2}} f_{1,j_2}(\textbf{x}_{I_1} )&\cdots&D_{x_{i_{r}}} f_{1,j_2}(\textbf{x}_{I_1} )\\
\vdots&\vdots&\ddots&\vdots\\
D_{x_{i_1}}f_{1,j_r}(\textbf{x}_{I_1} )&D_{x_{i_2}} f_{1,j_r}(\textbf{x}_{I_1} ) &\cdots& D_{x_{i_r}} f_{1,j_r}(\textbf{x}_{I_1} )
\end{array}\right)_{r\times r}.
\end{eqnarray*}
By \eqref{20260720for1}, we have
\begin{eqnarray}\label{20260720for2}
\left|\det(J_{I_2,I_1} (\textbf{x}_{I_1}))\right|=\left|\det (D (P_{I_2}P_{I_1 }^{-1})(\textbf{x}_{I_1 }))\right|,
\quad \left|\det(J_{I_1,I_2}(\textbf{x}_{I_2}))\right|=\left|\det (D (P_{I_1}P_{I_2 }^{-1})(\textbf{x}_{I_2 }))\right|.
\end{eqnarray}
Then \eqref{20241011for1} implies that
\begin{eqnarray}\label{20241011for2}
J_{I_1,I_2}(\textbf{x}_{I_2})J_{I_2,I_1}(\textbf{x}_{I_1})=
 \left(\begin{array}{cccc}
1&0&\cdots&0\\
0&1&\cdots&0\\
\vdots&\vdots&\ddots&\vdots\\
0&0&\cdots& 1
\end{array}\right)_{r\times r}.
\end{eqnarray}

Similar to \eqref{220817e1}, we denote by $\mu_{I_0}$ the restriction of the product measure $\prod\limits_{i\in I_0}\frac{e^{-\frac{x_i^2}{2a_i^2}}}{\sqrt{2\pi a_i^2}}dx_i$ on $P_{I_0}\ell^2$. Clearly, $\mu_{\mathbb{N}}$ is precisely the measure $P_t$ defined at \eqref{220817e1} for $t=1$.

\begin{lemma}\label{20241011lem1}
It holds that
\begin{eqnarray}\label{220820e1}
\mathrm{d}\mu_{I_1}(T\mathbf{x}_{I_2})= \left(\prod\limits_{i\in K_1}\frac{e^{-\frac{\left(f_{2,i}(\mathbf{x}_{I_2})\right)^2}{2a_i^2}}}{\sqrt{2\pi a_i^2}}\right)\cdot\left|\det J_{I_1,I_2}(\mathbf{x}_{I_2})\right|\left(\prod_{j\in K_2}\,\mathrm{d} x_{j}\right)\times\mathrm{d}\mu_{I_1\setminus K_1}(\mathbf{x}_{I_1\setminus K_1}),
\end{eqnarray}
on $\mathscr{B}\big(P_{I_2} W \big)$.
\end{lemma}
\begin{proof}
It suffices to prove the two measures on both sides of \eqref{220820e1} coincide on subsets of $P_{I_2} W $ in the following form:
\begin{eqnarray*}
A=\{  \textbf{x}_{I_1\setminus K_1}+\textbf{x}_{K_2}:\;\textbf{x}_{I_1\setminus K_1}\in U_{I_1\setminus K_1},\,\textbf{x}_{K_2}\in U_{K_2}\}\subset  P_{I_2} W,
\end{eqnarray*}
where $U_{I_1\setminus K_1}$ is an open subset of $P_{I_1\setminus K_1}\ell^2$ and $U_{K_2}$ is an open subset of $P_{K_2}\ell^2$. By (\ref{220916e1}), we obtain
$$TA=\left\{\textbf{x}_{I_1\setminus K_1}+\sum_{i\in K_1}f_{2,i}(\textbf{x}_{I_1\setminus K_1}+\textbf{x}_{K_2}) \textbf{e}_i:\;\textbf{x}_{I_1\setminus K_1}\in U_{I_1\setminus K_1},\,\textbf{x}_{K_2}\in U_{K_2}\right\}.
$$
Then,
\begin{eqnarray*}
&&\int_A \mathrm{d}\mu_{I_1}(T\textbf{x}_{I_2})= \mu_{I_1}(TA)=\int_{TA} \mathrm{d}\mu_{I_1}(\textbf{x}_{I_1})=\int_{TA}\mathrm{d}\mu_{K_1}(\textbf{x}_{K_1})\mathrm{d}\mu_{I_1\setminus K_1}(\textbf{x}_{I_1\setminus K_1}) \\
&&=\int_{TA}\left(\prod\limits_{i\in K_1}\frac{e^{-\frac{x_i^2}{2a_i^2}}}{\sqrt{2\pi a_i^2}}\cdot\mathrm{d} x_{i}\right)\mathrm{d}\mu_{I_1\setminus K_1}(\textbf{x}_{I_1\setminus K_1}) \\
&&=\int_{U_{I_1\setminus K_1}}\int_{(TA)_{\textbf{x}_{I_1\setminus K_1}}}\left(\prod\limits_{i\in K_1}\frac{e^{-\frac{x_i^2}{2a_i^2}}}{\sqrt{2\pi a_i^2}}\cdot\mathrm{d} x_{i}\right)\mathrm{d}\mu_{I_1\setminus K_1}(\textbf{x}_{I_1\setminus K_1}) \\
&&=\int_{U_{I_1\setminus K_1}}\int_{ U_{K_2}}\left(\prod\limits_{i\in K_1}\frac{e^{-\frac{\left(f_{2,i}(\textbf{x}_{I_2})\right)^2}{2a_i^2}}}{\sqrt{2\pi a_i^2}}\right)\cdot \left|\det J_{I_1,I_2}(\textbf{x}_{I_2})\right|\cdot \left(\prod_{j\in K_2}\,\mathrm{d} x_{j}\right)\mathrm{d}\mu_{I_1\setminus K_1}(\textbf{x}_{I_1\setminus K_1}) \\
&&=\int_{A}\left(\prod\limits_{i\in K_1}\frac{e^{-\frac{\left(f_{2,i}(\textbf{x}_{I_2})\right)^2}{2a_i^2}}}{\sqrt{2\pi a_i^2}}\right)\cdot \left|\det J_{I_1,I_2}(\textbf{x}_{I_2})\right|\cdot \left(\prod_{j\in K_2}\,\mathrm{d} x_{j}\right)\times\mathrm{d}\mu_{I_1\setminus K_1}(\textbf{x}_{I_1\setminus K_1}),
\end{eqnarray*}
where $(TA)_{\textbf{x}_{I_1\setminus K_1}}\triangleq \{\textbf{x}_{ K_1}\in P_{K_1}\ell^2:\textbf{x}_{I_1\setminus K_1}+\textbf{x}_{K_1}\in TA\}=\left\{\sum\limits_{i\in K_1}f_{2,i}(\textbf{x}_{I_1\setminus K_1}+\textbf{x}_{K_2}) \textbf{e}_i:\;\textbf{x}_{K_2}\in U_{K_2}\right\}$, the sixth equality follows from the finite dimensional change of variable formula for the following transformation
\begin{eqnarray*}
\textbf{x}_{K_1}=\sum_{i\in K_1}f_{2,i}(\textbf{x}_{I_1\setminus K_1}+\textbf{x}_{K_2}) \textbf{e}_i,\quad \forall\;\textbf{x}_{K_2}=\sum_{j\in K_2}x_j\textbf{e}_j\in U_{K_2},
\end{eqnarray*}
whose Jacobian matrix is given by \eqref{260119e1}, and therefore
$$
\int_{(TA)_{\textbf{x}_{I_1\setminus K_1}}}\left(\prod\limits_{i\in K_1}\frac{e^{-\frac{x_i^2}{2a_i^2}}}{\sqrt{2\pi a_i^2}}\cdot\mathrm{d} x_{i}\right)=
\int_{ U_{K_2}}\left(\prod\limits_{i\in K_1}\frac{e^{-\frac{\left(f_{2,i}(\textbf{x}_{I_2})\right)^2}{2a_i^2}}}{\sqrt{2\pi a_i^2}}\right)\cdot \left|\det J_{I_1,I_2}(\textbf{x}_{I_2})\right|\cdot \left(\prod_{j\in K_2}\,\mathrm{d} x_{j}\right).
$$
This completes the proof of Lemma \ref{20241011lem1}.
\end{proof}

Let \(I\) be a nonempty subset of \(\mathbb{N}\), and define
\begin{equation}\label{20250119def1}
C_I
\triangleq
\left\{
\mathbf{x}_I=\sum_{i\in I}x_i\mathbf{e}_i\in P_I\ell^2:
\sum_{i\in I}
\left|
\ln\frac{1}{\sqrt{2\pi a_i^2}}
-\frac{x_i^2}{2a_i^2}
\right|
<\infty
\right\}.
\end{equation}
Observe that, for every
\(\mathbf{x}_I=\sum_{i\in I}x_i\mathbf{e}_i\in C_I\), the infinite
product
\[
\prod_{i\in I}
\frac{1}{\sqrt{2\pi a_i^2}}
\exp\left(-\frac{x_i^2}{2a_i^2}\right)
\]
converges to a finite positive number. Moreover, \(C_I\) is a Borel
subset of \(P_I\ell^2\).

Define \(F_I:P_I\ell^2\to(0,\infty]\) by
\begin{equation}\label{20250206def1}
F_I(\mathbf{x}_I)
\triangleq
\begin{cases}
\displaystyle
\prod_{i\in I}
\frac{1}{\sqrt{2\pi a_i^2}}
\exp\left(-\frac{x_i^2}{2a_i^2}\right),
&
\mathbf{x}_I=\displaystyle\sum_{i\in I}x_i\mathbf{e}_i\in C_I,
\\[1.2em]
+\infty,
&
\mathbf{x}_I\in P_I\ell^2\setminus C_I.
\end{cases}
\end{equation}
Then \(F_I\) is a nonnegative extended-real-valued Borel measurable function on
\(P_I\ell^2\).

\begin{definition}\label{20241012def2}
Let \(S\) be a \(C^1\)-differentiable surface of \(\ell^2\) with
codimension \(\Gamma_{\mathbb{N}\setminus I}\), and let
\((S\cap U,I_1,f)\) be a local coordinate triple of \(S\).
Let
$
f\in
C_{P_{I_1}\ell^2}^{1}
\bigl(
U_{I_1};
P_{\mathbb{N}\setminus I_1}\ell^2
\bigr)
$
be the corresponding local graph map.
For every \(E\in\mathscr{B}(S\cap U)\), define
\begin{eqnarray*}
\mu(E,S\cap U,I_1,f)
\triangleq
\int_{P_{I_1}E}
& n_{I_1}
\bigl(
\mathbf{x}_{I_1}+f(\mathbf{x}_{I_1})
\bigr)
F_{\mathbb{N}\setminus I_1}
\bigl(
f(\mathbf{x}_{I_1})
\bigr)
\,\mathrm{d}\mu_{I_1}(\mathbf{x}_{I_1}).
\end{eqnarray*}
\end{definition}

\begin{remark}
The motivation for the preceding definition is essentially the same as
that in the finite-dimensional setting; see
Remark~\ref{20250208rem1}. The main difference is the presence of the
additional factor
\[
F_{\mathbb{N}\setminus I_1}
\bigl(f(\mathbf{x}_{I_1})\bigr).
\]
This factor is required because the ambient measure chosen on
\(\ell^2\) is anisotropic and has a nonconstant density in the
coordinate directions. When \(S\cap U\) is represented as a graph over
\(P_{I_1}\ell^2\), the factor
\(F_{\mathbb{N}\setminus I_1}(f(\mathbf{x}_{I_1}))\) accounts for the
contribution of the complementary coordinates.

By contrast, the measure considered in
Remark~\ref{20250208rem1} is the three-dimensional Lebesgue measure,
which is invariant under translations and rotations and has constant
density. Consequently, apart from the usual geometric Jacobian factor,
no additional density factor is needed when the surface is represented
over a lower-dimensional coordinate plane.
\end{remark}
\begin{lemma}\label{20241013lem1}
Retain the notation and assumptions of
Definition~\ref{20241012def2}. Let
$
(S\cap U_1,I_2,f')
$
be another local coordinate triple of \(S\) such that
$
S\cap U\cap U_1\neq\varnothing.
$
Let
$
f'
\in
C_{P_{I_2}\ell^2}^{1}
\bigl(
U_{I_2};
P_{\mathbb{N}\setminus I_2}\ell^2
\bigr)
$
be the corresponding local graph map.
For every \(E\in\mathscr{B}(S\cap U_1)\), define
\begin{eqnarray*}
\mu(E,S\cap U_1,I_2,f')
\triangleq
\int_{P_{I_2}E}
& n_{I_2}
\bigl(
\mathbf{x}_{I_2}+f'(\mathbf{x}_{I_2})
\bigr)
F_{\mathbb{N}\setminus I_2}
\bigl(
f'(\mathbf{x}_{I_2})
\bigr)
\,\mathrm{d}\mu_{I_2}(\mathbf{x}_{I_2}).
\end{eqnarray*}
Then
\[
\mu(E,S\cap U_1,I_2,f')
=
\mu(E,S\cap U,I_1,f)
\]
for every
$
E\in\mathscr{B}(S\cap U\cap U_1).
$
\end{lemma}

\begin{proof}
Let $W\triangleq S\cap U\cap U_1$. There exists $K_1\subset I_1$ and $K_2\subset \mathbb{N}\setminus I_1$ such that $|K_1|=|K_2|\in\mathbb{N}_0$ and $I_2=(I_1\setminus K_1)\cup K_2$.
Clearly, $T\triangleq P_{I_2}P_{I_1}^{-1}$ is a homeomorphism from $P_{I_1}W$ onto $P_{I_2}W$. Note that  for any Borel subset $E$ of $P_{I_2}W$, Borel measure $\mu$ on $P_{I_2}W$ and non-negative Borel measurable function $h$ on $P_{I_2}W$, it holds that (e.g., \cite[Theorem 2.27, p. 35]{LZ2021})
\begin{eqnarray}\label{20250205for1}
\int_E h \,\mathrm{d}\mu =\int_{T^{-1}E} h(T\omega)\,\mathrm{d}\mu(T\omega).
\end{eqnarray}
By \eqref{20250205for1}, we have
\begin{eqnarray*}
&&\int_{P_{I_2}E}n_{I_2}(\textbf{x}_{I_2}+f'(\textbf{x}_{I_2}))\cdot F_{\mathbb{N}\setminus I_2}(f'(\textbf{x}_{I_2}))\,\mathrm{d}\mu_{I_2}(\textbf{x}_{I_2})\\
&=&\int_{P_{I_1}E}n_{I_2}(P_{I_2}^{-1}(T\textbf{x}_{I_1}))\cdot F_{\mathbb{N}\setminus I_2}(f'(T\textbf{x}_{I_1}))\,\mathrm{d}\mu_{I_2}(T\textbf{x}_{I_1})\\
&=&\int_{P_{I_1}E}n_{I_2}( \textbf{x}_{I_1}+ f(\textbf{x}_{I_1}))\cdot F_{\mathbb{N}\setminus I_2}(f'(T\textbf{x}_{I_1}))\,\mathrm{d}\mu_{I_2}(T\textbf{x}_{I_1}).
\end{eqnarray*}
By Lemma \ref{20241012lem3}, $n_{I_2}( \textbf{x}_{I_1}+ f(\textbf{x}_{I_1}))=\left|\det (D(P_{I_1}P_{I_2}^{-1})(T\textbf{x}_{I_1}))\right|\cdot n_{I_1}(\textbf{x}_{I_1}+ f(\textbf{x}_{I_1}))$.
Then similar to \eqref{220916e1} and \eqref{220916e2}, for any $\textbf{x}_{I_1}\in P_{I_1}W$ and $\textbf{x}_{I_2}\in P_{I_2}W$, it holds that
\begin{eqnarray*}
T\textbf{x}_{I_1}
=\textbf{x}_{I_2\setminus K_2}+\sum_{i\in K_2}f_i(\textbf{x}_{I_2\setminus K_2}+\textbf{x}_{K_1}) \textbf{e}_i, \quad
T^{-1}\textbf{x}_{I_2}
=\textbf{x}_{I_2\setminus K_2}+\sum_{j\in K_1}f'_j(\textbf{x}_{I_2\setminus K_2}+\textbf{x}_{K_2}) \textbf{e}_j,
\end{eqnarray*}
where $f_i(\textbf{x}_{I_1})=\langle f(\textbf{x}_{I_1}),\textbf{e}_i\rangle$ for every $i\in K_2$, and $f'_j(\textbf{x}_{I_2})=\langle f'(\textbf{x}_{I_2}),\textbf{e}_j\rangle$ for every $j\in K_1$. Thus
\begin{eqnarray*}
F_{\mathbb{N}\setminus I_2}(f'(T\textbf{x}_{I_1}))&=&F_{\mathbb{N}\setminus I_2}\left(f'\left(P_{I_2}(\textbf{x}_{I_1}+f(\textbf{x}_{I_1}))\right) \right)\\
&=&\left(\prod_{i\in K_1}\frac{e^{-\frac{x_i^2}{2a_i^2}}}{\sqrt{2\pi a_i^2}}\right)\cdot  F_{(\mathbb{N}\setminus I_2)\setminus K_1}\left(P_{(\mathbb{N}\setminus I_2)\setminus K_1}f'\left(P_{I_2}(\textbf{x}_{I_1}+f(\textbf{x}_{I_1}))\right) \right)
\end{eqnarray*}
By Lemma \ref{20241011lem1},
\begin{eqnarray*}
\mathrm{d}\mu_{I_2}(T\textbf{x}_{I_1})= \left(\prod\limits_{i\in K_2}\frac{e^{-\frac{\left(f_i(\textbf{x}_{I_1})\right)^2}{2a_i^2}}}{\sqrt{2\pi a_i^2}}\right) \cdot\left|\det J_{I_2,I_1}(\textbf{x}_{I_1})\right|\left(\prod_{j\in K_1}\,\mathrm{d} x_{j}\right)\times\mathrm{d}\mu_{I_2\setminus K_2}(\textbf{x}_{I_2\setminus K_2}).
\end{eqnarray*}
Therefore,
\begin{eqnarray*}
&&\int_{P_{I_2}E}n_{I_2}(\textbf{x}_{I_2}+f'(\textbf{x}_{I_2}))\cdot F_{\mathbb{N}\setminus I_2}(f'(\textbf{x}_{I_2}))\,\mathrm{d}\mu_{I_2}(\textbf{x}_{I_2})\\
&=&\int_{P_{I_1}E}n_{I_2}( \textbf{x}_{I_1}+ f(\textbf{x}_{I_1}))\cdot F_{\mathbb{N}\setminus I_2}(f'(T\textbf{x}_{I_1}))\,\mathrm{d}\mu_{I_2}(T\textbf{x}_{I_1})\\
&=&\int_{P_{I_1}E} \left|\det (D(P_{I_1}P_{I_2}^{-1})(T\textbf{x}_{I_1}))\right|\cdot   n_{I_1}(\textbf{x}_{I_1}+ f(\textbf{x}_{I_1}))\cdot \left(\prod_{j\in K_1}\frac{e^{-\frac{x_j^2}{2a_j^2}}}{\sqrt{2\pi a_j^2}}\right)\\
&&\times F_{(\mathbb{N}\setminus I_2)\setminus K_1}\left(P_{(\mathbb{N}\setminus I_2)\setminus K_1}f'\left(P_{I_2}(\textbf{x}_{I_1}+f(\textbf{x}_{I_1}))\right) \right)\cdot\left(\prod\limits_{i\in K_2}\frac{e^{-\frac{\left(f_i(\textbf{x}_{I_1})\right)^2}{2a_i^2}}}{\sqrt{2\pi a_i^2}}\right)\cdot\left|\det J_{I_2,I_1}(\textbf{x}_{I_1})\right|\\
&&\left(\prod_{j\in K_1}\,\mathrm{d} x_{j}\right)\times\mathrm{d}\mu_{I_2\setminus K_2}(\textbf{x}_{I_2\setminus K_2})\\
&=&\int_{P_{I_1}E}  n_{I_1}(\textbf{x}_{I_1}+ f(\textbf{x}_{I_1}))\cdot  F_{\mathbb{N}\setminus I_1}(f(\textbf{x}_{I_1}))\mathrm{d}\mu_{I_1}(\textbf{x}_{I_1}),
\end{eqnarray*}
where the last two equalities follow from Lemma \ref{20241012lem3}, \eqref{20241013for1} and \eqref{20260720for2}, and the facts that
\begin{eqnarray*}
F_{\mathbb{N}\setminus I_1}(f(\textbf{x}_{I_1}))=F_{(\mathbb{N}\setminus I_2)\setminus K_1}\left(P_{(\mathbb{N}\setminus I_2)\setminus K_1}f'\left(P_{I_2}(\textbf{x}_{I_1}+f(\textbf{x}_{I_1}))\right) \right)\cdot\left(\prod\limits_{i\in K_2}\frac{e^{-\frac{\left(f_i(\textbf{x}_{I_1})\right)^2}{2a_i^2}}}{\sqrt{2\pi a_i^2}}\right),
\end{eqnarray*}
and $\mathrm{d}\mu_{I_1}(\textbf{x}_{I_1})=\left(\prod\limits_{j\in K_1}\frac{e^{-\frac{x_j^2}{2a_j^2}}}{\sqrt{2\pi a_j^2}}\right)\left(\prod\limits_{j\in K_1}\,\mathrm{d} x_{j}\right)\times\mathrm{d}\mu_{I_2\setminus K_2}(\textbf{x}_{I_2\setminus K_2})$. This completes the proof of Lemma \ref{20241013lem1}.
\end{proof}

Lemma~\ref{20241013lem1} shows that the quantity
\[
\mu(E,S\cap U,I,f)
\]
defined in Definition~\ref{20241012def2} is independent of the choice
of the local coordinate triple.
We therefore write
\[
\mu(E,S\cap U)
\]
instead of \(\mu(E,S\cap U,I,f)\).

For each local coordinate neighborhood \(S\cap U\), the set function
\[
\mu(\,\cdot\,,S\cap U):
\mathscr{B}(S\cap U)\longrightarrow[0,+\infty]
\]
is a Borel measure. We call it the
local surface measure on \(S\) associated with \(S\cap U\).

Moreover, if
$
S\cap U\cap U_1\neq\varnothing,
$
then Lemma~\ref{20241013lem1} implies that
\[
\mu(E,S\cap U)
=
\mu(E,S\cap U_1)
\]
for every
$
E\in\mathscr{B}(S\cap U\cap U_1).
$
Thus, the local measures agree on the overlaps of their domains.

\begin{proposition}[Existence of the global surface measure]
\label{20241013prop1}
Let \(S\) be a \(C^1\)-differentiable surface of \(\ell^2\) with
codimension \(\Gamma_{\mathbb{N}\setminus I_0}\). Then there exists a
unique Borel measure \(\mu_S\) on \(S\) such that, for every local
coordinate triple
$
(S\cap U,I,f)
$
of \(S\),
\[
\mu_S(E)
=
\mu(E,S\cap U)
\]
for every \(E\in\mathscr{B}(S\cap U)\). Equivalently,
\[
\left.\mu_S\right|_{S\cap U}
=
\mu(\,\cdot\,,S\cap U).
\]
We call \(\mu_S\) the surface measure on \(S\).
\end{proposition}
\begin{proof}
Since \(\ell^2\) is second countable, its subspace \(S\) is also
second countable and hence Lindel\"{o}f. Therefore, by \cite[Theorem 15, p. 49]{Kel}, from the open cover of
\(S\) by local coordinate neighborhoods, we may choose a countable
subcover
\[
S=\bigcup_{n=1}^{\infty}W_n,
\qquad
W_n\triangleq S\cap U_{i_n},
\]
where
$
(S\cap U_{i_n},I_{i_n},f_{i_n})
$
is a local coordinate triple of \(S\) for every \(n\in\mathbb{N}\).

Define
\[
Y_1\triangleq W_1,\qquad
Y_n
\triangleq
W_n\setminus\bigcup_{k=1}^{n-1}W_k,\quad n\geqslant 2.
\]
Then \(\{Y_n\}_{n=1}^{\infty}\) is a pairwise disjoint Borel
partition of \(S\), and
$
Y_n\subset W_n
$
for every \(n\in\mathbb{N}\).

For \(E\in\mathscr{B}(S)\), define
\[
\mu_S(E)
\triangleq
\sum_{n=1}^{\infty}
\mu(E\cap Y_n,W_n).
\]
For each \(n\), the set function
\[
E\longmapsto \mu(E\cap Y_n,W_n)
\]
is a Borel measure on \(S\). Consequently, \(\mu_S\), being a
countable sum of Borel measures, is itself a Borel measure on \(S\).
More explicitly, if \(\{E_m\}_{m=1}^{\infty}\subset\mathscr{B}(S)\)
is pairwise disjoint, then
\begin{eqnarray*}
\mu_S\left(\bigcup_{m=1}^{\infty}E_m\right)
=
\sum_{n=1}^{\infty}
\mu\left(
\bigcup_{m=1}^{\infty}(E_m\cap Y_n),
W_n
\right)=
\sum_{n=1}^{\infty}
\sum_{m=1}^{\infty}
\mu(E_m\cap Y_n,W_n) =
\sum_{m=1}^{\infty}\mu_S(E_m).
\end{eqnarray*}

It remains to verify that \(\mu_S\) agrees with every local surface
measure. Let \(W=S\cap U\) be a local coordinate neighborhood and let
\(E\in\mathscr{B}(W)\). Since \(W\) is open in \(S\), we may regard
\(E\) as an element of \(\mathscr{B}(S)\). For every \(n\),
\[
E\cap Y_n\subset W\cap W_n.
\]
Hence, by Lemma~\ref{20241013lem1}, the local measures agree on the
overlap, and therefore
\[
\mu(E\cap Y_n,W_n)
=
\mu(E\cap Y_n,W).
\]
Since the sets \(\{E\cap Y_n\}_{n=1}^{\infty}\) form a pairwise disjoint Borel partition
of \(E\), it follows that
\begin{eqnarray*}
\mu_S(E)=
\sum_{n=1}^{\infty}\mu(E\cap Y_n,W_n)=
\sum_{n=1}^{\infty}\mu(E\cap Y_n,W)=
\mu(E,W)=\mu(E,S\cap U).
\end{eqnarray*}
Thus \(\mu_S\) has the required local restriction property. If $\mu_S'$ is another Borel measure on $S$ that satisfies the properties in Proposition \ref{20241013prop1}, then for every \(E\in\mathscr{B}(S )\),
\begin{eqnarray*}
\mu_S(E)=
\sum_{n=1}^{\infty}\mu(E\cap Y_n,W_n)=
\sum_{n=1}^{\infty}\mu_S'(E\cap Y_n)=
\mu_S'(E),
\end{eqnarray*}
where the second equality follows from $\mu_S' |_{W_n}=\mu(\,\cdot\,,W_n)$, $E\cap Y_n\subset W_n$ and $\mu(E\cap Y_n,W_n)=\mu_S'(E\cap Y_n)$ for every $n\in\mathbb{N}$. Therefore, $\mu_S=\mu_S'$. This completes the proof of Proposition \ref{20241013prop1}.
\end{proof}


\begin{definition}\label{20250112def1}
Let \(S\) be a \(C^1\)-differentiable surface of \(\ell^2\) with
codimension \(\Gamma_{\mathbb{N}\setminus I_0}\), and let
$
(S\cap U,I,f)
$
be a local coordinate triple of \(S\). Let
$
f\in
C_{P_{I}\ell^2}^{1}
\bigl(
U_{I};
P_{\mathbb{N}\setminus I}\ell^2
\bigr)
$
be the corresponding local graph map.
We call \((S\cap U,I,f)\) a
bounded local coordinate triple of \(S\) if
\begin{equation}\label{20241013cond1}
\sup_{\mathbf{x}_{I}\in U_{I}}
n_{I}
\bigl(
\mathbf{x}_{I}+f(\mathbf{x}_{I})
\bigr)
F_{\mathbb{N}\setminus I}
\bigl(
f(\mathbf{x}_{I})
\bigr)
<\infty.
\end{equation}
Moreover, we say that \(S\) is a \(\sigma\)-finite surface if its surface
measure \(\mu_S\) is \(\sigma\)-finite.
\end{definition}

\begin{remark}
Combining \eqref{20250111for1} with the preceding definition, we see
that if \(\mathbb{N}\setminus I\) is finite, then every point of \(S\)
admits, after possibly shrinking the coordinate neighborhood, a
bounded local coordinate triple. Indeed, in this case the factor
\[
F_{\mathbb{N}\setminus I}
\]
is a finite-dimensional Gaussian density and is therefore bounded,
while \eqref{20250111for1} ensures that the geometric factor
\(n_I\) is locally bounded. Consequently,
\eqref{20241013cond1} holds on a sufficiently small local coordinate
neighborhood.

Since \(S\) is second countable, it can be covered by countably many
such bounded local coordinate neighborhoods. Hence every
\(C^1\)-differentiable surface in \(\ell^2\) of finite codimension is
a \(\sigma\)-finite surface.

The same argument applies in finite-dimensional ambient spaces. In
particular, after restricting to a sufficiently small coordinate
neighborhood, the usual geometric Jacobian factor is bounded, and no
additional nonconstant density factor is present when the ambient
measure is Lebesgue measure. The case of surfaces in
\(\mathbb{R}^3\) is discussed in
Remark~\ref{20250208rem1}.
\end{remark}


\begin{example}
Let \(S\) be a \(C^1\)-differentiable surface of \(\ell^2\) with
codimension \(\Gamma_{\{i\}}\) for some \(i\in\mathbb{N}\).
Thus, \(S\) is a \(C^1\)-differentiable hypersurface in \(\ell^2\).
Let
$
\bigl(
S\cap U,\,
\mathbb{N}\setminus\{i_0\},\,
P_{\mathbb{N}\setminus\{i_0\}}
\bigr)
$
be a local coordinate triple of \(S\), and let
$
f\in
C_{P_{\mathbb{N}\setminus\{i_0\}}\ell^2}^{1}
\left(
U_{\mathbb{N}\setminus\{i_0\}};
P_{\{i_0\}}\ell^2
\right)
$
be the corresponding local graph map. Write
\[
f(\mathbf{x}_{\mathbb{N}\setminus\{i_0\}})
=
f_{i_0}(\mathbf{x}_{\mathbb{N}\setminus\{i_0\}})
\mathbf{e}_{i_0}.
\]
By \eqref{20250205for2} in
Example~\ref{20250205exa1}, we have
\begin{eqnarray*}
&\sup\limits_{\mathbf{x}_{\mathbb{N}\setminus\{i_0\}}
\in U_{\mathbb{N}\setminus\{i_0\}}}
n_{\mathbb{N}\setminus\{i_0\}}
\left(
\mathbf{x}_{\mathbb{N}\setminus\{i_0\}}
+
f(\mathbf{x}_{\mathbb{N}\setminus\{i_0\}})
\right)
\cdot
F_{\{i_0\}}
\left(
f(\mathbf{x}_{\mathbb{N}\setminus\{i_0\}})
\right)
\\
&=
\sup\limits_{\mathbf{x}_{\mathbb{N}\setminus\{i_0\}}
\in U_{\mathbb{N}\setminus\{i_0\}}}
\sqrt{
1+
\left\|
Df_{i_0}
(\mathbf{x}_{\mathbb{N}\setminus\{i_0\}})
\right\|^2
}
\,
\frac{
\exp\left(
-\dfrac{
f_{i_0}(\mathbf{x}_{\mathbb{N}\setminus\{i_0\}})^2
}{
2a_{i_0}^2
}
\right)
}{
\sqrt{2\pi a_{i_0}^2}
}.
\end{eqnarray*}
The function
\[
\mathbf{x}_{\mathbb{N}\setminus\{i_0\}}
\longmapsto
\sqrt{
1+
\left\|
Df_{i_0}
(\mathbf{x}_{\mathbb{N}\setminus\{i_0\}})
\right\|^2
}
\,
\frac{
\exp\left(
-\dfrac{
f_{i_0}(\mathbf{x}_{\mathbb{N}\setminus\{i_0\}})^2
}{
2a_{i_0}^2
}
\right)
}{
\sqrt{2\pi a_{i_0}^2}
}
\]
is continuous on
\(U_{\mathbb{N}\setminus\{i_0\}}\), since
\(f_{i_0}\) is continuously Fr\'{e}chet differentiable. It is therefore
locally bounded. Consequently, after possibly shrinking the local
coordinate neighborhood, the above supremum is finite. Hence every
point of \(S\) admits a bounded local coordinate triple.
\end{example}

\begin{example}
Let \(S\) be a \(C^1\)-differentiable surface of \(\ell^2\) with codimension
\(\Gamma_{\{1,2,\ldots,n\}}\) for some \(n\in\mathbb{N}\). Then \(S\) is a
\(C^1\)-differentiable surface in \(\ell^2\) of codimension \(n\) in the usual
sense. We use the same notation as in Example \ref{20250110exa1}. For every \(\mathbf{x}_{I_0}\in U_{I_0}\), inequality
\eqref{20250111for1} in Example \ref{20250110exa1} yields
\begin{eqnarray*}
&&n_{I_0}\bigl(\mathbf{x}_{I_0}
+f(\mathbf{x}_{I_0})\bigr)
F_{I_0}\bigl(f(\mathbf{x}_{I_0})\bigr)\\
&\leqslant &
\left(
1+
\sum_{k=1}^{n}
\sum_{
\begin{gathered} K\subset\{1,2,\ldots,n\}, \\\card(K)=k  \end{gathered}
}
\prod_{i\in K}
\bigl\|Df_i(\mathbf{x}_{I_0})\bigr\|^2
\right)^{1/2}
\prod_{i=1}^{n}
\frac{
\exp \!\left(
-\dfrac{f_i(\mathbf{x}_{I_0})^2}{2a_i^2}
\right)
}{
\sqrt{2\pi a_i^2}
}.
\end{eqnarray*}

Since \(f\) is of class \(C^1\), the expression on the right-hand side is
continuous on \(U_{I_0}\), and hence it is locally bounded. Therefore, for any
fixed \(\mathbf{x}_{I_0}^{ 0}\in U_{I_0}\), after replacing \(U_{I_0}\), if
necessary, by a sufficiently small neighborhood of
\(\mathbf{x}_{I_0}^{ 0}\), we obtain
\[
\sup_{\mathbf{x}_{I_0}\in U_{I_0}}
n_{I_0}\bigl(\mathbf{x}_{I_0}
+f(\mathbf{x}_{I_0})\bigr)
F_{I_0}\bigl(f(\mathbf{x}_{I_0})\bigr)
<\infty.
\]
\end{example}

\begin{example}
We adopt the notation introduced in Example \ref{20250205exa2}.

For every \(\mathbf{x}_{I_0}\in P_{I_0}\ell^2\), we have
\begin{eqnarray*}
F_J\bigl(f(\mathbf{x}_{I_0})\bigr)
&=&
\prod_{j\in J_0}
\frac{1}{\sqrt{2\pi a_j^2}}
\exp\left(
-\frac{(x_j^0+x_1\Delta x_j^0)^2}{2a_j^2}
\right)\\
&=&
\exp\left\{
\sum_{j\in J_0}
\left(
\ln\frac{1}{\sqrt{2\pi}a_j}
-\frac{(x_j^0)^2}{2a_j^2}
\right)
-x_1\sum_{j\in J_0}
\frac{x_j^0\Delta x_j^0}{a_j^2}
-\frac{x_1^2}{2}
\sum_{j\in J_0}
\frac{(\Delta x_j^0)^2}{a_j^2}
\right\}.
\end{eqnarray*}
Combining this identity with \eqref{20250205for3} in
Example \ref{20250205exa2}, we obtain
\[
\sup_{\mathbf{x}_{I_0}\in U_{I_0}}
n_{I_0}\bigl(\mathbf{x}_{I_0}+f(\mathbf{x}_{I_0})\bigr)
F_{J_0}\bigl(f(\mathbf{x}_{I_0})\bigr)
<\infty
\]
for every bounded open subset \(U_{I_0}\) of \(P_{I_0}\ell^2\).
\end{example}

In order to obtain a global version of Gauss--Green-type Theorem, first, we need the existence of a kind of ``continuous unit normal vector field" on $S$. Suppose that $S$ is a $C^1$-differentiable surface of $\ell^2$ with codimension $\Gamma_{\mathbb{N}\setminus I}$. For $\textbf{x}\in S$, if $(S\cap U_1, P_{I_1})$ and $(S\cap U_2, P_{I_2})$ are two local coordinate triples of $S$ such that $\textbf{x}\in S\cap U_1\cap U_2$. Recall \eqref{20241013for1}, we have
$$
\det \left(D(P_{I_1}P_{I_2}^{-1})(\textbf{x}_{I_2})\right)\cdot \det \left(D(P_{I_2}P_{I_1}^{-1})(\textbf{x}_{I_1})\right)=1,
$$
where $\textbf{x}_{I_1}=P_{I_1}\textbf{x}$ and $\textbf{x}_{I_2}=P_{I_2}\textbf{x}$. Thus $\det \left(D(P_{I_1}P_{I_2}^{-1})(\textbf{x}_{I_2})\right)$ and $\det \left(D(P_{I_2}P_{I_1}^{-1})(\textbf{x}_{I_1})\right)$ are both positive numbers or negative numbers. Motivated by this fact and the equivalent definition of orientable finite dimensional manifold as in \cite[(3.18), p. 89]{CCL}, we have the following definition.
\section{A Local Version of the Gauss--Green-Type Theorem}

To establish a global version of the Gauss--Green-type theorem, we first need an appropriate analogue of a continuous unit normal vector field on the surface.

Let \(S\) be a \(C^1\)-differentiable surface in \(\ell^2\) of codimension
\(\Gamma_{\mathbb{N}\setminus I}\). Fix \(\mathbf{x}\in S\), and suppose that
\((S\cap U_1,I_1,f_1)\) and \((S\cap U_2,I_2,f_2)\) are two local coordinate triples of \(S\) such that
\[
\mathbf{x}\in S\cap U_1\cap U_2.
\]
By \eqref{20241013for1}, we have
\[
\det\left(
D\bigl(P_{I_1}  P_{I_2}^{-1}\bigr)
(\mathbf{x}_{I_2})
\right)\cdot
\det\left(
D\bigl(P_{I_2}  P_{I_1}^{-1}\bigr)
(\mathbf{x}_{I_1})
\right)
=1,
\]
where
$
\mathbf{x}_{I_1}\triangleq P_{I_1}\mathbf{x}
$ and $
\mathbf{x}_{I_2}\triangleq P_{I_2}\mathbf{x}.
$
Consequently,
\[
\det\left(
D\bigl(P_{I_1}  P_{I_2}^{-1}\bigr)
(\mathbf{x}_{I_2})
\right)
\quad\text{and}\quad
\det\left(
D\bigl(P_{I_2}  P_{I_1}^{-1}\bigr)
(\mathbf{x}_{I_1})
\right)
\]
have the same sign: they are either both positive or both negative.

Motivated by this observation and by the equivalent characterization of orientability for finite-dimensional manifolds given in
\cite[(3.18), p.~89]{CCL}, we introduce the following definition.
\begin{definition}\label{20241015def1}
Let $\mathcal {A}$ be a $C^1$-differentiable structure on a surface $S$  of $\ell^2$ with codimension
\(\Gamma_{\mathbb{N}\setminus I_0}\). We say that \(S\) is oriented if,
for every \(\mathbf{x}\in S\) and every pair of local coordinate triples
$
(S\cap U_1, I_1,f_1)$ and $(S\cap U_2, I_2,f_2)
$
in $\mathcal {A}$ satisfying
$
\mathbf{x}\in S\cap U_1\cap U_2,
$
the transition map satisfies
\begin{equation}\label{20241015for3}
\det\left(
D\bigl(P_{I_2}  P_{I_1}^{-1}\bigr)
(\mathbf{x}_{I_1})
\right)>0,
\end{equation}
where
$
\mathbf{x}_{I_1}\triangleq P_{I_1}\mathbf{x},
$
and
\begin{equation}\label{20250205for5}
n_{I_1}(\mathbf{x})
=
\left(
\sum_{I'\in\Gamma_I}
\left|
\det\left(
D\bigl(P_{I'} P_{I_1}^{-1}\bigr)
(\mathbf{x}_{I_1})
\right)
\right|^2
\right)^{1/2}
<\infty.
\end{equation}
\end{definition}

We shall also need the following notation in the definition of a region with boundary.

\begin{definition}\label{20250111def1}
Let \(S_{\mathbb{N}}\) denote the collection of all nonempty finite or
countably infinite sequences of distinct positive integers. More precisely,
an element \(\mathbf{I}\in S_{\mathbb{N}}\) is of the form
\[
\mathbf{I}=(i_k)_{k\in\Delta},
\]
where the index set $\Delta$ is one of the following:
$$
\{1,2,\ldots,n\},\qquad  \{0,1,2,\ldots,n\}(n\in\mathbb{N}),\qquad
\mathbb{N},\qquad \mathbb{N}_0,
$$
and the integers \(i_k\) are pairwise distinct.

For each \(\mathbf{I}=(i_k)_{k\in\Delta}\in S_{\mathbb{N}}\), there exists a unique
increasing rearrangement
\[
\overline{\mathbf{I}}=(i_k')_{k\in\Delta}\in S_{\mathbb{N}}
\]
such that
\[
\{i_k:k\in\Delta\}=\{i_k':k\in\Delta\}
\]
and
$
i_r'<i_s',
$
for all $s,t\in \Delta$ satisfying $s<t$.

Let
\[
|\mathbf{I}|\triangleq\{i_k:k\in\Delta\}
\]
denote the underlying set of \(\mathbf{I}\). A bijection
\[
\sigma:|\mathbf{I}|\longrightarrow |\mathbf{I}|
\]
is called a permutation of \(\mathbf{I}\), in analogy with its
finite-dimensional counterpart; see \cite[p.~8]{Lan1}. We define
\[
\sigma(\mathbf{I})\triangleq(\sigma(i_k))_{k\in\Delta}.
\]

A permutation \(\sigma\) of \(\mathbf{I}\) is called a transposition if there
exist two distinct elements \(p,q\in |\mathbf{I}|\) such that
\[
\sigma(p)=q,\qquad \sigma(q)=p,
\]
and
\[
\sigma(r)=r
\qquad
\text{for every }r\in |\mathbf{I}|\setminus\{p,q\}.
\]
This terminology is again consistent with the finite-dimensional case;
see \cite[p.~13]{Lan1}.

Let \(S_{\mathbb{N}}^{F}\) denote the collection of all
\(\mathbf{I}\in S_{\mathbb{N}}\) for which there exist finitely many transpositions
\(\sigma_1,\ldots,\sigma_m\) such that
\[
(\sigma_m\circ\cdots\circ\sigma_1)(\mathbf{I})=\overline{\mathbf{I}}.
\]
For \(\mathbf{I}\in S_{\mathbb{N}}^{F}\), define
\[
s(\mathbf{I})\triangleq (-1)^m,
\]
where \(\sigma_1,\ldots,\sigma_m\) are any transpositions satisfying the
preceding identity. The value \(s(\mathbf{I})\) is called the sign of \(\mathbf{I}\).
\end{definition}

\begin{remark}
For \(\mathbf{I}\in S_{\mathbb{N}}^{F}\), the definition of $s(\mathbf{I})$ is independent of the chosen transpositions. In the present setting, this fact can be reduced to the standard result for finite permutations. Although \(\mathbf{I}\) may be an infinite sequence, the permutation under consideration is a composition of only finitely many transpositions and therefore has finite support. By restricting it to a finite set containing its support, we may apply the uniqueness of the parity of a permutation in a finite symmetric group.
\end{remark}

Now we construct a unit normal vector field on an oriented surface \(S\), which is an essential ingredient in the Gauss--Green-type theorems.

\begin{definition}\label{20260801def1}
Let \(S\) be an oriented \(C^1\)-differentiable surface of \(\ell^2\) with codimension
\(\Gamma_{\mathbb{N}\setminus I_0}\). For each \(\mathbf{x}\in S\), choose a local coordinate triple
$
(S\cap U, I,f)
$ such that \(\mathbf{x}\in S\cap U\).
For each \(\mathbf{I}\in S_{\mathbb{N}}^F\) satisfying
$
|\mathbf{I}|\in\Gamma_{I_0},
$
define
\[
n_{\mathbf{I}}^S(\mathbf{x})
\triangleq
\frac{
s(\mathbf{I})\cdot
\det\left(
D\bigl(P_{|\mathbf{I}|}  P_{I}^{-1}\bigr)
(P_{I}\mathbf{x})
\right)
}{
n_{I}(\mathbf{x})
}.
\]
In addition, set
$
n_{\emptyset}^S(\mathbf{x})
\triangleq 0.
$ We then define
\[
\mathbf{n}^S(\mathbf{x})
\triangleq
\left(
\frac{
 \det\left(
D\bigl(P_{I_1}  P_{I}^{-1}\bigr)
(P_{I}\mathbf{x})
\right)
}{
n_{I}(\mathbf{x})
}
\right)_{I_1\in\Gamma_{I_0}}.
\]
The mapping
$
\mathbf{n}^S:S\longrightarrow \ell^2(\Gamma_{I_0})
$
is called the unit normal vector field on \(S\).
\end{definition}

First, by Definition \ref{20241012def1} and \eqref{20250205for5}, we have
\[
1\leqslant n_{I}(\mathbf{x})<\infty.
\]
Hence, the quotient appearing in the above definition is well defined as an expression. It remains to show that its value is independent of the choice of local coordinates.

Suppose that \((S\cap U_1, I_1,f_1)\) is another local coordinate triple such that \(\mathbf{x}\in S\cap U_1\cap U\). Combining \eqref{20250205for4} with Lemma \ref{20241012lem3}, we obtain
\begin{eqnarray*}
\frac{\det\!\left(D(P_{|\mathbf{I}|}P_{I}^{-1})(P_{I}\mathbf{x})\right)}
{n_{I}(\mathbf{x})}
&=
\frac{
\det\!\left(D(P_{|\mathbf{I}|}P_{I_1}^{-1})(P_{I_1}\mathbf{x})\right)\cdot
\det\!\left(D(P_{I_1}P_{I}^{-1})(P_{I}\mathbf{x})\right)}
{
n_{I_1}(\mathbf{x})
\left|
\det\!\left(D(P_{I_1}P_{I}^{-1})(P_{I}\mathbf{x})\right)
\right|
} \
&=
\frac{\det\!\left(D(P_{|\mathbf{I}|}P_{I_1}^{-1})(P_{I_1}\mathbf{x})\right)}
{n_{I_1}(\mathbf{x})},
\end{eqnarray*}
where the second equality follows from \eqref{20241015for3}. Therefore, the definition is independent of the choice of local coordinate triple.

Notice that
\[
\det\!\left(D(P_{|\mathbf{I}|}P_{I}^{-1})(P_{I}\mathbf{x})\right)
\]
is locally continuous. On the other hand, we can only conclude that \(n_{I}(\mathbf{x})\) is Borel measurable, whereas its finite-dimensional counterpart is a finite sum and is therefore locally continuous. Consequently, \(n_{\mathbf{I}}^{S}\) is Borel measurable.

By \eqref{20250205for6}, we have
\[
\mathbf{n}^{S}(\mathbf{x})\in L^{2}(\Gamma_{I_0})
\qquad\text{and}\qquad
\left\|\mathbf{n}^{S}(\mathbf{x})\right\|_{L^{2}(\Gamma_{I_0})}=1.
\]
Thus, \(\mathbf{n}^{S}\) is a mapping from \(S\) into the unit sphere of \(L^{2}(\Gamma_{I_0})\), and may be regarded as a Borel measurable unit normal vector field on \(S\).

\begin{definition}\label{20241009for1}
Let \(S\) be a \(\sigma\)-finite oriented \(C^{1}\)-differentiable surface of \(\ell^{2}\) with codimension \(\Gamma_{\mathbb{N}\setminus I_0}\). We also assume that every local coordinate triple of \(S\) is a bounded local coordinate triple.  A \textbf{region with boundary} in \(S\) is a closed subset \(D\) of \(\ell^{2}\), contained in \(S\), whose points are of the following two types:
\begin{itemize}
\item[\((1)\)] A point \(\mathbf{x}\in D\) is called an \textbf{interior point} of \(D\) if there exists a local coordinate triple
$
(S\cap U, I, f)
$
of \(S\) such that
$
\mathbf{x}\in S\cap U\subset D.
$

\item[\((2)\)] A point \(\mathbf{x}\in D\) is called a \textbf{boundary point} if there exist
\[
\mathbf{K}=(k_{0},k_{1},\ldots)\in S_{\mathbb{N}}^{F},
\quad
 i_{0} \in \mathbb{N}_0,\quad k_{i_0}\in |\mathbf{K}|,\quad \mathbf{K}\setminus\{k_{i_0}\}\triangleq (k_{0},k_{1},\ldots,k_{i_0-1},k_{i_0+1},\ldots)\in  S_{\mathbb{N}}^{F},
\]
an open neighborhood \(U\) of \(\mathbf{x}\) in \(\ell^{2}\), an open neighborhood \(U_{|\mathbf{K}|}\) of \(P_{|\mathbf{K}|}\mathbf{x}\) in \(P_{|\mathbf{K}|}\ell^2\), an open neighborhood
$
V_{|\mathbf{K}|\setminus\{k_{i_{0}}\}}
$
of \(P_{|\mathbf{K}|\setminus\{k_{i_{0}}\}}\mathbf{x}\) in \(P_{|\mathbf{K}|\setminus{\{k_{i_{0}}\}}}\ell^2\), and mappings
\[
\Phi\in C^{1}_{P_{|\mathbf{K}|\setminus\{k_{i_{0}}\}}\ell^2}
\left(
V_{|\mathbf{K}|\setminus\{k_{i_{0}}\}};
P_{\mathbb{N}\setminus|\mathbf{K}|}\ell^2
\right)
\]
and
$
h
\in
C^{1}_{P_{|\mathbf{K}|\setminus\{k_{i_{0}}\}}\ell^2}
\left(
V_{|\mathbf{K}|\setminus\{k_{i_{0}}\}};
\mathbb{R}
\right),
$
such that
$
s(\mathbf{K})=1,$ $
|\mathbf{K}|\in\Gamma_{I_0},
$ $P_{|\mathbf{K}|\setminus{\{k_{i_{0}}\}}}U_{|\mathbf{K}|}= V_{|\mathbf{K}|\setminus\{k_{i_{0}}\}}$,
and
\[
\mathbf{x}
=
P_{|\mathbf{K}|\setminus\{k_{i_{0}}\}}\mathbf{x}
+
h\!\left(
P_{|\mathbf{K}|\setminus\{k_{i_{0}}\}}\mathbf{x}
\right)\mathbf{e}_{k_{i_{0}}}
+
\Phi\!\left(
P_{|\mathbf{K}|\setminus\{k_{i_{0}}\}}\mathbf{x}
\right).
\]
Define
\[
f(P_{|\mathbf{K}|\setminus\{k_{i_{0}}\}}\mathbf{x})
\triangleq
h\!\left(
P_{|\mathbf{K}|\setminus\{k_{i_{0}}\}}\mathbf{x}
\right)\mathbf{e}_{k_{i_{0}}}
+
\Phi\!\left(
P_{|\mathbf{K}|\setminus\{k_{i_{0}}\}}\mathbf{x}
\right).
\]
Then $f\in C^{1}_{P_{|\mathbf{K}|\setminus\{k_{i_{0}}\}}\ell^2}
\left(
V_{|\mathbf{K}|\setminus\{k_{i_{0}}\}};
P_{(\mathbb{N}\setminus|\mathbf{K}|)\cup\{k_{i_{0}}\}}\ell^2
\right)$.
Moreover,
\begin{equation}\label{20241015for1}
D\cap U
=\left\{\mathbf{y}_{|\mathbf{K}|}+\Phi\!\left(
P_{|\mathbf{K}|\setminus\{k_{i_{0}}\}}
\mathbf{y}_{|\mathbf{K}|}
\right):
\mathbf{y}_{|\mathbf{K}|}\in U_{|\mathbf{K}|},
\quad
\rho(\mathbf{y}_{|\mathbf{K}|})\leqslant 0
\right\},
\end{equation}
where
$
\rho(\mathbf{y}_{|\mathbf{K}|})
\triangleq
(-1)^{i_{0}-1}
\left(\left\langle\mathbf{y}_{|\mathbf{K}|},\mathbf{e}_{k_{i_{0}}}\right\rangle-h\!\left(P_{|\mathbf{K}|\setminus\{k_{i_{0}}\}}
\mathbf{y}_{|\mathbf{K}|}\right)\right),
$
for every \(\mathbf{y}_{|\mathbf{K}|}\in U_{|\mathbf{K}|}\).
\end{itemize}

We denote the sets of interior and boundary points of \(D\) by \(D^{\circ}\) and \(\partial D\), respectively. We further require that
\begin{eqnarray*}
(\partial D)\cap U
&=&\left\{\mathbf{y}_{|\mathbf{K}|\setminus\{k_{i_{0}}\}}
+
h\!\left(\mathbf{y}_{|\mathbf{K}|\setminus\{k_{i_{0}}\}}\right)\mathbf{e}_{k_{i_{0}}}
+
\Phi\!\left(\mathbf{y}_{|\mathbf{K}|\setminus\{k_{i_{0}}\}}\right):
\mathbf{y}_{|\mathbf{K}|\setminus\{k_{i_{0}}\}}
\in V_{|\mathbf{K}|\setminus\{k_{i_{0}}\}}
\right\}\\
&=&\left\{\mathbf{y}_{|\mathbf{K}|}+\Phi\!\left(
P_{|\mathbf{K}|\setminus\{k_{i_{0}}\}}
\mathbf{y}_{|\mathbf{K}|}
\right):
\mathbf{y}_{|\mathbf{K}|}\in U_{|\mathbf{K}|},
\quad
\rho(\mathbf{y}_{|\mathbf{K}|})= 0
\right\};\\
D^{\circ}\cap U
&=&\left\{\mathbf{y}_{|\mathbf{K}|}+
\Phi\!\left(P_{|\mathbf{K}|\setminus\{k_{i_{0}}\}}\mathbf{y}_{|\mathbf{K}|}\right):
\mathbf{y}_{|\mathbf{K}|}\in U_{|\mathbf{K}|},
\quad
\rho(\mathbf{y}_{|\mathbf{K}|})<0
\right\};
\end{eqnarray*}
for every $\mathbf{y}=\mathbf{y}_{|\mathbf{K}|\setminus\{k_{i_{0}}\}}
+
h\!\left(\mathbf{y}_{|\mathbf{K}|\setminus\{k_{i_{0}}\}}\right)\mathbf{e}_{k_{i_{0}}}
+
\Phi\!\left(\mathbf{y}_{|\mathbf{K}|\setminus\{k_{i_{0}}\}}\right)\in(\partial D)\cap U$, there exists $\{y_{k_{i_0}}^n\}_{n=1}^{\infty}\in\mathbb{R}$ such that
\begin{eqnarray}\label{20260723for1}
\left\{\mathbf{y}^n_{|\mathbf{K}|}\triangleq y_{k_{i_0}}^n\mathbf{e}_{k_{i_0}}+P_{|\mathbf{K}|\setminus\{k_{i_0}\}}\mathbf{y}\right\}_{n=1}^{\infty}\subset U_{|\mathbf{K}|},\quad \left\{\mathbf{y}_n \triangleq P_{|\mathbf{K}|}^{-1}\mathbf{y}^n_{|\mathbf{K}|}\right\}_{n=1}^{\infty}\subset D^{\circ}\cap U,\quad\lim\limits_{n\to\infty}\mathbf{y}_n=\mathbf{y}.
\end{eqnarray}
Finally, the restriction
\[
\left.
P_{|\mathbf{K}|\setminus\{k_{i_{0}}\}}
\right|_{(\partial D)\cap U}
\]
is required to be a homeomorphism from \((\partial D)\cap U\) onto
$
V_{|\mathbf{K}|\setminus\{k_{i_{0}}\}}.
$
\end{definition}

In the setting of Definition~\ref{20241009for1}, define
\[
\begin{aligned}
\mathcal{A}_0
\triangleq
\bigl\{
&(S\cap U, I,f) :
(S\cap U,I,f) \text{ is a local coordinate triple of } S,\
& S\cap U\subset D
\bigr\}.
\end{aligned}
\]
Moreover, let \(\mathcal{A}_1\) be the collection of all triples of the form
\[
\bigl((\partial D)\cap U, \mathbf{K}\setminus\{k_{i_0}\} , f\bigr)
\]
satisfying the assumptions of part~\((2)\), and define
\[
|\mathcal{A}_1|\triangleq \{\bigl((\partial D)\cap U, |\mathbf{K}|\setminus\{k_{i_0}\} , f\bigr):\bigl((\partial D)\cap U, \mathbf{K}\setminus\{k_{i_0}\} , f\bigr)\in \mathcal{A}_1\}.
\]

First, it is straightforward to verify that \(\mathcal{A}_0\) defines a \(C^1\)-differentiable structure on \(D^\circ\). With respect to this structure, \(D^\circ\) is a \(\sigma\)-finite oriented \(C^1\)-differentiable surface of \(\ell^2\) with codimension
\(\Gamma_{\mathbb N\setminus I}\).

Second, the coordinate neighborhoods appearing in \(\mathcal{A}_1\) cover \(\partial D\). We shall prove that \(|\mathcal{A}_1|\) defines a \(C^1\)-differentiable structure on \(\partial D\), and that the additional sign information encoded in \(\mathcal{A}_1\) induces an orientation on this structure.

\begin{proposition}\label{20250109prop1}
Adopt the notation and assumptions of Definition~\ref{20241009for1}. Then \(|\mathcal A_1|\) defines a \(C^1\)-differentiable structure on \(\partial D\). With respect to this structure, \(\partial D\) is a \(C^1\)-differentiable surface of \(\ell^2\) with codimension \(\Gamma_{(\mathbb{N}\setminus I_0)\cup\{i\}}\) for some \(i\in  I_0\).

Moreover, for any two triples
$$((	\partial D)\cap U,\mathbf{K}\setminus\{k_{i_0}\},f), ((\partial D)\cap U_1,\mathbf{K}'\setminus\{k_{i_1}'\},g)\in \mathcal A_1$$
 such that $(\partial D)\cap U\cap U_1\neq \emptyset$. Then the transition map
$P_{|\mathbf{K}'|\setminus\{k_{i_1}'\}}P_{|\mathbf{K}|\setminus\{k_{i_0}\}}^{-1}$ is Fr\'{e}chet differentiable on its domain, and
\begin{eqnarray}\label{20250109for3}
s(\mathbf{K}\setminus\{k_{i_0}\})\cdot s(\mathbf{K}'\setminus\{k_{i_1}'\})\cdot \det \left(D( P_{|\mathbf{K}'|\setminus\{k_{i_1}'\}}P_{|\mathbf{K}|\setminus\{k_{i_0}\}}^{-1})(P_{|\mathbf{K}|\setminus\{k_{i_0}\}}\mathbf{x} )\right)>0,
\end{eqnarray}
for every $\mathbf{x} \in (	\partial D)\cap U\cap U_1$.
\end{proposition}

\begin{proof}
Let $F\left(\textbf{x}_{|\mathbf{K}|\setminus\{k_{i_0}\}}\right)\triangleq \textbf{x}_{|\mathbf{K}|\setminus\{k_{i_0}\}}+h(\textbf{x}_{|\mathbf{K}|\setminus\{k_{i_0}\}})\textbf{e}_{k_{i_0}}$, for every $\textbf{x}_{|\mathbf{K}|\setminus\{k_{i_0}\}}\in V_{|\mathbf{K}|\setminus\{k_{i_0}\}}$. Since $h$ is Fr\'{e}chet differentiable, so is $F$. Note that
$$
P_{|\mathbf{K}'| \setminus\{k_{i_1}'\}}P_{|\mathbf{K}|\setminus\{k_{i_0}\}}^{-1}
=P_{|\mathbf{K}'|\setminus\{k_{i_1}'\}}\circ(P_{|\mathbf{K}'|}P_{|\mathbf{K}|}^{-1})\circ F,
$$
which implies that it is a Fr\'{e}chet differentiable mapping. By the assumptions, there exists $n\in\mathbb{N}$ such that
$$
n>\max\{i_0,i_1\},\quad \mathbf{K}=(k_0,k_1,\ldots,k_n,k_{n+1},\cdots),\quad \mathbf{K}'=(k'_0,k'_1,\ldots,k_n',k_{n+1}',\cdots),
$$
$k_i=k_i'$, for every $i\geqslant n+1$, $k_{n+1}<k_{n+2}<\cdots<k_j<\cdots$, $k_{n+1}>\max\{k_0,k_1,\ldots,k_n,k'_0,k'_1,\ldots,k_n'\}$ and $\det\left(D(P_{|\mathbf{K}'|}P_{|\mathbf{K}|}^{-1})(P_{|\mathbf{K}|}\textbf{x})\right)>0$, for every $\textbf{x}\in  D\cap U\cap U_1.$ This implies that the following defines a strictly positive function on $P_{|\mathbf{K}|\setminus\{k_{i_0}\}}((\partial D)\cap U\cap U_1)$
\begin{eqnarray}\label{20250108for1}
s(\mathbf{K}) s(\mathbf{K}')
\begin{vmatrix}D_{x_{k_{0}}}\left( P_{\{k_0'\}}P_{|\mathbf{K}|}^{-1} \right)\circ F &\cdots&D_{x_{k_{i}}}\left( P_{\{k_0'\}}P_{|\mathbf{K}|}^{-1} \right)\circ F  &\cdots&D_{x_{k_{n}}}\left( P_{\{k_0'\}}P_{|\mathbf{K}|}^{-1} \right)\circ F  \\ \vdots&\ddots&\vdots&\ddots&\vdots\\
D_{x_{k_{0}}}\left( P_{\{k_j'\}}P_{|\mathbf{K}|}^{-1} \right)\circ F  &\cdots&D_{x_{k_{i}}}\left( P_{\{k_j'\}}P_{|\mathbf{K}|}^{-1} \right) \circ F &\cdots&D_{x_{k_{n}}}\left(P_{\{k_j'\}} P_{|\mathbf{K}|}^{-1} \right)\circ F  \\ \vdots&\ddots&\vdots&\ddots&\vdots\\
D_{x_{k_{0}}}\left( P_{\{k_n'\}}P_{|\mathbf{K}|}^{-1} \right)\circ F  &\cdots&D_{x_{k_{i}}}\left( P_{\{k_n'\}}P_{|\mathbf{K}|}^{-1} \right) \circ F &\cdots&D_{x_{k_{n}}}\left( P_{\{k_n'\}}P_{|\mathbf{K}|}^{-1} \right)\circ F  \end{vmatrix}  ,
\end{eqnarray}
where
$$
D_{x_{k_{i}}}\left( P_{\{k_j'\}}P_{|\mathbf{K}|}^{-1} \right)\triangleq  D_{x_{k_{i}}}\left( \left\langle P_{\{k_j'\}}P_{|\mathbf{K}|}^{-1}\cdot,\mathbf{e}_{k_j'}\right\rangle \right).
$$
For each
$$i\in\{0,\ldots,n\}\setminus\{i_0\}.$$
Multiply the $i_0$-th column of the matrix in \eqref{20250108for1} by $D_{x_{k_i}}(P_{\{k_{i_0}\}}P^{-1}_{|\mathbf{K}|\setminus\{k_{i_0}\}})$ and add the resulting column to the $i$-th column.
 For every $j\in\{0,\ldots, n\}$, we have
$$
(P_{\{k_j'\}}P^{-1}_{|\mathbf{K}|\setminus\{k_{i_0}\}})=(P_{\{k_j'\}}P^{-1}_{|\mathbf{K}|})\circ F.
$$
By the chain rule,
\begin{eqnarray*}
D_{x_{k_i}}(P_{\{k_j'\}}P^{-1}_{|\mathbf{K}|\setminus\{k_{i_0}\}})=D_{x_{k_i}}(P_{\{k_j'\}}P^{-1}_{|\mathbf{K}|})\circ F
+\left(D_{x_{k_i}}(P_{\{k_{i_0}\}}P^{-1}_{|\mathbf{K}|\setminus\{k_{i_0}\}})\right)
\cdot \left(D_{x_{k_{i_0}}}(P_{\{k_j'\}}P^{-1}_{|\mathbf{K}|})\circ F\right),
\end{eqnarray*}
and hence we obtain
\begin{eqnarray}\label{20250108for2}
s(\mathbf{K}) s(\mathbf{K}')
\begin{vmatrix}D_{x_{k_{0}}}\left( P_{\{k_0'\}}P^{-1}_{|\mathbf{K}|\setminus\{k_{i_0}\}}\right) &\cdots&D_{x_{k_{i_0}}}\left( P_{\{k_0'\}}P_{|\mathbf{K}|}^{-1} \right)\circ F  &\cdots&D_{x_{k_{n}}}\left( P_{\{k_0'\}}P^{-1}_{|\mathbf{K}|\setminus\{k_{i_0}\}}\right) \\ \vdots&\ddots&\vdots&\ddots&\vdots\\
D_{x_{k_{0}}}\left( P_{\{k_{i_1}'\}}P^{-1}_{|\mathbf{K}|\setminus\{k_{i_0}\}}\right)  &\cdots&D_{x_{k_{i_0}}}\left( P_{\{k_{i_1}'\}}P_{|\mathbf{K}|}^{-1} \right) \circ F &\cdots&D_{x_{k_{n}}}\left(P_{\{k_{i_1}'\}} P^{-1}_{|\mathbf{K}|\setminus\{k_{i_0}\}}\right) \\ \vdots&\ddots&\vdots&\ddots&\vdots\\
D_{x_{k_{0}}}\left( P_{\{k_n'\}}P^{-1}_{|\mathbf{K}|\setminus\{k_{i_0}\}}\right)  &\cdots&D_{x_{k_{i_0}}}\left( P_{\{k_n'\}}P_{|\mathbf{K}|}^{-1} \right) \circ F &\cdots&D_{x_{k_{n}}}\left( P_{\{k_n'\}}P^{-1}_{|\mathbf{K}|\setminus\{k_{i_0}\}}\right) \end{vmatrix} ,
\end{eqnarray}
is a strictly positive function on $P_{|\mathbf{K}|\setminus\{k_{i_0}\}}( (\partial D)\cap U\cap U_1)$.

Note that for every $\textbf{x}_{|\mathbf{K}|\setminus\{k_{i_0}\}}\in P_{|\mathbf{K}|\setminus\{k_{i_0}\}}( (\partial D)\cap U\cap U_1)$, it holds that
\begin{eqnarray*}
(P_{|\mathbf{K}'|}P_{|\mathbf{K}|\setminus\{k_{i_0}\}}^{-1})\textbf{x}_{|\mathbf{K}|\setminus\{k_{i_0}\}}
=(P_{|\mathbf{K}'|}P_{|\mathbf{K}'|\setminus\{k_{i_1}'\}}^{-1})(P_{|\mathbf{K}'|\setminus\{k_{i_1}'\}}P_{|\mathbf{K}|\setminus\{k_{i_0}\}}^{-1})\textbf{x}_{|\mathbf{K}|\setminus\{k_{i_0}\}},
\end{eqnarray*}
which implies that
\begin{eqnarray}\label{20250108for3}
D(P_{|\mathbf{K}'|}P_{|\mathbf{K}|\setminus\{k_{i_0}\}}^{-1})
=D(P_{|\mathbf{K}'|}P_{|\mathbf{K}'|\setminus\{k_{i_1}'\}}^{-1})(P_{|\mathbf{K}'|\setminus\{k_{i_1}'\}}P_{|\mathbf{K}|\setminus\{k_{i_0}\}}^{-1})\circ D(P_{|\mathbf{K}'|\setminus\{k_{i_1}'\}}P_{|\mathbf{K}|\setminus\{k_{i_0}\}}^{-1}),
\end{eqnarray}
on $P_{|\mathbf{K}|\setminus\{k_{i_0}\}}( (\partial D)\cap U\cap U_1)$. Observe that for every $i\in\{0,\ldots,n\}\setminus\{i_0\}$, \eqref{20250108for3} implies that
\begin{eqnarray}
&&D_{x_{k_{i}}}\left( P_{\{k_{i_1}'\}}P^{-1}_{|\mathbf{K}|\setminus\{k_{i_0}\}}\right)\nonumber\\
&=&\left\langle D(P_{|\mathbf{K}'|}P_{|\mathbf{K}|\setminus\{k_{i_0}\}}^{-1})\textbf{e}_{k_i},\textbf{e}_{k_{i_1}'}\right\rangle\nonumber\\
&=&\left\langle D(P_{|\mathbf{K}'|}P_{|\mathbf{K}'|\setminus\{k_{i_1}'\}}^{-1})\circ D(P_{|\mathbf{K}'|\setminus\{k_{i_1}'\}}P_{|\mathbf{K}|\setminus\{k_{i_0}\}}^{-1})\textbf{e}_{k_i},\textbf{e}_{k_{i_1}'}\right\rangle\nonumber\\
&=&\left\langle D(P_{|\mathbf{K}'|}P_{|\mathbf{K}'|\setminus\{k_{i_1}'\}}^{-1})(P_{|\mathbf{K}'|\setminus\{k_{i_1}'\}}P_{|\mathbf{K}|\setminus\{k_{i_0}\}}^{-1}) \sum_{k'\in |\mathbf{K}'|\setminus\{k_{i_1}'\} }D_{x_{k_i}}(P_{\{k'\}}P_{|\mathbf{K}|\setminus\{k_{i_0}\}}^{-1})\textbf{e}_{k'},\textbf{e}_{k_{i_1}'}\right\rangle\nonumber\\
&=& \sum_{k'\in |\mathbf{K}'|\setminus\{k_{i_1}'\} }D_{x_{k_i}}(P_{\{k'\}}P_{|\mathbf{K}|\setminus\{k_{i_0}\}}^{-1})\cdot\left\langle D(P_{|\mathbf{K}'|}P_{|\mathbf{K}'|\setminus\{k_{i_1}'\}}^{-1})(P_{|\mathbf{K}'|\setminus\{k_{i_1}'\}}P_{|\mathbf{K}|\setminus\{k_{i_0}\}}^{-1})\textbf{e}_{k'},\textbf{e}_{k_{i_1}'}\right\rangle\nonumber\\
&=& \sum_{k'\in |\mathbf{K}'|\setminus\{k_{i_1}'\} }D_{x_{k_i}}(P_{\{k'\}}P_{|\mathbf{K}|\setminus\{k_{i_0}\}}^{-1})\cdot D_{x_{k'}}(P_{\{k_{i_1}'\}}P_{|\mathbf{K}'|\setminus\{k_{i_1}'\}}^{-1})(P_{|\mathbf{K}'|\setminus\{k_{i_1}'\}}P_{|\mathbf{K}|\setminus\{k_{i_0}\}}^{-1}) \nonumber\\
&=& \sum_{j\in \{0,1,\ldots,n\}\setminus\{ i_1 \} }D_{x_{k_i}}(P_{\{k'_j\}}P_{|\mathbf{K}|\setminus\{k_{i_0}\}}^{-1})\cdot D_{x_{k'_j}}(P_{\{k_{i_1}'\}}P_{|\mathbf{K}'|\setminus\{k_{i_1}'\}}^{-1})(P_{|\mathbf{K}'|\setminus\{k_{i_1}'\}}P_{|\mathbf{K}|\setminus\{k_{i_0}\}}^{-1}),\label{20250109for1}
\end{eqnarray}
where the last equality follows from the fact that
$$D_{x_{k_i}}(P_{\{k'\}}P_{|\mathbf{K}|\setminus\{k_{i_0}\}}^{-1})=0$$
 for every $k'\in  |\mathbf{K}'|\setminus\{k_0',\ldots,k_{n}'\}.$ For each
 $$j\in\{0,\ldots,n\}\setminus\{i_1\},$$
 multiply the $j$-th row by $-D_{x_{k_j'}}(P_{\{k_{i_1 }'\}}P^{-1}_{|\mathbf{K}'|\setminus\{k_{i_1}'\}})(P_{|\mathbf{K}'|\setminus\{k_{i_1}'\}}P_{|\mathbf{K}|\setminus\{k_{i_0}\}}^{-1})$ and add the resulting row to the $i_1$-th row, then we have
\begin{eqnarray}\label{20250109for2}
s(\mathbf{K})s(\mathbf{K}')
\begin{vmatrix}D_{x_{k_{0}}}\left( P_{\{k_0'\}}P^{-1}_{|\mathbf{K}|\setminus\{k_{i_0}\}}\right) &\cdots&D_{x_{k_{i_0}}}\left( P_{\{k_0'\}}P_{|\mathbf{K}|}^{-1} \right)\circ F  &\cdots&D_{x_{k_{n}}}\left( P_{\{k_0'\}}P^{-1}_{|\mathbf{K}|\setminus\{k_{i_0}\}}\right) \\ \vdots&\ddots&\vdots&\ddots&\vdots\\
0&\cdots&G &\cdots&0\\ \vdots&\ddots&\vdots&\ddots&\vdots\\
D_{x_{k_{0}}}\left( P_{\{k_n'\}}P^{-1}_{|\mathbf{K}|\setminus\{k_{i_0}\}}\right)  &\cdots&D_{x_{k_{i_0}}}\left( P_{\{k_n'\}}P_{|\mathbf{K}|}^{-1} \right) \circ F &\cdots&D_{x_{k_{n}}}\left( P_{\{k_n'\}}P^{-1}_{|\mathbf{K}|\setminus\{k_{i_0}\}}\right) \end{vmatrix}
\end{eqnarray}
is a strictly positive function on $P_{|\mathbf{K}|\setminus\{k_{i_0}\}}((\partial D)\cap U\cap U_1)$, where for every $\mathbf{x}_{|\mathbf{K}|\setminus\{k_{i_0}\}}\in P_{|\mathbf{K}|\setminus\{k_{i_0}\}}((\partial D)$ $\cap U\cap U_1)$,
\begin{eqnarray*}
G(\mathbf{x}_{|\mathbf{K}|\setminus\{k_{i_0}\}} )&\triangleq&D_{x_{k_{i_0}}}\left( P_{\{k_{i_1}'\}}P_{|\mathbf{K}|}^{-1} \right) ( F(\mathbf{x}_{|\mathbf{K}|\setminus\{k_{i_0}\}} ))\\
&&-\sum\limits_{j\in\{0,\ldots,n\}\setminus\{i_1\}}\left(D_{x_{k_j'}}(P_{\{k_{i_1}'\}}P^{-1}_{|\mathbf{K}'|\setminus\{k_{i_1}'\}})
(P_{|\mathbf{K}'|\setminus\{k_{i_1}'\}}P_{|\mathbf{K}|\setminus\{k_{i_0}\}}^{-1}\mathbf{x}_{|\mathbf{K}|\setminus\{k_{i_0}\}} )\right)\\
&&\qquad \qquad\qquad\qquad\qquad\qquad\qquad\times D_{x_{k_{i_0}}}\left( P_{\{k_{j}'\}}P_{|\mathbf{K}|}^{-1} \right) (F(\mathbf{x}_{|\mathbf{K}|\setminus\{k_{i_0}\}} )),
\end{eqnarray*}
and the elements of the $i_1$-th row of the matrix in \eqref{20250109for2} are zero except the $i_0$-th entry. By Remark \ref{20250505rem2},
$$
s(\mathbf{K}\setminus\{k_{i_0}\})\cdot s(\mathbf{K}'\setminus\{k_{i_1}'\}) \cdot\det \left(D( P_{|\mathbf{K}'|\setminus\{k_{i_1}'\}}P_{|\mathbf{K}|\setminus\{k_{i_0}\}}^{-1})(\mathbf{x}_{|\mathbf{K}|\setminus\{k_{i_0}\}}  )\right)
$$
is equal to the \(n\times n\) determinant obtained by deleting the \(i_1\)-th row and the \(i_0\)-th column from the determinant on the right-hand side of \eqref{20250109for2}. Expanding that determinant along its \(i_1\)-th row, we obtain
\begin{eqnarray}\label{20250109for4}
&&(-1)^{i_0+i_1}\cdot G(\mathbf{x}_{|\mathbf{K}|\setminus\{k_{i_0}\}} ) \cdot s(\mathbf{K})\cdot s(\mathbf{K}') \nonumber\\
 &&\qquad\times s(\mathbf{K}\setminus\{k_{i_0}\})\cdot s(\mathbf{K}'\setminus\{k_{i_1}'\}) \cdot\det \left(D( P_{|\mathbf{K}'|\setminus\{k_{i_1}'\}}P_{|\mathbf{K}|\setminus\{k_{i_0}\}}^{-1})(\mathbf{x}_{|\mathbf{K}|\setminus\{k_{i_0}\}}  )\right) >0,
\end{eqnarray}
for every  $\mathbf{x}_{|\mathbf{K}|\setminus\{k_{i_0}\}} \in P_{|\mathbf{K}|\setminus\{k_{i_0}\}}((\partial D)\cap U\cap U_1)$. Since the product in \eqref{20250109for4} is strictly positive,
\begin{eqnarray*}
(-1)^{i_0+i_1}\cdot G(\mathbf{x}_{|\mathbf{K}|\setminus\{k_{i_0}\}} ) \cdot s(\mathbf{K})\cdot s(\mathbf{K}')
\end{eqnarray*}
and
\begin{eqnarray*}
s(\mathbf{K}\setminus\{k_{i_0}\})\cdot s(\mathbf{K}'\setminus\{k_{i_1}'\})\cdot \det \left(D( P_{|\mathbf{K}'|\setminus\{k_{i_1}'\}}P_{|\mathbf{K}|\setminus\{k_{i_0}\}}^{-1})(\mathbf{x}_{|\mathbf{K}|\setminus\{k_{i_0}\}}  )\right)
\end{eqnarray*}
have the same sign. Consequently, \eqref{20250109for3} is equivalent to
$$
(-1)^{i_0+i_1}\cdot G(P_{|\mathbf{K}|\setminus\{k_{i_0}\}}\mathbf{x} ) \cdot s(\mathbf{K})\cdot s(\mathbf{K}')>0,\,\textbf{x} \in (	\partial D)\cap U\cap U_1.
$$
 Recall that, by Definition \ref{20241009for1}, $s(\mathbf{K})=s(\mathbf{K}')=1$ and hence $s(\mathbf{K})\cdot s(\mathbf{K}')=1$, which implies that we only need to prove that
 $$(-1)^{i_0+i_1}\cdot G(\mathbf{x}_{|\mathbf{K}|\setminus\{k_{i_0}\}} ) >0,$$
  where $\mathbf{x}_{|\mathbf{K}|\setminus\{k_{i_0}\}} \triangleq P_{|\mathbf{K}|\setminus\{k_{i_0}\}}\mathbf{x}$ for every $\textbf{x} \in (	\partial D)\cap U\cap U_1.$

Fixed $\textbf{x}\in (\partial D)\cap U\cap U_1$. Define
$$
\Psi(\mathbf{y}_{|\mathbf{K}|})\triangleq \left\langle P_{\{k_{i_1}'\}}P_{|\mathbf{K}|}^{-1}\mathbf{y}_{|\mathbf{K}|}
-    (P_{\{k_{i_1}'\}}P_{|\mathbf{K}'|\setminus \{k_{i_1}'\}}^{-1}  )\circ(P_{|\mathbf{K}'|\setminus \{k_{i_1}'\}}P_{|\mathbf{K}|}^{-1})\mathbf{y}_{|\mathbf{K}|},\mathbf{e}_{k_{i_1}'}\right\rangle,\qquad
\mathbf{y}_{|\mathbf{K}|}\in V_{|\mathbf{K}|},
$$
where $V_{|\mathbf{K}|}$ is an open neighborhood of $P_{|\mathbf{K}|}\mathbf{x}$ such that $V_{|\mathbf{K}|}\subset U_{|\mathbf{K}|}$ and the above composition is well defined.

First, observe that
\begin{eqnarray}
G(\mathbf{x}_{|\mathbf{K}|\setminus\{k_{i_0}\}} )
&=&D_{x_{k_{i_0}}}\left( P_{\{k_{i_1}'\}}P_{|\mathbf{K}|}^{-1} \right) ( F(\mathbf{x}_{|\mathbf{K}|\setminus\{k_{i_0}\}} ))\nonumber\\
&-&\sum\limits_{j\in\{0,\ldots,n\}\setminus\{i_1\}}\left(D_{x_{k_j'}}(P_{\{k_{i_1}'\}}P^{-1}_{|\mathbf{K}'|\setminus\{k_{i_1}'\}})
(P_{|\mathbf{K}'|\setminus\{k_{i_1}'\}}P_{|\mathbf{K}|\setminus\{k_{i_0}\}}^{-1}\mathbf{x}_{|\mathbf{K}|\setminus\{k_{i_0}\}} )\right)\nonumber\\
&&\qquad\qquad\qquad \qquad\qquad\qquad\qquad \times D_{x_{k_{i_0}}}\left( P_{\{k_{j}'\}}P_{|\mathbf{K}|}^{-1} \right) ( F(\mathbf{x}_{|\mathbf{K}|\setminus\{k_{i_0}\}} ))\nonumber\\
&=&D_{x_{k_{i_0}}}\left( P_{\{k_{i_1}'\}}P_{|\mathbf{K}|}^{-1} \right) ( F(\mathbf{x}_{|\mathbf{K}|\setminus\{k_{i_0}\}} ))\nonumber\\
&-&\sum_{k'\in |\mathbf{K}'|\setminus\{k_{i_1}'\} }\left(D_{x_{k'}}(P_{\{k_{i_1}'\}}P^{-1}_{|\mathbf{K}'|\setminus\{k_{i_1}'\}})(P_{|\mathbf{K}'|\setminus\{k_{i_1}'\}}P_{|\mathbf{K}|\setminus\{k_{i_0}\}}^{-1}\mathbf{x}_{|\mathbf{K}|\setminus\{k_{i_0}\}} )\right)\nonumber\\
&&\qquad\qquad\qquad\qquad\qquad\qquad\qquad \times D_{x_{k_{i_0}}}\left( P_{\{k'\}}P_{|\mathbf{K}|}^{-1} \right) ( F(\mathbf{x}_{|\mathbf{K}|\setminus\{k_{i_0}\}} ))\nonumber\\
&=&D_{x_{k_{i_0}}} \Psi ( F(\mathbf{x}_{|\mathbf{K}|\setminus\{k_{i_0}\}} )).\label{20260724for1}
\end{eqnarray}
In passing from the first equality to the second, the sum is extended from
\[
\{k_j':j\in\{0,\ldots,n\}\setminus\{i_1\}\}
\]
to the whole set \( |\mathbf{K}'|\setminus\{k_{i_1}'\} \). This extension does not change the value of the sum, because
\[
D_{x_{k_{i_0}}}
\left(
P_{\{k'\}}P_{|\mathbf{K}|}^{-1}
\right)=0
\]
for every
$
k'\in |\mathbf{K}'|\setminus\{k_0',\ldots,k_n'\}.
$
The last equality follows from the chain rule.

By the assumption preceding \eqref{20260723for1}, there exists $\{x_{k_{i_0}}^m\}_{m=1}^{\infty}\in\mathbb{R}$ such that
$$
\left\{\textbf{x}^m_{|\mathbf{K}|}\triangleq x_{k_{i_0}}^m\textbf{e}_{k_{i_0}}+P_{|\mathbf{K}|\setminus\{k_{i_0}\}}\textbf{x}\right\}_{m=1}^{\infty}\subset U_{|\mathbf{K}|},\quad \left\{\textbf{x}_m \triangleq P_{|\mathbf{K}|}^{-1}\textbf{x}^m_{|\mathbf{K}|}\right\}_{m=1}^{\infty}\subset D^{\circ}\cap U\cap U_1,\quad\lim\limits_{m\to\infty} \textbf{x}_m=\textbf{x}.
$$
Therefore,
\begin{equation}\label{20260724for2}
\lim_{m\to\infty}x_{k_{i_0}}^m=\left\langle P_{\{k_{i_0}\}}\textbf{x},\textbf{e}_{k_{i_0}}\right\rangle,\quad (-1)^{i_0-1}\left(x_{k_{i_0}}^m-\left\langle P_{\{k_{i_0}\}}\textbf{x},\textbf{e}_{k_{i_0}}\right\rangle\right)<0
,\quad (-1)^{i_1-1}\Psi(\textbf{x}^m_{|\mathbf{K}|}) <0.
\end{equation}
Since \(\mathbf x\in(\partial D)\cap U\cap U_1\),
\[
\bigl(
P_{\{k'_{i_1}\}}
P_{|\mathbf{K}'|\setminus\{k'_{i_1}\}}^{-1}
\bigr)
\bigl(
P_{|\mathbf{K}'|\setminus\{k'_{i_1}\}}\mathbf x
\bigr)
=P_{\{k'_{i_1}\}}\mathbf x,
\]
and therefore
\begin{equation}\label{20260724for5}
\Psi(P_{|\mathbf{K}|}\mathbf x)=0.
\end{equation}
Since $\lim\limits_{m\to\infty}\mathbf x^m_{|\mathbf{K}|}=P_{|\mathbf{K}|}\mathbf x$,
\[
\mathbf x^m_{|\mathbf{K}|}
=P_{|\mathbf{K}|}\mathbf x
+
\left(
x_{k_{i_0}}^m-
\left\langle
P_{\{k_{i_0}\}}\mathbf x,
\mathbf e_{k_{i_0}}
\right\rangle
\right)\mathbf e_{k_{i_0}}\in V_{|\mathbf{K}|},
\]
for sufficiently large $m$. It follows from \eqref{20260724for1} that
\begin{eqnarray}
(-1)^{i_0+i_1}
G\bigl(
P_{|\mathbf{K}|\setminus\{k_{i_0}\}}\mathbf x
\bigr)
&=&
\lim_{m\to\infty}
\frac{
(-1)^{i_1-1}
\left(
\Psi(\mathbf x^m_{|\mathbf{K}|})
-
\Psi(P_{|\mathbf{K}|}\mathbf x)
\right)
}{
(-1)^{i_0-1}
\left(
x_{k_{i_0}}^m
-\left\langle
P_{\{k_{i_0}\}}\mathbf x,
\mathbf e_{k_{i_0}}
\right\rangle
\right)
}\nonumber\\
&=&
\lim_{m\to\infty}
\frac{
(-1)^{i_1-1}\Psi(\mathbf x^m_{|\mathbf{K}|})
}{
(-1)^{i_0-1}
\left(
x_{k_{i_0}}^m
-\left\langle
P_{\{k_{i_0}\}}\mathbf x,
\mathbf e_{k_{i_0}}
\right\rangle
\right)
},\label{20260724for6}
\end{eqnarray}
where the last equality follows from \eqref{20260724for5}.

Combining \eqref{20260724for2} and \eqref{20260724for6}, we obtain
$$(-1)^{i_0+i_1}\cdot G(P_{|\mathbf{K}|\setminus\{k_{i_0}\}}\textbf{x}) \geqslant 0.$$
 It follows from \eqref{20250109for4} that
$$(-1)^{i_0+i_1}\cdot G(P_{|\mathbf{K}|\setminus\{k_{i_0}\}}\textbf{x}) \neq 0.$$
 Consequently,
$$(-1)^{i_0+i_1}\cdot G(P_{|\mathbf{K}|\setminus\{k_{i_0}\}}\textbf{x})> 0.$$
This completes the proof of Proposition \ref{20250109prop1}.
\end{proof}

\begin{definition}\label{20241013def2}
The orientation on \(\partial D\) characterized by
\eqref{20250109for3} is called the \textbf{orientation on \(\partial D\)
induced by the orientation of \(S\)}.

By the constructions and results developed from
Definition~\ref{20241113def1} through
Proposition~\ref{20241013prop1}, one can define a
\(\sigma\)-finite Borel measure \(\mu_{D^\circ}\) on \(D^\circ\) and a
Borel measure \(\mu_{\partial D}\) on \(\partial D\).
At present, it is not known whether \(\mu_{\partial D}\) is
\(\sigma\)-finite.
\end{definition}

\begin{definition}\label{20260803def1}
Adopt the notation and assumptions of
Definition~\ref{20241009for1}. For each
\(\mathbf{x}\in\partial D\), choose a triple
$
\bigl((\partial D)\cap U,
\mathbf{K}\setminus{\{k_{i_0}\}}, f\bigr)\in\mathcal A_1
$
such that \(\mathbf{x}\in(\partial D)\cap U\).
For each \(\mathbf{I}\in S_{\mathbb N}^F\) satisfying
$
|\mathbf{I}|\in\Gamma_{|\mathbf{K}|\setminus \{k_{i_0}\}},
$
define
\[
n_{\mathbf{I}}^{\partial D}(\mathbf{x})
\triangleq
\frac{
s(\mathbf{I})\cdot s(\mathbf{K}\setminus\{k_{i_0}\})\cdot
\det\left(
D\bigl(P_{|\mathbf{I}|}  P_{|\mathbf{K}|\setminus\{k_{i_0}\}}^{-1}\bigr)
(P_{|\mathbf{K}|\setminus\{k_{i_0}\}}\mathbf{x})
\right)
}{
n_{|\mathbf{K}|\setminus\{k_{i_0}\}}(\mathbf{x})
}.
\]
In addition, set
$
n_{\emptyset}^{\partial D}(\mathbf{x})
\triangleq 0.
$
\end{definition}

It remains to verify that the quotient in the preceding definition is
well defined, namely, that it is independent of the choice of the triple
in \(\mathcal A_1\).

Let
$
\bigl((\partial D)\cap U_1 ,
\mathbf{K}'\setminus\{k_{i_1}'\} ,g\bigr)\in\mathcal A_1
$
be another triple such that
$
\mathbf{x}\in(\partial D)\cap U\cap U_1.
$
Set
\[
T
\triangleq
P_{|\mathbf{K}'|\setminus\{k_{i_1}'\}}
\circ
P_{|\mathbf{K}|\setminus\{k_{i_0}\}}^{-1}.
\]
By the chain rule,
\[
\begin{aligned}
&D\!\left(
P_{|\mathbf{I}|}
P_{|\mathbf{K}|\setminus\{k_{i_0}\}}^{-1}
\right)
\left(
P_{|\mathbf{K}|\setminus\{k_{i_0}\}}\mathbf{x}
\right)
=
D\!\left(
P_{|\mathbf{I}|}
P_{|\mathbf{K}'|\setminus\{k_{i_1}'\}}^{-1}
\right)
\left(
P_{|\mathbf{K}'|\setminus\{k_{i_1}'\}}\mathbf{x}
\right)
\circ
DT\!\left(
P_{|\mathbf{K}|\setminus\{k_{i_0}\}}\mathbf{x}
\right).
\end{aligned}
\]
Consequently, using \eqref{20250205for4},
Lemma~\ref{20241012lem3}, and \eqref{20250109for3}, we obtain
\[
\begin{aligned}
&\frac{
s\bigl(\mathbf{K}\setminus\{k_{i_0}\}\bigr)
\det\!\left(
D\!\left(
P_{|\mathbf{I}|}
\circ
P_{|\mathbf{K}|\setminus\{k_{i_0}\}}^{-1}
\right)
\left(
P_{|\mathbf{K}|\setminus\{k_{i_0}\}}\mathbf{x}
\right)
\right)
}{
n_{|\mathbf{K}|\setminus\{k_{i_0}\}}(\mathbf{x})
}\\
&\quad=
\frac{
s\bigl(\mathbf{K}\setminus\{k_{i_0}\}\bigr)
\det\!\left(
D\!\left(
P_{|\mathbf{I}|}
\circ
P_{|\mathbf{K}'|\setminus\{k_{i_1}'\}}^{-1}
\right)
\left(
P_{|\mathbf{K}'|\setminus\{k_{i_1}'\}}\mathbf{x}
\right)
\right)
}{
n_{|\mathbf{K}'|\setminus\{k_{i_1}'\}}(\mathbf{x})
}
\times
\frac{
\det\!\left(
DT\!\left(
P_{|\mathbf{K}|\setminus\{k_{i_0}\}}\mathbf{x}
\right)
\right)
}{
\left|
\det\!\left(
DT\!\left(
P_{|\mathbf{K}|\setminus\{k_{i_0}\}}\mathbf{x}
\right)
\right)
\right|
}\\
&\quad=
\frac{
s\bigl(\mathbf{K}'\setminus\{k_{i_1}'\}\bigr)
\det\!\left(
D\!\left(
P_{|\mathbf{I}|}
\circ
P_{|\mathbf{K}'|\setminus\{k_{i_1}'\}}^{-1}
\right)
\left(
P_{|\mathbf{K}'|\setminus\{k_{i_1}'\}}\mathbf{x}
\right)
\right)
}{
n_{|\mathbf{K}'|\setminus\{k_{i_1}'\}}(\mathbf{x})
}.
\end{aligned}
\]
Indeed, \eqref{20250109for3} yields
\[
s\bigl(\mathbf{K}\setminus\{k_{i_0}\}\bigr)
\frac{
\det\!\left(
DT\!\left(
P_{|\mathbf{K}|\setminus\{k_{i_0}\}}\mathbf{x}
\right)
\right)
}{
\left|
\det\!\left(
DT\!\left(
P_{|\mathbf{K}|\setminus\{k_{i_0}\}}\mathbf{x}
\right)
\right)
\right|
}
=s\bigl(\mathbf{K}'\setminus\{k_{i_1}'\}\bigr).
\]
Multiplying the resulting identity by \(s(\mathbf{I})\) shows that the value of
\(n_{\mathbf{I}}^{\partial D}(\mathbf{x})\) is independent of the chosen triple
in \(\mathcal A_1\). Hence \(n_{\mathbf{I}}^{\partial D}(\mathbf{x})\) is well
defined.

\begin{proposition}\label{20241013thm1}
We adopt the notation and assumptions of Definition~\ref{20241009for1}. Let $\{X_n\}_{n=1}^{\infty}$ be the sequence of functions on $\ell^2$ defined in \eqref{230409eq1}, and suppose that $\partial D$ is a $\sigma$-finite surface. Then the following assertions hold:
\begin{itemize}
\item[$(1)$] On $D^{\circ}$,
\[
\lim_{n\to\infty}X_n=1
\quad\text{and}\quad
\lim_{n\to\infty}
\sum_{i=1}^{\infty}
a_i^2
\left|
\frac{\partial X_n}{\partial x_i}
\right|^2
=0
\]
$\mu_{D^{\circ}}$-almost everywhere.

\item[$(2)$] On $\partial D$,
\[
\lim_{n\to\infty}X_n=1
\quad\text{and}\quad
\lim_{n\to\infty}
\sum_{i=1}^{\infty}
a_i^2
\left|
\frac{\partial X_n}{\partial x_i}
\right|^2
=0
\]
$\mu_{\partial D}$-almost everywhere.
\end{itemize}
\end{proposition}

\begin{proof}
Recall the sequence of positive numbers $\{c_i\}_{i=1}^{\infty}$ introduced in \eqref{20240920for1}, which satisfies $\lim\limits_{i\to \infty}c_i=0$ and $\sum\limits_{i=1}^{\infty}\frac{a_i^2}{c_i}<\infty$.

We first prove $(1)$. Assume that condition \eqref{20241013cond1} in Definition \ref{20250112def1} holds. Let $(D^{\circ}\cap U,I,f)$ be a bounded local coordinate triple of $D^{\circ}$, and set
\begin{eqnarray}\label{20250206for1}
C_1\triangleq \sup_{\textbf{x}_{I}\in U_{I}} n_{I}(\textbf{x}_{I}+f(\textbf{x}_{I}))\cdot F_{\mathbb{N}\setminus I}(f(\textbf{x}_{I}))<\infty.
\end{eqnarray}
Then, by the local coordinate representation and the definition of $\mu_{D^\circ}$, we have
\begin{eqnarray}
\int_{D^{\circ}\cap U}\left(\sum_{i\in I}\frac{x_i^2}{c_i}\right)\,\mathrm{d}\mu_{D^{\circ}}(\textbf{x})&= &\int_{P_{I}(D^{\circ}\cap U)}  \left(\sum_{i\in I}\frac{x_i^2}{c_i}\right)\cdot n_{I}(\textbf{x}_{I_1}+f(\textbf{x}_{I}))\cdot F_{\mathbb{N}\setminus I}(f(\textbf{x}_{I_1}))\,\mathrm{d}\mu_{I}(\textbf{x}_{I})\nonumber\\
&\leqslant &C_1\cdot \int_{P_{I}(D^{\circ}\cap U)} \left(\sum_{i\in I}\frac{x_i^2}{c_i}\right)\,\mathrm{d}\mu_{I}(\textbf{x}_{I})\nonumber\\
&\leqslant &C_1\cdot \int_{P_{I}\ell^2} \left(\sum_{i\in I}\frac{x_i^2}{c_i}\right)\,\mathrm{d}\mu_{I}(\textbf{x}_{I_1})\label{230413eq2}\\
&=&C_1\cdot  \left(\sum_{i\in I}\frac{a_i^2}{c_i}\right)\nonumber\\
&\leqslant&C_1\cdot   \left(\sum_{i=1}^{\infty}\frac{a_i^2}{c_i}\right)<\infty.\nonumber
\end{eqnarray}
Here the equality in the fourth line follows from the Monotone Convergence Theorem and the definition of the Gaussian product measure $\mu_{I}$.

Since the integrand is nonnegative and measurable, \eqref{230413eq2} implies that
\[
\sum_{i\in I}\frac{x_i^2}{c_i}<\infty
\]
for $\mu_{D^\circ}$-almost every $\mathbf{x}=(x_i)_{i\in\mathbb N}\in D^\circ\cap U$. Hence there exists a Borel set $E_U\subset D^\circ\cap U$ such that
\[
\mu_{D^\circ}\bigl((D^\circ\cap U)\setminus E_U\bigr)=0
,\qquad \text{and
}\qquad
\sum_{i\in I}\frac{x_i^2}{c_i}<\infty ,
\qquad \mathbf{x}=(x_i)_{i\in\mathbb N}\in E_U .
\]
Moreover, for every $\mathbf{x}=(x_i)_{i\in\mathbb N}\in E_U$, we may write
$
\mathbf{x}=\mathbf{x}_{I}+f(\mathbf{x}_{I}).
$
By \eqref{20250206for1} and the definition of
$F_{\mathbb{N}\setminus I}$ in \eqref{20250206def1}, we obtain
\[
\sum_{i\in\mathbb{N}\setminus I}
\left|
\ln\frac{1}{\sqrt{2\pi}a_i}
-\frac{x_i^2}{2a_i^2}
\right|
<\infty .
\]
Combining this with Lemma \ref{20250130lem1}, we further obtain
\[
\sum_{i\in\mathbb{N}\setminus I}\frac{x_i^2}{c_i}<\infty .
\]
Consequently,
\[
\sum_{i=1}^{\infty}\frac{x_i^2}{c_i}<\infty ,
\qquad \mathbf{x}=(x_i)_{i\in\mathbb N}\in E_U .
\]
Thus $E_U\subset K$, where $K$ is the set defined in \eqref{230708e1}.

Now let
$
\{(D^\circ\cap U_j,I_j,f_j)\}_{j=1}^{\infty}
$
be a countable family of bounded local coordinate triples covering $D^\circ$. Applying the preceding argument to each $D^\circ\cap U_j$, we can choose a Borel set
$
E_{U_j}\subset D^\circ\cap U_j
$
such that
\[
\mu_{D^\circ}\bigl((D^\circ\cap U_j)\setminus E_{U_j}\bigr)=0
\qquad\text{and}\qquad
E_{U_j}\subset K .
\]
Set
$
E\triangleq \bigcup\limits_{j=1}^{\infty}E_{U_j}.
$
Then $E$ is a Borel subset of $D^\circ$, and since the sets $D^\circ\cap U_j$ cover $D^\circ$, we have
\[
D^\circ\setminus E
\subset
\bigcup_{j=1}^{\infty}\bigl((D^\circ\cap U_j)\setminus E_{U_j}\bigr).
\]
Therefore,
\[
\mu_{D^\circ}(D^\circ\setminus E)
\leqslant
\sum_{j=1}^{\infty}
\mu_{D^\circ}\bigl((D^\circ\cap U_j)\setminus E_{U_j}\bigr)
=0 .
\]
Moreover, since each $E_{U_j}$ is contained in $K$, we have $E\subset K$.

Hence, for every $\mathbf{x}\in E\subset K$, \eqref{230409eq1} yields
\[
\lim\limits_{m\to\infty}X_m(\mathbf{x})=1,
\qquad\text{and}\qquad
\lim_{m\to\infty}
\sum_{i=1}^{\infty}
a_i^2
\left|
\frac{\partial X_m(\mathbf{x})}{\partial x_i}
\right|^2
=0 .
\]
Since $\mu_{D^\circ}(D^\circ\setminus E)=0$, it follows that
\[
X_m\to 1
,\qquad \text{and}\qquad
\sum_{i=1}^{\infty}
a_i^2
\left|
\frac{\partial X_m}{\partial x_i}
\right|^2
\to 0
\]
almost everywhere on $D^\circ$ with respect to $\mu_{D^\circ}$. This proves (1).

The proof of $(2)$ is analogous. Therefore, the proof of Proposition \ref{20241013thm1} is complete.
\end{proof}

Let \(I\) be a nonempty subset of \(\mathbb N\), and recall the definition of \(C_I\) in \eqref{20250119def1}. For
\[
\mathbf{x}_I=\sum_{i\in I}x_i\mathbf e_i\in P_I\ell^2,
\]
we have \(\mathbf{x}_I\in C_I\) if and only if
\[
\sum_{i\in I}
\left|
\ln\frac{1}{\sqrt{2\pi}a_i}
-\frac{x_i^2}{2a_i^2}
\right|<\infty.
\]

Define
\begin{equation}\label{20241012for4}
G_I(\mathbf{x}_I)
\triangleq
\sum_{i\in I}
\left(
\ln\frac{1}{\sqrt{2\pi}a_i}
-\frac{x_i^2}{2a_i^2}
\right),
\qquad
\mathbf{x}_I=\sum_{i\in I}x_i\mathbf e_i\in C_I.
\end{equation}
By the definition of \(F_I\) in \eqref{20250206def1}, we have
\[
F_I(\mathbf{x}_I)=\exp\bigl(G_I(\mathbf{x}_I)\bigr),
\qquad
\mathbf{x}_I\in C_I.
\]
For each \(\mathbf{x}_I=\sum_{i\in I}x_i\mathbf e_i\in C_I\), define the set of admissible increments by
\[
C_I(\mathbf{x}_I)
\triangleq
\left\{
\Delta\mathbf{x}_I\in P_I\ell^2:
\mathbf{x}_I+\Delta\mathbf{x}_I\in C_I
\right\}.
\]
Equivalently, if
\[
\Delta\mathbf{x}_I
=\sum_{i\in I}\Delta x_i\mathbf e_i \in P_I\ell^2,
\]
then \(\Delta\mathbf{x}_I\in C_I(\mathbf{x}_I)\) if and only if
\[
\sum_{i\in I}
\left|
\ln\frac{1}{\sqrt{2\pi}a_i}
-\frac{(x_i+\Delta x_i)^2}{2a_i^2}
\right|<\infty.
\]
Let
$
M_I\triangleq\sup\limits_{i\in I}a_i^2<\infty.
$
If \(\Delta\mathbf{x}_I=\sum\limits_{i\in I}\Delta x_i\mathbf e_i\in P_I\ell^2\) satisfies
\[
\sum_{i\in I}\frac{|\Delta x_i|^2}{a_i^4}<\infty,
\]
then
\[
\sum_{i\in I}\frac{|x_i\Delta x_i|}{a_i^2}
\leqslant
\left(\sum_{i\in I}x_i^2\right)^{1/2}
\left(\sum_{i\in I}\frac{|\Delta x_i|^2}{a_i^4}\right)^{1/2}
<\infty
\]
and
\begin{equation}
\sum_{i\in I}\frac{|\Delta x_i|^2}{2a_i^2}
\leqslant
\frac{M_I}{2}
\sum_{i\in I}\frac{|\Delta x_i|^2}{a_i^4}
<\infty.\label{20260725for1}
\end{equation}
Consequently,
\[
\begin{aligned}
\sum_{i\in I}
\left|
\ln\frac{1}{\sqrt{2\pi}a_i}
-\frac{(x_i+\Delta x_i)^2}{2a_i^2}
\right| \
\leqslant
\sum_{i\in I}
\left|
\ln\frac{1}{\sqrt{2\pi}a_i}
-\frac{x_i^2}{2a_i^2}
\right|
+
\sum_{i\in I}\frac{|x_i\Delta x_i|}{a_i^2}
+
 \sum_{i\in I}\frac{|\Delta x_i|^2}{2a_i^2}
<\infty.
\end{aligned}
\]
Thus \(\Delta\mathbf{x}_I\in C_I(\mathbf{x}_I)\). Moreover,
\begin{equation}\label{20241007for1}
G_I(\mathbf{x}_I+\Delta\mathbf{x}_I)
-G_I(\mathbf{x}_I)
=-\sum_{i\in I}\frac{x_i\Delta x_i}{a_i^2}
- \sum_{i\in I}\frac{(\Delta x_i)^2}{2a_i^2}.
\end{equation}

Define
\begin{eqnarray}
H_I^1
&\triangleq
\left\{
\mathbf y_I=\sum\limits_{i\in I}y_i\mathbf e_i\in P_I\ell^2:
\sum\limits_{i\in I}\frac{y_i^2}{a_i^2}<\infty
\right\},\qquad \nonumber\
H_I^2
&\triangleq
\left\{
\mathbf y_I=\sum_{i\in I}y_i\mathbf e_i\in P_I\ell^2:
\sum_{i\in I}\frac{y_i^2}{a_i^4}<\infty
\right\},\label{20241020for1}
\end{eqnarray}
and equip \(H_I^2\) with the norm
\[
\|\mathbf y_I\|_{H_I^2}
\triangleq
\left(
\sum_{i\in I}\frac{y_i^2}{a_i^4}
\right)^{1/2},\qquad
\mathbf y_I=\sum_{i\in I}y_i\mathbf e_i\in H_I^2.
\]

For a fixed \(\mathbf{x}_I=\sum\limits_{i\in I}x_i\mathbf e_i\in C_I\), define
\[
T_{\mathbf{x}_I}(\mathbf y_I)
\triangleq
-\sum_{i\in I}\frac{x_i y_i}{a_i^2},
\qquad
\mathbf y_I=\sum_{i\in I}y_i\mathbf e_i\in H_I^2.
\]
By the Cauchy--Schwarz inequality,
$
\left|T_{\mathbf{x}_I}(\mathbf y_I)\right|
\leqslant
\|\mathbf{x}_I\| \cdot
\|\mathbf y_I\|_{H_I^2}.
$
Hence
$
T_{\mathbf{x}_I}\in(H_I^2)^*,
$
which denotes the real dual of $H_I^2$.
Furthermore, \(H_I^2\subset C_I(\mathbf{x}_I)\), and it follows from
\eqref{20260725for1} and \eqref{20241007for1} that, for every
\(\Delta\mathbf{x}_I=\sum\limits_{i\in I}\Delta x_i\mathbf e_i\in H_I^2\),
\[
\begin{aligned}
\left|
G_I(\mathbf{x}_I+\Delta\mathbf{x}_I)
-G_I(\mathbf{x}_I)
-T_{\mathbf{x}_I}(\Delta\mathbf{x}_I)
\right|
=
\frac{1}{2}\sum_{i\in I}\frac{(\Delta x_i)^2}{a_i^2}
\leqslant
\frac{M_I}{2}
\|\Delta\mathbf{x}_I\|_{H_I^2}^2.
\end{aligned}
\]
Therefore,
\[
\frac{
\left|
G_I(\mathbf{x}_I+\Delta\mathbf{x}_I)
-G_I(\mathbf{x}_I)
-T_{\mathbf{x}_I}(\Delta\mathbf{x}_I)
\right|
}{
\|\Delta\mathbf{x}_I\|_{H_I^2}
}
\leqslant
\frac{M_I}{2}
\|\Delta\mathbf{x}_I\|_{H_I^2}
\longrightarrow 0
\]
as
\(\|\Delta\mathbf{x}_I\|_{H_I^2}\to0\).

Thus, in the sense of Definition \ref{20241115def1} and Definition
\ref{20250115def1}, \(G_I\) is Fr\'{e}chet differentiable along the $H_I^2$-direction at $\mathbf{x}_I$, with derivative
\[
D_{H_I^2}G_I(\mathbf{x}_I)
=T_{\mathbf{x}_I}.
\]
Finally, since \(F_I=\exp\circ G_I\), the chain rule yields that \(F_I\)
is Fr\'{e}chet differentiable along the $H_I^2$-direction at $\mathbf{x}_I$, and
\[
\begin{aligned}
D_{H_I^2}F_I(\mathbf{x}_I)(\mathbf y_I)
=F_I(\mathbf{x}_I)\cdot
D_{H_I^2}G_I(\mathbf{x}_I)(\mathbf y_I)
=F_I(\mathbf{x}_I)\cdot T_{\mathbf{x}_I}(\mathbf y_I),
\qquad
\mathbf y_I=\sum_{i\in I}y_i\mathbf e_i\in H_I^2.
\end{aligned}
\]

By the chain rule, we obtain the following result, which will be used in the sequel.

\begin{lemma}\label{20250113lem1}
Let \(O\) be a nonempty open subset of \(\mathbb{R}\), and let
\(g:O\rightarrow \ell^2\) satisfy \(g(O)\subset C_I\).
Assume that \(g\) is \(H_I^2\)-Fr\'{e}chet differentiable.
Then \(F_I\circ g\) is Fr\'{e}chet differentiable from $O$ into $\mathbb{R}$.
\end{lemma}

In order for a Gauss--Green type theorem to hold, we also require the following
additional regularity condition on the local coordinate triples of the surface.

\begin{definition}\label{20250115def2}
Suppose that $S$ is a $\sigma$-finite oriented $C^1$-differentiable surface
of $\ell^2$ with codimension $\Gamma_{\mathbb{N}\setminus I_0}$.
We call $S$ a
\textbf{smooth surface}, if there exists a family $\mathcal{A}$ of
bounded local coordinate triples of $S$ covering $S$ such that: for every
$(S\cap U, I,f)\in\mathcal{A}$,
\[
S\cap U=
\left\{
\mathbf{x}_{I}+f(\mathbf{x}_{I}):
\mathbf{x}_{I}=\sum_{i\in I}x_i\mathbf e_i\in U_{I}
\right\},
\]
where $U_{I}$ is an open subset of $P_{I}\ell^2$,
$
f\in C^1_{P_{I}\ell^2}(U_{I};P_{\mathbb{N}\setminus I}\ell^2),
$ and $
f(U_{I})\subset C_{\mathbb{N}\setminus I};
$
For every $j\in I$ and
$
\mathbf{x}_{I}^0\in U_{I},
$
define the one-dimensional slice function
\[
f_j(t)\triangleq
f(\mathbf{x}_{I}^0+t\mathbf e_j),
\qquad
t\in U_{\mathbf{x}_{I}^0,j},
\]
where
$
U_{\mathbf{x}_{I}^0,j}
\triangleq
\{t\in\mathbb R:
\mathbf{x}_{I}^0+t\mathbf e_j\in U_{I}\}
$ (Obviously, $0\in U_{\mathbf{x}_{I}^0,j}$);
For every $j\in I$, the function $f_j$ is
$H^2_{\mathbb{N}\setminus I}$-Fr\'{e}chet differentiable at $0$, and
\begin{equation}\label{20260729for1}
\langle f(\cdot),\mathbf{e}_k\rangle \in C_\mathcal{F}^2(U_{I}),\qquad k\in\mathbb{N}\setminus I.
\end{equation}

\end{definition}

\begin{remark}
If $J\subset \mathbb{N}$ is finite, then $P_J\ell^2$, $H^1_J$, and $H^2_J$ are Hilbert spaces with equivalent norms. Hence, if $\mathbb{N}\setminus I_0$ is finite, the weighted differentiability conditions appearing in Definition \ref{20250115def2} are automatically equivalent to the usual $C^1$-differentiability conditions. Consequently, the notion of a smooth surface is unnecessary in this case. 
\end{remark}

\begin{lemma}\label{20260728lem1}
We adopt the notation and assumptions of Definition~\ref{20250115def2}. Let $F$ be a real-valued function defined on $U$ and let $j\in I$. If $F$ is Fr\'{e}chet differentiable along the $H_{\mathbb{N}}^2$-direction at every point of $U$, then
\begin{eqnarray}\label{20250115for1}
D_{x_j}(F\circ P_{I}^{-1})= \sum_{k\in\mathbb{N}}(D_{x_k}F)\circ P_{I}^{-1} \cdot D_{x_j}(P_{\{k\}}P_{I}^{-1} ),\qquad \text{on}\quad U_{I},
\end{eqnarray}
where
$
 D_{x_j}(P_{\{k\}}P_{I}^{-1} )\triangleq  D_{x_j}\left(\left\langle P_{\{k\}}P_{I}^{-1}\cdot,\mathbf{e}_{k}\right\rangle \right).
$ Moreover, the series in \eqref{20250115for1} absolutely converge point-wise on $U_{I}$.
\end{lemma}
\begin{proof}
Suppose that $(S\cap U, I,f)$ is a local coordinate triple of $S$ given in Definition \ref{20250115def2}. Fix
$
\mathbf{x}_{I}^0\in U_{I},
$ and define
$$
\varphi_j(t)\triangleq P_{I}^{-1}(\mathbf{x}_{I}^0+t\mathbf{e}_j)=\mathbf{x}_{I}^0+t\mathbf{e}_j+f_j(t),\qquad t\in U_{\mathbf{x}_{I}^0,j}.
$$
By Definition \ref{20250115def2}, there exists $\mathbf{y}_{\mathbb{N}\setminus I }=\sum\limits_{k\in \mathbb{N}\setminus I}y_k\mathbf{e}_k\in H^2_{\mathbb{N}\setminus I}$ such that $f_j(t)-f_j(0)\in H^2_{\mathbb{N}\setminus I}$ for sufficiently small $t$, and
\begin{eqnarray}\label{20260728for2}
\lim_{t\to 0}\frac{\|f_j(t)-f_j(0)-t\mathbf{y}_{\mathbb{N}\setminus I }\|_{H^2_{\mathbb{N}\setminus I}}}{|t|}= 0 .
\end{eqnarray}
Let $\mathbf{y}\triangleq \mathbf{e}_j+\mathbf{y}_{\mathbb{N}\setminus I }\in H^2_{\mathbb{N} }$. Thus $D\varphi_j(0)=\mathbf{y}\in H^2_{\mathbb{N} }$,
$$
\varphi_j(t)-\varphi_j(0) =t\mathbf{e}_j+f_j(t)-f_j(0)\in t\mathbf{e}_j+H^2_{\mathbb{N}\setminus I}\subset  H^2_{\mathbb{N} }\,\,\text{for sufficiently small }t,
$$
and
\begin{eqnarray}\label{20260728for3}
\frac{\|\varphi_j(t)-\varphi_j(0)- t\mathbf{y} \|_{H^2_{\mathbb{N} }}}{|t|}
=\frac{\| f_j(t)-f_j(0)- t\mathbf{y}_{\mathbb{N}\setminus I }\|_{H^2_{\mathbb{N} }}}{|t|}
=\frac{\| f_j(t)-f_j(0)- t\mathbf{y}_{\mathbb{N}\setminus I }\|_{H^2_{\mathbb{N}\setminus I }}}{|t|}\to 0,
\end{eqnarray}
as $t\to 0.$ Hence $\varphi_j$ is
$H^2_{\mathbb{N}}$-Fr\'{e}chet differentiable at $0$.

Since $F$ is Fr\'{e}chet differentiable along the $H_{\mathbb{N}}^2$-direction at $P_{I}^{-1} \mathbf{x}_{I}^0= \mathbf{x}_{I}^0+ f_j(0)$, there exists
$\mathbf{z} =\sum\limits_{k\in \mathbb{N} }z_k\mathbf{e}_k\in H^2_{\mathbb{N} }$ such that
\begin{eqnarray}\label{20260728for1}
 \frac{|F(\mathbf{x}_{I}^0+\Delta\mathbf{x}_{I}+ f_j(0))-F(\mathbf{x}_{I}^0+  f_j(0))-\langle \Delta\mathbf{x}_{I},\mathbf{z}\rangle_{H^2_{\mathbb{N} }}|}{\|\Delta\mathbf{x}_{I}\|_{H^2_{\mathbb{N}}}}\to 0,
\end{eqnarray}
as $\Delta\mathbf{x}_{I}\in H^2_{\mathbb{N}}$ and $\|\Delta\mathbf{x}_{I}\|_{H^2_{\mathbb{N}}}\to 0$.

Hence,
\begin{eqnarray*}
(F\circ P_{I}^{-1})(\mathbf{x}_{I}^0+t\mathbf{e}_j)- (F\circ P_{I}^{-1})(\mathbf{x}_{I}^0)
 &=&F(\varphi_j(t))-F(\varphi_j(0))\\
 &=&\langle \varphi_j(t) - \varphi_j(0),\mathbf{z}\rangle_{H^2_{\mathbb{N} }}+o(\|\varphi_j(t) - \varphi_j(0)\|_{H_{\mathbb{N}}^2})\\
  &=&t \langle  \mathbf{y},\mathbf{z}\rangle_{H^2_{\mathbb{N} }}+o( |t|)+o(\|\varphi_j(t) - \varphi_j(0)\|_{H_{\mathbb{N}}^2}).
\end{eqnarray*}
By \eqref{20260728for3},
$$
\|\varphi_j(t) - \varphi_j(0)\|_{H_{\mathbb{N}}^2}=O( |t|)
$$
and therefore
$$o(\|\varphi_j(t) - \varphi_j(0)\|_{H_{\mathbb{N}}^2})=o( |t|).$$
Consequently,
\begin{eqnarray*}
(F\circ P_{I}^{-1})(\mathbf{x}_{I}^0+t\mathbf{e}_j)- (F\circ P_{I}^{-1})(\mathbf{x}_{I}^0)
=t \langle  \mathbf{y},\mathbf{z}\rangle_{H^2_{\mathbb{N} }}+o( |t|),
\end{eqnarray*}
which implies that
\begin{eqnarray*}
D_{x_j}(F\circ P_{I}^{-1})(\mathbf{x}_{I}^0)= \langle  \mathbf{y},\mathbf{z}\rangle_{H^2_{\mathbb{N} }}
=\frac{z_j}{a_j^4}+\sum\limits_{k\in \mathbb{N}\setminus I}\frac{y_kz_k}{a_k^4}.
\end{eqnarray*}
By \eqref{20260728for2} and \eqref{20260728for1}, we obtain
\begin{eqnarray*}
 (D_{x_k}F)\circ P_{I}^{-1}(\mathbf{x}_{I}^0)&=& \frac{z_k}{a_k^4},\qquad k\in\mathbb{N},\\
 D_{x_j}(P_{\{k\}}P_{I}^{-1} ) &=&\delta_{k,j},\qquad  k\in I,\quad\text{and}\quad  D_{x_j}(P_{\{k\}}P_{I}^{-1} ) =y_k,\qquad  k\in \mathbb{N}\setminus I.
\end{eqnarray*}
This proves \eqref{20250115for1}. Moreover, by the Cauchy--Schwarz inequality,
\begin{eqnarray*}
\frac{|z_j|}{a_j^4}+\sum\limits_{k\in \mathbb{N}\setminus I}\frac{|y_kz_k|}{a_k^4}
&\leqslant& \|\mathbf{y}\|_{H^2_{\mathbb{N} }}\cdot\|\mathbf{z}\|_{H^2_{\mathbb{N} }}<\infty,
\end{eqnarray*}
which proves the last conclusion of Lemma \ref{20260728lem1}. The proof of Lemma \ref{20260728lem1} is complete.
\end{proof}
We adopt the notation and assumptions of Definition~\ref{20250115def2}. Applying \eqref{20250115for1} to the function $F=F_{\mathbb{N}\setminus I}\circ P_{\mathbb{N}\setminus I}$, and by the argument before Lemma \ref{20250113lem1}, we obtain
\begin{eqnarray}\label{20250115for2}
 D_{x_{j}}\left( F_{\mathbb{N}\setminus I}\circ(P_{\mathbb{N}\setminus I}P_{I}^{-1} ) \right) = -F_{\mathbb{N}\setminus I}\circ(P_{\mathbb{N}\setminus I }P_{I}^{-1} ) \sum\limits_{k\in \mathbb{N}\setminus I} \frac{\langle P_{\{k\}}P_{I }^{-1}\cdot,\mathbf{e}_k\rangle}{a_k^2}  D_{x_j}(P_{\{k\}}P_{I}^{-1})\quad\text{on}\quad U_{I}.
\end{eqnarray}

\begin{lemma}\label{20250113lem2}
Let $\{X_n\}_{n=1}^{\infty}$ be the sequence of functions on $\ell^2$ defined in \eqref{230409eq1}. Then $\{X_n\}_{n=1}^{\infty}\subset C_{H}^{\infty}(\ell^2)$. Moreover, for every $n\in\mathbb{N}$, $X_n$ is Fr\'{e}chet differentiable along the $H_{\mathbb{N}}^2$-direction at every point of $\ell^2$.
\end{lemma}
\begin{proof}
First, observe that $H_{\mathbb{N}}^1=H$, where $H$ is defined at \eqref{20260725for2} and therefore, by Proposition \ref{partial derivative of Ptf}, we obtain $\{X_n\}_{n=1}^{\infty}\subset C_{H_{\mathbb{N}}^1}^{\infty}(\ell^2)$.

It remains to prove the differentiability assertion. Fix \(n\in\mathbb{N}\) and \(\mathbf{x}\in\ell^2\). By Proposition \ref{partial derivative of Ptf}, \(X_n\) is Fr\'{e}chet differentiable along the \(H_{\mathbb{N}}^1\)-direction at \(\mathbf{x}\). Hence, there exists a bounded linear functional
\[
D_{H_{\mathbb{N}}^1}X_n(\mathbf{x})
\in
\bigl(H_{\mathbb{N}}^1\bigr)^*
\]
such that
\[
X_n(\mathbf{x}+\mathbf{h})-X_n(\mathbf{x})
-D_{H_{\mathbb{N}}^1}X_n(\mathbf{x})(\mathbf{h})
=o\bigl(\|\mathbf{h}\|_{H_{\mathbb{N}}^1}\bigr)
\]
as \(\|\mathbf{h}\|_{H_{\mathbb{N}}^1}\to0\).

Since
\[
H_{\mathbb{N}}^2\subset H_{\mathbb{N}}^1
\quad\text{and}\quad
\|\mathbf{h}\|_{H_{\mathbb{N}}^1}
\leqslant
\left(\sup_{i\in\mathbb{N}}a_i\right)\|\mathbf{h}\|_{H_{\mathbb{N}}^2}
\qquad
\text{for all }\mathbf{h}\in H_{\mathbb{N}}^2,
\]
the restriction
\[
D_{H_{\mathbb{N}}^2}X_n(\mathbf{x})
\triangleq
D_{H_{\mathbb{N}}^1}X_n(\mathbf{x})
\big|_{H_{\mathbb{N}}^2}
\]
is a bounded linear functional on \(H_{\mathbb{N}}^2\). Moreover, for
\(\mathbf{h}\in H_{\mathbb{N}}^2\setminus\{\mathbf{0}\}\),
\[
\frac{
\left|
X_n(\mathbf{x}+\mathbf{h})-X_n(\mathbf{x})
-D_{H_{\mathbb{N}}^2}X_n(\mathbf{x})(\mathbf{h})
\right|
}{
\|\mathbf{h}\|_{H_{\mathbb{N}}^2}
}
\leqslant
\left(\sup_{i\in\mathbb{N}}a_i\right)\frac{
\left|
X_n(\mathbf{x}+\mathbf{h})-X_n(\mathbf{x})
-D_{H_{\mathbb{N}}^1}X_n(\mathbf{x})(\mathbf{h})
\right|
}{
\|\mathbf{h}\|_{H_{\mathbb{N}}^1}
}.
\]
The right-hand side tends to zero as
\(\|\mathbf{h}\|_{H_{\mathbb{N}}^2}\to0\). Thus, \(X_n\) is Fr\'{e}chet differentiable along the \(H_{\mathbb{N}}^2\)-direction at \(\mathbf{x}\). Since \(n\) and \(\mathbf{x}\) were arbitrary, the proof of Lemma \ref{20250113lem2} is complete.
\end{proof}

To establish a global version of the Gauss--Green-type theorem, we first prove its local counterpart. We recall the following notation. Let $U$ be a nonempty open subset of $\ell^2$. For $f\in C_F^1(U)$, define
\[
\delta_i f(\mathbf{x})
\triangleq
D_{x_i}f(\mathbf{x})
-\frac{x_i}{a_i^2}f(\mathbf{x}),
\qquad
\mathbf{x}=\sum_{k=1}^{\infty}x_k\mathbf e_k\in U,\quad i\in\mathbb N,
\]
which is the Gaussian derivative operator in the $i$-th coordinate direction. Moreover, for
\[\mathbf{I}=\left( i_{1},i_{2},\cdots \right)\in S_{\mathbb N}^{F},\]
and $i\in\mathbb{N}$, define
$$
(i,\mathbf{I})  \triangleq
\begin{cases}
\left(i, i_{1},i_{2},\cdots \right), &\text{ if }i\in\mathbb{N}\setminus |\mathbf{I}|, \\
 \emptyset,&\text{ if }i\in  |\mathbf{I}|.
\end{cases}
$$
\begin{proposition}[Local Gauss--Green-type theorem]
\label{20241102prop1}
Assume the hypotheses and notation of Definition~\ref{20241009for1}. Assume further that \(S\) is a smooth surface. Let \(\mathbf{x}_0\in D\).

\medskip
\noindent
\textbf{(I) Boundary-point case.}
Suppose that \(\mathbf{x}_0\in\partial D\). Then there exist
$$
\mathbf{K}=\left(k_0, k_{1},k_{2},\cdots  \right)\in S_{\mathbb N}^{F},
\qquad
 i_0 \in \mathbb{N}_0,\quad k_{i_0}\in |\mathbf{K}|,
$$
an open neighborhood \(U\) of \(\mathbf{x}_0\) in \(\ell^2\), an open neighborhood
$
U_{|\mathbf{K}|}\subset P_{|\mathbf{K}|}\ell^2
$
of \(P_{|\mathbf{K}|}\mathbf{x}_0\), an open neighborhood
$
V_{|\mathbf{K}|\setminus\{k_{i_0}\}}
\subset P_{|\mathbf{K}|\setminus\{k_{i_0}\}}\ell^2
$
of \(P_{|\mathbf{K}|\setminus\{k_{i_0}\}}\mathbf{x}_0\), and mappings
$
\Phi\in
C^1_{P_{|\mathbf{K}|\setminus\{k_{i_0}\}}\ell^2}
\left(
V_{|\mathbf{K}|\setminus\{k_{i_0}\}};
P_{\mathbb N\setminus|\mathbf{K}|}\ell^2
\right)
$
and
$
h\in
C^1_{P_{|\mathbf{K}|\setminus\{k_{i_0}\}}\ell^2}
\left(
V_{|\mathbf{K}|\setminus\{k_{i_0}\}};
\mathbb R
\right)
$
such that
$
s(\mathbf{K})=1,
|\mathbf{K}|\in\Gamma_{I_0},
$
and
$
P_{|\mathbf{K}|\setminus\{k_{i_0}\}}
 U_{|\mathbf{K}|}
=
V_{|\mathbf{K}|\setminus\{k_{i_0}\}}.
$
Moreover,
\[
\mathbf{x}_0
=P_{|\mathbf{K}|\setminus\{k_{i_0}\}}\mathbf{x}_0
+
h\left(
P_{|\mathbf{K}|\setminus\{k_{i_0}\}}\mathbf{x}_0
\right)\mathbf e_{k_{i_0}}
+
\Phi\left(
P_{|\mathbf{K}|\setminus\{k_{i_0}\}}\mathbf{x}_0
\right),
\]
and
\begin{eqnarray}
(\partial D)\cap U
&=&\Bigl\{
\mathbf y
+h(\mathbf y)\mathbf e_{k_{i_0}}
+\Phi(\mathbf y):
\mathbf y\in
V_{|\mathbf{K}|\setminus\{k_{i_0}\}}
\Bigr\}\\
&=&\Bigl\{
\mathbf z
+\Phi\left(
P_{|\mathbf{K}|\setminus\{k_{i_0}\}}\mathbf z
\right):
\mathbf z\in U_{|\mathbf{K}|},
\ \rho(\mathbf z)=0
\Bigr\},
\label{20241015for1ooo}\\
D^\circ\cap U
&=&\Bigl\{
\mathbf z
+\Phi\left(
P_{|\mathbf{K}|\setminus\{k_{i_0}\}}\mathbf z
\right):
\mathbf z\in U_{|\mathbf{K}|},
\ \rho(\mathbf z)<0
\Bigr\},
\end{eqnarray}
where
$
\rho(\mathbf z)
\triangleq
(-1)^{i_0-1}
\left(
\langle\mathbf z,\mathbf e_{k_{i_0}}\rangle
-h\left(
P_{|\mathbf{K}|\setminus\{k_{i_0}\}}\mathbf z
\right)
\right),
\,
\mathbf z\in U_{|\mathbf{K}|}.
$

Let \(\mathbf{I}\in S_{\mathbb N}^{F}\) satisfy
$
|\mathbf{I}|
\in
\Gamma_{|\mathbf{K}|\setminus\{k_{i_0}\}},
$
and let \(g\in B_{\ell^2}(U)\). Assume that:

\begin{enumerate}
\item there exists a nonempty closed ball \(B_0\subset U\) such that
$
g\equiv 0
\,\text{on }U\setminus B_0;
$

\item \(g\) is Fr\'{e}chet differentiable along the
\(H_{\mathbb N}^{2}\)-direction at every point of $U$, in the sense of Definition~\ref{20241115def1};

\item \(g\in C_{\mathcal F}^1(U)\), and the closure of \(g^{-1}(\mathbb{R}\setminus\{0\})\) in \(\ell^2\) is compact;

\item
\begin{eqnarray}
&&\sum_{i=1}^{\infty}
\left(
|D_{x_i}g|
+
\left|\frac{x_i}{a_i^2}g\right|
\right)
\left|
  n_{(i,\mathbf{I})}^{D^\circ}
\right|
\in
L^1(D^\circ\cap U,\mu_{D^\circ}),
\label{20250120for1}\\
&&g|_{(\partial D)\cap U}
\in
L^1((\partial D)\cap U,\mu_{\partial D}),
\qquad
g|_{D^\circ\cap U}
\in
L^1(D^\circ\cap U,\mu_{D^\circ}),
\label{20250120for2}\\
&&\left.
\left(
|D_{x_i}g|
+
\left|\frac{x_i}{a_i^2}g\right|
\right)
\right|_{D^\circ\cap U}
\in
L^1(D^\circ\cap U,\mu_{D^\circ}),
\qquad i\in\mathbb N.
\label{20250120for3}
\end{eqnarray}
\end{enumerate}

Then
\[
-\sum_{i=1}^{\infty}
\int_{D^\circ\cap U}
\delta_i g(\mathbf x)\,
  n_{(i,\mathbf{I})}^{D^\circ}(\mathbf x)
\,\mathrm d\mu_{D^\circ}(\mathbf x)
=\int_{(\partial D)\cap U}
g(\mathbf x)\,
  n_{\mathbf{I}}^{\partial D}(\mathbf x)
\,\mathrm d\mu_{\partial D}(\mathbf x).
\]

\medskip
\noindent
\textbf{(II) Interior-point case.}
Suppose that \(\mathbf{x}_0\in D^\circ\). Then there exists a local coordinate triple
$
(S\cap U, I,f)
$
of \(S\) such that
$
\mathbf{x}_0\in S\cap U\subset D^\circ.
$
Assume that \(\mathbf{I}\in S_{\mathbb N}^{F}\) satisfies $|(i,\mathbf{I})|\in \Gamma_{I_0}$ for every $i\in \mathbb{N}\setminus |\mathbf{I}|$.
Let \(u \in B_{\ell^2}(U)\). Assume that:

\begin{enumerate}
\item there exists a nonempty closed ball \(B_1\subset U\) such that
$
u\equiv0
\,\text{on }U\setminus B_1;
$

\item \(u\) is Fr\'{e}chet differentiable along the
\(H_{\mathbb N}^{2}\)-direction at every point of $U$;

\item \(u\in C_{\mathcal F}^1(U)\), and the closure of \(u^{-1}(\mathbb{R}\setminus\{0\})\) in \(\ell^2\) is compact;

\item
$
\sum\limits_{i=1}^{\infty}
\left(
|D_{x_i}u|
+
\left|\frac{x_i}{a_i^2}u\right|
\right)
\left|
  n_{(i,\mathbf{I})}^{D^\circ}
\right|
\in
L^1(D^\circ\cap U,\mu_{D^\circ}),\qquad
$ and $\qquad
u|_{D^\circ\cap U}
\in
L^1(D^\circ\cap U,\mu_{D^\circ});
$

\item
 $
\left.
\left(
|D_{x_i}u|
+
\left|\frac{x_i}{a_i^2}u\right|
\right)
\right|_{D^\circ\cap U}
\in
L^1(D^\circ\cap U,\mu_{D^\circ})
$
 for every \(i\in\mathbb N\).

\end{enumerate}

Then
\[
-\sum_{i=1}^{\infty}
\int_{U\cap D^\circ}
\delta_i u(\mathbf x)\,
  n_{(i,\mathbf{I})}^{D^\circ}(\mathbf x)
\,\mathrm d\mu_{D^\circ}(\mathbf x)
=0.
\]
\end{proposition}

\begin{proof}
By the assumptions, there exists $n\in\mathbb{N}$ such that
$$
\mathbf{I} =\left( i_{1},i_{2},\ldots ,i_{n},\ldots \right),\qquad \mathbf{K}=\left(k_0, k_{1},k_{2},\ldots ,k_{n},\ldots \right),
$$
$i_{s}=k_{s},$ for every $ s\geqslant n+1;$ $i_{n+r}<i_{n+r+1}$ and $k_{n+r}<k_{n+r+1}$, for every $r\in\mathbb{N}$; $i_0<n$, and
$$i_{n+1}>\max\{ i_{1},i_{2},\ldots ,i_{n}, k_{0},k_{1},\ldots ,k_{n}\}.$$
Since $s(\mathbf{K})=1$, we obtain
\begin{eqnarray}\label{20250121for7}
\begin{aligned}
&\sum\limits_{i=1}^{\infty} \int_{U\cap D^{\circ}}  \delta_{i} g\left( \textbf{x} \right) \cdot n_{\left( i,\mathbf{I} \right)}^{D^{\circ}}\left( \textbf{x} \right) \mathrm{d} \mu_{D^{\circ}} \left( \textbf{x} \right)\\
&=\int_{U\cap D^{\circ}} \sum\limits_{i=1}^{\infty} \left( D_{x_i}g\left( \textbf{x} \right) -\frac{x_i}{a_i^2}\cdot g\left( \textbf{x} \right) \right) \cdot n_{\left( i,\mathbf{I} \right)}^{D^{\circ}}\left( \textbf{x} \right) \mathrm{d} \mu_{D^{\circ}} \left( \textbf{x} \right)\\
& =\int_{\{\rho(\textbf{x}_{|\mathbf{K}|} )<0\}} \sum\limits_{i=1}^{\infty} D_{x_i}g\left( P_{|\mathbf{K}|}^{-1}\textbf{x}_{|\mathbf{K}|} \right)\cdot  s\left( \left( i,\mathbf{I}\right) \right)\cdot s(\mathbf{K})\\
 &\qquad\qquad\qquad\qquad\times\det \left( D \left( P_{|\left( i,\mathbf{I}\right)|} P_{|\mathbf{K}|}^{-1} \right) \right) \left( \textbf{x}_{|\mathbf{K}|}  \right)\cdot F_{\mathbb{N} \setminus |\mathbf{K}|}\circ \left( P_{\mathbb{N} \setminus |\mathbf{K}|} P_{|\mathbf{K}|}^{-1}\textbf{x}_{|\mathbf{K}|}  \right) \mathrm{d} \mu_{|\mathbf{K}|} \left( \textbf{x}_{|\mathbf{K}|}  \right)\\
&-\int_{\{\rho(\textbf{x}_{|\mathbf{K}|} )<0\}} \sum\limits_{i=1}^{\infty} \frac{x_i}{a_i^2}\cdot g\left( P_{|\mathbf{K}|}^{-1}\textbf{x}_{|\mathbf{K}|}  \right)\cdot s\left( \left( i,\mathbf{I}\right) \right)\cdot s(\mathbf{K})\\
 &\qquad\qquad\qquad\qquad\times \det \left( D\left( P_{|\left( i,\mathbf{I}  \right)|}P_{|\mathbf{K}|}^{-1} \right) \right) \left( \textbf{x}_{|\mathbf{K}|}  \right)\cdot F_{\mathbb{N} \setminus |\mathbf{K}|}\circ \left( P_{\mathbb{N} \setminus |\mathbf{K}|}P_{|\mathbf{K}|}^{-1}\textbf{x}_{|\mathbf{K}|}  \right) \mathrm{d} \mu_{|\mathbf{K}|} \left( \textbf{x}_{|\mathbf{K}|}  \right)\\
&= \int_{\{\rho <0\}} \sum\limits_{i=1}^{\infty} (D_{x_i}g)\circ P_{|\mathbf{K}|}^{-1} \cdot
\begin{vmatrix}D_{x_{k_{0}}}\left( P_{\{i\}}P_{|\mathbf{K}|}^{-1} \right) &D_{x_{k_{1}}}\left( P_{\{i\}}P_{|\mathbf{K}|}^{-1} \right)&\cdots&D_{x_{k_{n}}}\left( P_{\{i\}}P_{|\mathbf{K}|}^{-1} \right)\\ D_{x_{k_{0}}}\left( P_{\{i_{1}\}}P_{|\mathbf{K}|}^{-1} \right)&D_{x_{k_{1}}}\left( P_{\{i_{1}\}}P_{|\mathbf{K}|}^{-1} \right)&\cdots&D_{x_{k_{n}}}\left( P_{\{i_{1}\}}P_{|\mathbf{K}|}^{-1} \right)\\ \vdots&\vdots&\ddots&\vdots\\ D_{x_{k_{0}}}\left( P_{\{i_{n}\}}P_{|\mathbf{K}|}^{-1} \right)&D_{x_{k_{1}}}\left( P_{\{i_{n }\}}P_{|\mathbf{K}|}^{-1} \right)&\cdots&D_{x_{k_{n}}}\left( P_{\{i_{n }\}}P_{|\mathbf{K}|}^{-1} \right)\end{vmatrix}
\\
&\qquad\qquad\qquad\qquad\qquad\qquad \qquad\qquad\qquad  \qquad\qquad\qquad\qquad\qquad \times F_{\mathbb{N} \setminus |\mathbf{K}|}\circ\left( P_{\mathbb{N} \setminus |\mathbf{K}|} P_{|\mathbf{K}|}^{-1}  \right) \mathrm{d} \mu_{|\mathbf{K}|}\\
&- \int_{\{\rho <0\}} \sum\limits_{i=1}^{\infty} \frac{x_i}{a_i^2}\cdot  \left(g\circ P_{|\mathbf{K}|}^{-1}  \right) \cdot\begin{vmatrix}D_{x_{k_{0}}}\left( P_{\{i\}}P_{|\mathbf{K}|}^{-1} \right) &D_{x_{k_{1}}}\left( P_{\{i\}}P_{|\mathbf{K}|}^{-1} \right)&\cdots&D_{x_{k_{n}}}\left( P_{\{i\}}P_{|\mathbf{K}|}^{-1} \right)\\ D_{x_{k_{0}}}\left( P_{\{i_{1}\}}P_{|\mathbf{K}|}^{-1} \right)&D_{x_{k_{1}}}\left( P_{\{i_{1}\}}P_{|\mathbf{K}|}^{-1} \right)&\cdots&D_{x_{k_{n}}}\left( P_{\{i_{1}\}}P_{|\mathbf{K}|}^{-1} \right)\\ \vdots&\vdots&\ddots&\vdots\\ D_{x_{k_{0}}}\left( P_{\{i_{n}\}}P_{|\mathbf{K}|}^{-1} \right)&D_{x_{k_{1}}}\left( P_{\{i_{n }\}}P_{|\mathbf{K}|}^{-1} \right)&\cdots&D_{x_{k_{n}}}\left( P_{\{i_{n }\}}P_{|\mathbf{K}|}^{-1} \right)\end{vmatrix}
\\
&\qquad\qquad\qquad\qquad\qquad\qquad \qquad\qquad\qquad  \qquad\qquad\qquad\qquad\qquad \times F_{\mathbb{N} \setminus |\mathbf{K}|}\circ\left( P_{\mathbb{N} \setminus |\mathbf{K}|} P_{|\mathbf{K}|}^{-1}   \right) \mathrm{d} \mu_{|\mathbf{K}|}.
\end{aligned}
\end{eqnarray}
By
\eqref{20250120for1}, the series of absolute values is integrable. Hence Tonelli's theorem justifies the interchange of the infinite sum and the integral in
\eqref{20250121for7}.

Since \(g\) is Fr\'{e}chet differentiable along the
\(H_{\mathbb N}^{2}\)-direction at every point of $U$, by Lemma \ref{20260728lem1}, we have
$$ D_{x_{k_{j}}}(g\circ P_{|\mathbf{K}|}^{-1} )= \sum\limits_{i=1}^{\infty} (D_{x_i}g)\circ P_{|\mathbf{K}|}^{-1} \cdot D_{x_{k_{j}}}\left( P_{\{i\}}P_{|\mathbf{K}|}^{-1} \right),\qquad j\in\{0,1,\ldots,n\}.$$
Hence
$$
\begin{aligned}
&\sum\limits_{i=1}^{\infty} (D_{x_i}g)\circ P_{|\mathbf{K}|}^{-1} \cdot
\begin{vmatrix}D_{x_{k_{0}}}\left( P_{\{i\}}P_{|\mathbf{K}|}^{-1} \right) &D_{x_{k_{1}}}\left( P_{\{i\}}P_{|\mathbf{K}|}^{-1} \right)&\cdots&D_{x_{k_{n}}}\left( P_{\{i\}}P_{|\mathbf{K}|}^{-1} \right)\\ D_{x_{k_{0}}}\left( P_{\{i_{1}\}}P_{|\mathbf{K}|}^{-1} \right)&D_{x_{k_{1}}}\left( P_{\{i_{1}\}}P_{|\mathbf{K}|}^{-1} \right)&\cdots&D_{x_{k_{n}}}\left( P_{\{i_{1}\}}P_{|\mathbf{K}|}^{-1} \right)\\ \vdots&\vdots&\ddots&\vdots\\ D_{x_{k_{0}}}\left( P_{\{i_{n}\}}P_{|\mathbf{K}|}^{-1} \right)&D_{x_{k_{1}}}\left( P_{\{i_{n }\}}P_{|\mathbf{K}|}^{-1} \right)&\cdots&D_{x_{k_{n}}}\left( P_{\{i_{n }\}}P_{|\mathbf{K}|}^{-1} \right)\end{vmatrix}
\\
&=
\begin{vmatrix}
 D_{x_{k_{0}}}(g\circ P_{|\mathbf{K}|}^{-1}) &  D_{x_{k_{1}}}(g\circ  P_{|\mathbf{K}|}^{-1}) &\cdots& D_{x_{k_{n}}}(g\circ P_{|\mathbf{K}|}^{-1}) \\
  D_{x_{k_{0}}}\left( P_{\{i_{1}\}}P_{|\mathbf{K}|}^{-1} \right)&D_{x_{k_{1}}}\left( P_{\{i_{1}\}}P_{|\mathbf{K}|}^{-1} \right)&\cdots&D_{x_{k_{n}}}\left( P_{\{i_{1}\}}P_{|\mathbf{K}|}^{-1} \right)\\ \vdots&\vdots&\ddots&\vdots\\ D_{x_{k_{0}}}\left( P_{\{i_{n}\}}P_{|\mathbf{K}|}^{-1} \right)&D_{x_{k_{1}}}\left( P_{\{i_{n }\}}P_{|\mathbf{K}|}^{-1} \right)&\cdots&D_{x_{k_{n}}}\left( P_{\{i_{n }\}}P_{|\mathbf{K}|}^{-1} \right)
\end{vmatrix},
\end{aligned}
$$
and
$$
\begin{aligned}
&-\sum\limits_{i=1}^{\infty} \frac{x_i}{a_i^2}( g\circ P_{|\mathbf{K}|}^{-1} )   \begin{vmatrix}D_{x_{k_{0}}}\left( P_{\{i\}}P_{|\mathbf{K}|}^{-1} \right) &D_{x_{k_{1}}}\left( P_{\{i\}}P_{|\mathbf{K}|}^{-1} \right)&\cdots&D_{x_{k_{n}}}\left( P_{\{i\}}P_{|\mathbf{K}|}^{-1} \right)\\ D_{x_{k_{0}}}\left( P_{\{i_{1}\}}P_{|\mathbf{K}|}^{-1} \right)&D_{x_{k_{1}}}\left( P_{\{i_{1}\}}P_{|\mathbf{K}|}^{-1} \right)&\cdots&D_{x_{k_{n}}}\left( P_{\{i_{1}\}}P_{|\mathbf{K}|}^{-1} \right)\\ \vdots&\vdots&\ddots&\vdots\\ D_{x_{k_{0}}}\left( P_{\{i_{n}\}}P_{|\mathbf{K}|}^{-1} \right)&D_{x_{k_{1}}}\left( P_{\{i_{n }\}}P_{|\mathbf{K}|}^{-1} \right)&\cdots&D_{x_{k_{n}}}\left( P_{\{i_{n }\}}P_{|\mathbf{K}|}^{-1} \right)\end{vmatrix}\\
&=-\sum\limits_{j=0}^{n} \frac{x_{k_j}}{a_{k_j}^2} \left(g\circ P_{|\mathbf{K}|}^{-1} \right) \begin{vmatrix}D_{x_{k_{0}}}\left( P_{\{k_j\}}P_{|\mathbf{K}|}^{-1} \right) &D_{x_{k_{1}}}\left( P_{\{k_j\}}P_{|\mathbf{K}|}^{-1} \right)&\cdots&D_{x_{k_{n}}}\left( P_{\{k_j\}}P_{|\mathbf{K}|}^{-1} \right)\\ D_{x_{k_{0}}}\left( P_{\{i_{1}\}}P_{|\mathbf{K}|}^{-1} \right)&D_{x_{k_{1}}}\left( P_{\{i_{1}\}}P_{|\mathbf{K}|}^{-1} \right)&\cdots&D_{x_{k_{n}}}\left( P_{\{i_{1}\}}P_{|\mathbf{K}|}^{-1} \right)\\ \vdots&\vdots&\ddots&\vdots\\ D_{x_{k_{0}}}\left( P_{\{i_{n}\}}P_{|\mathbf{K}|}^{-1} \right)&D_{x_{k_{1}}}\left( P_{\{i_{n }\}}P_{|\mathbf{K}|}^{-1} \right)&\cdots&D_{x_{k_{n}}}\left( P_{\{i_{n }\}}P_{|\mathbf{K}|}^{-1} \right)\end{vmatrix}\\
&\quad -\sum\limits_{i\in \mathbb{N}\setminus |\mathbf{K}|} \frac{\langle P_{\{i\}}P_{|\mathbf{K}|}^{-1}\cdot,\mathbf{e}_i\rangle}{a_i^2} \left(g\circ P_{|\mathbf{K}|}^{-1} \right)\begin{vmatrix}D_{x_{k_{0}}}\left( P_{\{i\}}P_{|\mathbf{K}|}^{-1} \right) &D_{x_{k_{1}}}\left( P_{\{i\}}P_{|\mathbf{K}|}^{-1} \right)&\cdots&D_{x_{k_{n}}}\left( P_{\{i\}}P_{|\mathbf{K}|}^{-1} \right)\\ D_{x_{k_{0}}}\left( P_{\{i_{1}\}}P_{|\mathbf{K}|}^{-1} \right)&D_{x_{k_{1}}}\left( P_{\{i_{1}\}}P_{|\mathbf{K}|}^{-1} \right)&\cdots&D_{x_{k_{n}}}\left( P_{\{i_{1}\}}P_{|\mathbf{K}|}^{-1} \right)\\ \vdots&\vdots&\ddots&\vdots\\ D_{x_{k_{0}}}\left( P_{\{i_{n}\}}P_{|\mathbf{K}|}^{-1} \right)&D_{x_{k_{1}}}\left( P_{\{i_{n }\}}P_{|\mathbf{K}|}^{-1} \right)&\cdots&D_{x_{k_{n}}}\left( P_{\{i_{n }\}}P_{|\mathbf{K}|}^{-1} \right)\end{vmatrix}\\
&=-\sum\limits_{j=0}^{n} \frac{x_{k_j}}{a_{k_j}^2}\left(g\circ P_{|\mathbf{K}|}^{-1} \right) \begin{vmatrix}D_{x_{k_{0}}}\left( P_{\{k_j\}}P_{|\mathbf{K}|}^{-1} \right) &D_{x_{k_{1}}}\left( P_{\{k_j\}}P_{|\mathbf{K}|}^{-1} \right)&\cdots&D_{x_{k_{n}}}\left( P_{\{k_j\}}P_{|\mathbf{K}|}^{-1} \right)\\ D_{x_{k_{0}}}\left( P_{\{i_{1}\}}P_{|\mathbf{K}|}^{-1} \right)&D_{x_{k_{1}}}\left( P_{\{i_{1}\}}P_{|\mathbf{K}|}^{-1} \right)&\cdots&D_{x_{k_{n}}}\left( P_{\{i_{1}\}}P_{|\mathbf{K}|}^{-1} \right)\\ \vdots&\vdots&\ddots&\vdots\\ D_{x_{k_{0}}}\left( P_{\{i_{n}\}}P_{|\mathbf{K}|}^{-1} \right)&D_{x_{k_{1}}}\left( P_{\{i_{n }\}}P_{|\mathbf{K}|}^{-1} \right)&\cdots&D_{x_{k_{n}}}\left( P_{\{i_{n }\}}P_{|\mathbf{K}|}^{-1} \right)\end{vmatrix}\\
&\quad + \left( g\circ P_{|\mathbf{K}|}^{-1}  \right)\cdot\begin{vmatrix}\frac{D_{x_{k_{0}}}\left( F_{\mathbb{N}\setminus |\mathbf{K}|}\circ(P_{\mathbb{N}\setminus |\mathbf{K}|}P_{|\mathbf{K}|}^{-1}) \right)}{F_{\mathbb{N}\setminus |\mathbf{K}|}\circ(P_{\mathbb{N}\setminus |\mathbf{K}|}  P_{|\mathbf{K}|}^{-1})} &\frac{D_{x_{k_{1}}}\left( F_{\mathbb{N}\setminus |\mathbf{K}|}\circ(P_{\mathbb{N}\setminus |\mathbf{K}|}P_{|\mathbf{K}|}^{-1}) \right)}{F_{\mathbb{N}\setminus|\mathbf{K}|}\circ(P_{\mathbb{N}\setminus |\mathbf{K}|}P_{|\mathbf{K}|}^{-1})}&\cdots&\frac{D_{x_{k_{n}}}\left( F_{\mathbb{N}\setminus |\mathbf{K}|}\circ(P_{\mathbb{N}\setminus |\mathbf{K}|}P_{|\mathbf{K}|}^{-1}) \right)}{F_{\mathbb{N}\setminus |\mathbf{K}|}\circ(P_{\mathbb{N}\setminus |\mathbf{K}|}P_{|\mathbf{K}|}^{-1})}\\ D_{x_{k_{0}}}\left( P_{\{i_{1}\}}P_{|\mathbf{K}|}^{-1} \right)&D_{x_{k_{1}}}\left( P_{\{i_{1}\}}P_{|\mathbf{K}|}^{-1} \right)&\cdots&D_{x_{k_{n}}}\left( P_{\{i_{1}\}}P_{|\mathbf{K}|}^{-1} \right)\\ \vdots&\vdots&\ddots&\vdots\\ D_{x_{k_{0}}}\left( P_{\{i_{n}\}}P_{|\mathbf{K}|}^{-1} \right)&D_{x_{k_{1}}}\left( P_{\{i_{n }\}}P_{|\mathbf{K}|}^{-1} \right)&\cdots&D_{x_{k_{n}}}\left( P_{\{i_{n }\}}P_{|\mathbf{K}|}^{-1} \right)\\
\end{vmatrix},
\end{aligned}
$$
where the first equality follows from the fact that $D_{x_{k_{j}}}\left( P_{\{i\}}P_{|\mathbf{K}|}^{-1} \right)=0$ for all $j\in\{0,1,\ldots,n\}$ and $i\in |\mathbf{K}|\setminus \{k_0,k_1,\ldots,k_n\}$, the last equality follows from \eqref{20250115for2}  and by Lemma \ref{20260728lem1} all infinite sums in above procedure are absolutely convergent. Write $\psi\triangleq   \big(g\circ P_{|\mathbf{K}|}^{-1}\big) \cdot  \big(F_{\mathbb{N} \setminus |\mathbf{K}|}\circ(P_{\mathbb{N} \setminus |\mathbf{K}|}P_{|\mathbf{K}|}^{-1})\big) $. Thus we have
\begin{eqnarray}\label{20250121for6}
\begin{aligned}
&\int_{\{\rho <0\}} \sum\limits_{i=1}^{\infty} (D_{x_i}g)\circ P_{|\mathbf{K}|}^{-1} \cdot
\begin{vmatrix}D_{x_{k_{0}}}\left( P_{\{i\}}P_{|\mathbf{K}|}^{-1} \right) &D_{x_{k_{1}}}\left( P_{\{i\}}P_{|\mathbf{K}|}^{-1} \right)&\cdots&D_{x_{k_{n}}}\left( P_{\{i\}}P_{|\mathbf{K}|}^{-1} \right)\\ D_{x_{k_{0}}}\left( P_{\{i_{1}\}}P_{|\mathbf{K}|}^{-1} \right)&D_{x_{k_{1}}}\left( P_{\{i_{1}\}}P_{|\mathbf{K}|}^{-1} \right)&\cdots&D_{x_{k_{n}}}\left( P_{\{i_{1}\}}P_{|\mathbf{K}|}^{-1} \right)\\ \vdots&\vdots&\ddots&\vdots\\ D_{x_{k_{0}}}\left( P_{\{i_{n}\}}P_{|\mathbf{K}|}^{-1} \right)&D_{x_{k_{1}}}\left( P_{\{i_{n }\}}P_{|\mathbf{K}|}^{-1} \right)&\cdots&D_{x_{k_{n}}}\left( P_{\{i_{n }\}}P_{|\mathbf{K}|}^{-1} \right)\end{vmatrix}
\\
&\qquad\qquad\qquad\qquad\qquad\qquad \qquad\qquad\qquad  \qquad\qquad\qquad\qquad\qquad \times F_{\mathbb{N} \setminus |\mathbf{K}|}\circ\left( P_{\mathbb{N} \setminus |\mathbf{K}|} P_{|\mathbf{K}|}^{-1}  \right) \mathrm{d} \mu_{|\mathbf{K}|}\\
&- \int_{\{\rho <0\}} \sum\limits_{i=1}^{\infty} \frac{x_i}{a_i^2}\cdot  \left(g\circ P_{|\mathbf{K}|}^{-1}  \right) \cdot\begin{vmatrix}D_{x_{k_{0}}}\left( P_{\{i\}}P_{|\mathbf{K}|}^{-1} \right) &D_{x_{k_{1}}}\left( P_{\{i\}}P_{|\mathbf{K}|}^{-1} \right)&\cdots&D_{x_{k_{n}}}\left( P_{\{i\}}P_{|\mathbf{K}|}^{-1} \right)\\ D_{x_{k_{0}}}\left( P_{\{i_{1}\}}P_{|\mathbf{K}|}^{-1} \right)&D_{x_{k_{1}}}\left( P_{\{i_{1}\}}P_{|\mathbf{K}|}^{-1} \right)&\cdots&D_{x_{k_{n}}}\left( P_{\{i_{1}\}}P_{|\mathbf{K}|}^{-1} \right)\\ \vdots&\vdots&\ddots&\vdots\\ D_{x_{k_{0}}}\left( P_{\{i_{n}\}}P_{|\mathbf{K}|}^{-1} \right)&D_{x_{k_{1}}}\left( P_{\{i_{n }\}}P_{|\mathbf{K}|}^{-1} \right)&\cdots&D_{x_{k_{n}}}\left( P_{\{i_{n }\}}P_{|\mathbf{K}|}^{-1} \right)\end{vmatrix}
\\
&\qquad\qquad\qquad\qquad\qquad\qquad \qquad\qquad\qquad  \qquad\qquad\qquad\qquad\qquad \times F_{\mathbb{N} \setminus |\mathbf{K}|}\circ\left( P_{\mathbb{N} \setminus |\mathbf{K}|} P_{|\mathbf{K}|}^{-1}   \right) \mathrm{d} \mu_{|\mathbf{K}|},\\
&=\int_{\{\rho <0\}}\begin{vmatrix}
\delta_{k_{0}}\psi &\delta_{k_{1}}\psi&\cdots&\delta_{k_{n}}\psi\\ D_{x_{k_{0}}}\left( P_{\{i_{1}\}}P_{|\mathbf{K}|}^{-1} \right)&D_{x_{k_{1}}}\left( P_{\{i_{1}\}}P_{|\mathbf{K}|}^{-1} \right)&\cdots&D_{x_{k_{n}}}\left( P_{\{i_{1}\}}P_{|\mathbf{K}|}^{-1} \right)\\ \vdots&\vdots&\ddots&\vdots\\ D_{x_{k_{0}}}\left( P_{\{i_{n}\}}P_{|\mathbf{K}|}^{-1} \right)&D_{x_{k_{1}}}\left( P_{\{i_{n }\}}P_{|\mathbf{K}|}^{-1} \right)&\cdots&D_{x_{k_{n}}}\left( P_{\{i_{n }\}}P_{|\mathbf{K}|}^{-1} \right)\end{vmatrix}
 \mathrm{d} \mu_{|\mathbf{K}|}\\
&=\sum_{i=0}^{n}(-1)^i\cdot \int_{\{\rho <0\}}\delta_{k_{i}}\psi\cdot
\begin{vmatrix}
D_{x_{k_{0}}}\left( P_{\{i_{1}\}}P_{|\mathbf{K}|}^{-1} \right)&\cdots&\widehat{D_{x_{k_{i}}}\left( P_{\{i_{1}\}}P_{|\mathbf{K}|}^{-1} \right)}&\cdots&D_{x_{k_{n}}}\left( P_{\{i_{1}\}}P_{|\mathbf{K}|}^{-1} \right)\\ \vdots&\ddots&\vdots&\ddots&\vdots\\ D_{x_{k_{0}}}\left( P_{\{i_{n}\}}P_{|\mathbf{K}|}^{-1} \right)&\cdots&\widehat{D_{x_{k_{i}}}\left( P_{\{i_{n }\}}P_{|\mathbf{K}|}^{-1} \right)}&\cdots&D_{x_{k_{n}}}\left( P_{\{i_{n }\}}P_{|\mathbf{K}|}^{-1} \right)\end{vmatrix}
 \mathrm{d} \mu_{|\mathbf{K}|}.
\end{aligned}
\end{eqnarray}
Here, the second equality follows by expanding the third determinant in the preceding display along its first row, and the final determinant is the minor obtained by deleting the first row and the \((i+1)\)-st column of the corresponding matrix. Combining the preceding arguments with assumptions~\eqref{20250120for2} and~\eqref{20250120for3}, we conclude that all the integrals appearing in the preceding equations are finite.

Let $K_n'\triangleq\{k_{0},k_{1},\ldots,k_{n}\},\, K_n\triangleq|\mathbf{K}|\setminus K_n'$. Fix $i\in\{0,1,\ldots,n\}$, write
\begin{eqnarray*}
\phi^i \triangleq  \psi \cdot  \big(F_{K_n'}\circ P_{K_n'}\big)
\cdot\begin{vmatrix}
D_{x_{k_{0}}}\left( P_{\{i_{1}\}}P_{|\mathbf{K}|}^{-1} \right)&\cdots&\widehat{D_{x_{k_{i}}}\left( P_{\{i_{1}\}}P_{|\mathbf{K}|}^{-1} \right)}&\cdots&D_{x_{k_{n}}}\left( P_{\{i_{1}\}}P_{|\mathbf{K}|}^{-1} \right)\\ \vdots&\ddots&\vdots&\ddots&\vdots\\ D_{x_{k_{0}}}\left( P_{\{i_{n}\}}P_{|\mathbf{K}|}^{-1} \right)&\cdots&\widehat{D_{x_{k_{i}}}\left( P_{\{i_{n }\}}P_{|\mathbf{K}|}^{-1} \right)}&\cdots&D_{x_{k_{n}}}\left( P_{\{i_{n }\}}P_{|\mathbf{K}|}^{-1} \right)\end{vmatrix}.
\end{eqnarray*}
For $\textbf{x}_{K_n}\in P_{K_n} U_{|\mathbf{K}|}$, define
$$U_{\textbf{x}_{K_n}}\triangleq\{\textbf{x}_{K_n'}\in P_{K_n'} U_{|\mathbf{K}|}:\textbf{x}_{K_n'}+\textbf{x}_{K_n}\in U_{|\mathbf{K}|} \},$$
$$
\phi_{\textbf{x}_{K_n}}^i(\textbf{x}_{K_n'})\triangleq \phi^i(\textbf{x}_{K_n'}+\textbf{x}_{K_n}),\qquad \rho_{\textbf{x}_{K_n}}(\textbf{x}_{K_n'})\triangleq \rho(\textbf{x}_{K_n'}+\textbf{x}_{K_n}),\quad\, \textbf{x}_{K_n'}\in U_{\textbf{x}_{K_n}}.
$$
Then $U_{\textbf{x}_{K_n}}$ is a nonempty open set in $(n+1)$-dimensional linear space $P_{K_n'}\ell^2$.
Since $g\in C_{\mathcal F}^1(U)$, the closure of $\{\rho_{\textbf{x}_{K_n}}<0\}$ is contained in $U_{\textbf{x}_{K_n}}$, $\rho\in C^1_{P_{|\mathbf{K}| }\ell^2}
\left(
U_{|K| };
\mathbb R
\right)$ and
$|D_{x_{k_{i_0}}}\rho_{\textbf{x}_{K_n}}|=1$,
 we have $\phi_{\textbf{x}_{K_n}}^i\in C^1(U_{\textbf{x}_{K_n}})$, $\{\rho_{\textbf{x}_{K_n}}<0\}$ is an open set of $P_{K_n'}\ell^2$ with $C^1$ boundary
 $$
 \{\rho_{\textbf{x}_{K_n}}=0\}=\{\textbf{x}_{K_n'\setminus\{k_{i_0}\}} +h(\textbf{x}_{K_n'\setminus\{k_{i_0}\}}+\textbf{x}_{K_n})\mathbf{e}_{k_{i_0}}:\textbf{x}_{K_n'\setminus\{k_{i_0}\}} +\textbf{x}_{K_n}\in V_{|\mathbf{K}|\setminus\{k_{i_0}\}}\}.
 $$
Consequently, by the Gauss--Green formula in \cite[p. 711]{Eva},
\begin{eqnarray}
&&\int_{\{\rho_{\textbf{x}_{K_n}}<0\}}D_{x_{k_i}}\phi^i(\textbf{x}_{K_n'}+\textbf{x}_{K_n})\,\mathrm{d}\textbf{x}_{K_n'}\label{20250121for1}\\
&=&\int_{ \{\textbf{x}_{K_n'\setminus\{k_{i_0}\}}\in P_{K_n'\setminus\{k_{i_0}\}}\ell^2:\textbf{x}_{K_n'\setminus\{k_{i_0}\}} +\textbf{x}_{K_n}\in V_{|K|\setminus\{k_{i_0}\}}\}} \phi^i\left(\textbf{x}_{K_n'\setminus\{k_{i_0}\}}+\textbf{x}_{K_n}+h(\textbf{x}_{K_n'\setminus\{k_{i_0}\}}+\textbf{x}_{K_n})\mathbf{e}_{k_{i_0}}\right)\nonumber\\
&&\times\left(\frac{D_{x_{k_i}}\rho_{\textbf{x}_{K_n}}}{\sqrt{\sum_{j=0}^{n}|D_{x_{k_j}}\rho_{\textbf{x}_{K_n}}|^2}} \frac{\sqrt{\sum_{j=0}^{n}|D_{x_{k_j}}\rho_{\textbf{x}_{K_n}}|^2}}{|D_{x_{k_{i_0}}}\rho_{\textbf{x}_{K_n}}|}\right)\left(\textbf{x}_{K_n'\setminus\{k_{i_0}\}}+\textbf{x}_{K_n} +h(\textbf{x}_{K_n'\setminus\{k_{i_0}\}}+\textbf{x}_{K_n})\mathbf{e}_{k_{i_0}}\right)\,\mathrm{d}\textbf{x}_{K_n'\setminus\{k_{i_0}\}}\nonumber\\
&=&\int_{ \{\textbf{x}_{K_n'\setminus\{k_{i_0}\}}\in P_{K_n'\setminus\{k_{i_0}\}}\ell^2:\textbf{x}_{K_n'\setminus\{k_{i_0}\}} +\textbf{x}_{K_n}\in V_{|K|\setminus\{k_{i_0}\}}\}}  \phi^i\left(\textbf{x}_{K_n'\setminus\{k_{i_0}\}}+\textbf{x}_{K_n}+h(\textbf{x}_{K_n'\setminus\{k_{i_0}\}}+\textbf{x}_{K_n})\mathbf{e}_{k_{i_0}}\right)\nonumber\\
&&\times (D_{x_{k_i}}\rho_{\textbf{x}_{K_n}} )\left(\textbf{x}_{K_n'\setminus\{k_{i_0}\}}+\textbf{x}_{K_n}+h(\textbf{x}_{K_n'\setminus\{k_{i_0}\}}+\textbf{x}_{K_n})\mathbf{e}_{k_{i_0}}\right)\,\mathrm{d}\textbf{x}_{K_n'\setminus\{k_{i_0}\}}.\nonumber
\end{eqnarray}
where $\mathrm{d}\textbf{x}_{K_n'}$ and $\mathrm{d}\textbf{x}_{K_n'\setminus\{k_{i_0}\}}$ denote the Lebesgue measures on corresponding spaces, and the second equality follows from $|D_{x_{k_{i_0}}}\rho_{\textbf{x}_{K_n}}|=1$. Combining these observations that the closure of \(g^{-1}(\mathbb{R}\setminus\{0\})\) in \(\ell^2\) is compact, \eqref{20250120for2} and \eqref{20250120for3}, we obtain
\begin{eqnarray*}
\int_{\{\rho<0\}}|D_{x_{k_i}}\phi^i(\textbf{x}_{K_n'}+\textbf{x}_{K_n})|\,\mathrm{d}\textbf{x}_{K_n'}\mathrm{d}\mu_{K_n}(\textbf{x}_{K_n})
<\infty,
\end{eqnarray*}
and by Fubini's theorem, it holds that
\begin{eqnarray}
&&\int_{\{\rho<0\}}D_{x_{k_i}}\phi^i(\textbf{x}_{K_n'}+\textbf{x}_{K_n})\,\mathrm{d}\textbf{x}_{K_n'}\mathrm{d}\mu_{K_n}(\textbf{x}_{K_n})\nonumber\\
&=&\int_{P_{K_n}\{\rho<0\}}\left(\int_{\{\rho_{\textbf{x}_{K_n}}<0\}}D_{x_{k_i}}\phi^i(\textbf{x}_{K_n'}+\textbf{x}_{K_n})
\,\mathrm{d}\textbf{x}_{K_n'}\right)\mathrm{d}\mu_{K_n}(\textbf{x}_{K_n}).\label{20250121for2}
\end{eqnarray}
Combining \eqref{20250121for1} and \eqref{20250121for2}, we obtain
\begin{eqnarray}\label{20250121for3}
\begin{aligned}
&\sum_{i=0}^{n}(-1)^i\cdot \int_{\{\rho<0\}}\delta_{k_{i}}\psi
\\
&\quad\times\begin{vmatrix}
D_{x_{k_{0}}}\left( P_{\{i_{1}\}}P_{|\mathbf{K}|}^{-1} \right)&\cdots&\widehat{D_{x_{k_{i}}}\left( P_{\{i_{1}\}}P_{|\mathbf{K}|}^{-1} \right)}&\cdots&D_{x_{k_{n}}}\left( P_{\{i_{1}\}}P_{|\mathbf{K}|}^{-1} \right)\\ \vdots&\ddots&\vdots&\ddots&\vdots\\ D_{x_{k_{0}}}\left( P_{\{i_{n}\}}P_{|\mathbf{K}|}^{-1} \right)&\cdots&\widehat{D_{x_{k_{i}}}\left( P_{\{i_{n }\}}P_{|\mathbf{K}|}^{-1} \right)}&\cdots&D_{x_{k_{n}}}\left( P_{\{i_{n }\}}P_{|\mathbf{K}|}^{-1} \right)\end{vmatrix}
 \mathrm{d} \mu_{|\mathbf{K}|}\\
&+\sum_{i=0}^{n}(-1)^{i}\cdot \int_{\{\rho<0\}}    \psi
\\
&\quad\times \left(D_{x_{k_{i}}}\begin{vmatrix}
D_{x_{k_{0}}}\left( P_{\{i_{1}\}}P_{|\mathbf{K}|}^{-1} \right)&\cdots&\widehat{D_{x_{k_{i}}}\left( P_{\{i_{1}\}}P_{|\mathbf{K}|}^{-1} \right)}&\cdots&D_{x_{k_{n}}}\left( P_{\{i_{1}\}}P_{|\mathbf{K}|}^{-1} \right)\\ \vdots&\ddots&\vdots&\ddots&\vdots\\ D_{x_{k_{0}}}\left( P_{\{i_{n}\}}P_{|\mathbf{K}|}^{-1} \right)&\cdots&\widehat{D_{x_{k_{i}}}\left( P_{\{i_{n }\}}P_{|\mathbf{K}|}^{-1} \right)}&\cdots&D_{x_{k_{n}}}\left( P_{\{i_{n }\}}P_{|\mathbf{K}|}^{-1} \right)\end{vmatrix}\right)
 \mathrm{d} \mu_{|\mathbf{K}|}\\
&=\sum_{i=0}^{n}(-1)^i\cdot \int_{P_{K_n}\{\rho<0\}}\left(\int_{\{\rho_{\textbf{x}_{K_n}}<0\}}D_{x_{k_i}}\phi^i(\textbf{x}_{K_n'}+\textbf{x}_{K_n})
\,\mathrm{d}\textbf{x}_{K_n'}\right)\mathrm{d}\mu_{K_n}(\textbf{x}_{K_n}) \\
&=\sum_{i=0}^{n}(-1)^i\cdot \int_{V_{|\mathbf{K}|\setminus\{k_{i_0}\}}}\Phi^i\left(\textbf{x}_{|\mathbf{K}|\setminus\{k_{i_0}\}}+h(\textbf{x}_{|\mathbf{K}|\setminus\{k_{i_0}\}})\mathbf{e}_{k_{i_{0}}}\right)
\\
&\quad\times (D_{x_{k_i}}\rho )\left(\textbf{x}_{|\mathbf{K}|\setminus\{k_{i_0}\}}+h(\textbf{x}_{|\mathbf{K}|\setminus\{k_{i_0}\}})\mathbf{e}_{k_{i_0}}\right)
 \mathrm{d} \mu_{|\mathbf{K}|\setminus\{k_{i_0}\}}(\textbf{x}_{|\mathbf{K}|\setminus\{k_{i_0}\}}),
\end{aligned}
\end{eqnarray}
where
$$
\Phi^i\triangleq \frac{\phi^i}{  F_{K_n'\setminus\{k_{i_0}\}}\circ P_{K_n'\setminus\{k_{i_0}\}} },\qquad i=0,1,\ldots,n.
$$
Combining the facts that the closure of \(g^{-1}(\mathbb{R}\setminus\{0\})\) in \(\ell^2\)
 is compact and that \eqref{20250120for2} and \eqref{20250120for3} hold, we conclude that the integrals in the above equations are finite.

First, we have
\begin{eqnarray}\label{20250121for5}
\begin{aligned}
&\sum_{i=0}^{n}(-1)^{i}\cdot D_{x_{k_{i}}}\begin{vmatrix}
D_{x_{k_{0}}}\left( P_{\{i_{1}\}}P_{|\mathbf{K}|}^{-1} \right)&\cdots&\widehat{D_{x_{k_{i}}}\left( P_{\{i_{1}\}}P_{|\mathbf{K}|}^{-1} \right)}&\cdots&D_{x_{k_{n}}}\left( P_{\{i_{1}\}}P_{|\mathbf{K}|}^{-1} \right)\\ \vdots&\ddots&\vdots&\ddots&\vdots\\ D_{x_{k_{0}}}\left( P_{\{i_{n}\}}P_{|\mathbf{K}|}^{-1} \right)&\cdots&\widehat{D_{x_{k_{i}}}\left( P_{\{i_{n }\}}P_{|\mathbf{K}|}^{-1} \right)}&\cdots&D_{x_{k_{n}}}\left( P_{\{i_{n }\}}P_{|\mathbf{K}|}^{-1} \right)\end{vmatrix}\\
&=\sum_{i=0}^{n}(-1)^{i}\cdot\sum_{j=1}^{n}\begin{vmatrix}
D_{x_{k_{0}}}\left( P_{\{i_{1}\}}P_{|\mathbf{K}|}^{-1} \right)&\cdots&\widehat{D_{x_{k_{i}}}\left( P_{\{i_{1}\}}P_{|\mathbf{K}|}^{-1} \right)}&\cdots&D_{x_{k_{n}}}\left( P_{\{i_{1}\}}P_{|\mathbf{K}|}^{-1} \right)\\
\vdots&\ddots&\vdots&\ddots&\vdots\\
 D_{x_{k_{i}}}D_{x_{k_{0}}}\left( P_{\{i_{j}\}}P_{|\mathbf{K}|}^{-1} \right)&\cdots&\widehat{ D_{x_{k_{i}}}D_{x_{k_{i}}}\left( P_{\{i_{j}\}}P_{|\mathbf{K}|}^{-1} \right)}&\cdots& D_{x_{k_{i}}}D_{x_{k_{n}}}\left( P_{\{i_{j}\}}P_{|\mathbf{K}|}^{-1} \right)\\
 \vdots&\ddots&\vdots&\ddots&\vdots\\
D_{x_{k_{0}}}\left( P_{\{i_{n}\}}P_{|\mathbf{K}|}^{-1} \right)&\cdots&\widehat{D_{x_{k_{i}}}\left( P_{\{i_{n }\}}P_{|\mathbf{K}|}^{-1} \right)}&\cdots&D_{x_{k_{n}}}\left( P_{\{i_{n }\}}P_{|\mathbf{K}|}^{-1} \right)\end{vmatrix}\\
&=\sum_{j=1}^{n}\sum_{i=0}^{n}\sum_{r=0}^{i-1}(-1)^{i}\cdot(-1)^{j+r+1}\cdot  D_{x_{k_{i}}}D_{x_{k_{r}}}\left( P_{\{i_{j}\}}P_{|\mathbf{K}|}^{-1} \right)\\
&\times\begin{vmatrix}
D_{x_{k_{0}}}\left( P_{\{i_{1}\}}P_{|\mathbf{K}|}^{-1} \right)&\cdots&\widehat{D_{x_{k_{r}}}\left( P_{\{i_{1}\}}P_{|\mathbf{K}|}^{-1} \right)}&\cdots&\widehat{D_{x_{k_{i}}}\left( P_{\{i_{1}\}}P_{|\mathbf{K}|}^{-1} \right)}&\cdots&D_{x_{k_{n}}}\left( P_{\{i_{1}\}}P_{|\mathbf{K}|}^{-1} \right)\\
\vdots&\ddots&\vdots&\ddots&\vdots\\
\widehat{ D_{x_{k_{0}}}\left( P_{\{i_{j}\}}P_{|\mathbf{K}|}^{-1} \right)}&\cdots&\widehat{D_{x_{k_{r}}}\left( P_{\{i_{j}\}}P_{|\mathbf{K}|}^{-1} \right)}&\cdots&\widehat{ D_{x_{k_{i}}}\left( P_{\{i_{j}\}}P_{|\mathbf{K}|}^{-1} \right)}&\cdots& \widehat{D_{x_{k_{n}}}\left( P_{\{i_{j}\}}P_{|\mathbf{K}|}^{-1} \right)}\\
 \vdots&\ddots&\vdots&\ddots&\vdots\\
D_{x_{k_{0}}}\left( P_{\{i_{n}\}}P_{|\mathbf{K}|}^{-1} \right)&\cdots&\widehat{D_{x_{k_{r}}}\left( P_{\{i_{n}\}}P_{|\mathbf{K}|}^{-1} \right)}&\cdots&\widehat{D_{x_{k_{i}}}\left( P_{\{i_{n }\}}P_{|\mathbf{K}|}^{-1} \right)}&\cdots&D_{x_{k_{n}}}\left( P_{\{i_{n }\}}P_{|\mathbf{K}|}^{-1} \right)\end{vmatrix}\\
&-\sum_{j=1}^{n}\sum_{i=0}^{n}\sum_{r=i+1}^{n}(-1)^{i}\cdot(-1)^{j+r+1}\cdot  D_{x_{k_{i}}}D_{x_{k_{r}}}\left( P_{\{i_{j}\}}P_{|\mathbf{K}|}^{-1} \right)\\
&\times\begin{vmatrix}
D_{x_{k_{0}}}\left( P_{\{i_{1}\}}P_{|\mathbf{K}|}^{-1} \right)&\cdots&\widehat{D_{x_{k_{i}}}\left( P_{\{i_{1}\}}P_{|\mathbf{K}|}^{-1} \right)}&\cdots&\widehat{D_{x_{k_{r}}}\left( P_{\{i_{1}\}}P_{|\mathbf{K}|}^{-1} \right)}&\cdots&D_{x_{k_{n}}}\left( P_{\{i_{1}\}}P_{|\mathbf{K}|}^{-1} \right)\\
\vdots&\ddots&\vdots&\ddots&\vdots\\
\widehat{ D_{x_{k_{0}}}\left( P_{\{i_{j}\}}P_{|\mathbf{K}|}^{-1} \right)}&\cdots&\widehat{D_{x_{k_{i}}}\left( P_{\{i_{j}\}}P_{|\mathbf{K}|}^{-1} \right)}&\cdots&\widehat{ D_{x_{k_{r}}}\left( P_{\{i_{j}\}}P_{|\mathbf{K}|}^{-1} \right)}&\cdots& \widehat{D_{x_{k_{n}}}\left( P_{\{i_{j}\}}P_{|\mathbf{K}|}^{-1} \right)}\\
 \vdots&\ddots&\vdots&\ddots&\vdots\\
D_{x_{k_{0}}}\left( P_{\{i_{n}\}}P_{|\mathbf{K}|}^{-1} \right)&\cdots&\widehat{D_{x_{k_{i}}}\left( P_{\{i_{n}\}}P_{|\mathbf{K}|}^{-1} \right)}&\cdots&\widehat{D_{x_{k_{r}}}\left( P_{\{i_{n }\}}P_{|\mathbf{K}|}^{-1} \right)}&\cdots&D_{x_{k_{n}}}\left( P_{\{i_{n }\}}P_{|\mathbf{K}|}^{-1} \right)\end{vmatrix}\\
&=0.
\end{aligned}
\end{eqnarray}
Here, the second equality follows by expanding each determinant along its \(j\)-th row. For the last equality, we interchange the indices \(i\) and \(r\) in the second triple sum. The corresponding minors are identical, while the associated coefficients have opposite signs. Furthermore, \eqref{20260729for1} implies that
\[
D_{x_{k_i}}D_{x_{k_r}}
\left(P_{{i_j}}P_{|\mathbf{K}|}^{-1}\right)
=D_{x_{k_r}}D_{x_{k_i}}
\left(P_{{i_j}} P_{|\mathbf{K}|}^{-1}\right).
\]
Hence, the two triple sums cancel term by term, and the final expression is zero.

Moreover,
$$
\begin{aligned}
&\sum_{i=0}^{n}(-1)^i\cdot \int_{V_{|\mathbf{K}|\setminus\{k_{i_0}\}}}\Phi^i\left(\textbf{x}_{|\mathbf{K}|\setminus\{k_{i_0}\}}+h(\textbf{x}_{|\mathbf{K}|\setminus\{k_{i_0}\}})\mathbf{e}_{k_{i_0}}\right)
\\
&\quad\times (D_{x_{k_i}}\rho )\left(\textbf{x}_{|\mathbf{K}|\setminus\{k_{i_0}\}}+h(\textbf{x}_{|\mathbf{K}|\setminus\{k_{i_0}\}})\mathbf{e}_{k_{i_0}}\right)
 \mathrm{d} \mu_{|\mathbf{K}|\setminus\{k_{i_0}\}}(\textbf{x}_{|\mathbf{K}|\setminus\{k_{i_0}\}})\\
&= (-1)^{i_0}\cdot \int_{V_{|\mathbf{K}|\setminus\{k_{i_0}\}}}   (g\circ P_{|\mathbf{K}|\setminus\{k_{i_0}\}}^{-1})\cdot  F_{\mathbb{N} \setminus (|\mathbf{K}|\setminus\{k_{i_0}\})}\circ(P_{\mathbb{N} \setminus (|\mathbf{K}|\setminus\{k_{i_0}\})}P_{|\mathbf{K}|\setminus\{k_{i_0}\}}^{-1})\cdot (D_{x_{k_{i_0}}}\rho)\circ (P_{|\mathbf{K}|}  P_{|\mathbf{K}|\setminus\{k_{i_0}\}}^{-1})\\
&\quad\times\begin{vmatrix}
D_{x_{k_{0}}}\left( P_{\{i_{1}\}}P_{|\mathbf{K}|}^{-1} \right)&\cdots&\widehat{D_{x_{k_{i_0}}}\left( P_{\{i_{1}\}}P_{|\mathbf{K}|}^{-1} \right)}&\cdots&D_{x_{k_{n}}}\left( P_{\{i_{1}\}}P_{|\mathbf{K}|}^{-1} \right)\\ \vdots&\ddots&\vdots&\ddots&\vdots\\ D_{x_{k_{0}}}\left( P_{\{i_{n}\}}P_{|\mathbf{K}|}^{-1} \right)&\cdots&\widehat{D_{x_{k_{i_0}}}\left( P_{\{i_{n }\}}P_{|\mathbf{K}|}^{-1} \right)}&\cdots&D_{x_{k_{n}}}\left( P_{\{i_{n }\}}P_{|\mathbf{K}|}^{-1} \right)\end{vmatrix}\circ (P_{|\mathbf{K}|} P_{|\mathbf{K}|\setminus\{k_{i_0}\}}^{-1})
 \mathrm{d} \mu_{|\mathbf{K}|\setminus\{k_{i_0}\}}\\
 &\quad+\sum_{i\in \{0,1,\ldots, n\}\setminus\{i_0\}}(-1)^i\cdot \int_{V_{|\mathbf{K}|\setminus\{k_{i_0}\}}}  (g\circ P_{|\mathbf{K}|\setminus\{k_{i_0}\}}^{-1})\cdot  F_{\mathbb{N} \setminus (|\mathbf{K}|\setminus\{k_{i_0}\})}\circ(P_{\mathbb{N} \setminus (|\mathbf{K}|\setminus\{k_{i_0}\})}P_{|\mathbf{K}|\setminus\{k_{i_0}\}}^{-1})
\\
&\quad\times \left(D_{x_{k_{i}}}\rho \begin{vmatrix}
D_{x_{k_{0}}}\left( P_{\{i_{1}\}}P_{|\mathbf{K}|}^{-1} \right)&\cdots&\widehat{D_{x_{k_{i}}}\left( P_{\{i_{1}\}}P_{|\mathbf{K}|}^{-1} \right)}&\cdots&D_{x_{k_{n}}}\left( P_{\{i_{1}\}}P_{|\mathbf{K}|}^{-1} \right)\\ \vdots&\ddots&\vdots&\ddots&\vdots\\ D_{x_{k_{0}}}\left( P_{\{i_{n}\}}P_{|\mathbf{K}|}^{-1} \right)&\cdots&\widehat{D_{x_{k_{i}}}\left( P_{\{i_{n }\}}P_{|\mathbf{K}|}^{-1} \right)}&\cdots&D_{x_{k_{n}}}\left( P_{\{i_{n }\}}P_{|\mathbf{K}|}^{-1} \right)\end{vmatrix}\right)\circ (P_{|\mathbf{K}|} P_{|\mathbf{K}|\setminus\{k_{i_0}\}}^{-1})
 \mathrm{d} \mu_{|\mathbf{K}|\setminus\{k_{i_0}\}}\\
&= (-1)^{2i_0-1}\cdot \int_{V_{|\mathbf{K}|\setminus\{k_{i_0}\}}}  (g\circ P_{|\mathbf{K}|\setminus\{k_{i_0}\}}^{-1})\cdot  F_{\mathbb{N} \setminus (|\mathbf{K}|\setminus\{k_{i_0}\})}\circ(P_{\mathbb{N} \setminus (|\mathbf{K}|\setminus\{k_{i_0}\})}P_{|\mathbf{K}|\setminus\{k_{i_0}\}}^{-1})  \\
&\quad\times\begin{vmatrix}
D_{x_{k_{0}}}\left( P_{\{i_{1}\}}P_{|\mathbf{K}|}^{-1} \right)&\cdots&\widehat{D_{x_{k_{i_0}}}\left( P_{\{i_{1}\}}P_{|\mathbf{K}|}^{-1} \right)}&\cdots&D_{x_{k_{n}}}\left( P_{\{i_{1}\}}P_{|\mathbf{K}|}^{-1} \right)\\ \vdots&\ddots&\vdots&\ddots&\vdots\\ D_{x_{k_{0}}}\left( P_{\{i_{n}\}}P_{|\mathbf{K}|}^{-1} \right)&\cdots&\widehat{D_{x_{k_{i_0}}}\left( P_{\{i_{n }\}}P_{|\mathbf{K}|}^{-1} \right)}&\cdots&D_{x_{k_{n}}}\left( P_{\{i_{n }\}}P_{|\mathbf{K}|}^{-1} \right)\end{vmatrix}\circ (P_{|\mathbf{K}|} P_{|\mathbf{K}|\setminus\{k_{i_0}\}}^{-1})
 \mathrm{d} \mu_{|\mathbf{K}|\setminus\{k_{i_0}\}}\\
 &\quad+\sum_{i\in \{0,1,\ldots, n\}\setminus\{i_0\}}(-1)^{i+i_0}\cdot \int_{V_{|\mathbf{K}|\setminus\{k_{i_0}\}}}   (g\circ P_{|\mathbf{K}|\setminus\{k_{i_0}\}}^{-1})\cdot F_{\mathbb{N} \setminus (|\mathbf{K}|\setminus\{k_{i_0}\})}\circ(P_{\mathbb{N} \setminus (|\mathbf{K}|\setminus\{k_{i_0}\})}P_{|\mathbf{K}|\setminus\{k_{i_0}\}}^{-1}) \\
 &\qquad\quad\times D_{x_{k_i}}(P_{\{k_{i_0}\}}P^{-1}_{|\mathbf{K}|\setminus\{k_{i_0}\}})
\\
&\qquad\quad\times\begin{vmatrix}
D_{x_{k_{0}}}\left( P_{\{i_{1}\}}P_{|\mathbf{K}|}^{-1} \right)&\cdots&\widehat{D_{x_{k_{i}}}\left( P_{\{i_{1}\}}P_{|\mathbf{K}|}^{-1} \right)}&\cdots&D_{x_{k_{n}}}\left( P_{\{i_{1}\}}P_{|\mathbf{K}|}^{-1} \right)\\ \vdots&\ddots&\vdots&\ddots&\vdots\\ D_{x_{k_{0}}}\left( P_{\{i_{n}\}}P_{|\mathbf{K}|}^{-1} \right)&\cdots&\widehat{D_{x_{k_{i}}}\left( P_{\{i_{n }\}}P_{|\mathbf{K}|}^{-1} \right)}&\cdots&D_{x_{k_{n}}}\left( P_{\{i_{n }\}}P_{|\mathbf{K}|}^{-1} \right)\end{vmatrix}
 \circ (P_{|\mathbf{K}|} P_{|\mathbf{K}|\setminus\{k_{i_0}\}}^{-1})\mathrm{d} \mu_{|\mathbf{K}|\setminus\{k_{i_0}\}}\\
& =(-1)^{i_0-1}\int_{V_{|\mathbf{K}|\setminus\{k_{i_0}\}}}  (g\circ P_{|\mathbf{K}|\setminus\{k_{i_0}\}}^{-1})\cdot  F_{\mathbb{N} \setminus (|\mathbf{K}|\setminus\{k_{i_0}\})}\circ(P_{\mathbb{N} \setminus (|\mathbf{K}|\setminus\{k_{i_0}\})}P_{|\mathbf{K}|\setminus\{k_{i_0}\}}^{-1})\\
&\quad\times\begin{vmatrix}
-D_{x_{k_{0}}}(P_{\{k_{i_0}\}}P^{-1}_{|\mathbf{K}|\setminus\{k_{i_0}\}})&\cdots& 1 &\cdots&-D_{x_{k_{n}}}(P_{\{k_{i_0}\}}P^{-1}_{|\mathbf{K}|\setminus\{k_{i_0}\}})\\
D_{x_{k_{0}}}\left( P_{\{i_{1}\}}P_{|\mathbf{K}|}^{-1} \right)&\cdots& D_{x_{k_{i_0}}}\left( P_{\{i_{1}\}}P_{|\mathbf{K}|}^{-1} \right) &\cdots&D_{x_{k_{n}}}\left( P_{\{i_{1}\}}P_{|\mathbf{K}|}^{-1} \right)\\ \vdots&\ddots&\vdots&\ddots&\vdots\\ D_{x_{k_{0}}}\left( P_{\{i_{n}\}}P_{|\mathbf{K}|}^{-1} \right)&\cdots& D_{x_{k_{i_0}}}\left( P_{\{i_{n }\}}P_{|\mathbf{K}|}^{-1} \right)&\cdots&D_{x_{k_{n}}}\left( P_{\{i_{n }\}}P_{|\mathbf{K}|}^{-1} \right)\end{vmatrix}
 \mathrm{d} \mu_{|\mathbf{K}|\setminus\{k_{i_0}\}}.
\end{aligned}
$$
Here, the second equality follows from the identities
$
D_{x_{k_{i_0}}}\rho=(-1)^{i_0-1}
$
and
\[
(D_{x_{k_i}}\rho)
\circ
\bigl(
P_{|\mathbf{K}|}
P_{|\mathbf{K}|\setminus\{k_{i_0}\}}^{-1}
\bigr)
=(-1)^{i_0}
D_{x_{k_i}}
\bigl(
P_{\{k_{i_0}\}}
P_{|\mathbf{K}|\setminus\{k_{i_0}\}}^{-1}
\bigr),
\qquad
i\in\{0,1,\ldots,n\}\setminus\{i_0\}.
\]
The last equality is obtained by expanding the final determinant along
its first row.

For notational simplicity, in the final determinant we write
\begin{equation}\label{20260729for2}
D_{x_{k_s}}
\left(
P_{\{i_t\}} P_{|\mathbf{K}|}^{-1}
\right),
\qquad
s\in\{0,1,\ldots,n\},
\quad
t\in\{1,\ldots,n\},
\end{equation}
for the functions appearing in rows \(2\) through \(n+1\). More
precisely, each expression in \eqref{20260729for2} is to be evaluated
after composition with the local parametrization
$
P_{|\mathbf{K}|}
P_{|\mathbf{K}|\setminus\{k_{i_0}\}}^{-1},
$
that is, it stands for
\[
D_{x_{k_s}}
\left(
P_{\{i_t\}} P_{|\mathbf{K}|}^{-1}
\right)
\circ
\left(
P_{|\mathbf{K}|}
P_{|\mathbf{K}|\setminus\{k_{i_0}\}}^{-1}
\right).
\]
We shall use the abbreviated notation \eqref{20260729for2} throughout
the remainder of the proof.

We now derive an explicit formula for the final determinant obtained above, which we divide into the following two cases.

\textbf{Case 1.}  $k_{i_0}\notin\{i_1,\ldots,i_n\}$. Then
$$
\begin{aligned}
&\begin{vmatrix}
-D_{x_{k_{0}}}(P_{\{k_{i_0}\}}P^{-1}_{|\mathbf{K}|\setminus\{k_{i_0}\}})&\cdots& 1 &\cdots&-D_{x_{k_{n}}}(P_{\{k_{i_0}\}}P^{-1}_{|\mathbf{K}|\setminus\{k_{i_0}\}})\\
D_{x_{k_{0}}}\left( P_{\{i_{1}\}}P_{|\mathbf{K}|}^{-1} \right)&\cdots& D_{x_{k_{i_0}}}\left( P_{\{i_{1}\}}P_{|\mathbf{K}|}^{-1} \right) &\cdots&D_{x_{k_{n}}}\left( P_{\{i_{1}\}}P_{|\mathbf{K}|}^{-1} \right)\\ \vdots&\ddots&\vdots&\ddots&\vdots\\ D_{x_{k_{0}}}\left( P_{\{i_{n}\}}P_{|\mathbf{K}|}^{-1} \right)&\cdots& D_{x_{k_{i_0}}}\left( P_{\{i_{n }\}}P_{|\mathbf{K}|}^{-1} \right)&\cdots&D_{x_{k_{n}}}\left( P_{\{i_{n }\}}P_{|\mathbf{K}|}^{-1} \right)\end{vmatrix}\\
&=\begin{vmatrix}
-D_{x_{k_{0}}}(P_{\{k_{i_0}\}}P^{-1}_{|\mathbf{K}|\setminus\{k_{i_0}\}})&\cdots& 1 &\cdots&-D_{x_{k_{n}}}(P_{\{k_{i_0}\}}P^{-1}_{|\mathbf{K}|\setminus\{k_{i_0}\}})\\
D_{x_{k_{0}}}\left( P_{\{i_{1}\}}P_{|\mathbf{K}|\setminus\{k_{i_0}\}}^{-1} \right)&\cdots& 0 &\cdots&D_{x_{k_{n}}}\left( P_{\{i_{1}\}}P_{|\mathbf{K}|\setminus\{k_{i_0}\}}^{-1} \right)\\ \vdots&\ddots&\vdots&\ddots&\vdots\\ D_{x_{k_{0}}}\left( P_{\{i_{n}\}}P_{|\mathbf{K}|\setminus\{k_{i_0}\}}^{-1} \right)&\cdots& 0&\cdots&D_{x_{k_{n}}}\left( P_{\{i_{n }\}}P_{|\mathbf{K}|\setminus\{k_{i_0}\}}^{-1} \right)\end{vmatrix}\\
&=(-1)^{i_0}\begin{vmatrix}
D_{x_{k_{0}}}\left( P_{\{i_{1}\}}P_{|\mathbf{K}|\setminus\{k_{i_0}\}}^{-1} \right)&\cdots& \widehat{D_{x_{k_{i_0}}}\left( P_{\{i_{1 }\}}P_{|\mathbf{K}|\setminus\{k_{i_0}\}}^{-1} \right)} &\cdots&D_{x_{k_{n}}}\left( P_{\{i_{1}\}}P_{|\mathbf{K}|\setminus\{k_{i_0}\}}^{-1} \right)\\ \vdots&\ddots&\vdots&\ddots&\vdots\\ D_{x_{k_{0}}}\left( P_{\{i_{n}\}}P_{|\mathbf{K}|\setminus\{k_{i_0}\}}^{-1} \right)&\cdots& \widehat{D_{x_{k_{i_0}}}\left( P_{\{i_{n }\}}P_{|\mathbf{K}|\setminus\{k_{i_0}\}}^{-1} \right)} &\cdots&D_{x_{k_{n}}}\left( P_{\{i_{n }\}}P_{|\mathbf{K}|\setminus\{k_{i_0}\}}^{-1} \right)
\end{vmatrix}\\
&= (-1)^{i_0}\cdot  s(\mathbf{I}) \cdot s(\mathbf{K} \setminus\{k_{i_0}\})\cdot \det(D(P_{|\mathbf{I}|}P_{|\mathbf{K}|\setminus\{k_{i_0}\}}^{-1})).
\end{aligned}
$$
We now justify the first equality. Fix
\(j\in\{1,\ldots,n\}\). If \(i_j\in|\mathbf{K}|\), then the assumption
\(k_{i_0}\notin\{i_1,\ldots,i_n\}\) implies that
\(i_j\neq k_{i_0}\). Hence, for every
\(s\in\{0,1,\ldots,n\}\setminus\{i_0\}\),
\[
\begin{aligned}
&D_{x_{k_s}}
\left(P_{{i_j}} P_{|\mathbf{K}|}^{-1}\right)
\circ
\left(
P_{|\mathbf{K}|}
P^{-1}_{|\mathbf{K}|\setminus\{k_{i_0}\}}
\right)
=
\delta_{k_s,i_j}
=D_{x_{k_s}}
\left(
P_{\{i_j\}}
P^{-1}_{|\mathbf{K}|\setminus\{k_{i_0}\}}
\right),
\end{aligned}
\]
and
$
D_{x_{k_{i_0}}}
\left(P_{\{i_j\}}\circ P_{|\mathbf{K}|}^{-1}\right)=0.
$
If \(i_j\in\mathbb N\setminus|\mathbf{K}|\), the chain rule gives
\begin{eqnarray*}
 &&D_{x_{k_{s}}}\left( P_{\{i_{j}\}}  P_{|\mathbf{K}|\setminus\{k_{i_0}\}}^{-1}\right)\\
& =&D_{x_{k_{s}}}\left( P_{\{i_{j}\}}P_{|\mathbf{K}|}^{-1} \right)\circ (P_{|\mathbf{K}|} P_{|\mathbf{K}|\setminus\{k_{i_0}\}}^{-1})
+D_{x_{k_{i_0}}}\left( P_{\{i_{j}\}}P_{|\mathbf{K}|}^{-1} \right)\circ (P_{|\mathbf{K}|} P_{|\mathbf{K}|\setminus\{k_{i_0}\}}^{-1})
\cdot D_{x_{k_{s}}}\left( P_{\{k_{i_0}\}}P_{|\mathbf{K}|\setminus\{k_{i_0}\}}^{-1}\right)
\end{eqnarray*}
for every
\(s\in\{0,1,\ldots,n\}\setminus\{i_0\}\).
Accordingly, for each \(j\) such that \(i_j\notin|\mathbf{K}|\), multiply the
first row by
\[
-D_{x_{k_{i_0}}}
\left(P_{{i_j}} P_{|\mathbf{K}|}^{-1}\right)
\circ
\left(
P_{|\mathbf{K}|}
P^{-1}_{|\mathbf{K}|\setminus\{k_{i_0}\}}
\right)
\]
and add the resulting row to the \((j+1)\)-st row. This operation
replaces that row by the corresponding row in the second determinant
and makes its entry in the \((i_0+1)\)-st column equal to zero. Since
such row operations do not change the determinant, the first equality
follows.

The second equality is obtained by expanding the second determinant
along its \((i_0+1)\)-st column. The only nonzero entry in that column
is the \(1\) in the first row, whose cofactor has sign
\[
(-1)^{1+(i_0+1)}=(-1)^{i_0}.
\]
The last equality follows from the definitions of
\(s(\mathbf{I})\), \(s(\mathbf{K}\setminus\{k_{i_0}\})\), and the corresponding
coordinate Jacobian.

\textbf{Case 2.} $k_{i_0}\in\{i_1,\ldots,i_n\}$. Then there exists $k\in\{1,\ldots,n\}$ such that $k_{i_0}=i_k$, and we have
$$
\begin{aligned}
&\begin{vmatrix}
-D_{x_{k_{0}}}(P_{\{k_{i_0}\}}P^{-1}_{|\mathbf{K}|\setminus\{k_{i_0}\}})&\cdots& 1 &\cdots&-D_{x_{k_{n}}}(P_{\{k_{i_0}\}}P^{-1}_{|\mathbf{K}|\setminus\{k_{i_0}\}})\\
D_{x_{k_{0}}}\left( P_{\{i_{1}\}}P_{|\mathbf{K}|}^{-1} \right)&\cdots& D_{x_{k_{i_0}}}\left( P_{\{i_{1}\}}P_{|\mathbf{K}|}^{-1} \right) &\cdots&D_{x_{k_{n}}}\left( P_{\{i_{1}\}}P_{|\mathbf{K}|}^{-1} \right)\\ \vdots&\ddots&\vdots&\ddots&\vdots\\ D_{x_{k_{0}}}\left( P_{\{i_{n}\}}P_{|\mathbf{K}|}^{-1} \right)&\cdots& D_{x_{k_{i_0}}}\left( P_{\{i_{n }\}}P_{|\mathbf{K}|}^{-1} \right)&\cdots&D_{x_{k_{n}}}\left( P_{\{i_{n }\}}P_{|\mathbf{K}|}^{-1} \right)\end{vmatrix}\\
&=\begin{vmatrix}
-D_{x_{k_{0}}}(P_{\{k_{i_0}\}}P^{-1}_{|\mathbf{K}|\setminus\{k_{i_0}\}})&\cdots& 1 &\cdots&-D_{x_{k_{n}}}(P_{\{k_{i_0}\}}P^{-1}_{|\mathbf{K}|\setminus\{k_{i_0}\}})\\
D_{x_{k_{0}}}\left( P_{\{i_{1}\}}P_{|\mathbf{K}|}^{-1} \right)&\cdots& D_{x_{k_{i_0}}}\left( P_{\{i_{1}\}}P_{|\mathbf{K}|}^{-1} \right) &\cdots&D_{x_{k_{n}}}\left( P_{\{i_{1}\}}P_{|\mathbf{K}|}^{-1} \right)\\ \vdots&\ddots&\vdots&\ddots&\vdots\\
D_{x_{k_{0}}}\left( P_{\{i_{k}\}}P_{|\mathbf{K}|}^{-1} \right)&\cdots& D_{x_{k_{i_0}}}\left( P_{\{i_{k}\}}P_{|\mathbf{K}|}^{-1} \right) &\cdots&D_{x_{k_{n}}}\left( P_{\{i_{k}\}}P_{|\mathbf{K}|}^{-1} \right)\\ \vdots&\ddots&\vdots&\ddots&\vdots\\ D_{x_{k_{0}}}\left( P_{\{i_{n}\}}P_{|\mathbf{K}|}^{-1} \right)&\cdots& D_{x_{k_{i_0}}}\left( P_{\{i_{n }\}}P_{|\mathbf{K}|}^{-1} \right)&\cdots&D_{x_{k_{n}}}\left( P_{\{i_{n }\}}P_{|\mathbf{K}|}^{-1} \right)
\end{vmatrix}\\
&=\begin{vmatrix}
-D_{x_{k_{0}}}(P_{\{k_{i_0}\}}P^{-1}_{|\mathbf{K}|\setminus\{k_{i_0}\}})&\cdots& 1 &\cdots&-D_{x_{k_{n}}}(P_{\{k_{i_0}\}}P^{-1}_{|\mathbf{K}|\setminus\{k_{i_0}\}})\\
D_{x_{k_{0}}}\left( P_{\{i_{1}\}}P_{|\mathbf{K}|\setminus\{k_{i_0}\}}^{-1} \right)&\cdots& 0 &\cdots&D_{x_{k_{n}}}\left( P_{\{i_{1}\}}P_{|\mathbf{K}|\setminus\{k_{i_0}\}}^{-1} \right)\\ \vdots&\ddots&\vdots&\ddots&\vdots\\
0&\cdots& 1&\cdots&0\\ \vdots&\ddots&\vdots&\ddots&\vdots\\ D_{x_{k_{0}}}\left( P_{\{i_{n}\}}P_{|\mathbf{K}|\setminus\{k_{i_0}\}}^{-1} \right)&\cdots& 0&\cdots&D_{x_{k_{n}}}\left( P_{\{i_{n }\}}P_{|\mathbf{K}|\setminus\{k_{i_0}\}}^{-1} \right)
\end{vmatrix}\\
&=-\begin{vmatrix}
0&\cdots& 1 &\cdots&0\\
D_{x_{k_{0}}}\left( P_{\{i_{1}\}}P_{|\mathbf{K}|\setminus\{k_{i_0}\}}^{-1} \right)&\cdots& 0 &\cdots&D_{x_{k_{n}}}\left( P_{\{i_{1}\}}P_{|\mathbf{K}|\setminus\{k_{i_0}\}}^{-1} \right)\\ \vdots&\ddots&\vdots&\ddots&\vdots\\
-D_{x_{k_{0}}}(P_{\{k_{i_0}\}}P^{-1}_{|\mathbf{K}|\setminus\{k_{i_0}\}})&\cdots& 1&\cdots& -D_{x_{k_{n}}}(P_{\{k_{i_0}\}}P^{-1}_{|\mathbf{K}|\setminus\{k_{i_0}\}})\\ \vdots&\ddots&\vdots&\ddots&\vdots\\ D_{x_{k_{0}}}\left( P_{\{i_{n}\}}P_{|\mathbf{K}|\setminus\{k_{i_0}\}}^{-1} \right)&\cdots& 0&\cdots&D_{x_{k_{n}}}\left( P_{\{i_{n }\}}P_{|\mathbf{K}|\setminus\{k_{i_0}\}}^{-1} \right)
\end{vmatrix}\\
&=\begin{vmatrix}
0&\cdots& 1 &\cdots&0\\
D_{x_{k_{0}}}\left( P_{\{i_{1}\}}P_{|\mathbf{K}|\setminus\{k_{i_0}\}}^{-1} \right)&\cdots& 0 &\cdots&D_{x_{k_{n}}}\left( P_{\{i_{1}\}}P_{|\mathbf{K}|\setminus\{k_{i_0}\}}^{-1} \right)\\ \vdots&\ddots&\vdots&\ddots&\vdots\\
D_{x_{k_{0}}}(P_{\{k_{i_0}\}}P^{-1}_{|\mathbf{K}|\setminus\{k_{i_0}\}})&\cdots& -1&\cdots& D_{x_{k_{n}}}(P_{\{k_{i_0}\}}P^{-1}_{|\mathbf{K}|\setminus\{k_{i_0}\}})\\ \vdots&\ddots&\vdots&\ddots&\vdots\\ D_{x_{k_{0}}}\left( P_{\{i_{n}\}}P_{|\mathbf{K}|\setminus\{k_{i_0}\}}^{-1} \right)&\cdots& 0&\cdots&D_{x_{k_{n}}}\left( P_{\{i_{n }\}}P_{|\mathbf{K}|\setminus\{k_{i_0}\}}^{-1} \right)
\end{vmatrix}\\
&=(-1)^{i_0}\begin{vmatrix}
D_{x_{k_{0}}}\left( P_{\{i_{1}\}}P_{|\mathbf{K}|\setminus\{k_{i_0}\}}^{-1} \right)&\cdots& \widehat{D_{x_{k_{i_0}}}\left( P_{\{i_{1 }\}}P_{|\mathbf{K}|\setminus\{k_{i_0}\}}^{-1} \right)} &\cdots&D_{x_{k_{n}}}\left( P_{\{i_{1}\}}P_{|\mathbf{K}|\setminus\{k_{i_0}\}}^{-1} \right)\\ \vdots&\ddots&\vdots&\ddots&\vdots\\ D_{x_{k_{0}}}\left( P_{\{i_{n}\}}P_{|\mathbf{K}|\setminus\{k_{i_0}\}}^{-1} \right)&\cdots& \widehat{D_{x_{k_{i_0}}}\left( P_{\{i_{n }\}}P_{|\mathbf{K}|\setminus\{k_{i_0}\}}^{-1} \right)} &\cdots&D_{x_{k_{n}}}\left( P_{\{i_{n }\}}P_{|\mathbf{K}|\setminus\{k_{i_0}\}}^{-1} \right)
\end{vmatrix}\\
&=(-1)^{i_0}\cdot  s(\mathbf{I}) \cdot s(\mathbf{K}\setminus\{k_{i_0}\})\cdot\det(D(P_{|\mathbf{I}|}P_{|\mathbf{K}|\setminus\{k_{i_0}\}}^{-1})),
\end{aligned}
$$
where the second equality follows from the same argument as in Case 1; the third equality follows from interchanging the first row and the \(k+1\)-th row of the preceding determinant; the fourth equality follows from multiplying the $(k+1)$-th row by $-1$; and the fifth equality follows from expanding the determinant along its first row.

Consequently,
\begin{eqnarray}\label{20250121for4}
\begin{aligned}
&\sum_{i=0}^{n}(-1)^i\cdot\int_{V_{|\mathbf{K}|\setminus\{k_{i_0}\}}}\Phi^i\left(\textbf{x}_{|\mathbf{K}|\setminus\{k_{i_0}\}}+h(\textbf{x}_{|\mathbf{K}|\setminus\{k_{i_0}\}})\mathbf{e}_{k_{i_0}}\right)
\\
&\quad\times (D_{x_{k_i}}\rho )\left(\textbf{x}_{|\mathbf{K}|\setminus\{k_{i_0}\}}+h(\textbf{x}_{|\mathbf{K}|\setminus\{k_{i_0}\}})\mathbf{e}_{k_{i_0}}\right)
 \mathrm{d} \mu_{|\mathbf{K}|\setminus\{k_{i_0}\}}(\textbf{x}_{|\mathbf{K}|\setminus\{k_{i_0}\}})\\
& =(-1)^{i_0-1}\int_{V_{|\mathbf{K}|\setminus\{k_{i_0}\}}}  (g\circ P_{|\mathbf{K}|\setminus\{k_{i_0}\}}^{-1})\cdot  F_{\mathbb{N} \setminus (|\mathbf{K}|\setminus\{k_{i_0}\})}\circ(P_{\mathbb{N} \setminus (|\mathbf{K}|\setminus\{k_{i_0}\})}P_{|\mathbf{K}|\setminus\{k_{i_0}\}}^{-1})\\
&\qquad\qquad\times (-1)^{i_0}\cdot  s(\mathbf{I}) \cdot s(\mathbf{K} \setminus\{k_{i_0}\})\cdot\det(D(P_{|\mathbf{I}|}P_{|\mathbf{K}|\setminus\{k_{i_0}\}}^{-1}))
 \mathrm{d} \mu_{|\mathbf{K}|\setminus\{k_{i_0}\}}\\
& =-\int_{V_{|\mathbf{K}|\setminus\{k_{i_0}\}}}  (g\circ P_{|\mathbf{K}|\setminus\{k_{i_0}\}}^{-1})\cdot  \frac{ s(\mathbf{I}) \cdot s(\mathbf{K} \setminus\{k_{i_0}\})\cdot \det(D(P_{|\mathbf{I}|}P_{|\mathbf{K}|\setminus\{k_{i_0}\}}^{-1}))}{n_{|\mathbf{K}|\setminus\{k_{i_0}\}}\circ P_{|\mathbf{K}|\setminus\{k_{i_0}\}}^{-1}}\\
&\qquad\qquad\qquad\times   F_{\mathbb{N} \setminus (|\mathbf{K}|\setminus\{k_{i_0}\})}\circ(P_{\mathbb{N} \setminus (|\mathbf{K}|\setminus\{k_{i_0}\})}P_{|\mathbf{K}|\setminus\{k_{i_0}\}}^{-1})\cdot n_{|\mathbf{K}|\setminus\{k_{i_0}\}}\circ P_{|\mathbf{K}|\setminus\{k_{i_0}\}}^{-1}
 \mathrm{d} \mu_{|\mathbf{K}|\setminus\{k_{i_0}\}}\\
& =-\int_{ (\partial D)\cap U  }  g \cdot    n_{\mathbf{I}}^{\partial D}
 \mathrm{d} \mu_{\partial D}.
\end{aligned}
\end{eqnarray}
Here, the second equality uses
$
(-1)^{i_0-1}(-1)^{i_0}=-1,
$
while the last equality follows from the definition of
\(n_{\mathbf{I}}^{\partial D}\) and the local representation of the
measure \(\mu_{\partial D}\).

Since the closure of \(g^{-1}(\mathbb{R}\setminus\{0\})\) in \(\ell^2\) is compact, assumption~\eqref{20250120for2} implies that all the
integrals appearing in \eqref{20250121for4} and in the preceding
identities are finite.

Combining \eqref{20250121for3}, \eqref{20250121for5}, and
\eqref{20250121for4}, we obtain
\begin{eqnarray*}
\begin{aligned}
&\sum_{i=0}^{n}(-1)^i\cdot \int_{\{\rho<0\}}\delta_{k_{i}}\left( g\circ P_{|\mathbf{K}|}^{-1}  \cdot   F_{\mathbb{N} \setminus |\mathbf{K}|}\circ(P_{\mathbb{N} \setminus |\mathbf{K}|}P_{|\mathbf{K}|}^{-1}) \right)
\\
&\quad\times\begin{vmatrix}
D_{x_{k_{0}}}\left( P_{\{i_{1}\}}P_{|\mathbf{K}|}^{-1} \right)&\cdots&\widehat{D_{x_{k_{i}}}\left( P_{\{i_{1}\}}P_{|\mathbf{K}|}^{-1} \right)}&\cdots&D_{x_{k_{n}}}\left( P_{\{i_{1}\}}P_{|\mathbf{K}|}^{-1} \right)\\ \vdots&\ddots&\vdots&\ddots&\vdots\\ D_{x_{k_{0}}}\left( P_{\{i_{n}\}}P_{|\mathbf{K}|}^{-1} \right)&\cdots&\widehat{D_{x_{k_{i}}}\left( P_{\{i_{n }\}}P_{|\mathbf{K}|}^{-1} \right)}&\cdots&D_{x_{k_{n}}}\left( P_{\{i_{n }\}}P_{|\mathbf{K}|}^{-1} \right)\end{vmatrix}
 \mathrm{d} \mu_{|\mathbf{K}|}\\
&=-\int_{ (\partial D)\cap U  }  g \cdot    n_{\mathbf{I}}^{\partial D}
 \mathrm{d} \mu_{\partial D},
\end{aligned}
\end{eqnarray*}
Recalling \eqref{20250121for7} and \eqref{20250121for6}, this identity
can be rewritten as
\[
-\sum_{i=1}^{\infty}
\int_{D^\circ\cap U}
\delta_i g(\mathbf{x})\,
n_{(i,\mathbf{I})}^{D^\circ}(\mathbf{x})\,
\mathrm{d}\mu_{D^\circ}(\mathbf{x})
=\int_{(\partial D)\cap U}
g\,n_{\mathbf{I}}^{\partial D}\,
\mathrm{d}\mu_{\partial D}.
\]

Case~\((II)\) follows by the same argument as in case~\((I)\). In that
case, the required conclusion is already obtained at
\eqref{20250121for3}. This completes the proof of
Proposition~\ref{20241102prop1}.

\end{proof}

\section{Global Version of Gauss--Green-type Theorem}

We begin with the notion of a partition of unity on \(\ell^2\).

\begin{definition}\label{230415def1}
A sequence of real-valued functions
\[
\{\gamma_n\}_{n=1}^{\infty}
\subset C^{\infty}_{\ell^2}(\ell^2;[0,1])
\]
is called a \(C^{\infty}\)-partition of unity on \(\ell^2\) if the following
conditions are satisfied:
\begin{itemize}
\item[\((1)\)]
The family \(\{\gamma_n\}_{n=1}^{\infty}\) is locally finite; that is,
for every \(\mathbf{x}\in\ell^2\), there exists an open neighborhood
\(U\) of \(\mathbf{x}\) such that only finitely many of the functions
\(\gamma_n\) are nonzero on \(U\).

\item[\((2)\)]
For every \(\mathbf{x}\in\ell^2\),
\[
\sum_{n=1}^{\infty}\gamma_n(\mathbf{x})=1.
\]
\end{itemize}
\end{definition}

The construction of \(C^{\infty}\)-partitions of unity on \(\ell^2\) relies
on the following elementary separation property.

\begin{lemma}\label{230708lem1}
Let \(\mathbf{x}\in\ell^2\), and let \(r_1,r_2>0\) satisfy
\(r_1<r_2\). Then there exists a function
\[
f\in C^{\infty}_{\ell^2}(\ell^2;[0,1])
\]
such that
$
f\equiv1\,\text{on }B_{r_1}(\mathbf{x})
$
and
$
f\equiv0\,\text{on }\ell^2\setminus B_{r_2}(\mathbf{x}).
$
\end{lemma}

\begin{proof}
Choose \(h\in C^\infty(\mathbb R;[0,1])\) such that
\[
h(t)=1\quad\text{for }t\leqslant r_1^2,
\qquad
h(t)=0\quad\text{for }t\geqslant r_2^2.
\]
The map
$
\mathbf{y}\longmapsto
\|\mathbf{y}-\mathbf{x}\|^{2}
$
belongs to \(C^{\infty}_{\ell^2}(\ell^2;\mathbb R)\). Define
\[
f(\mathbf{y})
\triangleq
h\!\left(\|\mathbf{y}-\mathbf{x}\|^{2}\right),
\qquad
\mathbf{y}\in\ell^2.
\]
By the chain rule,
$
f\in C^{\infty}_{\ell^2}(\ell^2;[0,1]).
$
Moreover, if \(\mathbf{y}\in B_{r_1}(\mathbf{x})\), then
$
\|\mathbf{y}-\mathbf{x}\|^{2}<r_1^2,
$
and hence \(f(\mathbf{y})=1\). If
\(\mathbf{y}\notin B_{r_2}(\mathbf{x})\), then
$
\|\mathbf{y}-\mathbf{x}\|^{2}\geqslant r_2^2,
$
and therefore \(f(\mathbf{y})=0\). This proves the Lemma \ref{230708lem1}.
\end{proof}

The following lemma is presumably standard, although we have been
unable to locate a precise reference.

\begin{lemma}\label{second lemma for PU l^2}
Let \(\{U_{\alpha}:\alpha\in A\}\) be an open cover of \(\ell^2\).
Then there exist four countable, locally finite open covers
\[
\{V_i^1\}_{i=1}^\infty,\qquad
\{V_i^2\}_{i=1}^\infty,\qquad
\{V_i^3\}_{i=1}^\infty,\qquad
\{V_i^4\}_{i=1}^\infty
\]
of \(\ell^2\), together with a sequence
$
\{g_i\}_{i=1}^\infty
\subset C^{\infty}_{\ell^2}(\ell^2;[0,1]),
$
such that the following properties hold:
\begin{itemize}
\item[\((1)\)]
For every \(i\in\mathbb N\),
\[
\overline{V_i^1}
\subset V_i^2
\subset \overline{V_i^2}
\subset V_i^3
\subset \overline{V_i^3}
\subset V_i^4,
\]
and
$
g_i\equiv1\,\text{on }\overline{V_i^1},
\,
\operatorname{supp}g_i\subset V_i^3.
$

\item[\((2)\)]
The cover \(\{V_i^4\}_{i=1}^\infty\) refines
\(\{U_{\alpha}:\alpha\in A\}\).
\end{itemize}
\end{lemma}

\begin{proof}
For each \(\mathbf y\in\ell^2\), choose
\(\alpha(\mathbf y)\in A\) such that
\(\mathbf y\in U_{\alpha(\mathbf y)}\). Since
\(U_{\alpha(\mathbf y)}\) is open, there exist radii
$
0<r_1(\mathbf y)<r_2(\mathbf y)
$
such that
$
\overline{B}_{r_2(\mathbf y)}(\mathbf y)
\subset U_{\alpha(\mathbf y)}.
$
By Lemma~\ref{230708lem1}, there exists
$
\phi^{\mathbf y}\in C^{\infty}_{\ell^2}(\ell^2;[0,1])
$
such that
$
\phi^{\mathbf y}\equiv1
\,\text{on }\,B_{r_1(\mathbf y)}(\mathbf y)
$
and
$\phi^{\mathbf y}\equiv0
\,\text{on }\,\ell^2\setminus
B_{r_2(\mathbf y)}(\mathbf y).
$
Consequently,
$
\phi^{\mathbf y}(\mathbf y)=1,
\,
\operatorname{supp}\phi^{\mathbf y}
\subset
\overline{B}_{r_2(\mathbf y)}(\mathbf y)
\subset U_{\alpha(\mathbf y)}.
$

Define
\[
A_{\mathbf y}
\triangleq
\left\{
\mathbf x\in\ell^2:
\phi^{\mathbf y}(\mathbf x)>\frac12
\right\}.
\]
Then \(\{A_{\mathbf y}:\mathbf y\in\ell^2\}\) is an open cover of
\(\ell^2\). Since \(\ell^2\) is Lindel\"{o}f, there exists a sequence
\(\{\mathbf y_i\}_{i=1}^\infty\subset\ell^2\) such that
$
\ell^2=\bigcup_{i=1}^\infty A_{\mathbf y_i}.
$
For brevity, write
$
\phi_i\triangleq\phi^{\mathbf y_i}.
$

Choose \(f_1\in C^\infty(\mathbb R;[0,1])\) such that
\[
f_1(t)=0\quad\text{for }t\leqslant 0,
\qquad
f_1(t)=1\quad\text{for }t\geqslant\frac12.
\]
For every \(j\geqslant 2\), choose functions
\(a_j,b_j\in C^\infty(\mathbb R;[0,1])\) satisfying
$
a_j(t)=0\,\text{for }\,
t\leqslant\frac12-\frac1j,
\,
a_j(t)=1\,\text{for }t\geqslant\frac12,
$
and
$
b_j(t)=1\,\text{for }\,
t\leqslant\frac12+\frac1j,
\,
b_j(t)=0\,\text{for }
\,t\geqslant\frac12+\frac2j.
$
Define
\[
f_j(t_1,\ldots,t_j)
\triangleq
a_j(t_j)\prod_{q=1}^{j-1}b_j(t_q),
\qquad j\geqslant 2.
\]
Thus,
$
f_j(t_1,\ldots,t_j)=1
$
whenever
$t_j\geqslant \frac12
\,\text{and}\,
t_q\leqslant\frac12+\frac1j
\,\text{for every }q<j,
$
whereas
$
f_j(t_1,\ldots,t_j)=0
$
whenever
$
t_j\leqslant\frac12-\frac1j
$
or
$
t_q\geqslant\frac12+\frac2j
\,\text{for some }q<j.
$

Now define
\[
\psi_1(\mathbf x)
\triangleq f_1(\phi_1(\mathbf x))
\]
and, for \(j\geqslant 2\),
\[
\psi_j(\mathbf x)
\triangleq
f_j\bigl(\phi_1(\mathbf x),\ldots,\phi_j(\mathbf x)\bigr),
\qquad
\mathbf x\in\ell^2.
\]
Then
$
\psi_j\in C^{\infty}_{\ell^2}(\ell^2;[0,1])
\,\text{for every }j\in\mathbb N.
$

For \(i\in\mathbb N\) and \(k\in\{1,2,3,4\}\), set
\[
V_i^k
\triangleq
\left\{
\mathbf x\in\ell^2:
\psi_i(\mathbf x)>\frac{4-k}{4}
\right\}.
\]
Each \(V_i^k\) is open. Moreover, continuity of \(\psi_i\) gives
\[
\overline{V_i^1}
\subset
\left\{\psi_i\geqslant\frac34\right\}
\subset V_i^2,
\qquad
\overline{V_i^2}
\subset
\left\{\psi_i\geqslant\frac12\right\}
\subset V_i^3,
\qquad
\text{and}
\qquad
\overline{V_i^3}
\subset
\left\{\psi_i\geqslant\frac14\right\}
\subset V_i^4.
\]
Hence
$
\overline{V_i^1}
\subset V_i^2
\subset\overline{V_i^2}
\subset V_i^3
\subset\overline{V_i^3}
\subset V_i^4.
$

We next verify the covering property. Given
\(\mathbf x\in\ell^2\), let \(i(\mathbf x)\) be the smallest positive
integer \(i\) such that
\[
\phi_i(\mathbf x)>\frac12.
\]
If \(i(\mathbf x)=1\), then
\(\psi_1(\mathbf x)=1\). If \(i(\mathbf x)\geqslant 2\), then
$
\phi_q(\mathbf x)\leqslant\frac12
\quad\text{for every }q<i(\mathbf x),
$
and therefore the defining properties of \(f_{i(\mathbf x)}\) imply
that
$
\psi_{i(\mathbf x)}(\mathbf x)=1.
$
Thus,
$
\mathbf x\in V_{i(\mathbf x)}^1,
$
and hence \(\{V_i^1\}_{i=1}^\infty\) covers \(\ell^2\). Since
$
V_i^1\subset V_i^2\subset V_i^3\subset V_i^4,
$
each of the four families is an open cover of \(\ell^2\).

We now show that \(\{V_i^4\}_{i=1}^\infty\) refines the original
cover. If \(i=1\), then
\[
V_1^4=\{\psi_1>0\}\subset\{\phi_1>0\}.
\]
If \(i\geqslant 2\), then \(\psi_i(\mathbf x)>0\) implies
$
\phi_i(\mathbf x)>\frac12-\frac1i\geqslant 0.
$
In either case,
$
V_i^4
\subset\{\phi_i>0\}
\subset\operatorname{supp}\phi_i
\subset U_{\alpha(\mathbf y_i)}.
$
Therefore, \(\{V_i^4\}_{i=1}^\infty\) is a refinement of
\(\{U_{\alpha}:\alpha\in A\}\).

It remains to establish local finiteness. Fix
\(\mathbf x\in\ell^2\). Since the sets \(\{A_{\mathbf y_i}\}_{i=1}^{\infty}\) cover
\(\ell^2\), there exists \(m\in\mathbb N\) such that
$
\phi_m(\mathbf x)>\frac12.
$
By continuity, there exist an open neighborhood
\(\mathcal N_{\mathbf x}\) of \(\mathbf x\) and a number
\(a_{\mathbf x}>\frac12\) such that
$
\phi_m(\mathbf x')\geqslant a_{\mathbf x}
\,
\text{for every }\mathbf x'\in\mathcal N_{\mathbf x}.
$
Choose \(N>m\) sufficiently large that
$
\frac{2}{N}<a_{\mathbf x}-\frac12.
$
For every \(j\geqslant N\), we then have \(m<j\) and
$
\phi_m(\mathbf x')
\geqslant a_{\mathbf x}
>
\frac12+\frac2j
\,
\text{for every }\mathbf x'\in\mathcal N_{\mathbf x}.
$
The defining properties of \(f_j\) therefore imply that
$
\psi_j(\mathbf x')=0
\,
\text{for every }\mathbf x'\in\mathcal N_{\mathbf x}.
$
Consequently,
$
\mathcal N_{\mathbf x}\cap V_j^4=\varnothing
\,\text{for every }\,j\geqslant N.
$
Thus, \(\{V_i^4\}_{i=1}^\infty\) is locally finite. Since
\(V_i^k\subset V_i^4\) for \(k=1,2,3\), each of the other three
families is also locally finite.

Finally, choose \(h\in C^\infty(\mathbb R;[0,1])\) such that
$
h(t)=0\,\text{for }t\leqslant\frac12,
\,
h(t)=1\,\text{for }t\geqslant\frac34,
$
and define
\[
g_i(\mathbf x)\triangleq h(\psi_i(\mathbf x)),
\qquad
i\in\mathbb N,\quad\mathbf x\in\ell^2.
\]
Then
$
g_i\in C^{\infty}_{\ell^2}(\ell^2;[0,1]).
$
Since
$
\overline{V_i^1}
\subset\left\{\psi_i\geqslant\frac34\right\},
$
we have
$
g_i\equiv1\,\text{on }\overline{V_i^1}.
$
Furthermore,
$
\{g_i\neq0\}
\subset\left\{\psi_i>\frac12\right\}
=V_i^2,
$
so that
$
\operatorname{supp}g_i
\subset\overline{V_i^2}
\subset V_i^3.
$
This completes the proof of Lemma \ref{second lemma for PU l^2}.
\end{proof}

We are now ready to prove the existence of \(C^{\infty}\)-partitions of unity
on \(\ell^2\).

\begin{proposition}\label{prop:H-C1-partition-of-unity}
Let \(\{U_{\alpha}:\alpha\in  A\}\) be an open cover of
\(\ell^2\). Then there exists a \(C^{\infty}\)-partition of unity
$
\{\gamma_n\}_{n=1}^{\infty}
$
on \(\ell^2\) subordinate to \(\{U_{\alpha}:\alpha\in  A\}\);
that is, for every \(n\in\mathbb N\), there exists
\(\alpha(n)\in\mathscr A\) such that
$
\operatorname{supp}\gamma_n\subset U_{\alpha(n)}.
$
\end{proposition}

\begin{proof}
Let
$
\{V_i^j\}_{i=1}^{\infty},
\, j\in\{1,2,3,4\},
$
and
$
\{g_i\}_{i=1}^{\infty}
\subset C^{\infty}_{\ell^2}(\ell^2;[0,1])
$
be the open covers and functions provided by
Lemma~\ref{second lemma for PU l^2}. Define
\[
\gamma_1\triangleq g_1
\]
and, for \(i\geqslant 2\),
\[
\gamma_i
\triangleq
g_i\prod_{j=1}^{i-1}(1-g_j).
\]
Since each product is finite, we have
$
\gamma_i\in C^{\infty}_{\ell^2}(\ell^2;[0,1])
\,\text{for every }i\in\mathbb N.
$
Moreover,
\[
\operatorname{supp}\gamma_i
\subset
\operatorname{supp}g_i
\subset V_i^3
\subset V_i^4.
\]
By Lemma~\ref{second lemma for PU l^2}, the cover
\(\{V_i^4\}_{i=1}^{\infty}\) refines
\(\{U_{\alpha}:\alpha\in A\}\). Hence, for every
\(i\in\mathbb N\), there exists \(\alpha(i)\in  A\) such that
\[
\operatorname{supp}\gamma_i
\subset V_i^4
\subset U_{\alpha(i)}.
\]
Thus, \(\{\gamma_i\}_{i=1}^{\infty}\) is subordinate to the given
cover.

We next verify local finiteness. Since
\(\{V_i^3\}_{i=1}^{\infty}\) is locally finite, for every
\(\mathbf x\in\ell^2\), there exists an open neighborhood
\(\mathcal N_{\mathbf x}\) of \(\mathbf x\) that intersects only
finitely many sets \(V_i^3\). Because
\[
\{\gamma_i\neq0\}
\subset\operatorname{supp}\gamma_i
\subset V_i^3,
\]
only finitely many functions \(\gamma_i\) are nonzero on
\(\mathcal N_{\mathbf x}\). Therefore,
\(\{\gamma_i\}_{i=1}^{\infty}\) is locally finite.

It remains to prove that the functions sum to \(1\). For every
\(N\in\mathbb N\), the definition of \(\gamma_i\) gives the
telescoping identity
$
\sum_{i=1}^{N}\gamma_i
=1-\prod_{i=1}^{N}(1-g_i).
$
Fix \(\mathbf x\in\ell^2\). Since
\(\{V_i^1\}_{i=1}^{\infty}\) covers \(\ell^2\), there exists
\(m\in\mathbb N\) such that
$\mathbf x\in V_m^1.$
By Lemma~\ref{second lemma for PU l^2},
$
g_m(\mathbf x)=1.
$
Consequently, for every \(N\geqslant m\),
$
\prod_{i=1}^{N}\bigl(1-g_i(\mathbf x)\bigr)=0,
$
and hence
$
\sum_{i=1}^{N}\gamma_i(\mathbf x)=1.
$
It follows that
$
\sum_{i=1}^{\infty}\gamma_i(\mathbf x)=1.
$
Since \(\mathbf x\) was arbitrary,
\[
\sum_{i=1}^{\infty}\gamma_i\equiv1.
\]
Therefore, \(\{\gamma_i\}_{i=1}^{\infty}\) is a \(C^{\infty}\)-partition of
unity on \(\ell^2\) subordinate to
\(\{U_{\alpha}:\alpha\in A\}\).  The proof of Proposition \ref{prop:H-C1-partition-of-unity} is completed.
\end{proof}

\begin{remark}
The preceding proof extends to other regularity classes. More precisely,
the condition
\[
\{\gamma_n\}_{n=1}^{\infty}
\subset C^{\infty}_{\ell^2}(\ell^2;[0,1])
\]
in Definition~\ref{230415def1} may be replaced by any class of
real-valued functions on \(\ell^2\) satisfying the following two
properties:
\begin{itemize}
\item[\((1)\)]
the separation property stated in Lemma~\ref{230708lem1} holds within
that class;

\item[\((2)\)]
the class is closed under composition with smooth functions of finitely
many real variables.
\end{itemize}
Indeed, the proof uses the prescribed regularity only to construct the
functions \(\phi_i\), \(\psi_i\), and \(g_i\), and to form the finite
products
\[
\gamma_i
=g_i\prod_{j=1}^{i-1}(1-g_j).
\]
Consequently, under the two properties above, a partition of unity
belonging to the corresponding regularity class exists.
\end{remark}

\begin{theorem}[Global Gauss--Green-type theorem]
\label{20241015thm1}
Let \(D\) be a region with boundary in a smooth surface \(S\), and
suppose that \(\partial D\) is \(\sigma\)-finite. Let
\(\mathbf{I}\in S_{\mathbb N}^F\) satisfy the assumptions of
Proposition~\ref{20241102prop1}. Let \(O\) be an open subset of
\(\ell^2\) containing \(D\), and let \(f\in B_{\ell^2}(O)\).

Suppose that \(f\) is Fr\'{e}chet differentiable along the
\(H_{\mathbb N}^{2}\)-direction at every point of $O$, that \(f\in C_{\mathcal F}^1(O)\), and
that
\begin{eqnarray}
&&f \sqrt{ \sum\limits_{i=1}^{\infty} \frac{|n_{\left( i,\mathbf{I} \right)}^{D^{\circ}} |^2}{a_i^2}}\in L^1(D^{\circ},\mu_{D^{\circ}}),\quad f|_{ \partial D}\in L^1( \partial D ,\mu_{\partial D}),\quad f|_{D^{\circ}}\in L^1(D^{\circ} ,\mu_{ D^{\circ}}),\label{20250207for1}\\
&& \sum\limits_{i=1}^{\infty} \left( |D_{x_i}f  |+\left|\frac{x_i}{a_i^2} f \right| \right)   |n_{\left( i,\mathbf{I} \right)}^{D^{\circ}} |\in L^1(D^{\circ} ,\mu_{ D^{\circ}}),\,  \left( |D_{x_i}f  |+\left|\frac{x_i}{a_i^2}  f \right| \right)\bigg|_{D^{\circ}}\in L^1(D^{\circ} ,\mu_{ D^{\circ}}),\, \text{for every }\,i\in\mathbb{N}.\nonumber
\end{eqnarray}
Then
\begin{equation}\label{20250209for2}
-\sum_{i=1}^{\infty}
\int_{D^\circ}
\delta_i f\,
n_{(i,\mathbf{I})}^{D^\circ}\,
\mathrm d\mu_{D^\circ}
=\int_{\partial D}
f\,n_{\mathbf{I}}^{\partial D}\,
\mathrm d\mu_{\partial D}.
\end{equation}
\end{theorem}

\begin{proof}
For brevity, write
\[
N_i\triangleq n_{(i,\mathbf{I})}^{D^\circ},
\qquad
N_{\partial}\triangleq n_{\mathbf{I}}^{\partial D}.
\]

We first consider the case in which
\[
K\triangleq\operatorname{supp}f
\]
is compact in \(\ell^2\). By
Proposition~\ref{20241102prop1}, for every
\(\mathbf x^0\in\ell^2\), there exists an open neighborhood
\(U_{\mathbf x^0}\) of \(\mathbf x^0\) such that the local
Gauss--Green formula holds for every admissible function supported in
\(U_{\mathbf x^0}\).

The family
$
\{U_{\mathbf x^0}:\mathbf x^0\in\ell^2\}
$
is an open cover of \(\ell^2\). By
Proposition~\ref{prop:H-C1-partition-of-unity}, there exists a
\(C^{\infty}\)-partition of unity
$
\{\gamma_k\}_{k=1}^{\infty}
$
subordinate to this cover. Thus, for each \(k\in\mathbb N\), there
exists \(\mathbf x_k\in\ell^2\) such that
\[
\operatorname{supp}\gamma_k\subset U_{\mathbf x_k}.
\]
Consequently, the local formula gives
\[
-\sum_{i=1}^{\infty}
\int_{D^\circ}
\delta_i(f\gamma_k)\,N_i\,
\mathrm d\mu_{D^\circ}
=\int_{\partial D}
f\gamma_k\,N_{\partial}\,
\mathrm d\mu_{\partial D}.
\]

Since the family \(\{\operatorname{supp}\gamma_k\}_{k=1}^{\infty}\)
is locally finite and \(K\) is compact, only finitely many of these
supports intersect \(K\). Hence there exists a finite set
\(J\subset\mathbb N\) such that
\[
f=\sum_{k\in J}f\gamma_k.
\]
Summing the preceding local identities over \(k\in J\), we obtain
\[
\begin{aligned}
-\sum_{i=1}^{\infty}
\int_{D^\circ}
\delta_i f\,N_i\,
\mathrm d\mu_{D^\circ}
&=
-\sum_{k\in J}
\sum_{i=1}^{\infty}
\int_{D^\circ}
\delta_i(f\gamma_k)\,N_i\,
\mathrm d\mu_{D^\circ}
\\
&=
\sum_{k\in J}
\int_{\partial D}
f\gamma_k\,N_{\partial}\,
\mathrm d\mu_{\partial D}
\\
&=
\int_{\partial D}
f\,N_{\partial}\,
\mathrm d\mu_{\partial D}.
\end{aligned}
\]
Thus, \eqref{20250209for2} holds whenever
\(\operatorname{supp}f\) is compact.

We now consider the general case. Let \(\{X_k\}_{k=1}^{\infty}\) be
the cutoff sequence provided by Theorem~\ref{230215Th1} and
Proposition~\ref{20241013thm1}. In particular, for every $k\in\mathbb{N}$,
$
D\cap \supp X_k
$
is a compact subset of $\ell^2$, and
$$
D\cap \supp X_k\subset D\subset O.
$$
Thus, by applying the preceding compact-support argument to \(fX_k\), we obtain
\[
-\sum_{i=1}^{\infty}
\int_{D^\circ}
\delta_i(fX_k)\,N_i\,
\mathrm d\mu_{D^\circ}
=\int_{\partial D}
fX_k\,N_{\partial}\,
\mathrm d\mu_{\partial D}.
\]
Using the product rule
\[
\delta_i(fX_k)
=X_k\delta_i f
+
fD_{x_i}X_k,
\]
we obtain
\begin{equation}\label{eq:cutoff-GG}
\begin{aligned}
&-\sum_{i=1}^{\infty}
\int_{D^\circ}
X_k\,\delta_i f\,N_i\,
\mathrm d\mu_{D^\circ}
-\sum_{i=1}^{\infty}
\int_{D^\circ}
f\,D_{x_i}X_k\,N_i\,
\mathrm d\mu_{D^\circ}
=\int_{\partial D}
fX_k\,N_{\partial}\,
\mathrm d\mu_{\partial D}.
\end{aligned}
\end{equation}

By the Cauchy--Schwarz inequality,
\[
\begin{aligned}
\sum_{i=1}^{\infty}
|f|\cdot|D_{x_i}X_k|\cdot|N_i|
&\leqslant
|f|
\left(
\sum_{i=1}^{\infty}
a_i^2|D_{x_i}X_k|^2
\right)^{1/2}
\left(
\sum_{i=1}^{\infty}
\frac{|N_i|^2}{a_i^2}
\right)^{1/2}.
\end{aligned}
\]
The properties of the cutoff sequence established in
Theorem~\ref{230215Th1} and
Proposition~\ref{20241013thm1}, together with
\eqref{20250207for1} and the dominated convergence theorem, imply that
\[
\lim_{k\to\infty}
\sum_{i=1}^{\infty}
\int_{D^\circ}
f\,D_{x_i}X_k\,N_i\,
\mathrm d\mu_{D^\circ}
=0.
\]

Moreover, since
$
|\delta_i f|
\leqslant
|D_{x_i}f|
+
\left|\frac{x_i}{a_i^2}f\right|\,\,\text{for every }i\in\mathbb{N},
$
the fourth condition in \eqref{20250207for1} yields
\[
\sum_{i=1}^{\infty}
|\delta_i f|\cdot |N_i|
\in L^1(D^\circ,\mu_{D^\circ}).
\]
Because \(0\leqslant X_k\leqslant 1\) and \(X_k\to1\) $\mu_{D^{\circ}}$-almost everywhere, another application
of the dominated convergence theorem gives
\[
\lim_{k\to\infty}
\sum_{i=1}^{\infty}
\int_{D^\circ}
X_k\,\delta_i f\,N_i\,
\mathrm d\mu_{D^\circ}
=\sum_{i=1}^{\infty}
\int_{D^\circ}
\delta_i f\,N_i\,
\mathrm d\mu_{D^\circ}.
\]

Similarly, using the integrability of \(f|_{\partial D}\), \(|N_{\partial}|\leqslant 1\),\(\,0\leqslant X_k\leqslant 1\), and \(X_k\to1\) $\mu_{\partial D}$-almost everywhere, we obtain
\[
\lim_{k\to\infty}
\int_{\partial D}
fX_k\,N_{\partial}\,
\mathrm d\mu_{\partial D}
=\int_{\partial D}
f\,N_{\partial}\,
\mathrm d\mu_{\partial D}.
\]
Letting \(k\to\infty\) in \eqref{eq:cutoff-GG} therefore yields
\[
-\sum_{i=1}^{\infty}
\int_{D^\circ}
\delta_i f\,
n_{(i,\mathbf{I})}^{D^\circ}\,
\mathrm d\mu_{D^\circ}
=\int_{\partial D}
f\,n_{\mathbf{I}}^{\partial D}\,
\mathrm d\mu_{\partial D}.
\]
This proves \eqref{20250209for2} and completes the proof of Theorem \ref{20241015thm1}.
\end{proof}

\begin{corollary}\label{20260801corollary1}
The conclusion of Theorem~\ref{20241015thm1} remains valid for complex-valued
functions satisfying the corresponding assumptions.
\end{corollary}

\begin{proof}
Apply Theorem~\ref{20241015thm1} separately to the real and imaginary parts.
\end{proof}

\begin{remark}\label{20241126rem1}
Note that the corresponding unit normal vector fields appear on both sides of \eqref{20250209for2} in Theorem~\ref{20241015thm1}. Formulas of this type also arise in the finite-dimensional setting; see, for example, \cite[Theorem~16.34, p.~427]{Lee}.

It also follows from the preceding proof that, for surface with finite codimension, the assumption \(f\in C_{\mathcal F}^1(O)\) in Theorem~\ref{20241015thm1} is sufficient, and the additional assumption that \(f\) is Fr\'{e}chet differentiable along the
\(H_{\mathbb N}^{2}\)-direction at every point of $O$ is redundant. In particular, in the codimension-one case considered by Goodman in \cite[Theorem~2, p.~421]{Goo}, the regularity assumptions imposed here are weaker than the corresponding assumptions in Goodman's result.
\end{remark}

\begin{remark}
In the classical Gauss--Green theorem, the assumption that \(f\) is Borel measurable on \(U\) is redundant, since it follows automatically from the usual regularity assumptions. Whether an analogous conclusion holds in infinite-dimensional settings appears to be a substantially more difficult question. The discussion preceding Corollary~\ref{230720cor1} is intended to shed some light on this issue. One of the main difficulties is that several distinct and natural topologies arise in infinite-dimensional spaces, and the corresponding notions of continuity and measurability need not coincide.

In the main results of this paper, such as Theorem~\ref{20241015thm1}, we integrate certain functions with respect to Borel measures. Consequently, Borel measurability must be imposed in order for these integrals to be well defined. On the other hand, \(\mathcal F\)-continuity is the relevant regularity assumption in our main theorem. Corollary~\ref{230720cor1} shows that the condition
\[
f\in C_{\mathcal{F}}^1(O)
\]
does not, by itself, imply that \(f\) is Borel measurable. It remains unclear whether Borel measurability follows if one additionally assumes that \(f\) is Fr\'{e}chet differentiable along the
\(H_{\mathbb N}^{2}\)-direction at every point of $O$.

Combining this observation with Remark~\ref{20241126rem1}, we conclude that, for surfaces of $\ell^2$ with any codimension, it remains unclear whether the Borel measurability assumption in Theorem~\ref{20241015thm1} follows from the remaining regularity assumptions; therefore, this assumption cannot simply be omitted.
\end{remark}

The following example illustrates a surface satisfying the assumptions of Theorem~\ref{20241015thm1}.

\begin{example}
Let
\[
I\triangleq\{2n-1:n\in\mathbb N\},
\qquad
J\triangleq\mathbb N\setminus I.
\]
Choose
\[
\mathbf{x}_{J}^0
=\sum_{j\in J}x_j^0\mathbf e_j,
\qquad
\Delta\mathbf{x}_J^0
=\sum_{j\in J}\Delta x_j^0\mathbf e_j
\in P_J\ell^2
\]
such that
\[
\sum_{j\in J}
\left|
\ln\frac{1}{\sqrt{2\pi},a_j}
-\frac{(x_j^0)^2}{2a_j^2}
\right|
<\infty,
\qquad
\sum_{j\in J}\frac{|\Delta x_j^0|^2}{a_j^4}<\infty,
\qquad
\Delta x_j^0\neq 0
\quad\text{for all }j\in J.
\]
Define
\begin{eqnarray*}
S&\triangleq&\left\{\mathbf{x}_I+\mathbf{x}^0_J+x_1\cdot\Delta\mathbf{x}^0_J:\mathbf{x}_I=\sum\limits_{i\in I}x_i\mathbf{e}_i\in P_I\ell^2\right\},
\\
 D&\triangleq &\{\mathbf{x}\in S:\langle\mathbf{x},\mathbf{e}_1\rangle\leqslant 0\},\\
f(\mathbf{x}_I)&\triangleq&\mathbf{x}^0_J+x_1\cdot\Delta\mathbf{x}^0_J,\qquad  \mathbf{x}_I=\sum\limits_{i\in I}x_i\mathbf{e}_i\in U_{I}\triangleq P_I\ell^2.
\end{eqnarray*}
Then \(S\) is a surface in \(\ell^2\) of codimension \(\Gamma_J\), and
\(\{(S,I,f)\}\) is a family of local coordinate triples of \(S\).
By the definitions of \(n_I\) and \(F_J\), we have
\[
n_I(\mathbf{x})
=\left(
1+\sum_{j\in J}|\Delta x_j^0|^2
\right)^{1/2}
=\left(
1+\|\Delta\mathbf{x}_J^0\|^2
\right)^{1/2},
\]
for every
$
\mathbf{x}
=\mathbf{x}_I+\mathbf{x}_J^0
+x_1\Delta\mathbf{x}_J^0
\in S,
$
and
\begin{eqnarray*}
F_{J}(\mathbf{x}^0_J+x_1\cdot\Delta\mathbf{x}^0_J)
&=&\prod_{j\in J}\frac{1}{\sqrt{2\pi a_j^2}}e^{-\frac{(x_j^0+x_1\cdot (\Delta x_j^0))^2}{2a_j^2}}\\
&=&e^{\sum\limits_{j\in J}\left(\ln \frac{1}{\sqrt{2\pi} a_j}-\frac{(x_j^0)^2}{2a_j^2}\right) -x_1\cdot \left(\sum\limits_{j\in J}\frac{x_j^0\cdot (\Delta x_j^0)}{a_j^2}\right)-x_1^2\cdot \left(\sum\limits_{j\in J}\frac{(\Delta x_j^0)^2}{2a_j^2}\right)}.
\end{eqnarray*}
The assumptions above imply that
\[
\sum_{j\in J}
\frac{|x_j^0\Delta x_j^0|}{a_j^2}<\infty
\qquad\text{and}\qquad
\sum_{j\in J}
\frac{|\Delta x_j^0|^2}{a_j^2}<\infty.
\]
Consequently, for every bounded subset \(O\subset P_I\ell^2\),
\[
\sup_{\mathbf{x}_I\in O}
n_I\bigl(
\mathbf{x}_I+\mathbf{x}_J^0
+x_1\Delta\mathbf{x}_J^0
\bigr)
F_{\mathbb N\setminus I}\bigl(
\mathbf{x}_J^0+x_1\Delta\mathbf{x}_J^0
\bigr)
<\infty.
\]
Since there is only one local coordinate triple, \(S\) is an oriented \(\sigma\)-finite surface.

For \(j\in I\) and
$
\mathbf{x}_I^0
=\sum_{i\in I}x_i^0\mathbf e_i
\in U_I,
$
define
\[
f_j(t)
\triangleq
\begin{cases}
\mathbf{x}_J^0+(x_1^0+t)\Delta\mathbf{x}_J^0,
& j=1,\\
\mathbf{x}_J^0+x_1^0\Delta\mathbf{x}_J^0,
& j\neq 1,
\end{cases}
\qquad
t\in U_{\mathbf{x}_I^0,j},
\]
where
\[
U_{\mathbf{x}_I^0,j}
\triangleq
\left\{
t\in\mathbb R:
\mathbf{x}_I^0+t\mathbf e_j\in U_I
\right\}.
\]
Since \(U_I=P_I\ell^2\), we have
\(U_{\mathbf{x}_I^0,j}=\mathbb R\).
For every \(j\in I\), since $\Delta\mathbf{x}_J^0\in H_{J}^2$, the function \(f_j\) is
\(H_{J}^2\)-Fr\'{e}chet differentiable at \(0\).
Moreover,
\begin{equation}\label{20260729for1-1}
\langle f(\cdot),\mathbf e_k\rangle
\in C_{\mathcal F}^2(U_I),
\qquad
k\in\mathbb N\setminus I.
\end{equation}
Thus, \(S\) is a smooth surface.

Set
$
I_1\triangleq I\setminus\{1\}.
$
Since
\[
\partial D
=\left\{
\mathbf{x}_{I_1}+\mathbf{x}_J^0:
\mathbf{x}_{I_1}
=\sum_{i\in I_1}x_i\mathbf e_i
\in P_{I_1}\ell^2
\right\},
\]
it follows that \(\partial D\) is a \(\sigma\)-finite surface and that
\[
\mu_{\partial D}(E)
=C\mu_{I_1}(P_{I_1}E),
\qquad
E\in\mathscr B(\partial D),
\]
where
$
C
\triangleq
F_{J\cup\{1\}}(\mathbf{x}_J^0)
\in(0,\infty).
$
In particular, \(\mu_{\partial D}\) is a totally finite measure.

For every \(\mathbf{I}\in S_{\mathbb N}^F\) satisfying
\(|\mathbf{I}|\in\Gamma_{I_1}\), by \eqref{20260720for1}, and the definition of $n_{(i,\mathbf{I})}^{D^{\circ}}$ in Definition \eqref{20260801def1}, we have
$$
n_{(i,\mathbf{I})}^{D^{\circ}}(\mathbf{x})
\triangleq
\frac{
s((i,\mathbf{I}))\cdot
\det\left(
D\bigl(P_{|\mathbf{I}|}  P_{I}^{-1}\bigr)
(P_{I}\mathbf{x})
\right)
}{
n_{I}(\mathbf{x})
},
$$
$$
n_{(i,\mathbf{I})}^{D^{\circ}}(\mathbf{x})
=
\begin{cases}
\frac{s((i,\mathbf{I}))}{\left(
1+\|\Delta\mathbf{x}_J^0\|^2
\right)^{1/2}},& |(i,\mathbf{I})|=I,\\
\frac{s((i,\mathbf{I}))\cdot (\Delta x_j^0)}{\left(
1+\|\Delta\mathbf{x}_J^0\|^2
\right)^{1/2}},
& |(i,\mathbf{I})|=(I\setminus \{1\})\cup \{j\}\, \text{for some }j\in J,\\
0,&\text{otherwise}.
\end{cases}
$$
Hence,
$$
\left(
\sum_{i=1}^{\infty}
\frac{
\left|
n_{(i,\mathbf{I})}^{D^\circ}(\mathbf{x})
\right|^2
}{a_i^2}
\right)^{1/2}
=
\begin{cases}
 \left(
\frac{1}{a_1^2\left(
1+\|\Delta\mathbf{x}_J^0\|^2
\right)}+\sum\limits_{j\in J}
\frac{|\Delta x_j^0|^2}{a_j^2\left(
1+\|\Delta\mathbf{x}_J^0\|^2
\right)}
\right)^{1/2},& |\mathbf{I}|=I_1,\\
\frac{|\Delta x_j^0|}{a_j\left(
1+\|\Delta\mathbf{x}_J^0\|^2
\right)^{1/2}},
& \card(|\mathbf{I}|\bigtriangleup I_1)=2,\\
0,&\text{otherwise}.
\end{cases}
$$

\end{example}

Let
\[
\mathcal{A}
\triangleq
\left\{
(U,P_{\mathbb N}): U\text{ is an open subset of }\ell^2
\right\}.
\]
Then \(\mathcal{A}\) defines a \(C^1\)-differentiable structure on \(\ell^2\), under which \(\ell^2\) is a smooth, \(\sigma\)-finite, oriented \(C^1\)-differentiable surface of codimension \(\Gamma_{\emptyset}\).

For any nonempty open subset \(V\subset\ell^2\), define
\[
\mathcal{A}_V
\triangleq
\left
\{
(U,P_{\mathbb N}): U\text{ is an open subset of }V
\right\}.
\]
Then \(\mathcal{A}_V\) defines a \(C^1\)-differentiable structure on \(V\), under which \(V\) is a smooth, \(\sigma\)-finite, oriented \(C^1\)-differentiable surface of codimension \(\Gamma_{\emptyset}\). By the construction in Proposition~\ref{20241013prop1}, the corresponding surface measures on \(\ell^2\) and \(V\) are, respectively,
\[
\mu_{\mathbb N}
\qquad\text{and}\qquad
\left.\mu_{\mathbb N}\right|_V.
\]
When no confusion can arise, we use the same notation \(\mu_{\mathbb N}\) for both measures.

Suppose that \(\overline V\), the topological closure of \(V\) in \(\ell^2\), is a region with boundary. By Definition~\ref{20241013def2}, \(\partial V\) is then a smooth, oriented \(C^1\)-differentiable surface of codimension \(\Gamma_{\{1\}}\), and it carries a surface measure \(\mu_{\partial V}\). In fact, \(\partial V\) is a classical \(C^1\)-hypersurface in \(\ell^2\). It follows from the preceding construction that \(\partial V\) is a \(\sigma\)-finite surface and that \(\mu_{\partial V}\) is a \(\sigma\)-finite measure.

In this special setting, the additional hypotheses in Theorem~\ref{20241015thm1} concerning the region-with-boundary structure, the \(\sigma\)-finiteness of the boundary surface, and the Fr\'{e}chet differentiability of the integrand in the \(H_{\mathbb N}^2\)-directions are automatically satisfied or redundant. These hypotheses are used only in \eqref{20250115for1} and \eqref{20250115for2}. Consequently, Proposition~\ref{20241013thm1} and Theorem~\ref{20241015thm1} yield the following simplified version of the Gauss--Green formula which has the same form as  \cite[Corollary 2.1, p. 424]{Goo}.

\begin{corollary}\label{20241127cor1}
Let \(V\) be a nonempty open subset of \(\ell^2\) such that \(\overline V\) is a region with boundary. Let \(\widetilde V\) be an open subset of \(\ell^2\) satisfying
$
\overline V\subset\widetilde V,
$
and let \(f:\widetilde V\to\mathbb R\). Fix \(i\in\mathbb N\). Suppose that
$
f\in C_{\mathcal F}^1(\widetilde V)\cap B_{\ell^2}(\widetilde V),
$
that
$
\{\left.D_{x_i}f\right|_V,\,
\left.(x_i f)\right|_V,\,
\left.f\right|_V\}
\subset L^1(V,\mu_{\mathbb N}),
$
and that
$
\left.f\right|_{\partial V}
\in L^1(\partial V,\mu_{\partial V}).
$
Then
\begin{equation}\label{20260803for1}
\int_V D_{x_i}f\,\mathrm d\mu_{\mathbb N}
- \int_V\frac{x_i f}{a_i^2}\,\mathrm d\mu_{\mathbb N}
= \int_{\partial V}
f\,\nu_{ i }^{\partial V}
\,\mathrm d\mu_{\partial V},
\end{equation}
where for each $\mathbf{x}\in\partial V$, if we adopt the notation after \eqref{20241015for1},
$$
\nu_{ i }^{\partial V}(\mathbf{x})\triangleq \frac{D_{x_i}\rho( \mathbf{x})}{\sqrt{\sum_{j=1}^{\infty}|D_{x_j}\rho( \mathbf{x})|^2}}.
$$
\end{corollary}
\begin{proof}
Choose ${\mathbf I}_i\in S_{\mathbb{N}}^{F}$ satisfying $|{\mathbf I}_i|=\mathbb{N}\setminus\{i\}$ and ${\mathbf I}_i=\overline{{\mathbf I}_i}$.

Consider $D= \overline V$, then $D^{\circ}=V$, by Proposition~\ref{20241013thm1} and Theorem~\ref{20241015thm1}, we obtain
\[
-\sum_{j=1}^{\infty}
\int_{V}
\delta_j f\,
n_{(j,\mathbf{I}_i)}^{V}\,
\mathrm d\mu_{\mathbb{N}}
=\int_{\partial V}
f\,n_{\mathbf{I}_i}^{\partial V}\,
\mathrm d\mu_{\partial V}.
\]
We adopt the notation after \eqref{20241015for1}, and by Definition \ref{20260801def1},
\begin{eqnarray*}
-\sum_{j=1}^{\infty}
\int_{V}
\delta_j f\,
n_{(j,\mathbf{I}_i)}^{V}\,
\mathrm d\mu_{\mathbb{N}}
&=&
-\int_{V}
\delta_i f\,
n_{(i,\mathbf{I}_i)}^{V}\,
\mathrm d\mu_{\mathbb{N}}\\
&=&-\int_{V}\left(D_{x_i}f
-  \frac{x_i f}{a_i^2}\right)\frac{
s((i,\mathbf{I}_i))\cdot
\det\left(
D\bigl(P_{|(i,\mathbf{I}_i)|}  P_{\mathbb{N}}^{-1}\bigr)
(P_{\mathbb{N}}\mathbf{x})
\right)
}{
n_{\mathbb{N}}(\mathbf{x})
} \,\mathrm d\mu_{\mathbb N}\\
 &=&(-1)^i\int_{V}\left(D_{x_i}f
-\frac{x_i f}{a_i^2}\right)
 \,\mathrm d\mu_{\mathbb N}.
\end{eqnarray*}
By Definition \ref{20260803def1}, for every $\mathbf{x}\in\partial V$,
\begin{eqnarray*}
 n_{{\mathbf I}_i }^{\partial V}(\mathbf{x})
 &=&  \frac{
s(\mathbf{I}_i)\cdot s(\mathbf{K}\setminus\{k_{i_0}\})\cdot
\det\left(
D\bigl(P_{|\mathbf{I}|}  P_{|\mathbf{K}|\setminus\{k_{i_0}\}}^{-1}\bigr)
(P_{|\mathbf{K}|\setminus\{k_{i_0}\}}\mathbf{x})
\right)
}{
n_{|\mathbf{K}|\setminus\{k_{i_0}\}}(\mathbf{x})
} \\
&=& \frac{
  s(\mathbf{K}\setminus\{k_{i_0}\})\cdot
\det\left(
D\bigl(P_{\mathbb{N}\setminus\{i\}}  P_{ \mathbb{N} \setminus\{k_{i_0}\}}^{-1}\bigr)
(P_{\mathbb{N}\setminus\{k_{i_0}\}}\mathbf{x})
\right)
}{
n_{|\mathbb{N}|\setminus\{k_{i_0}\}}(\mathbf{x})
} \\
&=&  \frac{
  s(\mathbf{K}\setminus\{k_{i_0}\})\cdot
\det\left(
D\bigl(P_{\mathbb{N}\setminus\{i\}}  P_{ \mathbb{N} \setminus\{k_{i_0}\}}^{-1}\bigr)
(P_{\mathbb{N}\setminus\{k_{i_0}\}}\mathbf{x})
\right)
}{
\sqrt{\sum_{j=1}^{\infty}|D_{x_j}\rho(\mathbf{x})|^2}
}.
\end{eqnarray*}
If $i=k_{i_0}$, then
$$
\det\left(
D\bigl(P_{\mathbb{N}\setminus\{i\}}  P_{ \mathbb{N} \setminus\{k_{i_0}\}}^{-1}\bigr)(P_{|\mathbf{K}|\setminus\{k_{i_0}\}}\mathbf{x})\right)
=1=(-1)^{i_0-1}D_{  x_{k_{i_0}}}\rho(\mathbf{x}) =(-1)^{i+k_{i_0}+i_0-1}D_{  x_{i}}\rho(\mathbf{x}).
$$
If $i\neq k_{i_0}$, then by Remark \ref{20250505rem2} and simple computation, we obtain
$$
\det\left(
D\bigl(P_{\mathbb{N}\setminus\{i\}}  P_{ \mathbb{N} \setminus\{k_{i_0}\}}^{-1}\bigr)(P_{|\mathbf{K}|\setminus\{k_{i_0}\}}\mathbf{x})\right)
=(-1)^{i+k_{i_0}-1}  D_{x_i}h (P_{|\mathbf{K}|\setminus\{k_{i_0}\}}\mathbf{x})
=(-1)^{i+k_{i_0}+i_0-1} D_{  x_i}\rho(\mathbf{x}) .
$$
Since $s(\mathbf{K})=1$, $|\mathbf{K}|=\mathbb{N}$, and $
s(\mathbf{K})=(-1)^{i_0}\cdot s(\mathbf{K}\setminus\{k_{i_{0}}\})\cdot (-1)^{k_{i_0} -1},
$ it follows that
$$
s(\mathbf{K}\setminus\{k_{i_{0}}\})=(-1)^{k_{i_0}+i_0-1}.
$$
Hence,
\begin{eqnarray*}
\frac{
  s(\mathbf{K}\setminus\{k_{i_0}\})\cdot
\det\left(
D\bigl(P_{\mathbb{N}\setminus\{i\}}  P_{ \mathbb{N} \setminus\{k_{i_0}\}}^{-1}\bigr)
(P_{\mathbb{N}\setminus\{k_{i_0}\}}\mathbf{x})
\right)
}{
\sqrt{\sum_{j=1}^{\infty}|D_{x_j}\rho(\mathbf{x})|^2}
}=(-1)^i\frac{
D_{x_i}\rho(\mathbf{x})}{
\sqrt{\sum_{j=1}^{\infty}|D_{x_j}\rho(\mathbf{x})|^2}
}
=(-1)^i\nu_{ i }^{\partial V}(\mathbf{x}).
\end{eqnarray*}
Combining above, we arrive at \eqref{20260803for1}. This completes the proof of Corollary \ref{20241127cor1}.
\end{proof}
\begin{remark}
The regularity assumption
\[
f\in C_{\mathcal F}^1(\widetilde V)
\]
in Corollary~\ref{20241127cor1} may be weakened to include a more general class of functions. In some applications, the open set \(V\) is given by the intersection of the interiors of finitely many regions with boundary. Gauss--Green-type formulas adapted to such domains are therefore needed. We plan to investigate these problems in future work.
\end{remark}

\section{Top Differential Forms and a Stokes-Type Identity}

In this section, we introduce an analog of top-degree forms and establish an integral Stokes-type identity on surfaces of \(\ell^2\) with arbitrary codimension.

\begin{definition}\label{20260809def1}
Let \(S\) be an oriented \(C^1\)-differentiable surface of \(\ell^2\) with codimension
\(\Gamma_{\mathbb{N}\setminus I_0}\). Define
$$
I_S\triangleq \{\mathbf{I} \in S_{\mathbb N}^F:|\mathbf{I}|\in\Gamma_{I_0},\overline{\mathbf{I}}=\mathbf{I}\},
$$
and
$$
T_S\triangleq \left\{\mathbf{f}=(f_{\mathbf{I}})_{\mathbf{I}\in I_S} :f_{\mathbf{I}}\in L^1(S,\mu_S)\,\text{ for all }\mathbf{I}\in I_S, \sum_{\mathbf{I}\in I_S}|f_{\mathbf{I}}|\in L^1(S,\mu_S) \right\}.
$$
We call each element in $T_S$ a \textbf{top differential form} on $S$. Moreover, for every $\mathbf{f}=(f_{\mathbf{I}})_{\mathbf{I}\in I_S}\in T_S$, we define
$$
\int_S \mathbf{f}\triangleq \sum_{\mathbf{I}\in I_S}\int_S f_{\mathbf{I}} n_{\mathbf{I}}^S \,\mathrm{d}\mu_S.
$$
\end{definition}
\begin{definition}
Let \(D\) be a region with boundary in a smooth surface \(S\), and
suppose that \(\partial D\) is \(\sigma\)-finite. Let \(O\) be an open subset of
\(\ell^2\) containing \(D\). Let $\widetilde{T}_{D}$ denote the collection of all $\mathbf{f}=(f_{\mathbf{I}})_{\mathbf{I}\in I_{\partial D}}$ such that
\begin{itemize}
\item[$(1)$] for every $\mathbf{I}\in I_{\partial D}$, $f_{\mathbf{I}}$ satisfies the assumption in Theorem \ref{20241015thm1} for $f$;
\item[$(2)$]  $\sum\limits_{\mathbf{I}\in I_{\partial D}}  |f_{\mathbf{I}}\mid_{\partial D}|\in L^1(\partial D,\mu_{\partial D})$ and
$\sum\limits_{i\in\mathbb{N},\mathbf{I}\in I_{\partial D}}  |\delta_i f_{\mathbf{I}}\mid_{ D^{\circ}}|\in L^1( D^{\circ},\mu_{D^{\circ}})$.
\end{itemize}
We call each element in $\widetilde{T}_{D}$ a \textbf{sub-top differential form} on $D$.
Moreover,
for every  $\mathbf{f}=(f_{\mathbf{I}})_{\mathbf{I}\in I_{\partial D}}\in \widetilde{T}_{D}$, define
$$
i^*\mathbf{f}\triangleq (f_{\mathbf{I}}\mid_{\partial D})_{\mathbf{I}\in I_{\partial D}},\qquad\text{and}\qquad
\delta \mathbf{f}\triangleq\left(-\left.\sum\limits_{\begin{gathered} i\in\mathbb{N},\mathbf{I}\in I_{\partial D},\\ \overline{(i,\mathbf{I})} =\mathbf{J} \end{gathered}}s((i,\mathbf{I}))\delta_i f_{\mathbf{I}}\right|_{D^{\circ}}\right)_{\mathbf{J}\in I_{ D^{\circ}}},
$$
which are top differential forms on $\partial D$ and $D^{\circ}$, respectively.
\end{definition}
As an immediate consequence of Theorem \ref{20241015thm1}, we obtain the following Stokes-type identity.
\begin{theorem}[Stokes-type theorem]\label{20260809thm1}
For every  $\mathbf{f}=(f_{\mathbf{I}})_{\mathbf{I}\in I_{\partial D}}\in \widetilde{T}_{D}$,
$$
\int_{D^{\circ}}\delta f=\int_{\partial D }i^* f.
$$
\end{theorem}
\begin{proof}
For $\mathbf{I}\in I_{\partial D},\mathbf{J}\in I_{ D^{\circ}}$ and $i\in\mathbb{N}$, if $\overline{(i,\mathbf{I})} =\mathbf{J}$, then
$$
n_{(i,\mathbf{I})}^{D^{\circ}}=s((i,\mathbf{I}))n_{\mathbf J}^{D^{\circ}}.
$$
Therefore,
\begin{eqnarray*}
\int_{D^{\circ}}\delta f
&=&-\sum_{\mathbf{J}\in I_{ D^{\circ}}}\sum\limits_{\begin{gathered} i\in\mathbb{N},\mathbf{I}\in I_{\partial D},\\ \overline{(i,\mathbf{I})} =\mathbf{J} \end{gathered}}\int_{D^{\circ}}s((i,\mathbf{I}))\,\delta_i f_{\mathbf{I}}\,n_{\mathbf J}^{D^{\circ}}\,\mathrm{d}\mu_{D^{\circ}}\\
&=&-\sum_{ \mathbf{I}\in I_{\partial D}}\sum\limits_{ i\in\mathbb{N} }\int_{D^{\circ}} \delta_i f_{\mathbf{I}}\,n_{(i,\mathbf{I})}^{D^{\circ}}\,\mathrm{d}\mu_{D^{\circ}}\\
&=& \sum_{ \mathbf{I}\in I_{\partial D}}\int_{\partial D}
f_{\mathbf{I}}\,n_{\mathbf{I}}^{\partial D}\,
\mathrm d\mu_{\partial D}\\
&=& \int_{\partial D }i^* f,
\end{eqnarray*}
where the first and the last equalities are from Definition \ref{20260809def1}, and the third equality follows from Theorem \ref{20241015thm1} for each $\mathbf{I}\in I_{\partial D}$. This completes the proof of Theorem \ref{20260809thm1}.
\end{proof}

\end{document}